\documentclass[a4paper,reqno]{amsart}

\usepackage[T1]{fontenc}
\usepackage[utf8x]{inputenc}
\usepackage[english]{babel}
\usepackage{yfonts}
\usepackage{dsfont}
\usepackage{amscd,amssymb,amsmath,amsthm,amsfonts}
\usepackage{mathtools}
\usepackage{accents}
\usepackage{tensor}
\usepackage{graphicx}
\usepackage{tikz}
\usetikzlibrary{calc,matrix,arrows,decorations.pathmorphing}
\usepackage{calc}
\usepackage[cal=boondox,scr=boondoxo]{mathalfa}
\usepackage{enumitem}
\usepackage{marginnote}
\usepackage{booktabs}
\usepackage{scalerel}[2016/12/29]
\usepackage{url}
\usepackage{contour}
\usepackage[normalem]{ulem}

\usepackage{hyperref}
\hypersetup{colorlinks=true,pageanchor=false,
linkcolor=blue,citecolor=red,urlcolor=red}
\usepackage[all]{hypcap}

\allowdisplaybreaks

\calclayout

\newtheorem{theorem}{Theorem}[section]
\newtheorem{corollary}[theorem]{Corollary}
\newtheorem{proposition}[theorem]{Proposition}
\newtheorem{lemma}[theorem]{Lemma}

\theoremstyle{definition}
\newtheorem{definition}[theorem]{Definition}

\theoremstyle{remark}
\newtheorem{remark}[theorem]{Remark}

\newtheorem{conjecture}[theorem]{Conjecture}

\newtheorem{question}[theorem]{Question}
\newtheorem{theoremark}[theorem]{Theorem}

\newcommand{\N}{\mathbb{N}}
\newcommand{\Z}{\mathbb{Z}}
\newcommand{\Q}{\mathbb{Q}}
\newcommand{\R}{\mathbb{R}}

\newcommand{\rmC}{\mathrm{C}}

\newcommand{\rmH}{\mathrm{H}}

\newcommand{\rmT}{\mathrm{T}}

\newcommand{\rmV}{\mathrm{V}}

\newcommand{\bbC}{\mathbb{C}}

\newcommand{\bbH}{\mathbb{H}}

\newcommand{\bbM}{\mathbb{M}}

\newcommand{\bbV}{\mathbb{V}}

\DeclareRobustCommand{\bbSigma}{\mathbin{\text{\includegraphics[height=\heightof{$\mathbf{\Sigma}$}]{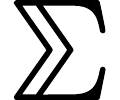}}}}

\newcommand{\bfzero}{\mathbf{0}}
\newcommand{\bfa}{\boldsymbol{a}}
\newcommand{\bfb}{\boldsymbol{b}}
\newcommand{\bfc}{\boldsymbol{c}}

\newcommand{\bfe}{\boldsymbol{e}}

\newcommand{\bfk}{\boldsymbol{k}}

\newcommand{\bfell}{\boldsymbol{\ell}}
\newcommand{\bfm}{\boldsymbol{m}}
\newcommand{\bfn}{\boldsymbol{n}}

\newcommand{\bfv}{\boldsymbol{v}}

\newcommand{\bfalpha}{\boldsymbol{\alpha}}
\newcommand{\bfbeta}{\boldsymbol{\beta}}

\newcommand{\bfGamma}{\boldsymbol{\Gamma}}
\newcommand{\bfDelta}{\boldsymbol{\Delta}}

\newcommand{\calB}{\mathcal{B}}

\newcommand{\calE}{\mathcal{E}}
\newcommand{\calF}{\mathcal{F}}

\newcommand{\calH}{\mathcal{H}}
\newcommand{\calI}{\mathcal{I}}
\newcommand{\calK}{\mathcal{K}}
\newcommand{\calL}{\mathcal{L}}
\newcommand{\calM}{\mathcal{M}}

\newcommand{\calX}{\mathcal{X}}

\newcommand{\frakg}{\mathfrak{g}}

\newcommand{\fraki}{\mathfrak{i}}

\newcommand{\fraku}{\mathfrak{u}}

\newcommand{\frakS}{\mathfrak{S}}

\renewcommand{\epsilon}{\varepsilon}
\renewcommand{\theta}{\vartheta}
\renewcommand{\phi}{\varphi}
\renewcommand{\Gamma}{\varGamma}
\renewcommand{\Sigma}{\varSigma}

\newcommand{\ad}{\mathrm{ad}}
\newcommand{\coad}{\mathrm{coad}}
\newcommand{\id}{\mathrm{id}}

\newcommand{\sgn}{\operatorname{sgn}}

\DeclareMathOperator{\Exists}{\exists}
\DeclareMathOperator{\Forall}{\forall}

\DeclareMathOperator{\bcs}{\natural}

\newcommand{\leqs}{\leqslant}
\newcommand{\geqs}{\geqslant}

\newcommand{\mods}[1]{\operatorname{\mathnormal{#1}-mod}}

\newcommand{\fsl}{\mathfrak{sl}}
\newcommand{\Mod}{\mathrm{Mod}}
\newcommand{\Diff}{\mathrm{Diff}}
\newcommand{\BM}{\mathrm{BM}}

\newcommand{\SL}{\mathrm{SL}}
\newcommand{\GL}{\mathrm{GL}}
\newcommand{\PGL}{\mathrm{PGL}}
\newcommand{\SU}{\mathrm{SU}}
\newcommand{\Aut}{\mathrm{Aut}}
\newcommand{\Hom}{\mathrm{Hom}}

\newcommand{\End}{\mathrm{End}}

\newcommand{\Vect}{\mathrm{Vect}}

\newcommand{\op}{\mathrm{op}}

\DeclareRobustCommand{\one}{\mathbin{\text{\includegraphics[height=\heightof{$\mathbf{1}$}]{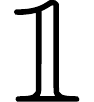}}}}

\newcommand{\Conf}{\mathrm{Conf}}
\newcommand{\sig}{n}

\newcommand{\RCob}{3\mathrm{Cob}^\natural}

\newcommand{\adCob}{3\check{\mathrm{C}}\mathrm{ob}^\sqcup}

\newcommand{\sqbinom}[2]{\left[ \begin{matrix} #1 \\ #2 \end{matrix} \right]}

\makeatletter
\newcommand{\subalign}[1]{
  \vcenter{
    \Let@ \restore@math@cr \default@tag
    \baselineskip\fontdimen10 \scriptfont\tw@
    \advance\baselineskip\fontdimen12 \scriptfont\tw@
    \lineskip\thr@@\fontdimen8 \scriptfont\thr@@
    \lineskiplimit\lineskip
    \ialign{\hfil$\m@th\scriptstyle##$&$\m@th\scriptstyle{}##$\crcr
      #1\crcr
    }
  }
}
\makeatother

\def\clap#1{\hbox to 0pt{\hss#1\hss}}

\def\mathclap{\mathpalette\mathclapinternal}

\def\mathclapinternal#1#2{%
\clap{$\mathsurround=0pt#1{#2}$}}

\newcommand{\Heisenberg}{\bbH}
\newcommand{\dHeis}{\bbH}

\newcommand{\srs}{q}
\newcommand{\homol}{\calH}
\newcommand{\simp}{\Gamma}
\newcommand{\bfsimp}{\bfGamma}

\newcommand{\pic}[2][0]{\raisebox{-0.5\height + 2.5pt + #1pt}{\includegraphics{#2.pdf}}}

\newcommand\arxiv[2]{\href{https://arXiv.org/abs/#1}{\texttt{arXiv:\allowbreak #1} #2}}
\newcommand\doi[2]{\href{https://doi.org/#1}{#2}}

\contourlength{0.625pt}

\DeclareRobustCommand{\myuline}[1]{
 \ifmmode \text{\uline{$\phantom{#1}$}\llap{\contour{white}{$#1$}}}
 \else \uline{\phantom{#1}}\llap{\contour{white}{#1}} \fi
}

\makeatletter
\def\namedlabel#1#2{\begingroup
    #2%
    \def\@currentlabel{#2}%
    \phantomsection\label{#1}\endgroup
}
\makeatother

\newcommand{\braid}[2]{\myuline{#1_{#2}}}
\newcommand{\twist}[2]{\tau_{#1_{#2}}}

\newcommand{\slt}{{\mathfrak{sl}_2}}

\newcommand{\homoldH}{\homol^\dHeis}

\newcommand{\thread}{\tilde{\boldsymbol{x}}}
\newcommand{\thrcmp}{\tilde{x}}
\newcommand{\perm}{\vartheta}
\newcommand{\chain}{\sigma}
\newcommand{\bloop}{\gamma}
\newcommand{\formv}{\mathbb{X}}
\newcommand{\cohom}{\omega}
\newcommand{\fld}{\bbC}

\newcommand{\basis}{\tilde{\bfsimp}}

\newcommand{\add}{\operatorname{add}}
\newcommand{\del}{\operatorname{del}}

\newcommand{\Ceq}{\stackrel{\mathclap{\eqref{E:cutting}}}{=}}
\newcommand{\Feq}{\stackrel{\mathclap{\eqref{E:fusion}}}{=}}

\newcommand{\Peq}{\stackrel{\mathclap{\eqref{E:permutation}}}{=}}
\newcommand{\Beq}{\stackrel{\mathclap{\eqref{E:braid}}}{=}}

\newcommand{\tCeq}{\stackrel{\mathclap{\eqref{E:tilde_cutting}}}{=}}

\begin{document}

\raggedbottom

\title{Homological Construction of Quantum Representations of Mapping Class Groups}

\author[M. De Renzi]{Marco De Renzi} 
\address{Institute of Mathematics, University of Zurich, Winterthurerstrasse 190, CH-8057 Zurich, Switzerland} 
\email{marco.derenzi@math.uzh.ch}

\author[J. Martel]{Jules Martel} 
\address{Institute of Mathematics, University of Zurich, Winterthurerstrasse 190, CH-8057 Zurich, Switzerland} 
\email{jules.martel-tordjman@math.uzh.ch}

\begin{abstract}
 We provide a homological model for a family of quantum representations of mapping class groups arising from non-semisimple TQFTs (Topological Quantum Field Theories). Our approach gives a new geometric point of view on these representations, and it gathers into one theory two of the most promising constructions for investigating linearity of mapping class groups. More precisely, if $\varSigma_{g,1}$ is a surface of genus $g$ with $1$ boundary component, we consider a (crossed) action of its mapping class group $\Mod(\varSigma_{g,1})$ on the homology of its configuration space $\Conf_n(\varSigma_{g,1})$ with twisted coefficients in the Heisenberg quotient $\dHeis_g$ of its surface braid group $\pi_1(\Conf_n(\varSigma_{g,1}))$. We show that this action intertwines an action of the quantum group of $\fsl_2$, that we define by purely homological means. For a finite-dimensional linear representation of $\dHeis_g$ (depending on a root of unity $\zeta$), we tweak the construction to obtain a projective representation of $\Mod(\varSigma_{g,1})$. Finally, we identify, by an explicit isomorphism, a subrepresentation of $\Mod(\varSigma_{g,1})$ that is equivalent to the quantum representation arising from the non-semi\-simple TQFT associated with quantum $\fsl_2$ at $\zeta$. In the process, we provide concrete bases and explicit formulas for the actions of all the standard generators of $\Mod(\varSigma_{g,1})$ and of quantum $\fsl_2$ on both sides of the equivalence, and answer a question by Crivelli, Felder, and Wieczerkowski. We also make sure that the restriction of these representations to the Torelli group $\calI(\varSigma_{g,1})$ are integral, in the sense that the actions have coefficients in the ring of cyclotomic integers $\Z[\zeta]$, when expressed in these bases.
\end{abstract}

\maketitle
\setcounter{tocdepth}{3}

\section{Introduction}\label{S:intro}

Let $\varSigma$ be a non-ex\-cep\-tion\-al surface\footnote{A surface is exceptional if it admits self-dif\-fe\-o\-mor\-phisms that are homotopic but not isotopic, see for instance \cite[Section~1.4]{FM12}}. Its \textit{mapping class group} $\Mod(\varSigma)$ is the group of self-dif\-fe\-o\-mor\-phisms of $\varSigma$, considered up to isotopy. When $\varSigma$ has boundary, we require self-diffeomorphisms and isotopies to restrict to the identity on it. A long-standing open problem concerning mapping class groups is the following.

\begin{question}\label{Q:The_Question}
 Is $\Mod(\varSigma)$ linear?
\end{question}

This question already appears in the appendix to Birman's book \cite[Problem~30]{B74}, as well as in Margalit's recent paper \cite[Question~1.1]{Ma18}, which collects problems, questions, and conjectures about mapping class groups. In order to establish linearity, one would need to look for a faithful finite-dimensional representation, that is, for an injective homomorphism from $\Mod(\varSigma)$ to a matrix group $\GL_\Bbbk(V)$, where $V$ is a finite-dimensional vector space over a field $\Bbbk$.

Question~\ref{Q:The_Question} can be justified by the observation that mapping class groups enjoy several properties that are typical byproducts of linearity (for instance, they are residually finite, see \cite[Theorem~6.11]{FM12}). However, expecting a positive answer is legitimate only if we have good candidates, namely if we are able to construct linear representations that have reasonable chances of being faithful. This is the general motivation for the present work, which investigates two (families of) representations of a very different kind (or at least so it would seem): twisted homological representations based on configuration spaces of surfaces (referred to as \textit{homological representations} here), and quantum representations arising from non-semisimple TQFTs\footnote{Short for Topological Quantum Field Theories.} (referred to as \textit{quantum representations} in the following). These provide, to the best of our knowledge, the two most important classes of candidates, in the current state-of-the-art.

\subsection{Homological representations}

A natural approach to the construction of mapping class group representations consists in exploiting the fact that $\Mod(\varSigma)$ acts on the homology of $\varSigma$. However, the kernel of this action is the \textit{Torelli group} of $\varSigma$, which is known to be a big part of $\Mod(\varSigma)$, when $\varSigma$ is not a torus. It contains, for instance, all separating Dehn twists, as well as all bounding pair maps.

The first example of a homological representation other than the standard one is probably the \textit{Burau representation} of the braid group $\calB_m$ on $m$ strands, which can be defined as the mapping class group of the disc with $m$ punctures $D^2_m$. This is a special case of the \textit{Magnus representation} of $\Mod(\varSigma)$, for a surface $\varSigma$ with one boundary component \cite{M39}. The idea consists in understanding $\Mod(\varSigma)$ as the group of automorphisms of the fundamental group of $\varSigma$ (thanks to the Dehn--Nielsen--Baer Theorem, see \cite[Theorem~8.1]{FM12}), and to use Fox's differential calculus for automorphisms of $\pi_1(\varSigma)$, which is a free group in this case. It turns out that the resulting mapping class group action is equivalent to a lifted action on the homology of the maximal abelian cover of $\varSigma$ \cite{S05a}, so it provides a first step towards twisted homological representations. The Burau representation of $\calB_m$ is known to be faithful for $m \in \{ 2,3 \}$, but not for $m \geqs 5$ \cite{B99}, with the case $m=4$ remaining very mysterious. For a positive genus surface $\varSigma$ with one boundary component, the Magnus representation is not faithful. Its kernel is not negligible, and has been well characterized \cite{S02,S05b}. Roughly speaking, it is parametrized by pairs of simple closed curves in $\varSigma$ with symplectic pairing $0$. Magnus representations enjoy \textit{symplectic transvections formulas} (see for instance \cite[Theorem~4.2]{S05b}), inherited from their homological definition, that can be generalized to the representations considered in this paper, as explained in \cite{DM22}. These formulas imply that the action of separating Dehn twists on twisted homologies has infinite order, which is not a common property. The same will apply also to the representations appearing in this paper.

Another idea is to consider the homology of the \textit{configuration space of $n$ points in $\varSigma$}, denoted $\Conf_n(\varSigma)$, rather than the homology of $\varSigma$ itself. This results in a family of representations parametrized by the integer $n \geqs 1$. For general surfaces $\varSigma$ (with one boundary component), this idea first appeared in the work of Moriyama \cite{Mo07}, in the case of ordered configurations. The kernel of the $n$th representation in this family is the $n$th term in the \textit{Johnson filtration} of the mapping class group, which is never trivial. Still, since the total intersection of all the terms in the Johnson filtration is trivial, \textit{Moriyama representations} are said to be \textit{asymptotically faithful}. Indeed, all representations based on the homology of configuration spaces of surfaces inherit this property, as do the ones of the present work.

In the case of punctured discs $D^2_m$ and braid groups $\calB_m$, the idea of considering homology of configuration spaces had already appeared before Moriyama's work. Indeed, in \cite{L90}, Lawrence introduced an action of $\calB_m$ on a twisted version (in the same spirit of Magnus representations) of the homology of $\Conf_n(D^2_m)$, that is, using maximal abelian covers. The result is again a family of representations parametrized by the integer $n \geqs 1$ that recovers, for $n=1$, the Burau representation (which is not globally faithful). Nevertheless, \textit{Lawrence representations} rose to fame ten years after their construction, thanks to the following spectacular result, independently proven by Bigelow and Krammer.

\begin{theoremark}[{Bigelow \cite[Theorem~1.1]{B00}, Krammer \cite[Theorem~B]{K02}}]
 For $n=2$, and for all $m \geqs 0$, the Lawrence representation of $\calB_m$ is faithful. In particular, braid groups are linear.
\end{theoremark}

Thus, braid groups form the first infinite family of positive answers to Question~\ref{Q:The_Question}. This substantiates the claim that representations based on the twisted homology of configuration spaces of surfaces with one boundary component provide very good candidates. Indeed, Bigelow's proof uses extensively the twisted homological nature of these representations, by showing that a natural twisted homological pairing (in the maximal abelian cover of $\Conf_2(D^2_m)$) detects geometrical intersection between arcs in $D^2_m$. Using the same setup, it was later shown that mapping class groups of punctured spheres and of the closed genus $2$ surface are also linear \cite{BB01}. Bigelow even computed matrices for braid generators that were used by Krammer for his proof, together with a Garside structure that is specific to braid groups. In this direction, actions of mapping class group generators on our representations are also computed in this paper.

One can imagine extending Bigelow's result to positive genus surfaces by studying representations of $\Mod(\varSigma)$ on the twisted homology of $\Conf_n(\varSigma)$. This provides the framework for the homological representations studied in this paper. However, for $\varSigma$ of positive genus, twisting to the level of the maximal abelian cover is no longer sufficient. Indeed, Lawrence representations are never faithful, see for instance the footnote to \cite[Question~1.2]{Ma18}, and some of Suzuki's kernel elements might also persist. It should be noted that Lawrence representations are not even defined on the whole mapping class group of $\varSigma$, since twisted homology is not functorial with respect to arbitrary diffeomorphisms. Here, we will consider projective representations of $\Mod(\varSigma)$ on the homology of $\Conf_n(\varSigma)$ twisted by non-abelian representations of the surface braid group $\pi_1(\Conf_n(\varSigma))$ that factor through its \textit{Heisenberg} quotient. The study of mapping class group actions at the level of the Heisenberg cover of $\Conf_n(\varSigma)$ was first explored in \cite{BPS21}. This is a setup where Bigelow's strategy has a chance to work, as will be investigated in \cite{DM22}.

At a higher categorical level, faithful representations of mapping class groups have been constructed in \cite{LOT10}. These representations are a categorification of the ones on the standard homology of the surface, and their definition is given in terms of bordered Floer homology. Therefore, the approach is geometric, and is based on arcs inside surfaces, in the same spirit of the homological calculus developed in this paper. The faithfulness of these categorical actions can be seen as a promising sign for the question of linearity of twisted homological representations.

\subsection{Quantum representations}

A completely different strategy for obtaining mapping class group representations comes from quantum topology, where these are typically one of the byproducts of a deeper construction. Indeed, TQFTs can be thought of as linear representations of categories of cobordisms, and they naturally contain quantum invariants of knots, links, and $3$-manifolds, as well as quantum representations of mapping class groups of surfaces. Generally speaking, TQFTs provide a highly organized categorical framework for understanding these different invariants as coherent layers of a unified theory, and constitute tools for performing functorial computations based on cut-and-paste methods.

TQFTs are typically constructed from algebraic ingredients that come with their own fully developed theory, and in particular with plenty of concrete examples. Such algebraic tools are usually provided by quantum groups and Hopf algebras, by their representation theory, and more generally by abstract notions modeled on the properties of these algebraic structures, such as modular categories. Therefore, another relevant feature of quantum topology is that its constructions are typically very general, and can be applied in a wide range of settings.

The first family of quantum representations of mapping class groups to appear in the literature is due to Reshetikhin and Turaev \cite[Section~4.6]{RT91}, and it made big waves. Although the construction can be performed starting from any modular fusion category, as explained in \cite[Chapter~IV, Section~4]{T94}, the best studied case is the one corresponding to the semisimple quotient of the category of representations of the \textit{small} quantum group of $\fsl_2$ at a root of unity of order $r \geqs 3$. It can be easily seen that these mapping class group representations send every Dehn twist to a matrix of order at most $r$, so in particular they are never faithful. However, they are \textit{asymptotically faithful}, in the sense that, for a fixed surface $\varSigma$, the intersection of the kernels of these representations for all $r \geqs 3$ is trivial (modulo the center, which is non-trivial only for the genus $2$ surface, in which case its only non-trivial element is the hyperelliptic involution). This was first proved by Andersen \cite[Theorem~1]{A02}, using gauge theoretic methods, and shortly afterwards by Freedman, Walker, and Wang \cite[Theorem~1.1]{FWW02}, using the skein theoretic approach of Lickorish \cite{L93} and Blanchet, Habegger, Masbaum, and Vogel \cite{BHMV95}. In particular, the main ingredient of the second proof is the fact that the mapping class group $\Mod(\varSigma)$ acts faithfully on the curve complex of $\varSigma$. This kind of argument seems to be one of the best suited for studying faithfulness of quantum representations. As we will see, our diagrammatic approach to computations naturally features the action of $\Mod(\varSigma)$ on the curve complex of $\varSigma$.

Although small quantum groups at roots of unity provide modular fusion categories that can be used to obtain quantum representations of mapping class groups,  following Reshetikhin and Turaev, their categories of representations are never semisimple. Therefore, the first step in the original construction is always a quotient operation that sacrifices a great deal of algebraic information. Lyubashenko and Majid were the first to come up with an alternative construction that did not require semisimplicity, see \cite[Theorem~1.1]{LM94} for the torus, and \cite[Section~4]{L94} for arbitrary genus surfaces. In the case of $\fsl_2$, their quantum representations were studied by Kerler, who already showed faithfulness for the torus at all values of $r$, see \cite[Conjecture~1]{K94} and \cite[Theorem~15]{K96}. What is clear from this result is that the shift to the non-semisimple setting pays off, since no single quantum representation had any hope of being faithful, in the semisimple case of Reshetikhin and Turaev.

This indication was then reinforced by \cite[Proposition~5.1]{DGGPR20}, where it was shown that, again in the case of small quantum $\fsl_2$, and for all values of $r$, Lyubashenko's representations send every essential Dehn twist (both separating and non-separating ones) to a matrix of infinite order. This kind of phenomenon was actually not new to non-semisimple quantum topology. Indeed, it was first observed by Blanchet, Costantino, Geer, and Patureau, see \cite[Theorems~1.3 \& 1.4]{BCGP14}, in the context of a more sophisticated family of quantum representations arising from the TQFTs constructed out of the \textit{unrolled} quantum group of $\fsl_2$. The hope that these non-semisimple TQFTs would yield faithful mapping class group representations had indeed even been made explicit in \cite[Questions~1.5 \& 1.7]{BCGP14}. An advantage of this construction is that the corresponding quantum invariants depend holomorphically on a family of complex variables parametrizing the choice of a cohomology class with coefficients in $\bbC/2\Z$, see \cite[Section~3.5]{CGP12}. This property has been used effectively to show faithfulness of the resulting quantum representations of the mapping class group of the sphere with four punctures \cite[Theorem~1]{M20b}, and to establish irreducibility of the corresponding quantum representations of Kauffman bracket skein algebras of closed surfaces \cite[Theorem~1.3]{KK22}. As we will explain, our construction can be deformed using a set variables that can be interpreted as evaluations of a cohomology class, and the resulting representations depend polynomially on these parameters. This means our framework is also well-suited for applying these kinds of arguments.

Despite these promising features of non-semisimple TQFTs, the original semisimple construction of Reshetikhin and Turaev remains the most studied to date. Indeed, the interest in semisimple quantum representations extends far beyond linearity questions. For instance, the Andersen--Masbaum--Ueno Conjecture \cite[Question~1.1]{AMU05} predicts that the collection of all the semisimple quantum representations of $\Mod(\varSigma)$ arising from quantum $\fsl_2$ detects the dynamical nature of diffeomorphisms of $\varSigma$ with respect to the Nielsen--Thurston classification, and even measures the stretching factor of pseudo-Anosov maps. A homological approach like ours could be useful for recovering such dynamical information from non-semisimple TQFTs, as suggested by the results of \cite{BB06}.

Furthermore, a major strength of semisimple quantum representations and TQFTs is provided by their \textit{integrality} properties. Indeed, the approach of Reshetikhin and Turaev can be refined to obtain TQFTs with coefficients in the ring of cyclotomic integers associated with the root of unity underlying the construction, see \cite{G01,GM04,BCL10}. For quantum representations of mapping class groups, this has important consequences: in \cite{KS15}, the property is used to answer a question by Looijenga on generators of homologies of finite covers of surfaces, while in \cite{GM16}, it is used to study in detail some representations of symplectic groups over finite fields. In this paper, we recover non-semisimple quantum representations from homology, and in the process we establish their integrality for Torelli subgroups. Indeed, by their own very nature, homological bases pinpoint the precise spot where actions are integral, since only integral coefficients are allowed at any given stage of the homological construction. These integrality properties were not known to hold in the non-semisimple framework, at least to the best of our knowledge.

\subsection{Isomorphism in low genus}

Even though Question~\ref{Q:The_Question} is generally stated in the case of closed surfaces, this paper will focus on surfaces with one boundary component. Notice that there are explicit relations between the two families of mapping class groups (see for instance the capping homomorphism of \cite[Proposition~3.19]{FM12} and the Birman exact sequence of \cite[Theorem~4.6]{FM12}). Although we will not study them here, these relations suggest that tackling the question of linearity for surfaces with boundary might also be important. Indeed, this was even crucial in the case of genus $0$ and $2$.

The study of surfaces with a single boundary component starts with punctured discs and braid groups. As it was mentioned, by interpreting $\calB_m$ as the mapping class group $\Mod(D^2_m)$, Lawrence defined twisted homological representations that were later shown to be faithful by Bigelow and Krammer. Another (maybe more common) definition of braid groups identifies $\calB_m$ with the Artin group with generators $\{ \sigma_i \mid 1 \leqs i \leqs m-1 \}$ and relations 
\begin{align*}
 \sigma_i \sigma_j &= \sigma_j \sigma_i \quad \Forall |i-j| \geqs 2, &
 \sigma_i \sigma_j \sigma_i &= \sigma_j \sigma_i \sigma_j \quad \Forall |i-j| = 1.
\end{align*}
This algebraic definition underlies the construction of a very large family of representations of $\calB_m$ based on R-matrices and representations of quantum groups, see for instance \cite[Section~X.6]{K95}. These quantum representations of braid groups constitute one of the main building blocks for quantum representations of mapping class groups, and for TQFTs in general.

Recently, the second author has shown that Lawrence representations recover a family of quantum representations of braid groups associated with quantum $\fsl_2$ in a wide generality, see \cite[Theorem~1.5]{M20} (an isomorphism in a smaller subcase was previously established by Jackson and Kerler \cite{JK09}). The equivalence includes an action of quantum $\slt$ on the twisted homology of $\Conf_n(D^m)$ that is built by purely homological means, and that commutes with the homological action of $\calB_m$. This paper will generalize this result, thus extending the explicit relation between homological and quantum representations of mapping class groups to all surfaces. The equivalence will still feature a commuting action of quantum $\fsl_2$ defined by the same homological operators.

Hints for the first step in this generalization can be found in the work of Crivelli, Felder, and Wieczerkowski. Indeed, in \cite{FW91}, an action of the mapping class group $\Mod(T_1)$ of the torus $T_1$ with $1$ boundary component is introduced on a ``vector space of homology cycles with twisted coefficients on configuration spaces of $T_1$'', which is generated by families of non-intesecting loops in $T_1$ subject to certain relations. In \cite[Conjecture~6.1]{CFW93a}, it is conjectured that these vector spaces should in fact be understood as twisted homologies of configuration spaces of $T_1$, as their name suggests. In \cite{CFW93b}, the action of $\Mod(T_1)$ is transported to the restricted quantum group of $\fsl_2$ by means of an explicit isomorphism, and similarities with the representation defined by Lyubashenko and Majid in \cite{LM94} are remarked. Notice that Kerler and Lyubashenko also seemed to expect an interesting relation between the two approaches, see \cite[Section~3.5]{K96} and \cite{L94}. Therefore, our work (see Theorem~\ref{T:main_result}) will also give a homological status to the construction of Crivelli, Felder, and Wieczerkowski by answering their conjecture, confirming the relation with quantum representations, and generalizing the result to higher genus surfaces.

\subsection{Main constructions and results}

In this paper, we establish an explicit isomorphism between two specific families of homological and quantum representations of mapping class groups. More precisely, we fix a compact, connected, oriented surface $\varSigma_{g,1}$ of genus $g$, with $1$ boundary component. On the one hand, on the homological side, we consider a projective action of $\Mod(\varSigma_{g,1})$ on a family of finite-dimensional specializations, denoted\footnote{It should be pointed out that, hidden from the notation, there is a secret dependence of $\calH_{n,g}^V$ on an odd integer parameter $r \geqs 3$.} $\calH_{n,g}^V$ (see Definition~\ref{D:linear_homology}), of twisted homology groups of configuration spaces $\Conf_n(\varSigma_{g,1})$, denoted $\calH_{n,g}^\dHeis$ (see Definition~\ref{D:Heisenberg_homology}). On the other hand, on the quantum side, we focus on the family of projective representations of $\Mod(\varSigma_{g,1})$ obtained from the non-semisimple TQFT associated with the small quantum group $\fraku_\zeta \fsl_2$ at the root of unity $\zeta = e^{\frac{2 \pi \fraki}{r}}$ for $r \geqs 3$ odd (see Section~\ref{S:quantum_sl2_algebra}). Our main contributions are the following:
\begin{enumerate}
 \item We construct homologically an action of an integral version of the quantum group $U_q \fsl_2$ on
  \[
   \calH_g^\bbH = \bigoplus_{n \geqs 0} \calH_{n,g}^\bbH,
  \]
  see Theorem~\ref{T:Uq_homological_representation}. These twisted homology groups have coefficients in the Heisenberg group ring $\Z[\dHeis_g]$ of Section~\ref{S:Heisenberg_groups}, which, in particular, features $q$ among its generators.
 \item We define a homomorphism from $\Mod(\varSigma_{g,1})$ to $\PGL_{r^g}(\bbC)$ that induces a projective action of $\Mod(\varSigma_{g,1})$ on the twisted homology group $\calH_{n,g}^V$ obtained from $\calH_{n,g}^\dHeis$ by specializing $q$ to $\zeta$, and by correspondingly representing $\dHeis_g$ into $\GL_{r^g}(\mathbb{Z}[\zeta])$, see Theorem~\ref{T:mcg_homological_projective_rep}. We also remark that the direct sum of these actions of $\Mod(\varSigma_{g,1})$ naturally commutes with the action of $U_\zeta \fsl_2$ on
  \[
   \calH_g^V = \bigoplus_{n \geqs 0} \calH_{n,g}^V,
  \]
  which is obtained from the one of $U_q \fsl_2$ on $\calH_g^\bbH$ by specialization.
 \item We explicitly compute the two actions with respect to convenient bases of these twisted homologies, and we exhibit an explicit $\mathbb{Z}[\zeta]$-linear isomorphism between a finite-dimensional subspace of $\calH_g^V$, denoted $\calH_g^{V(r)}$, and $\fraku_\zeta \fsl_2^{\otimes g}$. Finally, we show in Theorem~\ref{T:The_Theorem} that this linear isomorphism identifies these homological actions with the $g$th tensor power of the adjoint representation of $\fraku_\zeta \fsl_2$, denoted $\ad^{\otimes g}$, and with the corresponding projective representation of $\Mod(\varSigma_{g,1})$ of Lyubashenko, respectively. In the process, we provide explicit formulas also for the quantum versions of the actions of these two groups, that might be of independent interest.
 \item We obtain, as a direct consequence, that the non-semisimple quantum representation associated with $\fraku_\zeta \fsl_2$ is integral when restricted to the Torelli group $\calI(\varSigma_{g,1})$, meaning that the coefficients of the corresponding action actually belong to $\Z[\zeta]$, when these are computed in the convenient bases considered above, see Corollary~\ref{C:non-semisimple_TQFTs_are_integral}. This property, which was already known in the semisimple case and led to important results, appears to be new on the non-semisimple side.
\end{enumerate}
We point out that, up to minor differences, the homological representations considered here should coincide with those constructed in \cite{BPS21}, where the twisted homology groups $\calH_{n,g}^\bbH$ were introduced, together with their \textit{finite-dimensional Schrödinger} specializations corresponding to $\calH_{n,g}^V$. An implicit procedure for uncrossing the mapping class group action on $\calH_{n,g}^V$ is given in \cite[Theorem~C]{BPS21}, although the result is only discussed for even roots of unity. In this paper, we focus on odd roots of unity, and we use an explicit uncrossing morphism, inspired from similar formulas provided by \cite{CFW93b} in the case of the torus, see Equation~\eqref{E:psi}. On the quantum side, the quantum representations considered here were first defined in \cite{L94}, and later shown to be part of a family of non-semisimple TQFTs in \cite{DGGPR20}. The results listed above are summarized by the following statement.

\begin{theorem}\label{T:main_result}
 The $\fraku_\zeta$-module isomorphism
 \begin{align*}
  \Phi_g^V : \calH_g^{V(r)} &\to \ad^{\otimes g}
 \end{align*}
 defined in Equation~\eqref{E:isomorphism} induces a commutative diagram
 \begin{center}
  \begin{tikzpicture}[descr/.style={fill=white}]
   \node (P1) at (0,0) {$\Mod(\varSigma_{g,1})$};
   \node (P2) at (3,0.75) {$\PGL_{\fraku_\zeta}(\homol_g^{V(r)})$};
   \node (P3) at (3,-0.75) {$\PGL_{\fraku_\zeta}(\ad^{\otimes g})$};
   \draw
   (P1) edge[->] node[above] {\scriptsize $\bar{\rho}_g^{V(r)}$} (P2)
   (P2) edge[->] node[right] {\scriptsize $\Phi_g^V \circ \_ \circ (\Phi_g^V)^{-1}$}(P3)
   (P1) edge[->] node[below] {\scriptsize $\bar{\rho}_g^{\fraku_\zeta}$} (P3);
  \end{tikzpicture}
 \end{center}
 providing an isomorphism between the homological representation $\bar{\rho}_g^{V(r)}$ of Theorem~\ref{T:uzeta_homological_representation} and the quantum representation $\bar{\rho}_g^{\fraku_\zeta}$ of Proposition~\ref{P:quantum_action_of_mcg_generators}.
\end{theorem}

As a consequence of Theorem~\ref{T:main_result}, both families of tools (quantum and homological ones), and all arguments for tackling faithfulness, are now gathered into one coherent theory, where they can be used and developed jointly.

In the process of establishing this result, Proposition~\ref{P:homology_structure} provides a basis for the twisted homology of the configuration space $\Conf_n(\varSigma_{g,1})$, which is presented by labeled diagrams in the surface $\varSigma_{g,1}$. It should be noted that this basis persists under all twisting, meaning that, for instance, it remains a basis under all specializations of the variable $q$ to any root of unity. Furthermore, Proposition~\ref{P:computation_rules} provides computation rules for these diagrams that translate identities coming from homological calculus. It is by means of these diagrammatic rules that we obtain explicit formulas for the homological action of Dehn twists, in Section~\ref{S:homological_mcg_computation}. The isomorphism between mapping class group representations is then established by comparing homological equations with those expressing the quantum action of Dehn twists in (tensor powers of) integral bases of $\ad^{\otimes g}$, which are derived in Section~\ref{S:quantum_mcg_computation}.

We also stress the fact that our result provides a homological model for all tensor powers of the adjoint representation of $\fraku_\zeta \fsl_2$, which we identify with fi\-nite-di\-men\-sion\-al subrepresentations of some in\-fi\-nite-di\-men\-sion\-al representations obtained by specializing $U_q \fsl_2$-modules with coefficients in Heisenberg group rings. These representations have the remarkable property of intertwining actions of (subgroups of) mapping class groups, and thus seem to be a fragment of a non-semisimple TQFT defined by homological means, whose state spaces are infinite-dimensional. Understanding these representations, and unveiling such a TQFT, will be the subject of future work.

Finally, we highlight Section~\ref{S:recovering_BCGP}, where we construct homological representations of Torelli groups using a slightly different choice of twisted coefficients, obtained by perturbing our original representation of the Heisenberg group $\dHeis_g$ using a set of $2g$ formal variables. These parameters can be specialized, and interpreted as evaluations of a cohomology class on the surface $\varSigma_{g,1}$. The resulting twisted homology groups still support an action of the Torelli subgroup of $\Mod(\varSigma_{g,1})$ that intertwines an action of the quantum group $U_\zeta \fsl_2$. This allows us to conjecture that the resulting homological representations are equivalent to quantum representations arising from another non-semisimple TQFT due to Blanchet, Costantino, Geer, and Patureau \cite{BCGP14}.

\subsection{Structure of the paper}

The first part of this paper is devoted to the action of $\Mod(\varSigma_{g,1})$ on the twisted homology of the configuration space $\Conf_n(\varSigma_{g,1})$. Defining twisted homology requires choosing a representation of the surface braid group, which is the fundamental group of $\Conf_n(\varSigma_{g,1})$. This is done in Section~\ref{S:heisenberg_rep_of_braids}. The structure of the resulting twisted homology is discussed in Section~\ref{S:twisted_homology_modules}, with an explicit basis defined in terms of diagrams representing twisted homology classes given in Section~\ref{S:structure_and_bases}, and diagrammatic rules for computations presented in Section~\ref{S:diagrammatic_rules}. Since the natural action of $\Mod(\varSigma_{g,1})$ on $\Conf_n(\varSigma_{g,1})$ simultaneously affects twisted coefficients, we have to \textit{uncross} it appropriately. This is achieved (projectively) in Theorem~\ref{T:mcg_homological_projective_rep}, in Section~\ref{S:linear_homological_representations}. Before doing this, we also construct a homological action of the quantum group $U_q \fsl_2$ on the direct sum, over all $n \geqs 0$, of these twisted homologies. This action commutes with the projective one of the mapping class group, and is the subject of Theorem~\ref{T:Uq_homological_representation}, in Section~\ref{S:homological_action_of_quantum_sl2}. Its construction involves a purely homological definition of the operators corresponding to the generators of $U_q \fsl_2$, which give a surprising new point of view on the representation theory of quantum groups.

In Section~\ref{S:homological_computations}, we compute these homological actions explicitly with respect to the bases given by Proposition~\ref{P:homology_structure} and Remark~\ref{R:infinite_linear_basis}. This is done both for the generators of the quantum group $U_q \fsl_2$, on the one hand, and for those of the mapping class group $\Mod(\varSigma_{g,1})$, on the other. These computations are based on our graphical calculus at all stages: bases are composed of diagrams, and diagrammatic rules are extensively used.

In Section~\ref{S:quantum_representations}, we recall how to construct projective representations of $\Mod(\varSigma_{g,1})$ out of a factorizable ribbon Hopf algebra $H$. These are the byproducts of a non-semisimple TQFT that can be built out of the transmutation of $H$, which is a braided Hopf algebra in the category of left $H$-modules, whose underlying object is the adjoint representation of $H$. The setup is recalled in Section~\ref{S:Hopf_algebras}, while the TQFT construction is explained in Section~\ref{S:quantum_representations_of_mcg}. In particular, the category $\RCob$ of framed connected cobordisms between connected surfaces, due to Crane, Yetter, and Kerler, is recalled in Section~\ref{S:cobordisms}. Then, Habiro's diagrammatic notation for morphisms of $\RCob$, based on top tangles in handlebodies, is recalled in Section~\ref{S:top_tangles}. In Section~\ref{S:quantum_diagrammatic_calculus}, an algorithm for computing the TQFT associated with the transmutation of $H$ is explained in terms of a diagrammatic calculus whose flavor is quite different from the homological one (being much more algebraic, and much less topological). This leads to Proposition~\ref{P:quantum_action_of_mcg_generators}, which provides formulas for the action of generators of the mapping class group in Hopf algebraic terms. These formulas will be then computed for the small quantum group $\fraku_\zeta \fsl_2$ at an odd root of unity $\zeta$, whose definition is recalled in Section~\ref{S:small_quantum_sl2}.

In Section~\ref{S:quantum_computations}, we compute quantum actions with respect to bases in the explicit case of $\fraku_\zeta \fsl_2$. Namely, we fix a (carefully) chosen basis of state space $\ad^{\otimes g}$ of the surface $\varSigma_{g,1}$, and we compute explicitly the action of generators of the small quantum group $\fraku_\zeta \fsl_2$, on the one hand, and of generators of the mapping class group $\Mod(\varSigma_{g,1})$, on the other. These computations are much more algebraic in nature than the corresponding homological ones.

Our main result is Theorem~\ref{T:The_Theorem}, which is proved in Section~\ref{S:isomorphism}. It identifies homological representations with quantum ones. Namely, we give a simple explicit isomorphism between twisted homologies and state spaces of surfaces, which is a diagonal correspondence between the chosen bases on both sides. This isomorphism intertwines the two mapping class group actions: the quantum one arising from non-semisimple TQFTs, on one side, and the twisted homological one, on the other. It also intertwines the two actions of the small quantum group of $\fsl_2$: the tensor power of the adjoint representation, on the quantum side, and the newly-defined twisted homological one, on the homological side. Our result allows us to deduce integrality properties for the non-semisimple quantum representations of Torelli groups, as explained in Corollary~\ref{C:non-semisimple_TQFTs_are_integral}. We also conjecture that a generalization of our construction should recover the representations of Blanchet, Costantino, Geer, and Patureau, see Conjecture~\ref{C:BCGP_conjecture}.

Sections~\ref{S:homological_representations} and~\ref{S:homological_computations} only deal with twisted homologies of configuration spaces of surfaces. On the other hand, Sections~\ref{S:quantum_representations} and~\ref{S:quantum_computations} are only concerned with quantum representations of mapping class group arising from non-semisimple TQFTs. Therefore, \textbf{each pair of sections can be read independently}.

\subsection*{Acknowledgments}

J.M. is very grateful to R.~Detcherry for many invaluable discussions on homological representations in general, and to Q.~Faes for fruitful remarks on their computations and on mapping class groups. J.M. thanks A.~Beliakova for inviting him to Zürich, which greatly contributed to the progress on this project. Part of this work was achieved while J.M. was supported by the project ``AlMaRe'' (ANR-19-CE40-0001-01). M.D. is grateful to G.~Massuyeau for inviting him to Dijon, thus allowing this project to start. M.D. was supported by Grant n.~200020\_207374 of the Swiss National Science Foundation (SNSF). Both authors were supported by the National Center of Competence in Research (NCCR) SwissMAP. Both authors are thankful to A.~Beliakova and C.~Blanchet for organizing a workshop on this topic, and to CIRM for hosting the event. No spleens were used in the making of this project.

\section{Homological representations of mapping class groups}\label{S:homological_representations}

Let $\varSigma_{g,1}$ denote the compact connected surface of genus $g$ with $1$ boundary component, represented as follows:
\[
 \pic{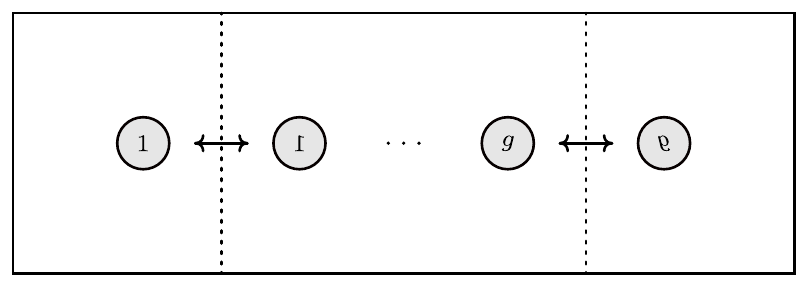}
\]
In this picture, $2g$ embedded discs are identified in pairs through reflections along $g$ dotted vertical axes. When $g=1$, we will drop the numbering, and simply represent $\varSigma_{1,1}$ as shown:
\[
 \pic{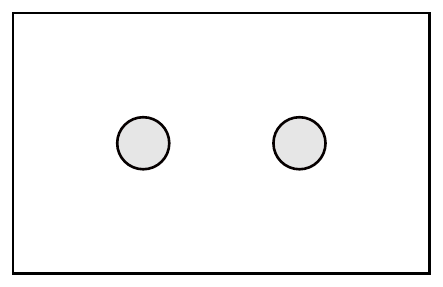}
\]
We recall that, for every integer $n \geqs 1$, the \textit{configuration space of $n$ undordered points in $\varSigma_{g,1}$}, also called the \textit{$n$th configuration space of $\varSigma_{g,1}$}, is defined as
\begin{equation}\label{E:X_def}
 X_{n,g} := \Conf_n(\varSigma_{g,1}) = \{ (x_1,\ldots,x_n) \in \varSigma_{g,1}^{\times n} \mid \Forall 1 \leqs i < j \leqs n \quad x_i \neq x_j \} / \frakS_n,
\end{equation}
where the symmetric group $\frakS_n$ acts by permutation of coordinates. Elements of $X_{n,g}$ are denoted $\underline{x} = \{ x_1,\ldots,x_n \}$. For $n = 0$, we set $X_{0,g} := \varnothing$. Notice that $X_{n,g}$ is a non-compact connected $2n$-dimensional manifold.

In Section~\ref{S:twisted_homology_modules}, we will construct twisted homology groups of $X_{n,g}$. Then, we will equip them with an action of the quantum group of $\fsl_2$, in Section~\ref{S:homological_action_of_quantum_sl2}, and with
a commuting projective action of the mapping class group of $\varSigma_{g,1}$, in Section~\ref{S:homological_representations_of_mcg}. Since twisted homology requires a representation of the (group ring of the) fundamental group of $X_{n,g}$, we will first need to introduce Heinsenberg group representations, in Section~\ref{S:heisenberg_rep_of_braids}.

\subsection{Heisenberg representations of surface braid groups}\label{S:heisenberg_rep_of_braids}

In this section, we will construct and study representations of surface braid groups onto so-called Heisenberg groups.

\subsubsection{Surface braid groups}

Let us fix a basepoint $\braid{\xi}{} = \{ \xi_1,\ldots,\xi_n \} \in \partial X_{n,g}$ in the boundary of $X_{n,g}$ of the following form:
\[
 \pic{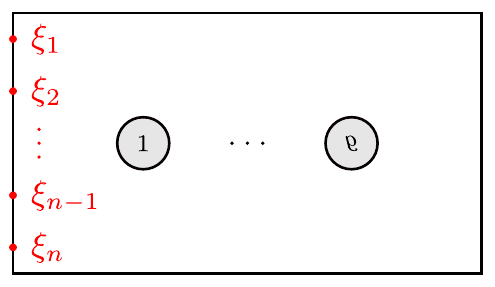}
\]
The \textit{$n$th braid group of $\varSigma_{g,1}$} is the fundamental group
\[
 \pi_{n,g} := \pi_1(X_{n,g},\braid{\xi}{}).
\]
An explicit presentation of $\pi_{n,g}$ can be found in \cite[Theorem~2.1]{BG05}, based on \cite[Theorem~1.1]{B01}. Let us recall it here for convenience. For all $n \geqs 1$, the group $\pi_{n,g}$ is generated by
\[
 \{ \braid{\sigma}{i}, \braid{\alpha}{j}, \braid{\beta}{j} \mid 1 \leqs i \leqs n-1, 1 \leqs j \leqs g \},
\]
where:
\begin{enumerate}
 \item $\braid{\sigma}{i}$ keeps all coordinates fixed except for the $i$th and $(i+1)$th ones, which move along the following paths:
  \[
   \pic{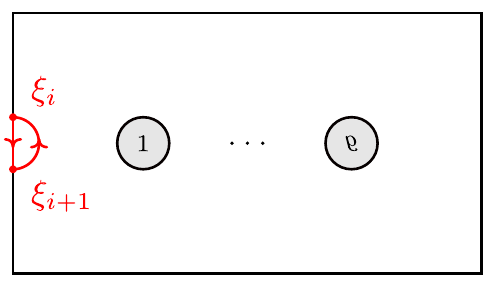}
  \]
 \item $\braid{\alpha}{j}$ keeps all coordinates fixed except for the last one, which moves along the following path:
  \[
   \pic{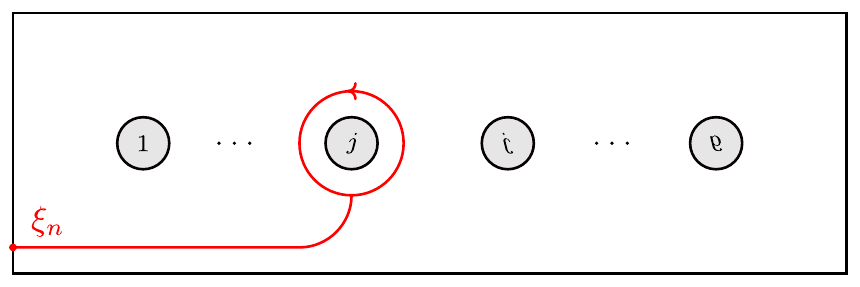}
  \]
 \item $\braid{\beta}{j}$ keeps all coordinates fixed except for the last one, which moves along the following path:
  \[
   \pic{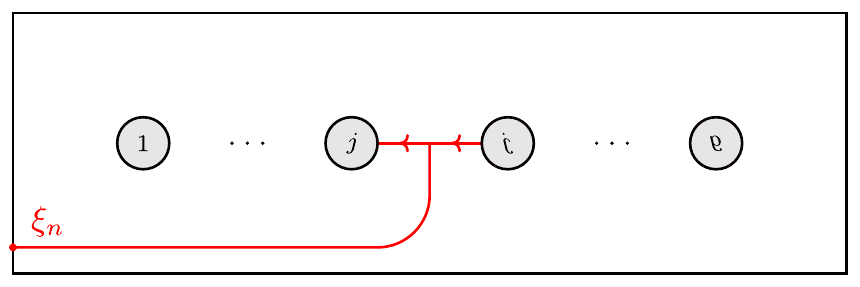}
  \]
\end{enumerate}
These generators satisfy the following complete list of relations\footnote{Our convention for path concatenation is fixed in Appendix~\ref{A:twisted}.}
\begin{description}
 \item[\namedlabel{R:BR1}{(BR1)}] $[\braid{\sigma}{i},\braid{\sigma}{j}] = 1$ for all $1 \leqs i,j \leqs n-1$ such that $\lvert i-j \rvert \geqs 2$;
 \item[\namedlabel{R:BR2}{(BR2)}] $\braid{\sigma}{i} \ast \braid{\sigma}{i+1} \ast \braid{\sigma}{i} = \braid{\sigma}{i+1} \ast \braid{\sigma}{i} \ast \braid{\sigma}{i+1}$ for every $1 \leqs i \leqs n-2$;
 \item[\namedlabel{R:CR1}{(CR1)}] $[\braid{\sigma}{i},\braid{\alpha}{j}] = [\braid{\sigma}{i},\braid{\beta}{j}] = 1$ for all $1 \leqs i \leqs n-2$ and $1 \leqs j \leqs g$;
 \item[\namedlabel{R:CR2}{(CR2)}] $[\braid{\alpha}{j},\braid{\sigma}{n-1} \ast \braid{\alpha}{j} \ast \braid{\sigma}{n-1}] = [\braid{\beta}{j},\braid{\sigma^{-1}}{n-1} \ast \braid{\beta}{j} \ast \braid{\sigma^{-1}}{n-1}] = 1$ for every $1 \leqs j \leqs g$;
 \item[\namedlabel{R:CR3}{(CR3)}] $[\braid{\alpha}{j},\braid{\sigma}{n-1} \ast \braid{\alpha}{k} \ast \braid{\sigma^{-1}}{n-1}] = [\braid{\alpha}{j},\braid{\sigma}{n-1} \ast \braid{\beta}{k} \ast \braid{\sigma^{-1}}{n-1}] = [\braid{\beta}{j},\braid{\sigma}{n-1} \ast \braid{\alpha}{k} \ast \braid{\sigma^{-1}}{n-1}] = [\braid{\beta}{j},\braid{\sigma}{n-1} \ast \braid{\beta}{k} \ast \braid{\sigma^{-1}}{n-1}] = 1$ for all $1 \leqs j < k \leqs g$;
 \item[\namedlabel{R:SCR}{(SCR)}] $\braid{\sigma}{n-1} \ast \braid{\alpha}{j} \ast \braid{\sigma}{n-1} \ast \braid{\beta}{j} = \braid{\beta}{j} \ast \braid{\sigma^{-1}}{n-1} \ast \braid{\alpha}{j} \ast \braid{\sigma}{n-1}$ for every $1 \leqs j \leqs g$.
\end{description}
This presentation is obtained from \cite[Theorem~2.1]{BG05} by setting
\begin{align*}
 \braid{\sigma}{i} &= \sigma_{n-i+1}, &
 \braid{\alpha}{j} &= \delta_{2g-2j+1}, &
 \braid{\beta}{j} &= \delta_{2g-2j+2}^{-1}
\end{align*}
for all integers $1 \leqs i \leqs n-1$ and $1 \leqs j \leqs g$. Notice that relation~\ref{R:CR3} is equivalent to the one appearing in \cite[Theorem~2.1]{BG05} thanks to relation~\ref{R:CR2}. When $n=1$, there is no standard braid generator $\braid{\sigma}{i}$, and $\pi_{1,g}$ reduces to the fundamental group $\pi_1(\varSigma_{g,1},\xi)$, which is the free group with generators $\braid{\alpha}{j}$, $\braid{\beta}{j}$ for $1 \leqs j \leqs g$. In general, it is useful to fix names for a few other elements of $\pi_{n,g}$. For all integers $1 \leqs j \leqs g$ and $1 \leqs k \leqs g-1$, we set
\begin{align}
 \braid{\tilde{\alpha}}{j} &:=\braid{\beta^{-1}}{j} \ast \braid{\alpha}{j} \ast \braid{\beta}{j}, &
 \braid{\gamma}{k} &:= \braid{\alpha}{k+1} \ast \braid{\tilde{\alpha}^{-1}}{k}, &
 \braid{\delta}{j} &:= 
 \begin{cases}
  \braid{\alpha}{1} & \mbox{ if } j = 1, \\
  \braid{\gamma}{j-1} \ast \ldots \ast \braid{\gamma}{1} \ast \braid{\alpha}{1} & \mbox{ if } 2 \leqs j \leqs g.
 \end{cases} \label{E:other_braids}
\end{align}
Graphically, $\braid{\tilde{\alpha}}{j}$, $\braid{\gamma}{k}$, and $\braid{\delta}{j}$ are given by
\begin{gather*}
 \pic{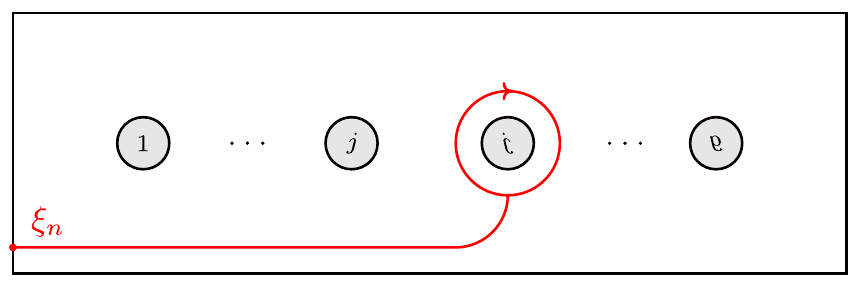} \\*
 \pic{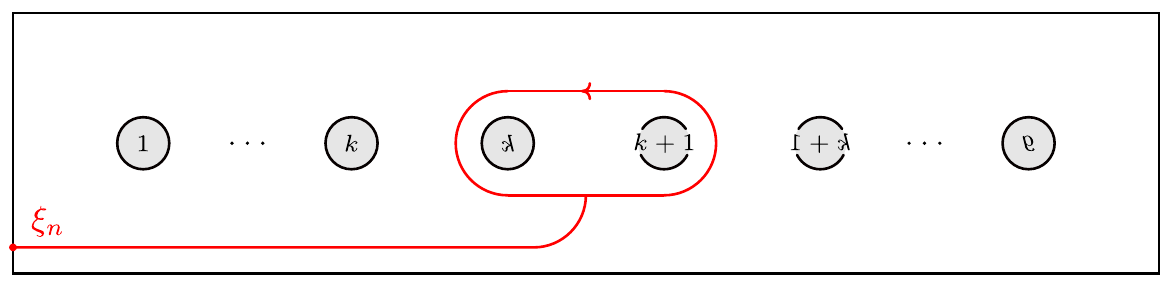} \\*
 \pic{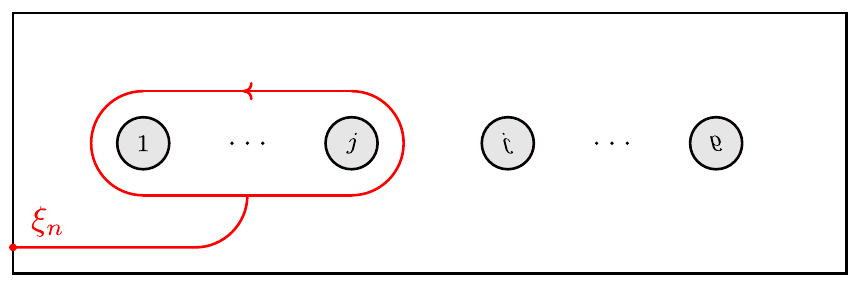}
\end{gather*}
respectively.

\begin{remark}\label{R:alternative_generators_pi}
 Notice that
 \[
  \{ \braid{\sigma}{i}, \braid{\beta}{j}, \braid{\delta}{j} \mid 1 \leqs i \leqs n-1, 1 \leqs j \leqs g \}
 \]
 is also a system of generators of $\pi_{n,g}$. Indeed, for every $1 \leqs k \leqs g-1$, we have  
 \[
  \braid{\gamma}{k} = \braid{\delta}{k+1} \ast \braid{\delta^{-1}}{k},
 \]
 and, for every $1 \leqs j \leqs g$, we can recursively obtain
 \[
  \begin{cases}
   \braid{\alpha}{1} = \braid{\delta}{1}, & \\
   \braid{\tilde{\alpha}}{1} = \braid{\beta^{-1}}{1} \ast \braid{\delta}{1} \ast \braid{\beta}{1}, & \\
   \braid{\alpha}{j} = \braid{\gamma}{j-1} \ast \braid{\tilde{\alpha}}{j-1} & \mbox{ if } 2 \leqs j \leqs g, \\
   \braid{\tilde{\alpha}}{j} = \braid{\beta^{-1}}{j} \ast \braid{\alpha}{j} \ast \braid{\beta}{j} & \mbox{ if } 2 \leqs j \leqs g.
  \end{cases}
 \]
\end{remark}

\subsubsection{Heisenberg groups}\label{S:Heisenberg_groups}

Our next goal is to represent (by a surjective homomorphism) the surface braid group $\pi_{n,g}$ onto some simpler group (although we will actually work with group rings). Our choice for the latter is the Heisenberg group associated with $\varSigma_{g,1}$, which is richer than the abelianization of $\pi_{n,g}$, and will allow us to recover the action of a quantum group. For every integer $g \geqs 1$ and every commutative ring $R$, the \textit{Heisenberg group of genus $g$ with coefficients in $R$} is the group $\Heisenberg_g(R)$ of matrices of the form
\[
 \left\{ 
 \begin{pmatrix}
  1 & 2\bfa^\rmT & c \\ 0 & I_g & 2\bfb \\ 0 & 0 & 1
 \end{pmatrix} \in \GL_{g+2}(R) \Biggm\vert
 \bfa = 
 \begin{pmatrix}
  a_1 \\ \vdots \\ a_g
 \end{pmatrix},
 \bfb = 
 \begin{pmatrix}
  b_1 \\ \vdots \\ b_g
 \end{pmatrix} \in R^g,
  c \in R
 \right\}.
\]
Notice that $\Heisenberg_g(R)$ is not a commutative group, since
\[
 \begin{pmatrix}
  1 & 2\bfa^\rmT & c \\ 0 & I_g & 2\bfb \\ 0 & 0 & 1 
 \end{pmatrix} 
 \begin{pmatrix}
  1 & 2\bfa'^\rmT & c' \\ 0 & I_g & 2\bfb' \\ 0 & 0 & 1 
 \end{pmatrix} =
 \begin{pmatrix}
  1 & 2(\bfa + \bfa')^\rmT & c + c' + 4\bfa \cdot \bfb' \\ 0 & I_g & 2(\bfb + \bfb') \\ 0 & 0 & 1 
 \end{pmatrix}.
\]
The group $\dHeis_g := \Heisenberg_g(\Z)$ is called the \textit{discrete} Heisenberg group. Let us fix the elements
\begin{align}
 \srs &:=
 \begin{pmatrix}
  1 & \bfzero^\rmT & 1 \\ 0 & I_g & \bfzero \\ 0 & 0 & 1
 \end{pmatrix}, &
 \alpha_j &:=
 \begin{pmatrix}
  1 & 2\bfe_j^\rmT & 0 \\ 0 & I_g & \bfzero \\ 0 & 0 & 1
 \end{pmatrix}, &
 \beta_j &:= 
 \begin{pmatrix}
  1 & \bfzero^\rmT & 0 \\ 0 & I_g & 2\bfe_j \\ 0 & 0 & 1
 \end{pmatrix}, \label{E:Heisenberg_generators}
\end{align}
where $\bfe_j$ denotes the column vector of size $g$ whose $k$th entry is $\delta_{j,k}$.

\begin{lemma}\label{L:discrete_Heisenberg_presentation}
 The discrete Heisenberg group $\dHeis_g$ is isomorphic to the group with generators
 \[
  \{ \srs, \alpha_j, \beta_j \mid 1 \leqs j \leqs g \}
 \]
 and the following complete list of relations:
 \begin{description}
  \item[\namedlabel{R:DC1}{\normalfont (DC1)}] $[\srs,\alpha_j] = [\srs,\beta_j] = 1$ for every $1 \leqs j \leqs g$;
  \item[\namedlabel{R:DC2}{\normalfont (DC2)}] $[\alpha_j,\alpha_k] = [\alpha_j,\beta_k] = [\beta_j,\alpha_k] = [\beta_j,\beta_k] = 1$ for all $1 \leqs j < k \leqs g$;
  \item[\namedlabel{R:DSC}{\normalfont (DSC)}] $[\alpha_j,\beta_j] = \srs^4$ for every $1 \leqs j \leqs g$.
 \end{description}
\end{lemma}

\begin{proof}
 It can be easily verified that relations~\ref{R:DC1}--\ref{R:DSC} are satisfied in $\dHeis_g$. Then, it is sufficient to check that
 \[
  \begin{pmatrix}
   1 & 2\bfa^\rmT & c \\ 0 & I_g & 2\bfb \\ 0 & 0 & 1 
  \end{pmatrix} \mapsto
  \srs^{c-4\bfa \cdot \bfb} \bfalpha^{\bfa} \bfbeta^{\bfb}
 \]
 defines the inverse homomorphism, where
 \begin{align*}
  \bfalpha^{\bfa} &:= \alpha_1^{a_1} \cdots \alpha_g^{a_g}, &
  \bfbeta^{\bfb} &:= \beta_1^{b_1} \cdots \beta_g^{b_g}.
 \end{align*}
 Notice that a product
 \[
  \begin{pmatrix}
   1 & 2\bfa^\rmT & c \\ 0 & I_g & 2\bfb \\ 0 & 0 & 1 
  \end{pmatrix}
  \begin{pmatrix}
   1 & 2\bfa'^\rmT & c' \\ 0 & I_g & 2\bfb' \\ 0 & 0 & 1 
  \end{pmatrix}
 \]
 is sent to
 \begin{align*}
  \srs^{c-4\bfa \cdot \bfb} \bfalpha^{\bfa} \bfbeta^{\bfb} \srs^{c'-4\bfa' \cdot \bfb'} \bfalpha^{\bfa'} \bfbeta^{\bfb'}
  &= \srs^{c+c'-4(\bfa \cdot \bfb + \bfa' \cdot \bfb + \bfa' \cdot \bfb')} \bfalpha^{\bfa+\bfa'} \bfbeta^{\bfb+\bfb'} 
  = \srs^{c+c'+4\bfa \cdot \bfb'-4(\bfa+\bfa') \cdot (\bfb+\bfb')} \bfalpha^{\bfa+\bfa'} \bfbeta^{\bfb+\bfb'},
 \end{align*}
 which is the image of 
 \[
  \begin{pmatrix}
   1 & 2(\bfa + \bfa')^\rmT & c + c' + 4\bfa \cdot \bfb' \\ 0 & I_g & 2(\bfb + \bfb') \\ 0 & 0 & 1 
  \end{pmatrix}. \qedhere
 \]
\end{proof}

The following statement was already proved, with slightly different conventions, in \cite[Lemma~4.4]{BGG11}, and it was used in \cite[Proposition~7 \& Corollary~8]{BPS21} in order to define twisted homology groups.

\begin{lemma}\label{L:phi}
 There exists a unique ring homomorphism $\varphi_{n,g}^\dHeis : \Z[\pi_{n,g}] \to \Z[\dHeis_g]$ satisfying
 \begin{align*}
  \braid{\sigma}{i} &\mapsto \sigma := -\srs^{-2}, &
  \braid{\alpha}{j} &\mapsto \alpha_j, &
  \braid{\beta}{j} &\mapsto \beta_j.
 \end{align*}
\end{lemma}

\begin{proof}
 In order to check that the assignment gives well-defined homomorphisms, we simply need to verify that relations~\ref{R:BR1}--\ref{R:SCR} are satisfied by $-\srs^{-2}$, $\alpha_j$, $\beta_j$ for every integer $1 \leqs j \leqs g$. These are easy computations, which are left to the reader. Notice that, when $n = 1$, the surface braid group $\pi_{n,g}$ reduces to the fundamental group of $\varSigma_{g,1}$, which is a free group, so there is nothing to check in this case.
\end{proof}

The following observation is not crucial for the present work. It is useful, however, for interpreting the twisted homology groups defined in Section~\ref{S:twisted_homology_modules} in terms of standard homology groups of the corresponding regular covers, see for instance Remark~\ref{R:interpretation_of_homology_intermsof_covers}.

\begin{remark}\label{R:quotient_group}
 When $n \geqs 2$, we have a commutative diagram
 \begin{center}
  \begin{tikzpicture}[descr/.style={fill=white}]
   \node (P1) at (-1.5,1) {$\Z[\pi_{n,g}]$};
   \node (P2) at (1.5,1) {$\Z[\dHeis_g]$};
   \node (P3) at (0,0) {$\Z[\pi_{n,g}/K_{n,g}]$};
   \draw
   (P1) edge[->] node[above] {\scriptsize $\varphi_{n,g}^\dHeis$} (P2)
   (P1) edge[->>] node[left,xshift=-7.5pt] {\scriptsize $\varphi_{n,g}^{\pi/K}$} (P3)
   (P3) edge[right hook->] (P2);
  \end{tikzpicture}
 \end{center}
 where $\varphi_{n,g}^{\pi/K}$ is induced by the projection to the quotient of $\pi_{n,g}$ with respect to the normal subgroup $K_{n,g} \triangleleft \pi_{n,g}$  with generators
 \[
  \{ [\braid{\sigma}{n-1},\braid{\sigma}{i}], 
  [\braid{\sigma}{n-1},\braid{\alpha}{j}], 
  [\braid{\sigma}{n-1},\braid{\beta}{j}] \mid
  1 \leqs i \leqs n-2, 1 \leqs j \leqs g \}.
 \]
 Indeed, $\varphi_{n,g}^\dHeis$ clearly vanishes on $\Z[K_{n,g}]$, and it is easy to see that $\pi_{n,g}/K_{n,g}$ is isomorphic to the group with generators
 \[
  \{ \sigma, \alpha_j, \beta_j \mid 1 \leqs j \leqs g \}
 \]
 and the following complete list of relations:
 \begin{description}
  \item[\namedlabel{R:QC1}{\normalfont (QC1)}] $[\sigma,\alpha_j] = [\sigma,\beta_j] = 1$ for every $1 \leqs j \leqs g$;
  \item[\namedlabel{R:QC2}{\normalfont (QC2)}] $[\alpha_j,\alpha_k] = [\alpha_j,\beta_k] = [\beta_j,\alpha_k] = [\beta_j,\beta_k] = 1$ for all $1 \leqs j < k \leqs g$;
  \item[\namedlabel{R:QSC}{\normalfont (QSC)}] $[\alpha_j,\beta_j] = \sigma^{-2}$ for every $1 \leqs j \leqs g$.
 \end{description}
 The analogous statement holds for $n=1$, provided we define $K_{1,g} \triangleleft \pi_{1,g}$ as the normal subgroup generated by the following list of elements:
 \begin{enumerate}
  \item $[\braid{\alpha}{j},\braid{\beta}{j}][\braid{\alpha}{k},\braid{\beta}{k}]^{-1}$ for all integers $1 \leqs j,k \leqs g$;
  \item $[\braid{\alpha}{j},\braid{\alpha}{k}]$, $[\braid{\alpha}{j},\braid{\beta}{k}]$, $[\braid{\beta}{j},\braid{\alpha}{k}]$, $[\braid{\beta}{j},\braid{\beta}{k}]$ for all integers $1 \leqs j < k \leqs g$;
  \item $[[\braid{\alpha}{j},\braid{\beta}{j}],\braid{\alpha}{k}]$, $[[\braid{\alpha}{j},\braid{\beta}{j}],\braid{\beta}{k}]$ for all integers $1 \leqs j,k \leqs g$.
 \end{enumerate}
\end{remark}

\subsection{Discrete Heisenberg homology}\label{S:twisted_homology_modules}

We are ready to define homology groups of $X_{n,g}$ twisted by the homomorphism $\varphi_{n,g}^\dHeis$, but first, we need to specify a relative part in the boundary of $X_{n,g}$. To this end, let us denote by $\partial_- \varSigma_{g,1}$ the subset of $\partial \varSigma_{g,1}$ represented by the magenta arc
\[
 \pic{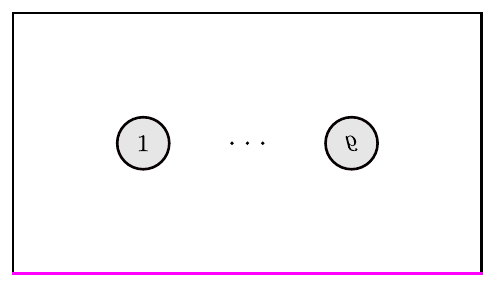}
\]
and let us consider the submanifold
\begin{equation}\label{E:Y_def}
 Y_{n,g} := \{ \underline{x} = \{ x_1,\ldots,x_n \} \in X_{n,g} \mid \Exists 1 \leqs i \leqs n \quad x_i \in \partial_- \varSigma_{g,1} \}
\end{equation}
of $\partial X_{n,g}$.

\begin{definition}\label{D:Heisenberg_homology}
 For every integer $n \geqs 1$, the \textit{$n$th discrete Heisenberg homology group} of $\varSigma_{g,1}$ is the $\Z[\dHeis_g]$-module
 \[
  \homol_{n,g}^{\dHeis} := H^\BM_n(X_{n,g},Y_{n,g};\varphi_{n,g}^{\dHeis}),
 \]
 where $X_{n,g}$ and $Y_{n,g}$ are defined in Equations~\eqref{E:X_def} and \eqref{E:Y_def} respectively, and $\varphi_{n,g}^{\dHeis}$ is defined in Lemma~\ref{L:phi}. For $n = 0$, we set 
 \[
  \homol_{0,g}^{\dHeis} := {\dHeis_g}.
 \]
 The \textit{total discrete Heisenberg homology group} of $\varSigma_{g,1}$ is the direct sum
 \[
  \homol_g^{\dHeis} := \bigoplus_{n \geqs 0} \homol_{n,g}^{\dHeis}.
 \]
\end{definition}

We highlight the fact that, instead of standard singular homology, we are using \textit{Borel--Moore homology} (see Appendix~\ref{A:Borel--Moore}) with \textit{twisted coefficients} (see Appendix~\ref{A:twisted}).

\subsubsection{Diagrammatic notation}\label{S:diagrammatic_notation}

We represent twisted homology classes in $\homoldH_{n,g}$ using diagrams composed of embedded disjoint curves in the surface $\varSigma_{g,1}$ labeled by integers, and equipped with paths connecting them to the base point $\braid{\xi}{} \in \partial X_{n,g}$. Such diagrams are a generalization to positive genus surfaces of the those introduced for punctured discs in \cite[Section~3.1]{M20}. Let us explain how to represent elements of $\homol_{n,g}^\dHeis$ this way. First of all, let us denote by
\[
 \Delta^k := \{ (t_1,\ldots,t_k) \in \R^k \mid 0 < t_1 < \ldots < t_k < 1 \}
\]
the standard open $k$-dimensional simplex in $\R^k$, which we can think of as the configuration space of $k$ (ordered) points in the open interval $]0,1[$.

\begin{definition}\label{D:embedded_twisted_cycle}
 An \textit{embedded twisted cycle} of dimension $n \geqs 0$ in $\varSigma_{g,1}$ is a $4$-tuple $(m,\bfsimp,\bfk,\thread)$ where:
 \begin{enumerate}
  \item $m \geqs 0$ is an integer; 
  \item $\bfsimp$ is an \textit{$m$-multisimplex}, which is given by an ordered family $(\simp_1,\ldots,\simp_m)$ of disjoint proper embeddings $\simp_1,\ldots,\simp_m : [0,1] \to \varSigma_{g,1}$ such that $\simp_\ell$ embeds $\{ 0,1 \}$ into $\partial_- \varSigma_{g,1}$ for $1 \leqs \ell \leqs m$. 
  \item $\bfk$ is an \textit{$m$-partition} of $n$, which is given by an ordered family $\bfk = (k_1,\ldots,k_m) \in \N^m$ of integers satisfying $k_1 + \ldots + k_m = n$, and which provides a \textit{labeling} of the components of the $m$-multisimplex $\bfsimp$.
  \item $\thread$ is a \textit{thread} of the $\bfk$-labeled $m$-multisimplex $\bfsimp$, which is given by an ordered family $(\thrcmp_1,\ldots,\thrcmp_n)$ of disjoint embeddings of $\thrcmp_1,\ldots,\thrcmp_m : [0,1] \to \varSigma_{g,1}$ satisfying $\thrcmp_i(0) = \xi_i$ and $\thrcmp_i(1) \in \bfsimp^{\bfk}(\bfDelta^{\bfk}(\frac 12,\ldots,\frac 12))$ for all integers $1 \leqs i \leqs n$, where 
   \[
    \bfDelta^{\bfk} : {]0,1[}^{\times n} \to \Delta^{k_1} \times \ldots \times \Delta^{k_m}
   \]
   is the homeomorphism defined in Equation~\eqref{E:homeomorphism_hypercube_simplices}, and 
   \[
    \bfsimp^{\bfk} : \Delta^{k_1} \times \ldots \times \Delta^{k_m} \to X_{n,g}
   \]
   is the embedding defined in Equation~\eqref{E:embedding_simplices}.
 \end{enumerate}
\end{definition}

We can represent an embedded twisted cycle $(m,\bfsimp,\bfk,\thread)$ diagrammatically by drawing dashed-dotted curves that follow the images of the proper embeddings $\simp_1, \ldots, \simp_m$, oriented from $0$ to $1$, and carrying the labels $k_1, \ldots, k_m$ respectively, and by drawing solid curves that follow the thread $\thread$. We will often use multiple colors for a multisimplex, in order to help the reader distinguish different components, but we will always reserve the color red for the thread. Here is an example of a diagram of an embedded twisted cycle for $n=5$ and $g=1$:
\begin{equation}\label{fig:example_homology_class}
 \pic{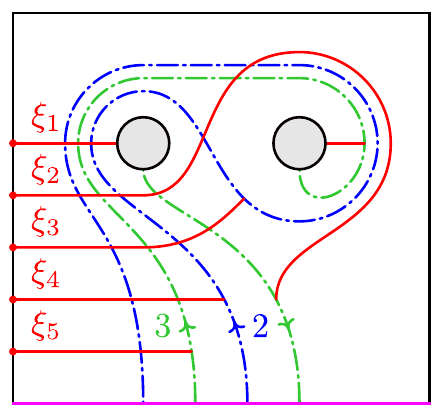}
\end{equation}

The rest of this section is devoted to explaining how to interpret these diagrams as relative homology classes in $H^\BM_n(X_{n,g},Y_{n,g};\varphi_{n,g}^\pi)$, where $\varphi^\pi : \Z[\pi_{n,g}] \to \Z[\pi_{n,g}]$ denotes the identity (or equivalently in $H^\BM_n(X_{n,g},Y_{n,g};\varphi^{\pi/K})$, where $\varphi_{n,g}^{\pi/K} : \Z[\pi_{n,g}] \to \Z[\pi_{n,g}/K_{n,g}]$ denotes the ring homomorphism introduced in Remark~\ref{R:interpretation_of_homology_intermsof_covers}). If we forget threads, then an embedded twisted cycle still determines a relative homology class, but one in $H^\BM_n(X_{n,g},Y_{n,g})$ instead of $H^\BM_n(X_{n,g},Y_{n,g};\varphi_{n,g}^\pi)$. Indeed, every $m$-multisimplex $\bfsimp = (\simp_1,\ldots,\simp_m)$, together with an $m$-partition $\bfk = (k_1,\ldots,k_m)$ of $n$, induces an embedding of the form
\begin{center}
 \begin{tikzpicture}
  \node (P1) at (0,0) {${]0,1[}^{\times n}$};
  \node (P2) at (3,0) {$\Delta^{k_1} \times \ldots \times \Delta^{k_m}$};
  \node (P3) at (6,0) {$X_{n,g}$,};
  \draw
  (P1) edge[->] node[above] {\scriptsize $\bfDelta^{\bfk}$} (P2)
  (P2) edge[->] node[above] {\scriptsize $\bfsimp^{\bfk}$} (P3);
 \end{tikzpicture}
\end{center}
where $\bfDelta^{\bfk}$ identifies the open hypercube ${]0,1[}^{\times n}$ with the product of simplices $\Delta^{k_1} \times \ldots \times \Delta^{k_m}$, and $\bfsimp^{\bfk}$ embeds the latter into $X_{n,g}$. More precisely, for every integer $k \geqs 0$, we have a homeomorphism
\begin{align*}
 \bfDelta^k : {]0,1[}^{\times k} &\to \Delta^k \\*
 (t_1,\ldots,t_k) &\mapsto (\Delta^k_1(t_1),\ldots,\Delta^k_k(t_k)),
\end{align*}
where $\Delta^k_1(t_1) := t_1$ and, recursively, $\Delta^k_i(t_i) := t_i(1-\Delta^k_{i-1}(t_{i-1})) + \Delta^k_{i-1}(t_{i-1})$ for all integers $1 < i \leqs k$. Therefore, we can define the homeomorphism
\begin{align}
 \bfDelta^{\bfk} : {]0,1[}^{\times n} &\to \Delta^{k_1} \times \ldots \times \Delta^{k_m} \label{E:homeomorphism_hypercube_simplices} \\*
 (t_1,\ldots,t_n) &\mapsto (\Delta^{\bfk}_1(t_1),\ldots,\Delta^{\bfk}_n(t_n)), \nonumber
\end{align}
where, if $i = k_1 + \ldots + k_{\ell-1} + i_\ell$ for some integers $1 \leqs \ell \leqs m$ and $1 \leqs i_\ell \leqs k_\ell$, then $\Delta^{\bfk}_i := \Delta^{k_\ell}_{i_\ell}$, and similarly, we can define the embedding
\begin{align}
 \bfsimp^{\bfk} : \Delta^{k_1} \times \ldots \times \Delta^{k_m} &\to X_{n,g}, \label{E:embedding_simplices} \\*
 (t_1,\ldots,t_n) &\mapsto \{ \simp^{\bfk}_1(t_1),\ldots,\simp^{\bfk}_n(t_n) \} \nonumber
\end{align}
where, if $i = k_1 + \ldots + k_{\ell-1} + i_\ell$ for some integers $1 \leqs \ell \leqs m$ and $1 \leqs i_\ell \leqs k_\ell$, then $\simp^{\bfk}_i := \simp_\ell$. Since faces of simplices are sent either to infinity (that is, to one of the diagonals of $\varSigma_{g,1}^{\times n}$ corresponding to the collision of two coordinates, or particles) or to $Y_{n,g}$, the embedding $\bfsimp^{\bfk} \circ \bfDelta^{\bfk}$ determines a Borel--Moore cycle in $X_{n,g}$ relative to $Y_{n,g}$. Indeed, in Borel--Moore homology, open submanifolds whose boundary goes to infinity are always cycles. Then, we will denote the homology class of $\bfsimp^{\bfk} \circ \bfDelta^{\bfk}$ by 
\[
 \bfsimp(\bfk) \in H^\BM_n(X_{n,g},Y_{n,g}).
\]

Next, let us explain how to use threads to lift homology classes from $H^\BM_n(X_{n,g},Y_{n,g})$ to $H^\BM_n(X_{n,g},Y_{n,g};\varphi_{n,g}^\pi)$. Recall that, by definition, a thread of a $\bfk$-labeled $m$-multisimplex $\bfsimp$ determines a path $\thread : [0,1] \to X_{n,g}$ from $\myuline{\xi}$ to $\bfsimp^{\bfk}(\bfDelta^{\bfk}((\frac 12,\ldots,\frac 12))$. In other words, a thread is a point $\thread \in \tilde{X}_{n,g}$ in the fiber $\tilde{p}^{-1}(\bfsimp^{\bfk}(\bfDelta^{\bfk}(\frac 12,\ldots,\frac 12)))$. Notice that a thread also naturally determines a permutation $\perm_{\thread} \in \frakS_n$, together with an associated map 
\begin{align*}
 \perm_{\thread} : {]0,1[}^{\times n} &\to {]0,1[}^{\times n} \\*
 (t_1,\ldots,t_n) &\mapsto (t_{\perm_{\thread}(1)},\ldots,t_{\perm_{\thread}(n)}),
\end{align*}
where $\perm_{\thread}(i)$ is the unique integer satisfying $\thrcmp_i(1) = \simp^{\bfk}_{\perm_{\thread}(i)}(\Delta^{\bfk}_{\perm_{\thread}(i)}(\frac 12))$. Therefore, it determines uniquely a lift 
\[
 \tilde{\bfsimp}^{\bfk}_{\thread} : \Delta^{k_1} \times \ldots \times \Delta^{k_m} \to \tilde{X}_{n,g}
\]
of $\bfsimp^{\bfk} : \Delta^{k_1} \times \ldots \times \Delta^{k_m} \to X_{n,g}$ satisfying $\tilde{\bfsimp}^{\bfk}_{\thread}(\bfDelta^{\bfk}(\frac 12,\ldots,\frac 12)) = \thread$. Clearly, the embedding $\tilde{\bfsimp}^{\bfk}_{\thread} \circ \bfDelta^{\bfk} \circ \perm_{\thread}^{-1}$ still determines a Borel--Moore cycle in $\tilde{X}_{n,g}$ relative to $\tilde{p}^{-1}(Y_{n,g})$. Then, we will denote the corresponding homology class by 
\[
 \tilde{\bfsimp}_{\thread}(\bfk) \in H^\BM_n(X_{n,g},Y_{n,g};\varphi_{n,g}^\pi),
\]
although most of the time we will drop the thread $\thread$ from the notation.

\begin{remark}\label{R:thread_and_orientation}
 Notice that the thread $\thread$ serves a double role: not only it specifies a lift $\tilde{\bfsimp}^{\bfk}_{\thread}$ of the embedding $\bfsimp^{\bfk}$, but it also dictates, via the permutation $\perm_{\thread}$, how the components of the hypercube ${]0,1[}^{\times n}$ are distributed onto it. Thus, in $H^\BM_n(X_{n,g},Y_{n,g};\varphi_{n,g}^\pi)$, we have
 \[
  \tilde{\bfsimp}_{\thread'} (\bfk) = \sgn(\perm_{\thread} \circ \perm_{\thread'}^{-1}) (\thread' \ast \thread^{-1}) \cdot \tilde{\bfsimp}_{\thread} (\bfk),
 \]
 where $\sgn : \frakS_n \to \{ +1,-1 \}$ denotes the sign, or parity, of permutations.
\end{remark}

In order to keep our graphical notation as light as possible, let us adopt the following convention: whenever we label a component of a red thread by $k$, the picture should be understood as representing $k$ disjoint parallel red components in the surface, starting at $k$ consecutive points of $\braid{\xi}{}$, and ending at $k$ distinct neighboring points along a single component of a multisimplex. Such a labeled thread will be called a \textit{multithread}. Furthermore, from now on, whenever we draw a single unlabeled component of a red multithread ending at a point of a $k$-labeled component of a multisimplex (which represents therefore an embedded $k$-dimensional simplex), the red component should be understood as carrying the label $k$. For example, we have
\[
 \pic{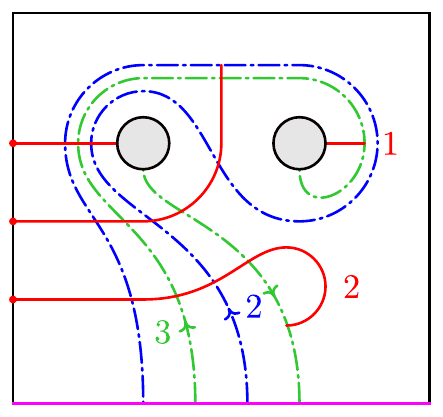} = \pic{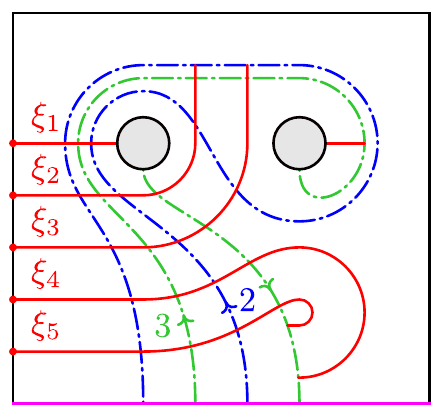}
\]

\begin{remark}
 In the notation of Remark~\ref{R:interpretation_of_homology_intermsof_covers}, an embedded twisted cycle also defines a relative homology class in $H^\BM_n(X_{n,g},Y_{n,g};\varphi_{n,g}^{\pi/K})$ given by $\hat{\bfsimp}^{\bfk}_{\thread} \circ \bfDelta^{\bfk} \circ \perm_{\thread}$, where
 \[
  \hat{\bfsimp}^{\bfk}_{\thread} : {]0,1[}^{\times n} \to \hat{X}_{n,g}
 \]
 denotes the unique lift of $\bfsimp^{\bfk} : {]0,1[}^{\times n} \to X_{n,g}$ sending $(\frac 12,\ldots,\frac 12)$ to the projection of $\thread$ to $\hat{X}_{n,g} = \tilde{X}_{n,g}/K_{n,g}$. This lift is clearly the projection of $\tilde{\bfsimp}^{\bfk}_{\thread}$ to $\hat{X}_{n,g}$.
\end{remark}

\subsubsection{Structure and bases}\label{S:structure_and_bases}

For all $\bfa=(a_1,\ldots,a_g),\bfb=(b_1,\ldots,b_g) \in \N^{\times g}$, let us set
\[
 \left| \bfa + \bfb \right| = \sum_{j=1}^g (a_j+b_j).
\]
If $\left| \bfa + \bfb \right| = n$, meaning that $(\bfa,\bfb)$ is a $2g$-partition of $n$, then let us consider the relative homology class 
\[
 \basis(\bfa,\bfb) \in H^\BM_n(X_{n,g},Y_{n,g};\varphi_{n,g}^\pi)
\]
corresponding to the embedded twisted cycle represented by
\begin{equation}\label{E:basis}
 \basis(\bfa,\bfb) := \pic{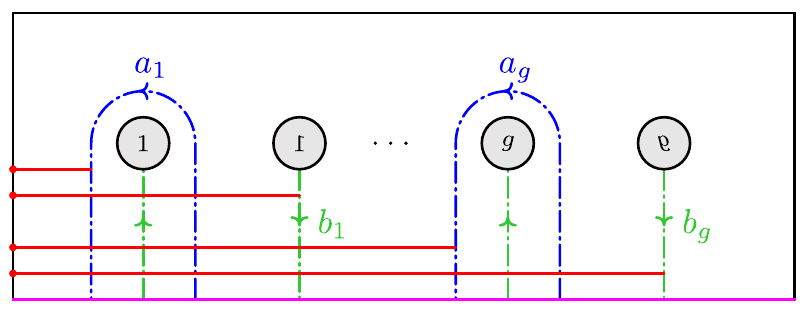}
\end{equation}
where green and blue colors are only used to clearly distinguish arcs. Notice that we will not explicitly refer to the multithread in the notation, in order to keep it as light as possible. The next statement is a direct generalization of \cite[Proposition~3.6]{M20}, which is based on \cite[Lemma~3.1]{B00}. The result was already proved in \cite[Theorem~A.(a)]{BPS21}, but we repeat here the sketch of a proof, for convenience of the reader.

\begin{proposition}\label{P:homology_structure}
 The $\Z[\pi_{n,g}]$-module $H^\BM_n(X_{n,g},Y_{n,g};\varphi_{n,g}^\pi)$ is free with basis
 \[
  \left\{ \basis(\bfa,\bfb) \Bigm|
  \bfa, \bfb \in \N^{\times g}, \left| \bfa + \bfb \right| = n \right\},
 \]
 while $H^\BM_k(X_{n,g},Y_{n,g};\varphi_{n,g}^\pi) = \{ 0 \}$ for all $k \neq n$. 
\end{proposition}

\begin{proof}[Sketch of proof] 
 The surface $\varSigma_{g,1}$ retracts onto a wedge of $2g$ circles, which can be embedded into $\varSigma_{g,1}$ as (the solid version of) the blue and green curves appearing in the definition of $\basis(\bfa,\bfb)$ above, together with the magenta interval $\partial_- \varSigma_{g,1}$. This gives a basis of $H^\BM_n(X_{n,g},Y_{n,g};\varphi_{n,g}^\pi)$ following Bigelow's trick \cite[Lemma~3.1]{B00}, by extending the retraction to configuration spaces, and applying excision.
\end{proof}

\begin{remark}
 Since $H^\BM_k(X_{n,g},Y_{n,g};\varphi_{n,g}^\pi)$ is a free $\Z[\pi_{n,g}]$-module for all $k \geqs 0$, then we have
 \[
  H^\BM_k(X_{n,g},Y_{n,g};\varphi^M_{n,g}) \cong H^\BM_k(X_{n,g},Y_{n,g};\varphi_{n,g}^\pi) \otimes_{\Z[\pi_{n,g}]} M_{n,g}
 \]
 for every left $\Z[\pi_{n,g}]$-module $M_{n,g}$ determined by $\varphi^M_{n,g} : \Z[\pi_{n,g}] \to \End_\Z(M_{n,g})$, see Remark~\ref{R:twisted_universal_coeff_free}. In other words, the relative twisted homology of $(X_{n,g},Y_{n,g})$ with coefficients in any left $\Z[\pi_{n,g}]$-module $M_{n,g}$ is also concentrated in degree $n$. Furthermore, if $M_{n,g}$ is a free $R$-module with basis $\{ m_1, \ldots, m_k \}$ over some ring $R$, and $\varphi^M_{n,g}$ takes values in $\End_R(M_{n,g})$ (meaning that $M_{n,g}$ is a $(Z[\pi_{n,g}],R)$-bimodule), then a basis of $H^\BM_n(X_{n,g},Y_{n,g};\varphi^M_{n,g})$ as a free $R$-module is given by 
 \[
  \left\{ \basis(\bfa,\bfb) \otimes m_i \Bigm| 
  \bfa, \bfb \in \N^{\times g}, \left| \bfa + \bfb \right| = n, 1 \leqs i \leqs k
  \right\}.
 \]
\end{remark}

\begin{remark}\label{R:interpretation_of_homology_intermsof_covers}
 If $\hat{p} : \hat{X}_{n,g} \to X_{n,g}$ denotes the regular cover corresponding to the normal subgroup $K_{n,g} \triangleleft \pi_{n,g}$ of Remark~\ref{R:quotient_group}, then the proof of Proposition~\ref{P:homology_structure} can be easily adapted to show that $H^\BM_k(X_{n,g},Y_{n,g};\varphi_{n,g}^{\pi/K})$ is a free $\Z[\dHeis_g]$-module for all $k \geqs 0$. This yields a natural isomorphism 
 \begin{align*}
  \homol_{n,g}^\dHeis &\cong H^\BM_n(X_{n,g},Y_{n,g};\varphi_{n,g}^{\pi/K}) \otimes_{\Z[\pi_{n,g}/K_{n,g}]} \Z[\dHeis_g].
 \end{align*}
 Roughly speaking, the twisted homology group $\homol_{n,g}^\dHeis$ is nothing else than the standard homology group of a regular cover of $X_{n,g}$ with extended coefficients (namely, we can multiply by a square root of $-\sigma$, called $q^{-1}$).
\end{remark}

\begin{corollary}\label{C:Heisenberg_basis}
 The $\Z[\dHeis_g]$-module $\homoldH_{n,g}$ is free with basis
 \[
  \left\{ \basis(\bfa,\bfb) \Bigm| 
  \bfa, \bfb \in \N^{\times g}, \left| \bfa + \bfb \right| = n \right\}.
 \]
\end{corollary}

\subsubsection{Diagrammatic calculus}\label{S:diagrammatic_rules}

The diagrammatic notation that was introduced in Section~\ref{S:diagrammatic_notation} comes with a set of rules. These are a diagrammatic translation of the rules of homological calculus, meaning that they allow us to manipulate diagrams without changing the corresponding twisted homology class. These rules are a direct generalization, to positive genus surfaces, of the ones derived in \cite[Section~4.2]{M20}. They provide a complete picture for the homological model developed in this paper, in the sense that they are sufficient to perform all the computations we will carry out in the following.

These diagrammatic rules involve \textit{quantum integers}, \textit{factorials}, and \textit{binomials}. If $\Z[q,q^{-1}]$ denotes the ring of integral Laurent polynomials in the free variable $q$, then, for every integer $n \in \Z$, and for all positive integers $k \geqs \ell \geqs 0$, we recall the notation 
\begin{align*}
 \{ n \}_q &:= q^n - q^{-n}, &
 [n]_q &:= \frac{\{ n \}_q}{\{ 1 \}_q}, &
 [k]_q! &:= \prod_{j=1}^k [j]_q, &
 \sqbinom{k}{\ell}_q &:= \frac{[k]_q!}{[\ell]_q![k-\ell]_q!}.
\end{align*}

\begin{proposition}\label{P:computation_rules}
 Diagrams representing embedded twisted cycles corresponding to elements of $\homol_{n,g}^\dHeis$ satisfy the following list of relations:
 \begin{description}
  \item[\normalfont Cutting rule.] For every integer $0 \leqs k \leqs n$ we have
   \[
    \pic{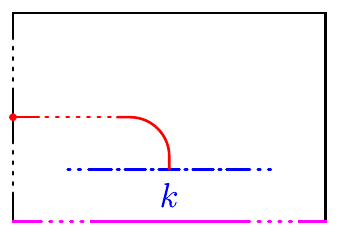} 
    = \sum_{\ell = 0}^k \pic{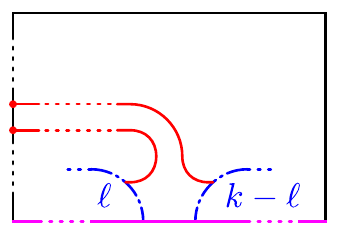} \tag{C}\label{E:cutting}
   \]
   where the pair of dotted red arcs on the right-hand side of the equality runs parallel in surface;
  \item[\normalfont Fusion rule.] For all integers $0 \leqs k,\ell,k+\ell \leqs n$ we have
   \[
    \pic{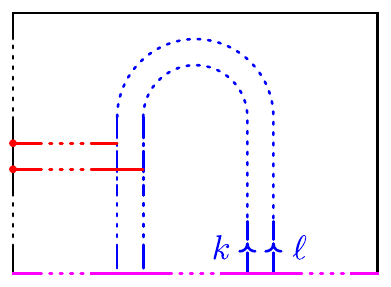} 
    = \pic{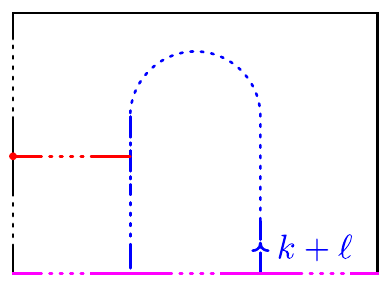} \otimes \sqbinom{k+\ell}{k}_q q^{k\ell} \tag{F}\label{E:fusion}
   \]
   where the pairs of dotted red arcs and of dotted blue arcs on the left-hand side of the equality run parallel in surface;
  \item[\normalfont Orientation rule.] For every integer $0 \leqs k \leqs n$ we have
   \[
    \pic{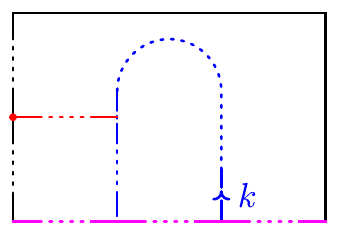} 
    = (-1)^k \pic{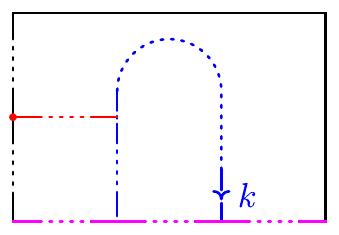} \tag{O}\label{E:orientation}
   \]
  \item[\normalfont Permutation rule.] For all integers $0 \leqs k,\ell \leqs n$ we have
   \[
    \pic{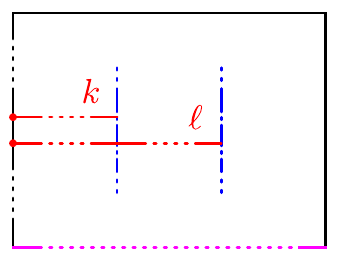} 
    = \pic{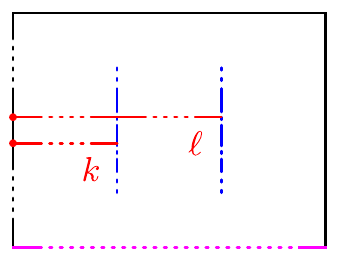} \otimes q^{2k\ell} \tag{P}\label{E:permutation}
   \]
   where the left-most pairs of dotted red arcs on both sides of the equality run parallel in surface;
  \item[\normalfont Braid rule.] For every integer $0 \leqs k \leqs n$ we have
   \[
    \pic{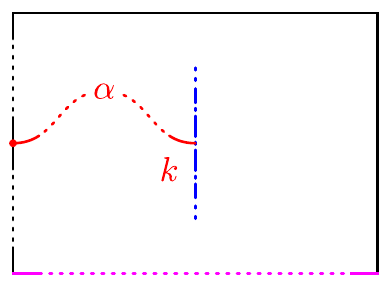} 
    = \pic{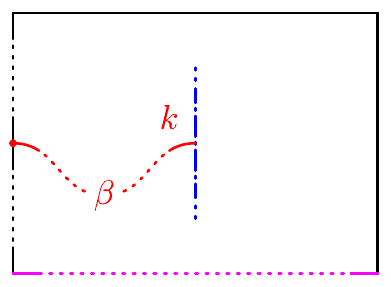} \otimes \varphi^\dHeis_{n,g}(\braid{(\beta \ast \alpha^{-1})}{(i)})^k \tag{B}\label{E:braid}
   \]
 \end{description}
 where the dotted red arcs $\alpha$ and $\beta$ are disjoint from the rest of the multithread, intersect each other only twice, transversely and with opposite sign, at $\alpha(0) = \beta(0) = \xi_i$ for some $1 \leqs i \leqs n$ and at $\alpha(1) = \beta(1)$, and where $\braid{(\beta \ast \alpha^{-1})}{(i)} \in \pi_{n,g}$ keeps all coordinates of $\braid{\xi}{}$ fixed except for the $i$th one, which moves along $\beta \ast \alpha^{-1}$.
\end{proposition}

These rules can be established in the exact same way as in \cite[Section~4.2]{M20}, where analogue relations are proved in the case of discs. Indeed, all modifications actually happen inside a disc embdedded into $\varSigma_{g,1}$ (only the last rule requires a remark on surface braids). Still, we give elements of proofs in Appendix~\ref{A:graphs_and_rules}. We will make extensive use of these rules while performing computations, and we will always add the initials of the rules that are being applied on top of every equality involving them.

\subsection{Homological representations of quantum \texorpdfstring{$\fsl_2$}{sl(2)}}\label{S:homological_action_of_quantum_sl2}

For a fixed genus $g$, we will now equip the direct sum of the Heisenberg homology groups $\homol_{n,g}^\dHeis$ (taken over all possible numbers of configuration points) with an action of the quantum group of $\slt$. This representation of quantum $\fsl_2$ will be defined in terms of homological operators, and computed by means of homological calculus. Later, we will show that a $\Z[\zeta]$-linear version of this representation contains a finite-dimensional submodule that will be identified with the adjoint representation of the small version of quantum $\fsl_2$.

\subsubsection{Quantum \texorpdfstring{$\fsl_2$}{sl(2)}}\label{S:quantum_sl2_algebra}

Let us consider the ring $\Z[q,q^{-1}]$ of integral Laurent polynomials in the free variable $q$. Let $U_q = U_q \fsl_2$ be the $\Z[q,q^{-1}]$-algebra with generators $\{ E,F^{(k)},K,K^{-1} \mid k \in \N \}$ and relations
\begin{gather*}
 F^{(k)}F^{(\ell)} = \sqbinom{k+\ell}{k}_q F^{(k+\ell)}, \qquad K K^{-1} = K^{-1} K = 1, \\*
 K E K^{-1} = q^2 E, \qquad K F^{(k)} K^{-1} = q^{-2k} F^{(k)}, \qquad
 [E,F^{(k+1)}] = F^{(k)}(q^{-k}K - q^kK^{-1}).
\end{gather*}
Notice that
\begin{equation}\label{E:divided_power}
 (F^{(1)})^k = [ k ]_q! F^{(k)}.
\end{equation}

This is an integral version of quantum $\fsl_2$, in the sense that it has coefficients in a ring of integral Laurent polynomials, rather than in a field. It lies somewhat between Lusztig's integral version \cite{Lu88} and De Concini--Procesi's one \cite{DP93}, since it involves divided powers of the standard generator $F$, which are denoted $F^{(k)}$, but not those of $E$. It was introduced by Habiro in \cite{Ha06} (up to a minor difference), and was already related to homology (in the precise form reported here) in \cite{JK09,M20}.

Let us also consider the primitive $r$th root of unity $\zeta = e^{\frac{2 \pi \fraki}{r}}$, where $r \geqs 3$ is an odd integer, and the $\Z[\zeta]$-algebra $U_\zeta = U_q \otimes \Z[\zeta]$ obtained by specializing $q = \zeta$. Let $\bar{U}_\zeta = \bar{U}_\zeta \fsl_2$ be the $\Z[\zeta]$-subalgebra of $U_\zeta$ generated by $\{ E,F^{(1)},K \}$. Notice that, in $\bar{U}_\zeta$, Equation~\eqref{E:divided_power} implies $(F^{(1)})^r = 0$, since $[r]_\zeta! = 0$.

\subsubsection{Homological action of quantum \texorpdfstring{$\fsl_2$}{sl(2)} generators}\label{S:homological_quantum_sl2_operators}

Let us construct linear operators $\calE$, $\calF^{(1)}$, and $\calK$ on the total discrete Heisenberg homology $\homol_g^\dHeis$ following the pattern
\begin{center}
 \begin{tikzpicture}[descr/.style={fill=white}]
  \node (P1) at (0,0) {$\cdots$};
  \node (P2) at (2.5,0) {$\homol_{n-1,g}^\dHeis$};
  \node (P3) at (5,0) {$\homol_{n,g}^\dHeis$};
  \node (P4) at (7.5,0) {$\homol_{n+1,g}^\dHeis$};
  \node (P5) at (10,0) {$\cdots$};
  \draw
  (P2) edge[->,bend right] node[above] {\scriptsize $\calE$} (P1)
  (P3) edge[->,bend right] node[above] {\scriptsize $\calE$} (P2)
  (P4) edge[->,bend right] node[above] {\scriptsize $\calE$} (P3)
  (P2) edge[->,bend right] node[below] {\scriptsize $\calF^{(1)}$} (P3)
  (P3) edge[->,bend right] node[below] {\scriptsize $\calF^{(1)}$} (P4)
  (P4) edge[->,bend right] node[below] {\scriptsize $\calF^{(1)}$} (P5)
  (P2) edge[->,loop above] node[above] {\scriptsize $\calK$} (P2)
  (P3) edge[->,loop above] node[above] {\scriptsize $\calK$} (P3)
  (P4) edge[->,loop above] node[above] {\scriptsize $\calK$} (P4);
 \end{tikzpicture}
\end{center}

For every $n \geqs 0$, let us consider the triple $X_{n,g} \supset Y_{n,g} \supset Z_{n,g}$, where
\[
 Z_{n,g} := \{ \underline{x} = \{ x_1,\ldots,x_n \} \in Y_{n,g} \mid \Exists 1 \leqs i < j \leqs n \quad x_i, x_j \in \partial_- \varSigma_{g,1} \}.
\]
Now let $x_0$ be the leftmost point of $\partial_-\varSigma_{g,1}$, and let us consider the map
\begin{align*}
 \add_{x_0} : X_{n-1,g} / Y_{n-1,g} &\to Y_{n,g}/ Z_{n,g} \\*
 [ \braid{x}{} ] &\mapsto 
 \begin{cases}
  [ \{ x_0 \} \cup \braid{x}{} ] & \mbox{ if } \braid{x}{} \in X_{n-1,g} \smallsetminus Y_{n-1,g}, \\
  [ Z_{n,g} ] & \mbox{ if } \braid{x}{} \in Y_{n-1,g}.
 \end{cases} 
\end{align*}
Notice that $\add_{x_0}$ is a homeomorphism. We want to lift it to (the corresponding quotient of) the universal cover $\tilde{p}: \tilde{X}_{n,g} \to X_{n,g}$, whose points $\braid{\tilde{x}}{}$ are isotopy classes (relative to endpoints) of configurations of paths starting at $\myuline{\xi} \in \partial X_{n,g}$. Therefore, since we need to choose a point in the fiber of $x_0$, we let $\tilde{x}_0$ denote the simple path in $\partial \varSigma_{g,1}$ starting at $\xi_n$, ending at $x_0$, and running counterclockwise (up to parametrization it is unique). Then, let us consider the map
\begin{align*}
 \add_{\tilde{x}_0} : \tilde{X}_{n-1,g} / \tilde{p}^{-1}(Y_{n-1,g}) &\to \tilde{p}^{-1}(Y_{n,g}) / \tilde{p}^{-1}(Z_{n,g}) \\*
 [ \braid{\tilde{x}}{} ] &\mapsto 
 \begin{cases}
  [ \{ \tilde{x}_0 \} \cup \braid{\tilde{x}}{} ] & \mbox{ if } \braid{\tilde{x}}{} \in \tilde{X}_{n-1,g} \smallsetminus \tilde{p}^{-1}(Y_{n-1,g}), \\
  [ \tilde{p}^{-1}(Z_{n,g}) ] & \mbox{ if } \braid{\tilde{x}}{} \in \tilde{p}^{-1}(Y_{n-1,g}).
 \end{cases} 
\end{align*} 
Notice that we can avoid parametrization issues by requiring, without loss of generality, that every configuration of paths does not to intersect the boundary of $X_{n,g}$ outside of its ends. The map $\add_{\tilde{x}_0}$ is then a lift of $\add_{x_0}$ to universal covers, and is again a homeomorphism.

The long exact sequence of the triple considered above is
\begin{center}
 \begin{tikzpicture}[descr/.style={fill=white}]
  \node[anchor=west] (P1) at (0,0) {$\cdots$};
  \node (P2) at (3,0) {$H^\BM_n(X_{n,g},Z_{n,g};\varphi_{n,g}^{\dHeis})$};
  \node (P3) at (6,0) {$\homol_{n,g}^{\dHeis}$};
  \node (P4) at (9,0) {$H^\BM_{n-1}(Y_{n,g},Z_{n,g};\varphi_{n,g}^{\dHeis})$};
  \node[anchor=east] (P5) at (12,0) {$\cdots$};
  \draw
  (P1) edge[->] (P2)
  (P2) edge[->] (P3)
  (P3) edge[->] node[above] {\scriptsize $\partial_*$} (P4)
  (P4) edge[->] (P5);
 \end{tikzpicture}
\end{center}
where $\partial_*$ denotes the connection homomophism, and $H^\BM_\ast(Y_{n,g},Z_{n,g};\varphi_{n,g}^{\dHeis})$ denotes the homology of the complex
\[
 C^\BM_\ast(Y_{n,g},Z_{n,g};\varphi_{n,g}^{\dHeis}) := \varprojlim_{K \in \calK(Y_{n,g})} C_\ast(\tilde{p}^{-1}(Y_{n,g}),\tilde{p}^{-1}(Z_{n,g} \cup (Y_{n,g} \smallsetminus K))) \otimes_{\Z[\pi_{n,g}]} \Z[\dHeis_g],
\]
in the notation of Appendix~\ref{A:Borel--Moore}.

\begin{definition}\label{D:homological_E}
 For every integer $n \geqs 1$, the operator
 \[
  \calE : \homol_{n,g}^\dHeis \to \homol_{n-1,g}^\dHeis
 \]
 is the composition
 \begin{center}
  \begin{tikzpicture}[descr/.style={fill=white}]
   \node (P1) at (0,0) {$\homol_{n,g}^\dHeis$};
   \node (P2) at (4,0) {$H^\BM_{n-1}(Y_{n,g},Z_{n,g};\varphi_{n,g}^{\dHeis})$};
   \node (P3) at (8,0) {$\homol_{n-1,g}^{\dHeis},$};
   \draw
   (P1) edge[->] node[above] {\scriptsize $(-1)^{n-1} \partial_*$} (P2)
   (P2) edge[->] node[above] {\scriptsize $\del_{\tilde{x}_0}$} (P3);
  \end{tikzpicture}
 \end{center}
 where $\del_{\tilde{x}_0}$ is the unique map fitting in the commutative diagram of isomorphisms
 \begin{center}
  \begin{tikzpicture}[descr/.style={fill=white}]
   \node (P1) at (0,0) {$H^\BM_{n-1}(Y_{n,g},Z_{n,g};\varphi_{n,g}^{\dHeis})$};
   \node (P2) at (6,0) {$H^\BM_{n-1}(Y_{n,g}/Z_{n,g},Z_{n,g}/Z_{n,g};\varphi_{n,g}^{\dHeis})$};
   \node (P3) at (0,-1.5) {$\homol_{n-1,g}^{\dHeis}$};
   \node (P4) at (6,-1.5) {$H^\BM_{n-1}(X_{n-1,g}/Y_{n-1,g},Y_{n-1,g}/Y_{n-1,g};\varphi_{n-1,g}^{\dHeis})$};
   \draw
   (P1) edge[->] node[above] {$\sim$} (P2)
   (P1) edge[->] node[left] {\scriptsize $\del_{\tilde{x}_0}$} (P3)
   (P3) edge[->] node[above] {$\sim$} (P4)
   (P2) edge[->] node[right] {\scriptsize $(\add_{\tilde{x}_0}^{-1})_*$} (P4);
  \end{tikzpicture}
 \end{center}
 with horizontal arrows given by identifications of relative homologies of good pairs.
\end{definition}

Roughly speaking, the operator $\calE$ computes the boundary map $\partial_*$, followed by some standard homological identifications.

\begin{definition}\label{D:homological_F^(k)}
 For all integers $n,k \geqs 0$, the operator
 \[
  \calF^{(k)} : \homol_{n,g}^\dHeis \to \homol_{n+k,g}^\dHeis
 \]
 is the $\Z[\dHeis_g]$-linear map sending the basis vector $\tilde{\bfsimp}(\bfa,\bfb)$ of $\homol_{n,g}^\dHeis$ to the vector 
 \[
  \pic{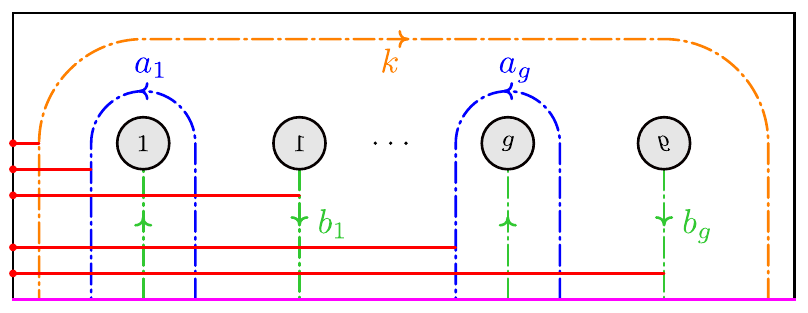} \otimes q^{\frac{k(k-1)}{2}+2kg}
 \]
 of $\homol_{n+k,g}^\dHeis$.
\end{definition}

Notice that the definition of the operator $\calF^{(k)}$ is actually independent of the choice of the basis of $\homol_{n,g}^\dHeis$. 
Indeed, if $\partial_+ \varSigma_{g,1}$ denotes the subset of $\partial \varSigma_{g,1}$ represented by the orange arc
\[
 \pic{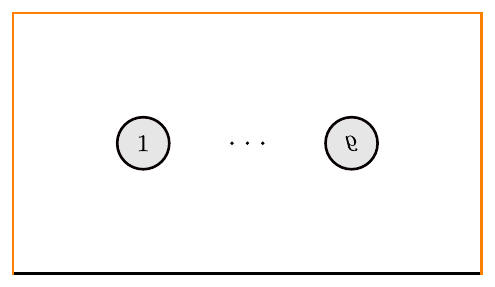}
\]
then, up to isotopy, we can make sure that every relative homology class in $\homol_{n,g}^\dHeis$ avoids a fixed collar of the submanifold
\[
 \partial X_{n,g} \smallsetminus Y_{n,g} = \{ \underline{x} = \{ x_1,\ldots,x_n \} \in X_{n,g} \mid \Exists 1 \leqs i \leqs n \quad x_i \in \partial_+ \varSigma_{g,1} \}
\]
in $X_{n,g}$, and we can insert the curve defining the operator $\calF^{(k)}$ inside such fixed collar.

\begin{definition}\label{D:homological_K}
 For every integer $n \geqs 0$, the operator 
 \[
  \calK : \homol_{n,g}^\dHeis \to \homol_{n,g}^\dHeis
 \]
 is the scalar multiple of the identity sending $\tilde{\chain} \otimes h$ to $\tilde{\chain} \otimes q^{-2(n+g)} h$ for all homology classes $\tilde{\chain} \in H^\BM_n(X_{n,g},Y_{n,g};\varphi_{n,g}^\pi)$ and coefficients $h \in \Z[\dHeis_g]$.
\end{definition}

\subsubsection{Proof of quantum \texorpdfstring{$\fsl_2$}{sl(2)} relations}

Now that homological operators $\calE, \calF^{(k)}$ and $\calK$ have been defined, we can show that they satisfy relations for generators of $U_q$.

\begin{proposition}\label{P:divided_power_relation}
 For all integers $k,\ell \geqs 0$ we have
 \[
  \calF^{(k)} \circ \calF^{(\ell)} = \sqbinom{k+\ell}{k}_q \calF^{(k+\ell)}.
 \]
\end{proposition}

\begin{proof}
 This is an immediate consequence of the fusion rule in Proposition~\ref{P:computation_rules}. Indeed, we have
 \begin{align*}
  &\pic{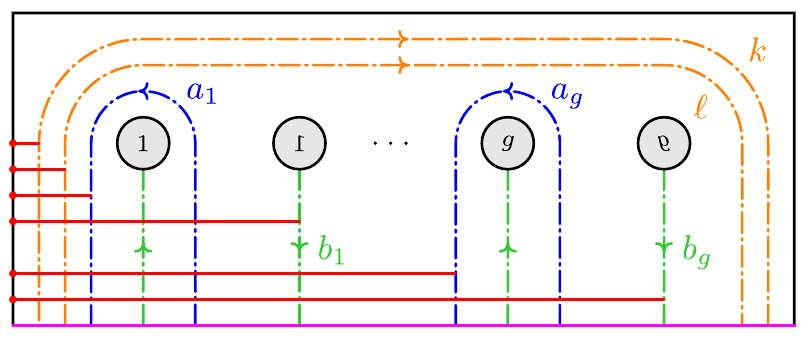} \otimes q^{\frac{k(k-1)}{2}+2kg+\frac{\ell(\ell-1)}{2}+2 \ell g} \\
  &\hspace*{\parindent} \Feq \pic{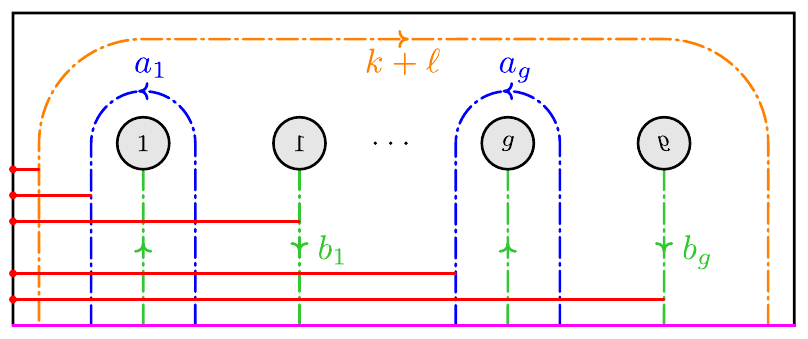} \otimes \sqbinom{k+\ell}{k}_q q^{\frac{(k+\ell)(k+\ell-1)}{2}+2(k+\ell)g}. \qedhere 
 \end{align*}
\end{proof}

\begin{remark}\label{R:easy_relations}
 Notice that relations
 \begin{align*}
  \calK \circ \calK^{-1} = \calK^{-1} \circ \calK &= \id_{\homol_{n,g}^\dHeis}, \\
  \calK \circ \calE \circ \calK^{-1} &= q^2 \calE, \\
  \calK \circ \calF^{(k)} \circ \calK^{-1} &= q^{-2k} \calF^{(k)}
 \end{align*}
are clearly satisfied for every integer $k \geqs 0$.
\end{remark}

\begin{proposition}\label{P:fundamental_relation}
 For every integer $k \geqs 0$ we have
 \[
  \calE \circ \calF^{(k+1)} - \calF^{(k+1)} \circ \calE = \calF^{(k)} \circ \left( q^{-k} \calK - q^k \calK^{-1} \right).
 \]
\end{proposition}

\begin{proof}
 If we denote by
 \[
  \tilde{\Phi}(k+1) \times \tilde{\bfsimp}(\bfa,\bfb) : {]0,1[}^{\times (n+k+1)} \to \tilde{X}_{n,g}
 \]
 the relative homology class satisfying
 \[
  \calF^{(k+1)}(\tilde{\bfsimp}(\bfa,\bfb)) = \left( \tilde{\Phi}(k+1) \times \tilde{\bfsimp}(\bfa,\bfb) \right) \otimes q^{\frac{(k+1)k}{2}+2(k+1)g},
 \]
 then we have
 \begin{align*}
  &\partial_* \left( \tilde{\Phi}(k+1) \times \tilde{\bfsimp}(\bfa,\bfb) \right) 
  = \partial_*(\tilde{\Phi}(k+1)) \times \tilde{\bfsimp}(\bfa,\bfb) 
  + (-1)^{k+1} \tilde{\Phi}(k+1) \times \partial_*(\tilde{\bfsimp}(\bfa,\bfb)). 
 \end{align*}
 This means that 
 \begin{align*}
  &\left( \calE \circ \calF^{(k+1)} - \calF^{(k+1)} \circ \calE \right)(\tilde{\bfsimp}(\bfa,\bfb)) \\
  &\hspace*{\parindent}= (-1)^{n+k} \del_{\tilde{x}_0} \left( \partial_* \left( \tilde{\Phi}(k+1) \times \tilde{\bfsimp}(\bfa,\bfb) \right) \right) \otimes q^{\frac{(k+1)k}{2}+2(k+1)g} \\*
  &\hspace*{2\parindent}- \tilde{\Phi}(k+1) \times (-1)^{n-1} \del_{\tilde{x}_0} (\partial_*(\tilde{\bfsimp}(\bfa,\bfb))) \otimes q^{\frac{(k+1)k}{2}+2(k+1)g} \\
  &\hspace*{\parindent}= (-1)^{n+k} \del_{\tilde{x}_0} \left( \partial_* \left( \tilde{\Phi}(k+1) \times \tilde{\bfsimp}(\bfa,\bfb) \right) - (-1)^{k+1} \tilde{\Phi}(k+1) \times \partial_*(\tilde{\bfsimp}(\bfa,\bfb)) \right) \otimes q^{\frac{(k+1)k}{2}+2(k+1)g} \\
  &\hspace*{\parindent}= (-1)^{n+k} \del_{\tilde{x}_0} \left( \partial_*(\tilde{\Phi}(k+1)) \times \tilde{\bfsimp}(\bfa,\bfb) \right) \otimes q^{\frac{(k+1)k}{2}+2(k+1)g}. 
 \end{align*}
 It remains to compute $\partial_*(\tilde{\Phi}(k+1))\times \tilde{\bfsimp}(\bfa,\bfb)$, and to apply $\del_{\tilde{x}_0}$ to it.

 On one hand, we have
 \begin{align*}
  &\pic{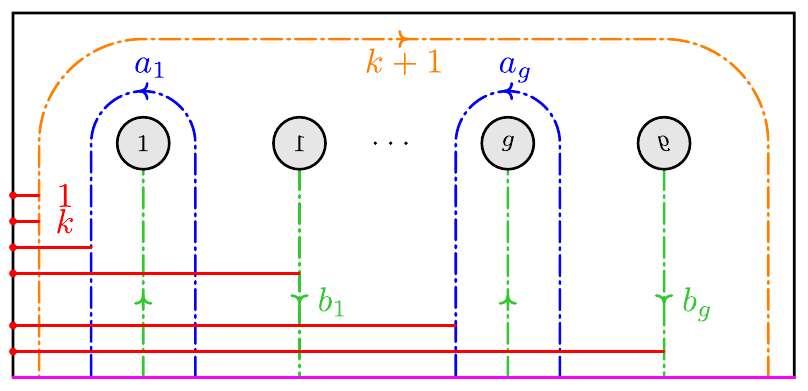} \\
  &\hspace*{\parindent} = \pic{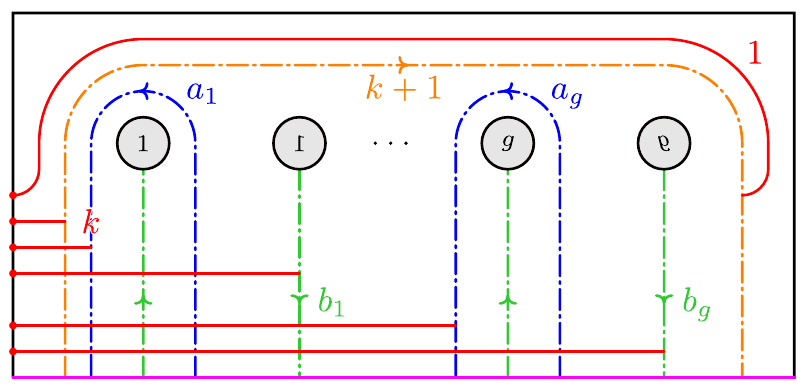} \\
  &\hspace*{\parindent} \Beq \pic{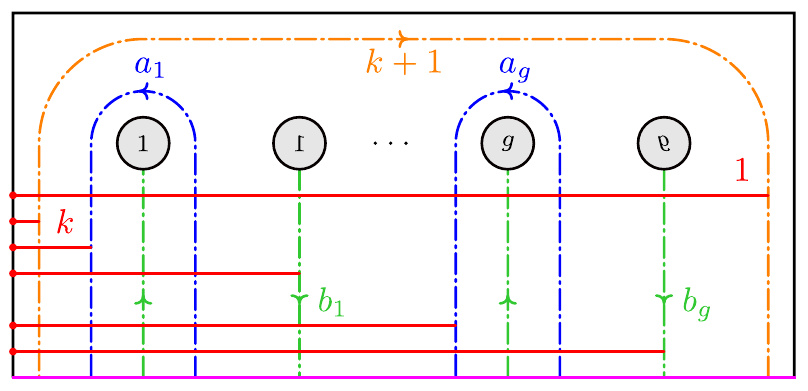} \otimes \prod_{j=1}^g [\beta_j^{-1},\alpha_j^{-1}] \\
  &\hspace*{\parindent} \Peq \pic{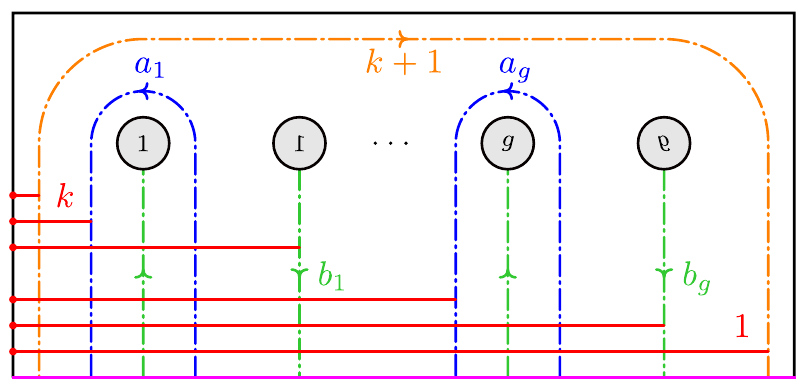} \otimes q^{-2(n+k+2g)},
 \end{align*}
 where the second and third equalities follow from an application of the braid and permutation rules in Proposition~\ref{P:computation_rules}, respectively.

 On the other hand, we have
 \begin{align*}
  &\pic{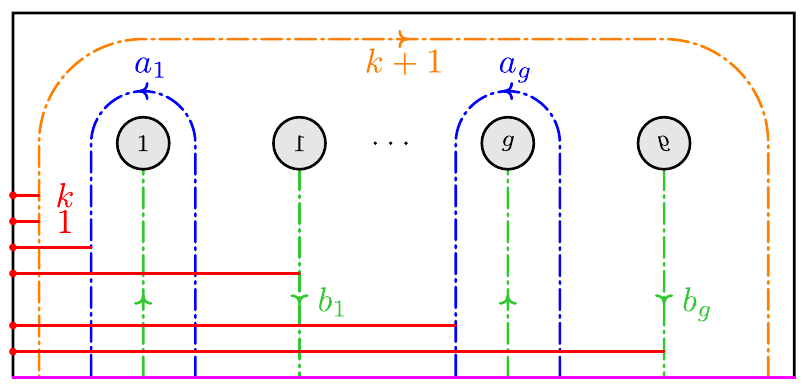} \\
  &\hspace*{\parindent} \Peq \pic{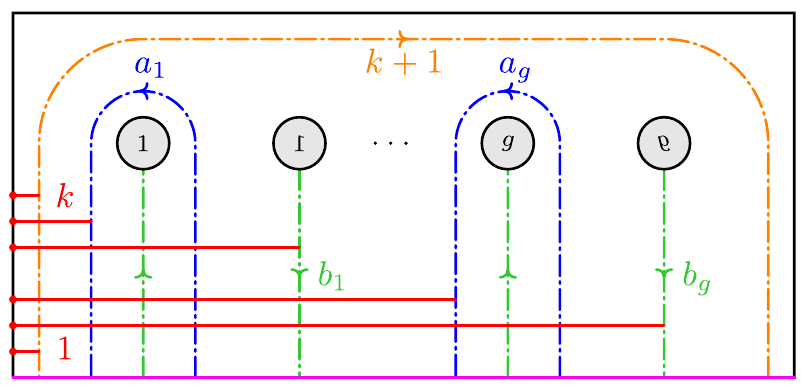} \otimes q^{2n},
 \end{align*}
 by applying the permutation rule in Proposition~\ref{P:computation_rules}.

 In the end, we have
 \begin{align*}
  &(\calE \circ \calF^{(k+1)} - \calF^{(k+1)} \circ \calE)(\tilde{\bfsimp}(\bfa,\bfb))  
  = \calF^{(k)} \left( \tilde{\bfsimp}(\bfa,\bfb) \otimes q^{k+2g} (q^{-2(n+k+2g)}-q^{2n}) \right) \\
  &\hspace*{\parindent} = \calF^{(k)} \left( \tilde{\bfsimp}(\bfa,\bfb) \otimes (q^{-k-2(n+g)} - q^{k+2(n+g)}) \right) . \qedhere
 \end{align*}
\end{proof}

\begin{theorem}\label{T:Uq_homological_representation}
 $\homol_g^\dHeis$ is a $U_q$-module, with action of $E$, $F^{(k)}$, and $K$ given by the operators $\calE$, $\calF^{(k)}$, and $\calK$ respectively.
\end{theorem}

\begin{proof}
 Comparing with the definition of $U_q$ given in Section~\ref{S:quantum_sl2_algebra}, the first claim is a direct consequence of Propositions~\ref{P:divided_power_relation} and \ref{P:fundamental_relation}, together with Remark~\ref{R:easy_relations}.
\end{proof}

Notice that, since the operators $\calE$, $\calF^{(k)}$, and $\calK$ are $\Z[\dHeis_g]$-linear, $\homol_g^\dHeis$ is in fact a $(\Z[\dHeis_g],U_q)$-bimodule.

\subsection{Homological representations of mapping class groups}\label{S:homological_representations_of_mcg}

Let us consider the mapping class group of $\varSigma_{g,1}$, denoted $\Mod(\varSigma_{g,1})$. By definition, it is the group of positive self-diffeomorphisms of $\varSigma_{g,1}$ fixing the boundary pointwise, modulo isotopies fixing the boundary pointwise (see Appendix~\ref{A:twisted_homological_action_of_mcg}). As proved in \cite{L64}, it is generated by the positive Dehn twists
\[
 \{ \twist{\alpha}{j}, \twist{\beta}{j}, \twist{\gamma}{k} \mid 1 \leqs j \leqs g, 1 \leqs k \leqs g-1 \}
\]
along the simple closed curves
\begin{gather*}
 \pic{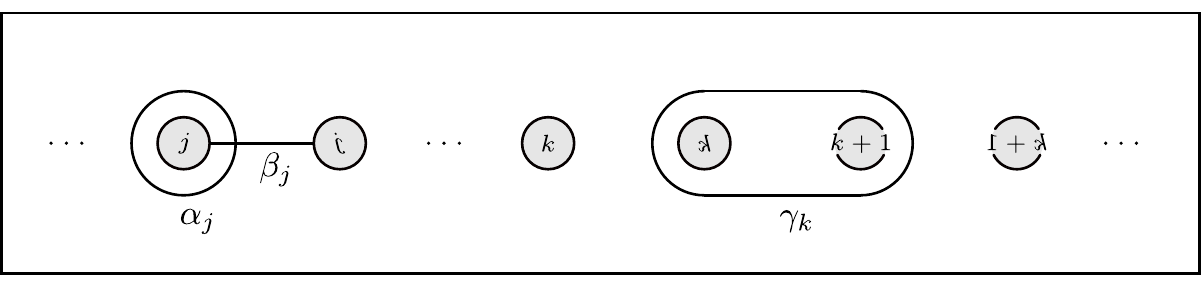}
\end{gather*}
In this section, we will linearize the Heisenberg homology group $\homol_{n,g}^\dHeis$, and equip it with a projective action of the mapping class group $\Mod(\varSigma_{g,1})$.

\subsubsection{Mapping class subgroup representations with Heisenberg coefficients}\label{S:Heisenberg_homological_representations}

First of all, notice that $\Mod(\varSigma_{g,1})$ embeds into $\Mod(\tilde{X}_{n,g})$ via
\[
 \begin{array}{ccccc}
  \Mod(\varSigma_{g,1}) & \hookrightarrow & \Mod(X_{n,g}) & \hookrightarrow & \Mod(\tilde{X}_{n,g}), \\
  f & \mapsto & f^{\times n} & \mapsto & \tilde{f}^{\times n}
 \end{array}
\]
where, by abuse of notation, $f^{\times n}\{ x_1,\ldots,x_n \} := \{ f(x_1),\ldots,f(x_n) \}$ for all $\{ x_1,\ldots,x_n \} \in X_{n,g}$, and $\tilde{f}^{\times n}$ denotes the unique lift of $f^{\times n} \circ \tilde{p}$ fixing the constant path in the fiber $\tilde{p}^{-1}(\braid{\xi}{})$. This induces an action of $\Mod(\varSigma_{g,1})$ onto both $\Z[\pi_{n,g}]$ and $H^\BM_n(X_{n,g},Y_{n,g};\varphi_{n,g}^\pi)$. For every $f \in \Mod(\varSigma_{g,1})$, let 
\begin{align*}
 (f^{\times n})_* &\in \Aut(\Z[\pi_{n,g}]), & 
 (\tilde{f}^{\times n})_* &\in \Aut(H^\BM_n(X_{n,g},Y_{n,g};\varphi_{n,g}^\pi))
\end{align*} 
denote the associated automorphisms of $\Z[\pi_{n,g}]$ and of $H^\BM_n(X_{n,g},Y_{n,g};\varphi_{n,g}^\pi)$ respectively. For all integers $n \geqs 1$, let us also define
\begin{align*}
 \calM_{0,g}^\dHeis 
 &:= \Mod(\varSigma_{g,1}), &
 \calM_{n,g}^\dHeis 
 &:= \left\{ f \in \Mod(\varSigma_{g,1}) \mid \phi_{n,g}^\dHeis \circ (f^{\times n})_* = \phi_{n,g}^\dHeis \right\}, &
 \calM_g^\dHeis 
 &:= \bigcap_{n \geqs 0} \calM_{n,g}^\dHeis.
\end{align*}

\begin{proposition}\label{P:mcg_action_on_dHeis}
 There exists a $\Z[\dHeis_g]$-linear action of $\calM_{n,g}^\dHeis$ on the $n$th discrete Heisenberg homology $\homol_{n,g}^\dHeis$ of Definition~\ref{D:Heisenberg_homology} given by the homomorphism
 \[
  \rho_{n,g}^{\dHeis} : \calM_{n,g}^\dHeis \to \GL_{\Z[\dHeis_g]}(\homol_{n,g}^\dHeis)
 \]
 defined by 
 \[
  \rho_{n,g}^{\dHeis}(f)(\tilde{\chain} \otimes h) = (\tilde{f}^{\times n})_*(\tilde{\chain}) \otimes h
 \]
 for all diffeomorphisms $f \in \calM_{n,g}^\dHeis$, homology classes $\tilde{\chain} \in H^\BM_n(X_{n,g},Y_{n,g};\varphi_{n,g}^\pi)$, and coefficients $h \in \Z[\dHeis_g]$. Furthermore, the direct sum, for $n \geqs 0$, of these representations defines a $\Z[\dHeis_g]$-linear action
 \[
  \rho_g^\dHeis : \calM_g^\dHeis \to \GL_{U_q}(\homol_g^\dHeis)
 \]
 of $\calM_g^\dHeis$ on the total discrete Heisenberg homology $\homol_g^\dHeis$ by invertible $U_q$-module endomorphisms.
\end{proposition}

\begin{proof}
 The first claim follows directly from Proposition~\ref{P:specializing}.$(i)$, by letting the embedding $\Mod(\varSigma_{g,1}) \hookrightarrow \Mod(X_{n,g})$ play the role of the homomorphism $\chi : F \to \Mod(X)$, and the group ring $\Z[\dHeis_g]$ play both the role of the ring $R$ and of the $R$-module $M$.
 
 The fact that the action of $\calM_g^\dHeis$ commutes with the action of $U_q$ follows immediately from the homological definition of the operators $\calE$, $\calF^{(k)}$, and $\calK$, in Section~\ref{S:homological_quantum_sl2_operators}. Indeed, the first one is essentially the connection homomorphism in the long exact sequence of a triple of spaces which is preserved by every diffeomorphism in $\calM_g^\dHeis$. Similarly, the second one essentially amounts to a cross product with the homology class of a multisimplex contained in a collar of the boundary, which is also fixed by every diffeomorphism in $\calM_g^\dHeis$. Finally, the operator $\calK$ acts by a scalar on every subrepresentation $\homol_{n,g}^{\dHeis}$.
\end{proof}

\begin{remark}\label{R:extensions_of_Chillingworth_action}
 Let us provide a few comments about this action.
 \begin{enumerate}
  \item The group $\calM^\dHeis_g$ is identified with the \textit{Chillingworth} subgroup of $\Mod(\varSigma_{g,1})$ in \cite[Proposition~27]{BPS21}. The same definition given above yields an action of the whole mapping class group that is not $\Z[\dHeis_g]$-linear, but rather a \textit{crossed} representation, as explained in \cite[Theorem~A.(b)]{BPS21}.
  \item By enlarging coefficients to $\Z[\dHeis_g \rtimes \Aut(\dHeis_g)]$, it is possible to extend this action to a $\Z[\dHeis_g \rtimes \Aut(\dHeis_g)]$-linear representation of $\Mod(\varSigma_{g,1})$. The procedure is explained in \cite{DM22}, where a tautological $\bbC$-linearization, induced by a $\bbC$-linear representation of $\dHeis_g \rtimes \Aut(\dHeis_g)$, is also considered. This naturally yields a $\bbC$-linear representation of $\Mod(\varSigma_{g,1})$.
  \item Even though the $\bbC$-linear representation of $\dHeis_g \rtimes \Aut(\dHeis_g)$ of the previous point is faithful, it is also shown in \cite{DM22} that a faithful $\bbC$-linear representation of $\Z[\dHeis_g]$ cannot exist. This implies that the process of $\bbC$-linearization might increase the kernel of the mapping class group representation. 
 \end{enumerate}
\end{remark}

\subsubsection{Mapping class group representations with complex coefficients}\label{S:linear_homological_representations}

Let us now consider the primitive $r$th root of unity $\zeta = e^{\frac{2 \pi \fraki}{r}}$, where $r \geqs 3$ is an odd integer. Our next goal is to obtain a $\fld$-vector space from $\homol_{n,g}^{\dHeis}$ by representing $\Z[\dHeis_g]$, and to equip it with a (projective) $\fld$-linear action of the whole mapping class group $\Mod(\varSigma_{g,1})$, rather than just a subgroup. We will first look for a representation of $\dHeis_g$, which will induce a representation of the group ring $\Z[\pi_{n,g}]$ into $\End_{\fld}(V_g)$ for some finite-dimensional $\fld$-vector space $V_g$, and then we will look for a homomorphism from $\Mod(\varSigma_{g,1})$ to $\PGL_{\fld}(V_g)$. This will extend the $\Z[\dHeis_g]$-linear representation of $\calM^\dHeis_{n,g}$ defined in Proposition~\ref{P:mcg_action_on_dHeis} to a $\fld$-linear projective one of $\Mod(\varSigma_{g,1})$.

When $n=1$, let us consider the matrices $A_1^{(1)},B_1^{(1)} \in M_{r \times r}(\Z[\zeta])$ defined as
\begin{align}
 A_1^{(1)} &:= \left(
 \begin{matrix}
  1 & 0 & \cdots & 0 \\
  0 & \zeta^4 & \ddots & \vdots \\
  \vdots & \ddots & \ddots & 0 \\
  0 & \cdots & 0 & \zeta^{4(r-1)}
 \end{matrix} \right), &
 B_1^{(1)} &:= \left(
 \begin{matrix}
  0 & 0 & \cdots & 0 & 1 \\
  1 & 0 & \ddots & & 0 \\
  0 & \ddots & \ddots & \ddots & \vdots \\
  \vdots & \ddots & \ddots & 0 & 0 \\
  0 & \cdots & 0 & 1 & 0
 \end{matrix} \right). \label{E:A_1_1-B_1_1}
\end{align}
When $n>1$, notice that
\[
 M_{r^n \times r^n}(\Z[\zeta]) \cong M_{r \times r}(M_{r^{n-1} \times r^{n-1}}(\Z[\zeta])) \cong M_{r^{n-1} \times r^{n-1}}(\Z[\zeta]) \otimes M_{r \times r}(\Z[\zeta]).
\]
Then, let us consider the matrices $A_1^{(n)}, B_1^{(n)}, \ldots, A_n^{(n)}, B_n^{(n)} \in M_{r^n \times r^n}(\Z[\zeta])$ defined, for all $1 \leqs j \leqs n-1$, as
\begin{align}
 A_j^{(n)} &:= \left(
 \begin{matrix}
  A_j^{(n-1)} & 0 & \cdots & 0 \\
  0 & A_j^{(n-1)} & \ddots & \vdots \\
  \vdots & \ddots & \ddots & 0 \\
  0 & \cdots & 0 & A_j^{(n-1)}
 \end{matrix} \right), &
 B_j^{(n)} &:= \left(
 \begin{matrix}
  B_j^{(n-1)} & 0 & \cdots & 0 \\
  0 & B_j^{(n-1)} & \ddots & \vdots \\
  \vdots & \ddots & \ddots & 0 \\
  0 & \cdots & 0 & B_j^{(n-1)}
 \end{matrix} \right), \label{E:A_j_n-B_j_n} \\
 A_n^{(n)} &:= \left(
 \begin{matrix}
  I_{r^{n-1}} & 0 & \cdots & 0 \\
  0 & \zeta^4 I_{r^{n-1}} & \ddots & \vdots \\
  \vdots & \ddots & \ddots & 0 \\
  0 & \cdots & 0 & \zeta^{4(r-1)} I_{r^{n-1}}
 \end{matrix} \right), &
 B_n^{(n)} &:= \left(
 \begin{matrix}
  0 & 0 & \cdots & 0 & I_{r^{n-1}} \\
  I_{r^{n-1}} & 0 & \ddots & & 0 \\
  0 & \ddots & \ddots & \ddots & \vdots \\
  \vdots & \ddots & \ddots & 0 & 0 \\
  0 & \cdots & 0 & I_{r^{n-1}} & 0
 \end{matrix} \right). \label{E:A_n_n-B_n_n}
\end{align}
In other words, we are setting
\begin{align*}
 A_j^{(n)} &:= A_j^{(n-1)} \otimes I_r, &
 B_j^{(n)} &:= B_j^{(n-1)} \otimes I_r, &
 A_n^{(n)} &:= I_{r^{n-1}} \otimes A_1^{(1)}, &
 B_n^{(n)} &:= I_{r^{n-1}} \otimes B_1^{(1)}.
\end{align*}
Let us abuse notation from now on, and denote $A_j^{(n)}$ and $B_j^{(n)}$ simply by $A_j$ and $B_j$ respectively, for every $n \geqs 1$.

\begin{lemma}\label{L:local_system}
 The family of matrices $A_1, B_1, \ldots, A_n, B_n \in M_{r^n \times r^n}(\Z[\zeta])$ defined by Equations~\eqref{E:A_1_1-B_1_1}--\eqref{E:A_n_n-B_n_n} satisfies the following list of properties:
 \begin{enumerate}
 \item They either commute or $\zeta^4$-commute, more precisely:
  \begin{align}
   A_j A_k &= A_k A_j, &
   A_j B_k  &= \zeta^{4 \delta_{j,k}} B_k A_j, &
   B_j B_k  &= B_k B_j; \label{E:AB_zeta-commute} 
  \end{align}
 \item They have order $r$:
  \begin{align}
   A_1^r = B_1^r = \ldots = A_n^r = B_n^r = I_{r^n} &\in \GL_{r^n}(\Z[\zeta]); \label{E:AB_order_r} 
  \end{align}
 \item Their centralizer in $\GL_{r^n}(\Z[\zeta])$ is the subgroup of invertible scalar matrices:
  \begin{align}
   C_{\GL_{r^n}(\Z[\zeta])} \left( A_1,B_1,\ldots,A_n,B_n \right) &= (\Z[\zeta])^\times. \label{E:AB_centralizer}
  \end{align}
\end{enumerate}
\end{lemma}

\begin{proof}
 The claim is easily established by induction on $n \geqs 1$.

 When $n=1$, Equation~\eqref{E:AB_zeta-commute} follows from an easy computation. Equation~\eqref{E:AB_order_r} is clear. To check Equation~\eqref{E:AB_centralizer}, remark that any matrix commuting with $A_1$ has to be diagonal, and that any diagonal matrix commuting with $B_1$ has to have a single eigenvalue.

 When $n>1$, Equations~\eqref{E:AB_zeta-commute}--\eqref{E:AB_centralizer} are an immediate consequence of the induction hypothesis.
\end{proof}

Now, thanks to Lemma~\ref{L:local_system}, we can endow the $\fld$-vector space $V_g := \fld^{r^g}$ with a left $\Z[\pi_{n,g}]$-module structure determined by the $\fld$-linear representation
\begin{align}
 \varphi_{n,g}^V : \Z[\pi_{n,g}] & \to \End_{\fld}(V_g) \label{E:varphi} \\*
 \braid{\sigma}{i} & \mapsto -\zeta^{-2} I_{r^g} \nonumber \\*
 \braid{\alpha}{j} & \mapsto A_j \nonumber \\*
 \braid{\beta}{j} & \mapsto B_j \nonumber 
\end{align}
This should be compared with \cite[Theorem~2.6]{GH14}. Notice that, in order to infer that $\varphi_{n,g}^V$ is well-defined, we only need Equations~\eqref{E:AB_zeta-commute} and \eqref{E:AB_order_r}. Equation~\eqref{E:AB_centralizer}, which states that $V_g$ is a simple $\Z[\pi_{n,g}]$-module, will only be crucial later, in order to define a projective representation of $\Mod(\varSigma_{g,1})$ on the Borel--Moore homology of $X_{n,g}$ with local coefficients in $V_g$, using Proposition~\ref{P:lifting_actions}.

\begin{lemma}\label{L:rho_alpha_prime}
 For all integers $1 \leqs j \leqs g$ and $1 \leqs k \leqs g-1$ we have
 \begin{align}
  \varphi_{n,g}^V(\braid{\tilde{\alpha}}{j}) &= \zeta^4 A_j, &
  \varphi_{n,g}^V(\braid{\gamma}{k}) &= \zeta^{-4} A_k^{-1} A_{k+1}, &
  \varphi_{n,g}^V(\braid{\delta}{j}) &= \zeta^{-4(j-1)} A_j, \label{E:rho_other_braids}
 \end{align}
 where $\braid{\tilde{\alpha}}{j}$, $\braid{\gamma}{k}$, and $\braid{\delta}{j}$ are defined in Equation~\eqref{E:other_braids}.
\end{lemma}

The proof is a direct computation based on Equations~\eqref{E:other_braids} and \eqref{E:AB_zeta-commute}, which is left to the reader.

Let us now consider the free group 
\begin{equation}
 F_g := \langle \twist{\alpha}{j}, \twist{\beta}{j}, \twist{\gamma}{k} \mid 1 \leqs j \leqs g, 1 \leqs k \leqs g-1 \rangle, \label{E:free_group}
\end{equation}
and the homomorphism
\begin{align}
 \psi_g^V : F_g &\to \GL_{\fld}(V_g) \label{E:psi} \\*
 \twist{\alpha}{j} &\mapsto \sum_{\ell=0}^{r-1} \zeta^{-2\ell(\ell-1)} A_j^\ell, \nonumber \\*
 \twist{\beta}{j} &\mapsto \sum_{\ell=0}^{r-1} \zeta^{-2(\ell+1)\ell} B_j^\ell, \nonumber \\*
 \twist{\gamma}{k} &\mapsto \sum_{\ell=0}^{r-1} \zeta^{-2(\ell+1)\ell} A_k^{-\ell} A_{k+1}^\ell, \nonumber 
\end{align}
where $A_1,B_1,\ldots,A_g,B_g \in \GL_{r^g}(\Z[\zeta])$ are defined in Equations~\eqref{E:A_1_1-B_1_1}--\eqref{E:A_n_n-B_n_n}. Notice that we are not claiming that
\[
 \{ \twist{\alpha}{j}, \twist{\beta}{j}, \twist{\gamma}{k} \mid 1 \leqs j \leqs g, 1 \leqs k \leqs g-1 \}
\]
generates a free subgroup of $\Mod(\varSigma_{g,1})$, because it does not. We are merely considering the free group of formal words in these Dehn twists, and instead of introducing new symbols for the corresponding free generators, we abuse notation a little.

\begin{remark}\label{R:integral_homological}
 We point out that
\begin{align}
 \psi_g^V(\twist{\alpha}{j})^{-1} &= \frac{1}{r} \sum_{\ell=0}^{r-1} \zeta^{2(\ell+1)\ell} A_j^\ell, &
 \psi_g^V(\twist{\beta}{j})^{-1} &= \frac{1}{r} \sum_{\ell=0}^{r-1} \zeta^{2\ell(\ell-1)} B_j^\ell, &
 \psi_g^V(\twist{\gamma}{k})^{-1} &= \frac{1}{r} \sum_{\ell=0}^{r-1} \zeta^{2\ell(\ell-1)} A_k^{-\ell} A_{k+1}^\ell, \label{E:psi_inverses}
\end{align}
so that in fact the image of $\psi_g^V$ is actually contained in $\GL_{r^g}(\Z[\zeta,\frac{1}{r}])$.
\end{remark}

\begin{proposition}\label{P:MCG_action}
 The homomorphisms $\varphi_{n,g}^V : \Z[\pi_{n,g}] \to \End_{\fld}(V_g)$ and $\psi_g^V : F_g \to \GL_{\fld}(V_g)$, defined in Equations~\eqref{E:varphi} and \eqref{E:psi} respectively, satisfy, for all $f \in F_g$ and $\bloop \in \pi_{n,g}$, the identity
 \begin{equation}\label{E:fundamental_identity}
  \psi_g^V(f) \circ \varphi_{n,g}^V(\bloop) = 
  \varphi_{n,g}^V((f^{\times n})_*(\bloop)) \circ \psi_g^V(f)
 \end{equation}
\end{proposition}

\begin{proof}
 It is sufficient to establish Equation~\eqref{E:fundamental_identity} for a fixed pair of sets of generators of $F_g$ and of $\pi_{n,g}$. Then, thanks to Remark~\ref{R:alternative_generators_pi}, it is sufficient to check it for all 
 \[
  f \in \{ \twist{\alpha}{j}, \twist{\beta}{j}, \twist{\gamma}{k} \mid 1 \leqs j \leqs g, 1 \leqs k \leqs g-1 \}
 \]
 and all
 \[
  \bloop \in \{ \braid{\sigma}{i},\braid{\beta}{j},\braid{\delta}{j} \mid 1 \leqs i \leqs n, 1 \leqs j \leqs g \}.
 \]
 First of all, $(\twist{\alpha}{j}^{\times n})_*(\braid{\beta}{j}) = \braid{\alpha}{j} \ast \braid{\beta}{j}$, since it is given by
 \[
  \pic{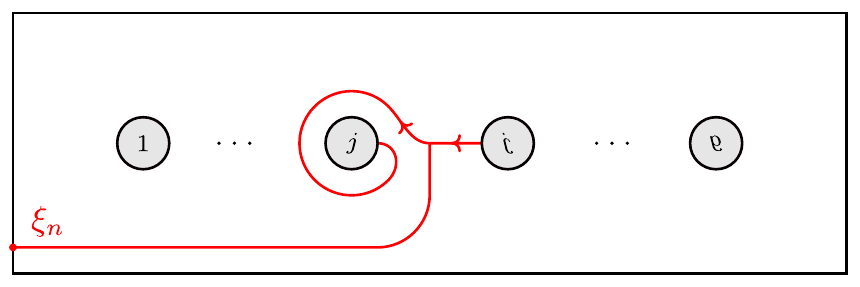}
 \]
 Next, $(\twist{\beta}{j}^{\times n})_*(\braid{\delta}{j}) = \braid{\beta^{-1}}{j} \ast \braid{\delta}{j}$, since it is given by
 \[
  \pic{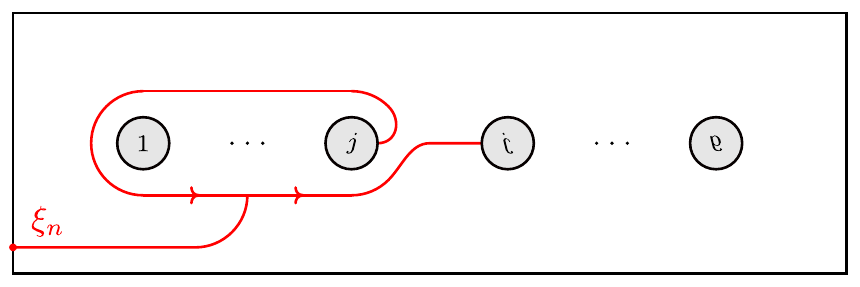}
 \]
 Furthermore, $(\twist{\gamma}{k}^{\times n})_*(\braid{\beta}{k}) = \braid{\beta}{k} \ast \braid{\gamma^{-1}}{k}$ and $(\twist{\gamma}{k}^{\times n})_*(\braid{\beta}{k+1}) = \braid{\gamma}{k} \ast \braid{\beta}{k+1}$, since they are given by
 \begin{gather*}
  \pic{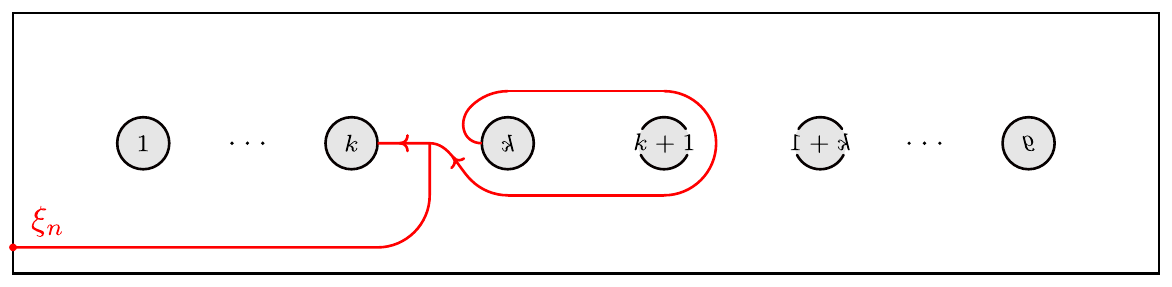} \\*
  \pic{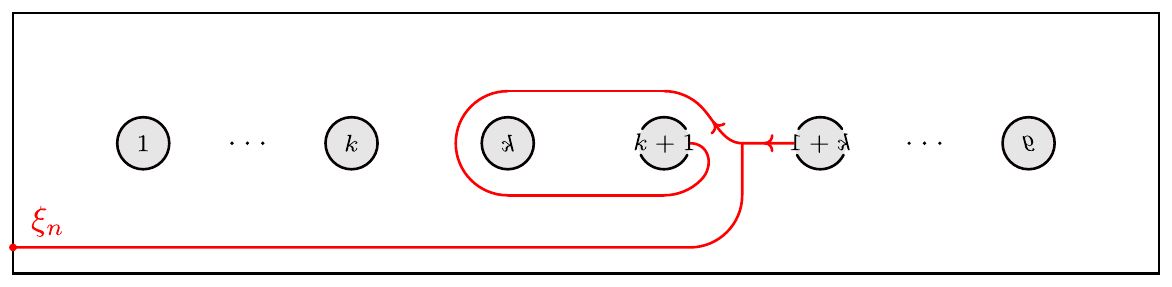} 
 \end{gather*}
 respectively. Clearly, in all other cases, $(f^{\times n})_*(\bloop)$ is simply given by $\bloop$.

 Therefore, the claim is a direct consequence of Equations~\eqref{E:AB_zeta-commute}. Indeed, we have
 \begin{align*}
  &\varphi_{n,g}^V((\twist{\alpha}{j}^{\times n})_*(\braid{\beta}{j})) \psi_g^V(\twist{\alpha}{j}) 
  = \varphi_{n,g}^V(\braid{\alpha}{j}) \varphi_{n,g}^V(\braid{\beta}{j}) \psi_g^V(\twist{\alpha}{j}) 
  = A_j B_j \left( \sum_{\ell=0}^{r-1} \zeta^{-2\ell(\ell-1)} A_j^\ell \right) \\*
  &\hspace*{\parindent} = \left( \sum_{\ell=0}^{r-1} \zeta^{-2\ell(\ell-1)-4\ell} A_j^{\ell+1} \right) B_j 
  = \left( \sum_{\ell=0}^{r-1} \zeta^{-2(\ell+1)\ell} A_j^{\ell+1} \right) B_j \\*
  &\hspace*{\parindent} = \left( \sum_{m=0}^{r-1} \zeta^{-2m(m-1)} A_j^m \right) B_j 
  = \psi_g^V(\twist{\alpha}{j}) \varphi_{n,g}^V(\braid{\beta}{j}).
 \end{align*}
 Next, we have
 \begin{align*}
  &\varphi_{n,g}^V((\twist{\beta}{j}^{\times n})_*(\braid{\delta}{j})) \psi_g^V(\twist{\beta}{j}) 
  = \varphi_{n,g}^V(\braid{\beta}{j})^{-1} \varphi_{n,g}^V(\braid{\delta}{j}) \psi_g^V(\twist{\beta}{j}) 
  = \zeta^{-4(j-1)} B_j^{-1} A_j \left( \sum_{\ell=0}^{r-1} \zeta^{-2(\ell+1)\ell} B_j^\ell \right) \\*
  &\hspace*{\parindent} = \zeta^{-4(j-1)} \left( \sum_{\ell=0}^{r-1} \zeta^{-2(\ell+1)\ell+4\ell} B_j^{\ell-1} \right) A_j 
  = \zeta^{-4(j-1)} \left( \sum_{\ell=0}^{r-1} \zeta^{-2\ell(\ell-1)} B_j^{\ell-1} \right) A_j \\*
  &\hspace*{\parindent} = \zeta^{-4(j-1)} \left( \sum_{m=0}^{r-1} \zeta^{-2(m+1)m} B_j^m \right) A_j 
  = \psi_g^V(\twist{\beta}{j}) \varphi_{n,g}^V(\braid{\delta}{j}).
 \end{align*}
 Similarly, we have
 \begin{align*}
  &\varphi_{n,g}^V((\twist{\gamma}{k}^{\times n})_*(\braid{\beta}{k})) \psi_g^V(\twist{\gamma}{k}) 
  = \varphi_{n,g}^V(\braid{\beta}{k}) \varphi_{n,g}^V(\braid{\gamma}{k})^{-1} \psi_g^V(\twist{\gamma}{k}) 
  = \zeta^4 B_k A_k A_{k+1}^{-1} \left( \sum_{\ell=0}^{r-1} \zeta^{-2(\ell+1)\ell} A_k^{-\ell} A_{k+1}^\ell \right) \\*
  &\hspace*{\parindent} = \left( \sum_{\ell=0}^{r-1} \zeta^{-2(\ell+1)\ell+4\ell} A_k^{-\ell+1} A_{k+1}^{\ell-1} \right) B_k 
  = \left( \sum_{\ell=0}^{r-1} \zeta^{-2\ell(\ell-1)} A_k^{-\ell+1} A_{k+1}^{\ell-1} \right) B_k \\*
  &\hspace*{\parindent} = \left( \sum_{m=0}^{r-1} \zeta^{-2(m+1)m} A_k^{-m} A_{k+1}^m \right) B_k 
  = \psi_g^V(\twist{\gamma}{k}) \varphi_{n,g}^V(\braid{\beta}{k}).
 \end{align*}
 Finally, we have
 \begin{align*}
  &\varphi_{n,g}^V((\twist{\gamma}{k}^{\times n})_*(\braid{\beta}{k+1})) \psi_g^V(\twist{\gamma}{k})  
  = \varphi_{n,g}^V(\braid{\gamma}{k}) \varphi_{n,g}^V(\braid{\beta}{k+1}) \psi_g^V(\twist{\gamma}{k}) 
  = \zeta^{-4} A_k^{-1} A_{k+1} B_{k+1} \left( \sum_{\ell=0}^{r-1} \zeta^{-4(\ell+1)\ell} A_k^{-\ell} A_{k+1}^\ell \right) \\*
  &\hspace*{\parindent} = \left( \sum_{\ell=0}^{r-1} \zeta^{-2(\ell+1)\ell-4(\ell+1)} A_k^{-\ell-1} A_{k+1}^{\ell+1} \right) B_{k+1} 
  = \left( \sum_{\ell=0}^{r-1} \zeta^{-2(\ell+2)(\ell+1)} A_k^{-\ell-1} A_{k+1}^{\ell+1} \right) B_{k+1} \\*
  &\hspace*{\parindent} = \left( \sum_{m=0}^{r-1} \zeta^{-2(m+1)m} A_k^{-m} A_{k+1}^m \right) B_{k+1} 
  = \psi_g^V(\twist{\gamma}{k}) \varphi_{n,g}^V(\braid{\beta}{k+1}). \qedhere
 \end{align*}
\end{proof}

\begin{corollary}\label{C:bar_psi}
 There exists a homomorphism $\bar{\psi}_g^V : \Mod(\varSigma_{g,1}) \to \PGL_{\fld}(V_g)$ that fits into the commutative diagram 
 \begin{center}
  \begin{tikzpicture}[descr/.style={fill=white}] \everymath{\displaystyle}
   \node (P0) at (0,0) {$F_g$};
   \node (P1) at (3,0) {$\GL_{\fld}(V_g)$};
   \node (P2) at (0,-1) {$\Mod(\varSigma_{g,1})$};
   \node (P3) at (3,-1) {$\PGL_{\fld}(V_g)$};
   \draw
   (P0) edge[->] node[above] {\scriptsize $\psi_g^V$}(P1)
   (P0) edge[->>] node[left] {} (P2)
   (P1) edge[->>] node[right] {} (P3)
   (P2) edge[->] node[below] {\scriptsize $\bar{\psi}_g^V$} (P3);
  \end{tikzpicture}
 \end{center}
\end{corollary}

\begin{proof}
 If $f \in F_g$ satisfies $[f] = [\id] \in \Mod(\varSigma_{g,1})$, then Equation~\eqref{E:fundamental_identity} implies
 \[
  \psi_g^V(f) \circ \varphi_{n,g}^V(\bloop)
  = \varphi_{n,g}^V((f^{\times n})_*(\bloop)) \circ \psi_g^V(f)
  = \varphi_{n,g}^V(\bloop) \circ \psi_g^V(f)
 \]
 for every $\bloop \in \pi_{n,g}$, which means
 \[
  \psi_g^V(f) \in C_{\GL_{r^n}(\Z[\zeta])} \left( A_1,B_1,\ldots,A_n,B_n \right) 
  = (\Z[\zeta])^\times
 \]
 thanks to Equation~\eqref{E:AB_centralizer}.
\end{proof}

\begin{definition}\label{D:linear_homology}
 For every integer $n \geqs 1$, the \textit{$n$th modular Heisenberg homology group} of $\varSigma_{g,1}$ is the $\fld$-vector space
 \[
  \homol_{n,g}^V := H^\BM_n(X_{n,g},Y_{n,g};\phi_{n,g}^V),
 \]
 where $\varphi_{n,g}^V$ is defined in Equation~\eqref{E:varphi}. For $n = 0$, we set 
 \[
  \homol_{0,g}^V := V_g.
 \]
 The \textit{total modular Heisenberg homology group} of $\varSigma_{g,1}$ is the direct sum
 \[
  \homol_g^V := \bigoplus_{n \geqs 0} \homol_{n,g}^V.
 \]
\end{definition}

\begin{remark}\label{R:infinite_linear_basis}
 A basis of $\homol_{n,g}^V$ is given by
 \[
  \left\{ \basis(\bfa,\bfb) \otimes \bfv_{\bfc} \Bigm|
  \bfa, \bfb \in \N^{\times g}, \left| \bfa + \bfb \right| = n,
  \bfc \in (\Z/r\Z)^{\times g} \right\},
 \]
 where $\{ \bfv_{\bfc} \mid \bfc \in (\Z/r\Z)^{\times g} \}$ is the canonical basis of $V_g = \fld^{r^g}$ in which matrices of Lemma~\ref{L:local_system} were written, meaning we have
 \begin{align*}
  A_j \bfv_{\bfc} &= \zeta^{4 c_j} \bfv_{\bfc}, &
  B_j \bfv_{\bfc} &= \bfv_{\bfc+\bfe_j}
 \end{align*}
 for every integer $1 \leqs j \leqs g$, where we recall that $\bfe_j$ denotes the column vector of size $g$ whose $k$th entry is $\delta_{j,k}$.
\end{remark}

The next statement should be compared with \cite[Theorems~62 \& 70]{BPS21}, where an implicit version of the homomorphism $\bar{\psi}_g^V : \Mod(\varSigma_{g,1}) \to \PGL_{\fld}(V_g)$ is used for even roots of unity (and lifted to a linear representation of the stably universal central extension of $\Mod(\varSigma_{g,1})$).

\begin{theorem}\label{T:mcg_homological_projective_rep}
 There exists a projective action of $\Mod(\varSigma_{g,1})$ on the $n$th modular Heisenberg homology $\homol_{n,g}^V$ of Definition~\ref{D:linear_homology} determined by the commutative diagram 
 \begin{center}
  \begin{tikzpicture}[descr/.style={fill=white}] \everymath{\displaystyle}
   \node (P0) at (0,0) {$F_g$};
   \node (P1) at (3,0) {$\GL_{\fld}(\homol_{n,g}^V)$};
   \node (P2) at (0,-1) {$\Mod(\varSigma_{g,1})$};
   \node (P3) at (3,-1) {$\PGL_{\fld}(\homol_{n,g}^V)$};
   \draw
   (P0) edge[->] node[above] {\scriptsize $\rho_{n,g}^V$}(P1)
   (P0) edge[->>] node[left] {} (P2)
   (P1) edge[->>] node[right] {} (P3)
   (P2) edge[->] node[below] {\scriptsize $\bar{\rho}_{n,g}^V$} (P3);
  \end{tikzpicture}
 \end{center}
 where
 \[
  \rho_{n,g}^V(f)(\tilde{\chain} \otimes v) = (\tilde{f}^{\times n})_*(\tilde{\chain}) \otimes \psi_g^V(f)(v)
 \]
 for all diffeomorphisms $f \in F_g$, homology classes $\tilde{\chain} \in H^\BM_n(X_{n,g},Y_{n,g};\varphi_{n,g}^\pi)$, and coefficients $v \in V_g$. Furthermore, the direct sum, for $n \geqs 0$, of these representations defines a projective action
 \[
  \bar{\rho}_g^V : \Mod(\varSigma_{g,1}) \to \PGL_{U_\zeta}(\homol_g^V)
 \]
 of $\Mod(\varSigma_{g,1})$ on the total modular Heisenberg homology $\homol_g^V$ by invertible $U_\zeta$-module endomorphisms, considered up to non-zero scalars.
\end{theorem}

\begin{proof}
 The first claim follows directly from Proposition~\ref{P:lifting_actions}, by considering the homomorphism
 \begin{align*}
  \chi_{n,g} : F_g & \to \Mod(X_{n,g}). \\*
  f & \mapsto f^{\times n}
 \end{align*} 
 Notice that $\Mod(\varSigma_{g,1}) \cong F_g / \ker \chi_{n,g}$, and that Equations~\eqref{E:f_gamma} and \eqref{E:centralizer} are satisfied thanks to Proposition~\ref{P:MCG_action} and to Lemma~\ref{L:local_system}.$(iii)$ respectively. 
\end{proof}

\subsubsection{Finite-dimensional mapping class group representations}\label{S:fin-dim_homol_rep}

We say an embedded twisted cycle $(m,\bfsimp,\bfk,\thread)$ of dimension $n$ in $\varSigma_{g,1}$ is \textit{small} if the $m$-partition $\bfk = (k_1,\ldots,k_m)$ of $n$ satisfies $0 \leqs k_i \leqs r-1$ for all integers $1 \leqs i \leqs m$.

\begin{definition}\label{D:fin_dim_homol}
 The \textit{small Heisenberg homology} of $\varSigma_{g,1}$ is the $\fld$-vector space $\homol_g^{V(r)} \subset \homol_g^V$ spanned by homology classes of small embedded twisted cycles.
\end{definition}

\begin{theorem}\label{T:uzeta_homological_representation}
 $\homol_g^{V(r)}$ is a $\bar{U}_\zeta$-submodule of $\homol_g^V$ with basis
 \[
  \left\{ \basis(\bfa,\bfb) \otimes \bfv_{\bfc} \Bigm|
  \bfa, \bfb \in \{ 0,1,\ldots,r-1 \}^{\times g},
  \bfc \in (\Z/r\Z)^{\times g} \right\},
 \]
 and the projective representation $\bar{\rho}_g^V : \Mod(\varSigma_{g,1}) \to \PGL_{U_\zeta}(\homol_g^V)$ of Theorem~\ref{T:mcg_homological_projective_rep} restricts to a projective action
 \[
  \bar{\rho}_g^{V(r)} : \Mod(\varSigma_{g,1}) \to \PGL_{\bar{U}_\zeta}(\homol_g^{V(r)}),
 \]
 of $\Mod(\varSigma_{g,1})$ on the small Heisenberg homology $\homol_g^{V(r)}$ by invertible $\bar{U}_\zeta$-module endomorphisms, considered up to non-zero scalars.
\end{theorem}

\begin{proof}
 First of all, notice that $\homol_g^{V(r)}$ is clearly invariant for the operators $\calE$, $\calF^{(1)}$, and $\calK$. Indeed, $\calE$ can only decrease labels of small multisimplices, while $\calF^{(1)}$ always adds a component of label $1$ to every small labeled multisimplex, and $\calK$ acts diagonally. In particular, all of them preserve the property of being small.

 For what concerns the basis, notice that the proposed set is clearly contained in $\homol_g^{V(r)}$, and it is free thanks to Remark~\ref{R:infinite_linear_basis}. To see that it generates $\homol_g^{V(r)}$, notice that the (image of the) multisimplex underlying all homology classes in this free set is a deformation retract of $\varSigma_{g,1}$. This means we can reduce every small embedded twisted cycle to a linear combination of homology classes of this form using the rules of Proposition~\ref{P:computation_rules}. In particular, thanks to the fusion rule, every label exceeding $r-1$ comes with a vanishing binomial coefficient in front.

 Finally, all diffeomorphisms clearly preserve the property of being small too.
\end{proof}

The $\fld$-vector space $\homol_g^{V(r)}$ is thus endowed with commuting actions of $\fraku_\zeta$ and of $\Mod(\varSigma_{g,1})$. The main goal of this paper is to identify these module structures respectively with the $g$th tensor power of the adjoint representation of $\fraku_\zeta$ and with a projective representation of $\Mod(\varSigma_{g,1})$ first introduced by Lyubashenko in \cite{L94}, which can be recovered from the non-semisimple TQFT constructed from $\fraku_{\zeta}$ in \cite{DGP17,DGGPR19}, as shown in \cite{DGGPR20}.

In a future work, we will investigate these intertwined actions at the higher level of $\homol_g^\dHeis$, since $U_q$-modules with $\Z[\dHeis_g]$-coefficients seem to be interesting and new in the literature.

\begin{remark}
 The reader should notice that the fact that the actions of the mapping class group and of the quantum group commute is a straightforward consequence of the purely homological definition of the operators $\calE$, $\calF^{(k)}$, and $\calK$. This gives a new simple and conceptual proof of this important property which, on the quantum side, follows from the transmutation procedure introduced by Majid \cite{M91}. On the other hand, we will see later that $E^r$ acts by $0$. This behavior is not apparent yet on the homological side, and should be established by computation, while it will be a direct consequence of the algebraic definition of the small quantum group $\fraku_{\zeta}$ of Section~\ref{S:small_quantum_sl2}, on the quantum side.
\end{remark}

In order to find a more intrinsic characterization of the small Heisenberg homology $\homol_g^{V(r)} \subset \homol_g^V$ of Definition~\ref{D:fin_dim_homol}, we might want to focus on the standard version of homology, rather than the Borel--Moore one. The main difficulty is that we cannot apply the proof of Proposition~\ref{P:homology_structure}, which provides bases that persist under all choices of twisted coefficients. Nevertheless, let us set
\begin{align*}
 \calH_{n,g}^{\dHeis,\dagger} &:= H_n(X_{n,g},Y_{n,g};\varphi_{n,g}^{\dHeis}), &
 \calH_g^{\dHeis,\dagger} &:= \bigoplus_{n \in \N} \calH_{n,g}^{\dHeis,\dagger},
\end{align*} 
and similarly
\begin{align*}
 \calH_{n,g}^{V,\dagger} &:= H_n(X_{n,g},Y_{n,g};\varphi_{n,g}^V), &
 \calH_g^{V,\dagger} &:= \bigoplus_{n \in \N} \calH_{n,g}^{V,\dagger},
\end{align*}
There exist natural maps
\begin{align*}
 \iota^{\dHeis}_g : \calH_g^{\dHeis,\dagger} &\to \calH^{\dHeis}_g, &
 \iota^V_g : \calH_g^{V,\dagger} &\to \calH^V_{g}
\end{align*}
induced by inclusions of chain complexes, since standard chains are also Borel--Moore chains. We conjecture the following characterization of $\calH_g^{V(r)}$.

\begin{conjecture}\label{C:r_characterization_conjecture}
 $\calH_g^{V(r)} = \iota^V_g \left( \calH_g^{V,\dagger} \right)$.
\end{conjecture}

Let us explain where this expectation comes from, by first introducing special embedded twisted cycles represented by
\[
 \iota^{\dHeis}_g \left( \tilde{\bfsimp}^\dagger(\bfa,\bfb) \right) := \pic{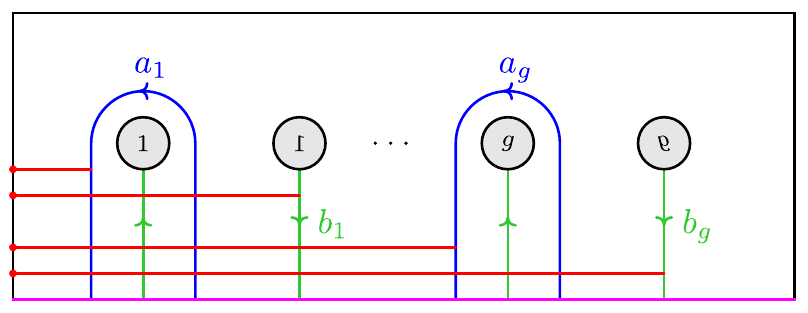}
\]
for every $2g$-partition $(\bfa,\bfb)$ of $n$. Notice that the only difference between these diagrams and the ones appearing in Equation~\eqref{E:basis} is that, here, blue and green curves are solid, instead of being dashed-dotted. We adopt the convention that a solid curve labeled by an integer $k$ should represent $k$ parallel copies, in the surface $\varSigma_{g,1}$, of the solid curve. Therefore, these diagrams represent closed hypercubes embedded into $X_{n,g}$, with faces embedded into $Y_{n,g}$. In particular, they also define standard cycles, and we denote by $\tilde{\bfsimp}^{\dagger}(\bfa, \bfb) \otimes v$ the associated vectors in $\calH_{g,n}^{V,\dagger}$ for every vector $v \in V_g$.

\begin{remark}\label{R:automatic_truncation_at_roots}
 Whenever $a_j \geqs r$ or $b_j \geqs r$ for some $1 \leqs j \leqs g$, we have
 \[
  \iota^V_g \left( \tilde{\bfsimp}^{\dagger}(\bfa,\bfb) \otimes \bfv_{\bfm} \right) 
  = \iota^{\dHeis}_g \left( \tilde{\bfsimp}^{\dagger}(\bfa, \bfb) \right) \otimes \bfv_{\bfm} 
  \propto \left( \prod_{j=1}^g [a_j]_\zeta ! [b_j]_\zeta ! \right) \tilde{\bfsimp}(\bfa, \bfb) \otimes \bfv_{\bfm} = 0,
 \]
 as a straightforward consequence of the fusion rule in Proposition~\ref{P:computation_rules}. 
\end{remark}

In \cite{DM22}, it will be shown that there exists a bilinear form
\[
 \langle \_ , \_ \rangle_{\dHeis_g} : \calH_{n,g}^{\dHeis} \otimes \calH_{n,g}^{\dHeis,\dagger} \to \Z[\dHeis_g]
\]
which is left non-degenerate (meaning that the left radical is trivial), as a direct consequence of Proposition~\ref{P:homology_structure} and of the identity
\[
 \langle \tilde{\bfsimp}(\bfa, \bfb),\tilde{\bfsimp}^{\dagger}(\bfa', \bfb') \rangle_{\dHeis_g} = \delta_{(\bfa, \bfb),(\bfa', \bfb')},
\]
where $\delta$ is the Kronecker symbol for $2g$-partitions of $n$. Since the cycles $\tilde{\bfsimp}(\bfa, \bfb)$ form a $\Z[\dHeis_g]$-basis of $\calH_g^{\dHeis}$, the family
\[
 \lbrace \tilde{\bfsimp}^{\dagger}(\bfa, \bfb) \mid \bfa, \bfb \in \N^{\times g}, \lvert \bfa + \bfb \rvert = n \rbrace
\]
yields a free $\Z[\dHeis_g]$-submodule of $\calH_{n,g}^{\dHeis,\dagger}$. It is not clear that this should coincide with the whole $\calH_{n,g}^{\dHeis,\dagger}$. Nevertheless, if we set
\[
 \calH_g^{V(r),\dagger} := \langle \tilde{\bfsimp}^{\dagger}(\bfa, \bfb) \otimes \bfv_{\bfm} | \bfa, \bfb, \bfm \in \N^{\times g} \rangle_{\fld} \subset \calH_g^{\dHeis,\dagger},
\]
then Remark~\ref{R:automatic_truncation_at_roots} immediately yields to following result.

\begin{proposition}
 The small Heisenberg homology satisfies
 \[
  \calH_g^{V(r)} = \iota^{V_g} \left( \calH_g^{V(r),\dagger} \right).
 \]
\end{proposition}

This provides the first part of the answer to Conjecture~\ref{C:r_characterization_conjecture}. This characterization of $\calH_g^{V(r)}$ (as the image of a subspace of the standard homology) is still artificial, but if one could show that $\calH_g^{V(r),\dagger}$ actually spans $\calH_{g}^{V,\dagger}$, Conjecture~\ref{C:r_characterization_conjecture} would find a positive answer. A possibility would be to show that twisted Borel--Moore and standard homologies are \textit{generically} isomorphic for configuration spaces of surfaces, with a notion of genericity that would have to be suitably characterized. For punctured discs, this is indeed the case, see \cite{K86}.

\section{Computation of homological actions}\label{S:homological_computations}

In this section, we compute the homological action of (some of) the generators of $U_q$ and of $\Mod(\varSigma_{g,1})$ on $\homol_{n,g}^\dHeis$ and on $\smash{\homol_g^V}$, respectively. Both these actions were defined in Section~\ref{S:homological_representations}, in Theorems~\ref{T:Uq_homological_representation} and \ref{T:mcg_homological_projective_rep}, respectively. We compute them in the homological bases provided by Corollary~\ref{C:Heisenberg_basis} and by Remark~\ref{R:infinite_linear_basis}, respectively. These are defined by diagrams representing embedded twisted cycles, that were introduced in Definition~\ref{D:embedded_twisted_cycle}. The formulas we obtain here will be later used to identify these homological representations with quantum ones, that will be recalled in Section~\ref{S:quantum_representations}, and explicitly computed in Section~\ref{S:quantum_computations}. We will focus on the action of $U_q$ in Section~\ref{S:homological_adjoint_computation}, and on the one of $\Mod(\varSigma_{g,1})$ in Section~\ref{S:homological_mcg_computation}. Many of these formulas involve \textit{quantum multinomials}, which are defined, for all positive integers $k \geqs 0$ and $\ell_1, \ldots, \ell_j \geqs 0$ satisfying $\ell_1 + \ldots + \ell_j \leqs k$, as
\[
 \sqbinom{k}{\ell_1,\ldots,\ell_j}_q := \frac{[k]_q!}{[\ell_1]_q! \cdots [\ell_j]_q![k-\ell_1-\ldots-\ell_j]_q!}.
\]

\subsection{Homological action of quantum \texorpdfstring{$\fsl_2$}{sl(2)}}\label{S:homological_adjoint_computation}

Let us consider the linear map
\[
 \tilde{\calF}^{(k)} : \homol_{n,1}^\dHeis \to \homol_{n+k,1}^\dHeis
\]
sending every basis vector $\basis(a,b)$ with $a+b=n$ to
\[
 \pic{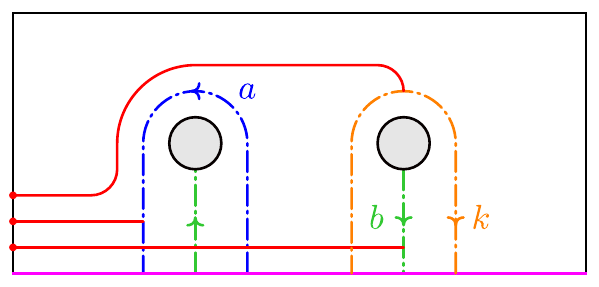}
\]

\begin{lemma}\label{L:F_right_regular}
 For all integers $a,b,k \geqs 0$ we have 
 \begin{align}\label{E:F_right_regular}
  \tilde{\calF}^{(k)}(\basis(a,b)) &= q^{-(k+3)k-(2a+b)k} \sum_{j=0}^k \sum_{i=0}^{k-j} (-1)^j \basis(a+i,b+k-i) \nonumber \\*
  &\hspace*{\parindent} \otimes \sqbinom{a+i}{a}_q \sqbinom{b+k-i}{b,j}_q q^{ij+(a+b)i+(2b+k+3)j} \beta^i \alpha^j.
 \end{align}
\end{lemma}

\begin{proof}
 Using Proposition~\ref{P:computation_rules} and Remark~\ref{R:computation_rules}, we obtain
 \begin{align*}
  &\tilde{\calF}^{(k)}(\basis(a,b)) 
  = \pic{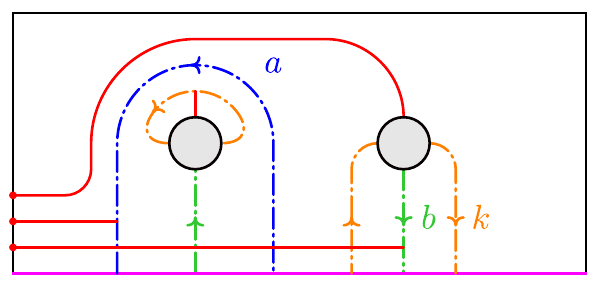} 
  \Ceq \sum_{j=0}^k \pic{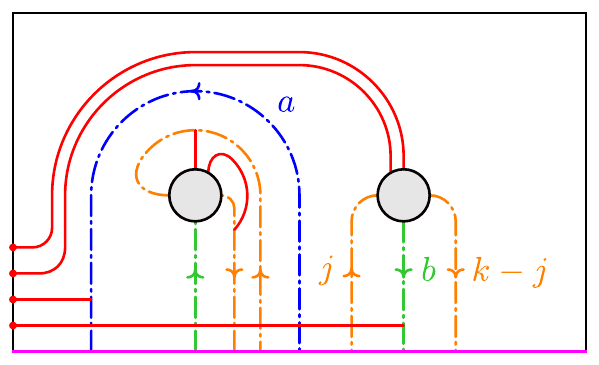} \\
  &\hspace*{\parindent}\Ceq \sum_{j=0}^k \sum_{i=0}^{k-j} \pic{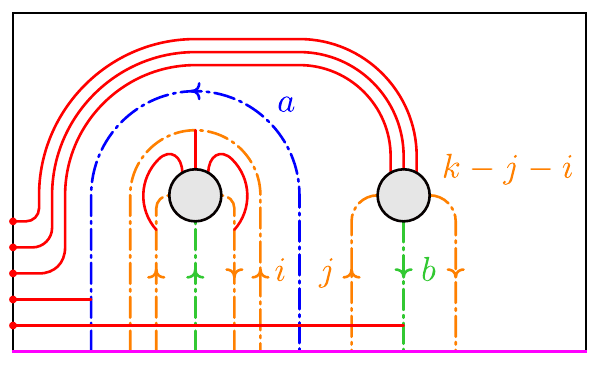} 
  = \sum_{j=0}^k \sum_{i=0}^{k-j} \pic{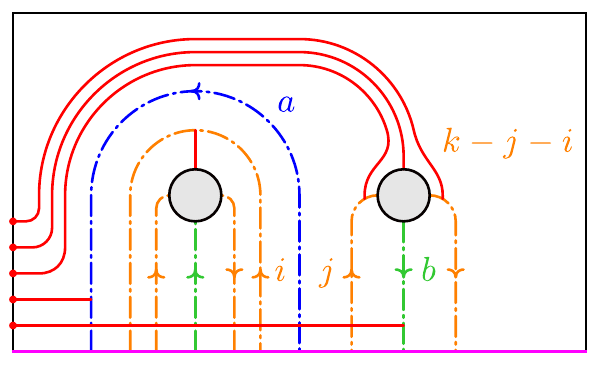} \\
  &\hspace*{\parindent}\stackrel{\substack{\mathclap{\eqref{E:hook}} \\ \mathclap{\eqref{E:braid}}}}{=} \sum_{j=0}^k \sum_{i=0}^{k-j} \pic{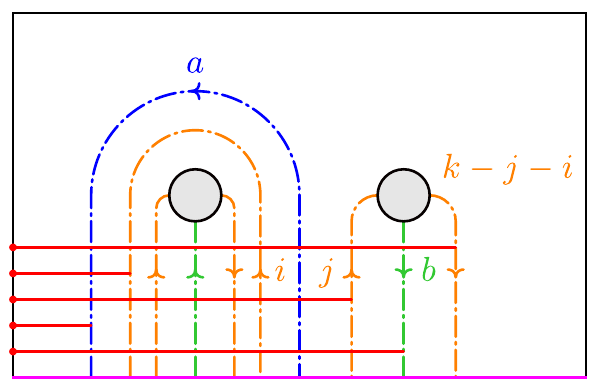} 
  \otimes 
  q^{-i(i-1)-(k-j-i)(k-j-i-1)} [\beta^{-1},\alpha^{-1}]^{k-j-i} (\alpha^{-1} \beta \alpha)^i \alpha^j \\
  &\hspace*{\parindent}\Peq q^{-k(k-1)} \sum_{j=0}^k \sum_{i=0}^{k-j} \pic{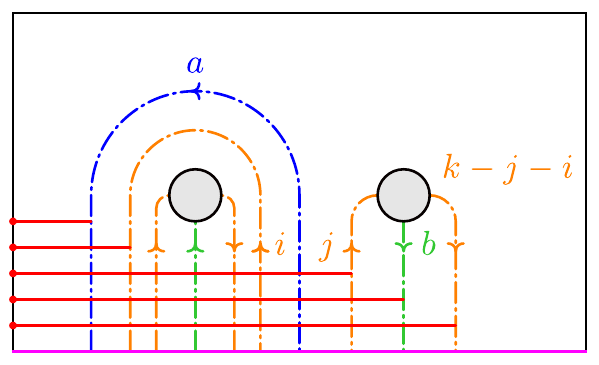} \\*
  &\hspace*{2\parindent} \otimes 
  q^{-2i^2-(j+1)j-2ij+2(i+j)k-2(k-j-i)(a+b+i+j)-2a(i+j)-4(k-j-i)-4i} \beta^i \alpha^j \\
  &\hspace*{\parindent}= q^{-(k+3)k-2(a+b)k} \sum_{j=0}^k \sum_{i=0}^{k-j} (-1)^j \basis(a+i,b+k-i) \\*
  &\hspace*{2\parindent} \otimes \sqbinom{a+i}{a}_q \sqbinom{b+k-i}{b,j}_q q^{ai+bj+(b+j)(k-j-i)+(j+3)j+2ij+2b(i+j)} \beta^i \alpha^j. \qedhere
 \end{align*}
\end{proof}

For every vector $\bfa = (a_1,\ldots,a_g)$ with integer coordinates $a_1,\ldots,a_g \geqs 0$ and every integer $1 \leqs j \leqs g$ let us set
\begin{align*}
 \left| \bfa \right|
 &:= \sum_{k=1}^g a_k, &
 \left| \bfa \right|_{<j} 
 &:= \sum_{k=1}^{j-1} a_k, &
 \left| \bfa \right|_{>j} 
 &:= \sum_{k=j+1}^g a_k.
\end{align*}

\begin{lemma}\label{L:homological_adjoint}
 For all vectors $\bfa = (a_1,\ldots,a_g)$, $\bfb = (b_1,\ldots,b_g)$ with integer coordinates $a_1,\ldots,a_g, b_1,\ldots,b_g \geqs 0$ we have
 \begin{align*}
  \calE(\basis(\bfa,\bfb)) 
  &= \sum_{j=1}^g \basis(\bfa-\bfe_j,\bfb) 
  \otimes \left( q^{2b_j} - q^{-2(a_j-b_j-1)} \alpha_j \right) q^{2
  \left| \bfa+\bfb \right|_{>j}} \\*
  &\hspace*{\parindent} + \basis(\bfa,\bfb-\bfe_j) 
  \otimes \left( 1 - q^{2(b_j-1)} \beta_j \right) q^{2 \left| \bfa+\bfb \right|_{>j}}, \\*
  \calF^{(1)}(\basis(\bfa,\bfb))
  &= \sum_{j=1}^g \basis(\bfa+\bfe_j,\bfb) 
  \otimes [a_j+1]_q \left( - q^{a_j} + q^{-a_j-4} \beta_j \right) q^{-2 \left| \bfa+\bfb \right|_{<j}+2(g-2(j-1))} \\*
  &\hspace*{\parindent} + \basis(\bfa,\bfb+\bfe_j) 
  \otimes [b_j+1]_q \left( q^{-2a_j-b_j-4}-q^{-2a_j+b_j} \alpha_j \right) q^{-2 \left| \bfa+\bfb \right|_{<j}+2(g-2(j-1))}, \\
  \calK(\basis(\bfa,\bfb)) 
  &= \basis(\bfa,\bfb) \otimes q^{-2 \left| \bfa+\bfb \right|-2g}.
 \end{align*}
\end{lemma}

\begin{proof}
 First of all, using Proposition~\ref{P:computation_rules}, we obtain
 \begin{align*}
  &\pic{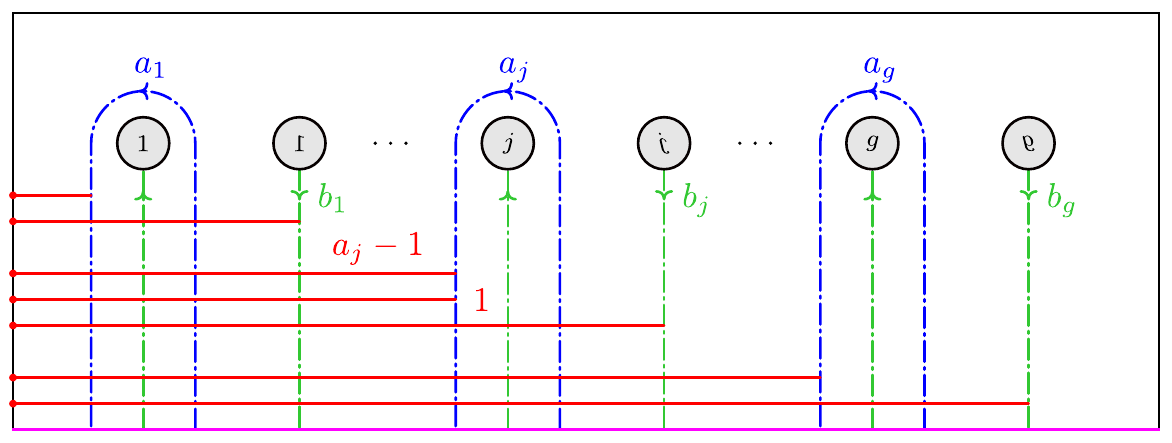} \\
  &\Peq \pic{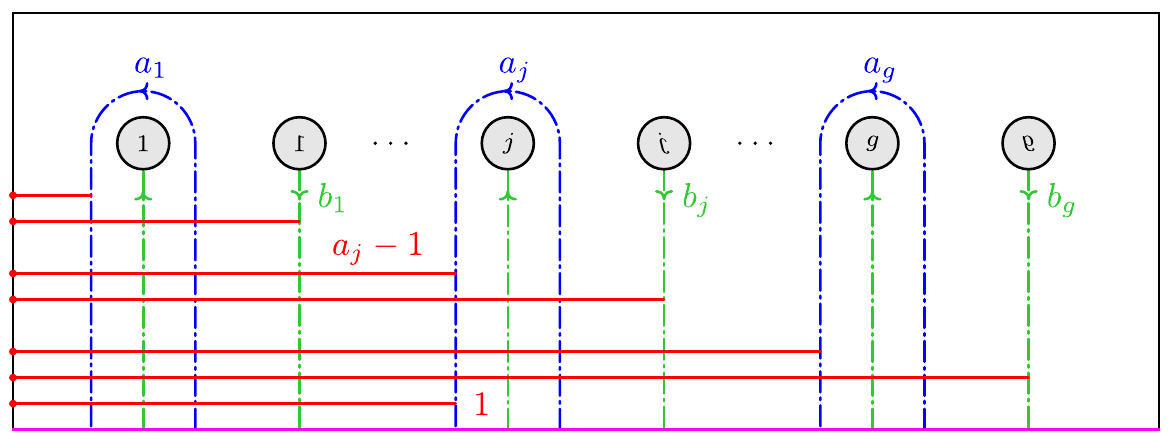} 
  \otimes q^{2b_j+2 \left| \bfa+\bfb \right|_{>j}}
 \end{align*}
 Furthermore, we obtain
 \begin{align*}
  &\pic{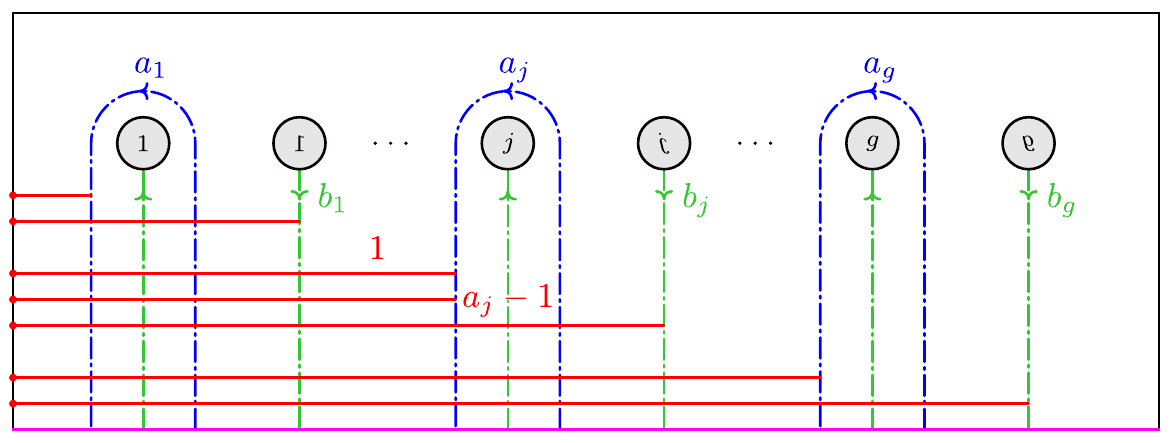} \\
  &= \pic{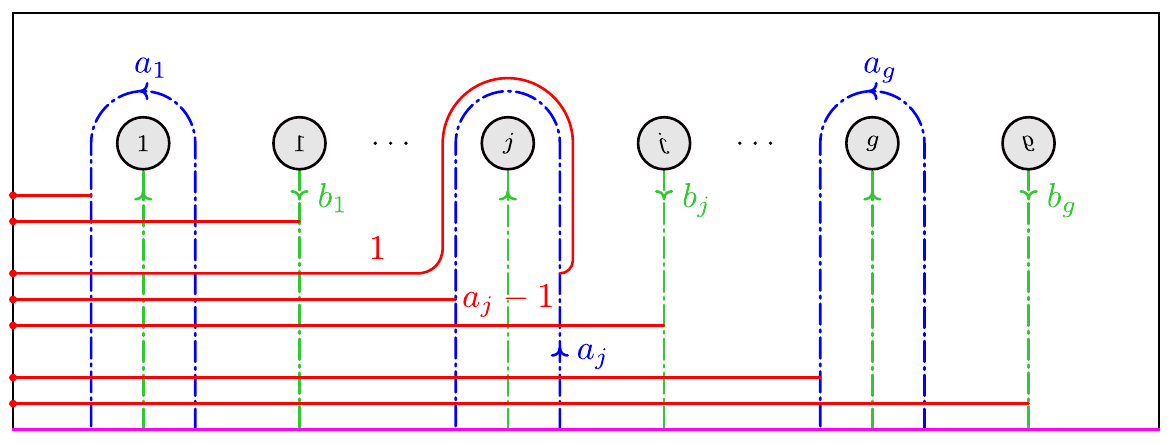} \\
  &\Beq \pic{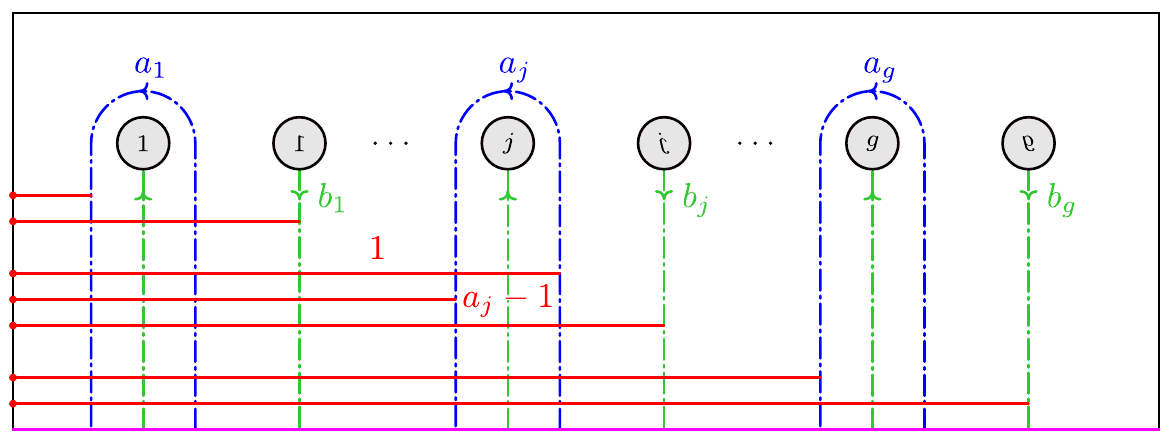} \\*
  &\hspace*{\parindent} \otimes \alpha_j \\
  &\Peq \pic{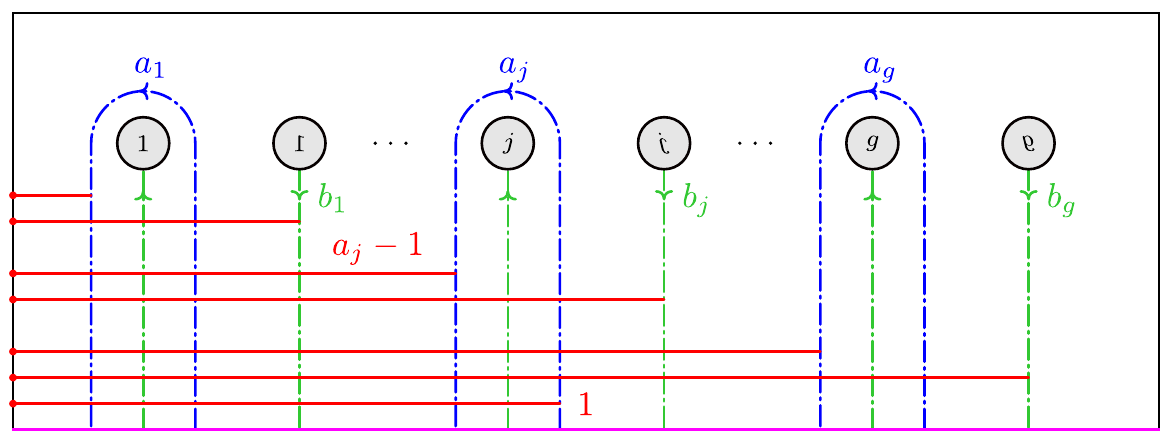} 
  \otimes q^{-2(a_j-b_j-1)+2 \left| \bfa+\bfb \right|_{>j}} \alpha_j
 \end{align*}
 Similarly, we obtain
 \begin{align*}
  &\pic{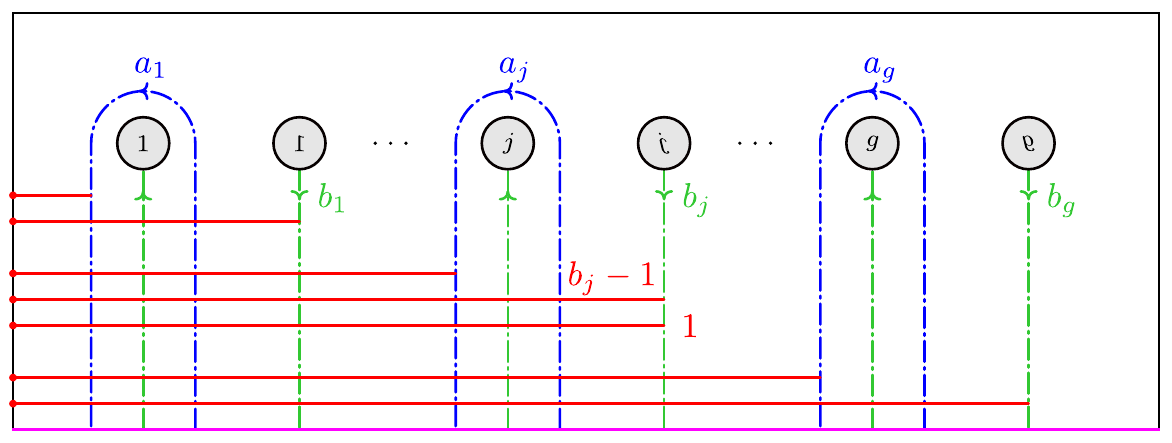} \\
  &\Peq \pic{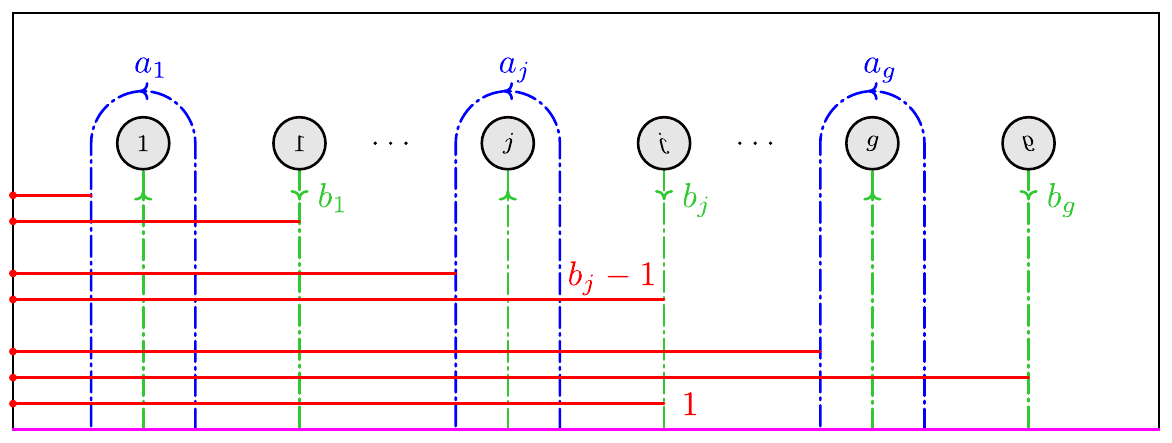} 
  \otimes q^{2 \left| \bfa+\bfb \right|_{>j}}
 \end{align*}
 Finally, we obtain
 \begin{align*}
  &\pic{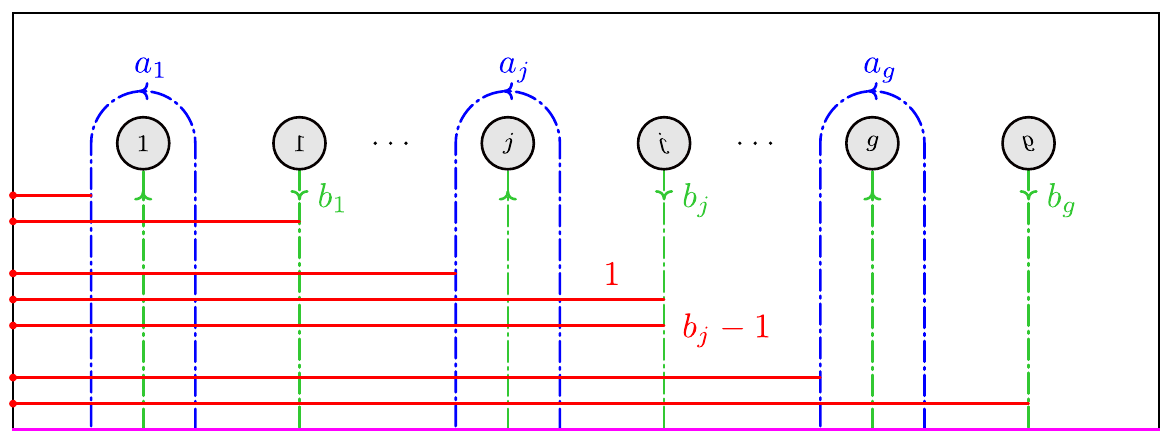} \\
  &= \pic{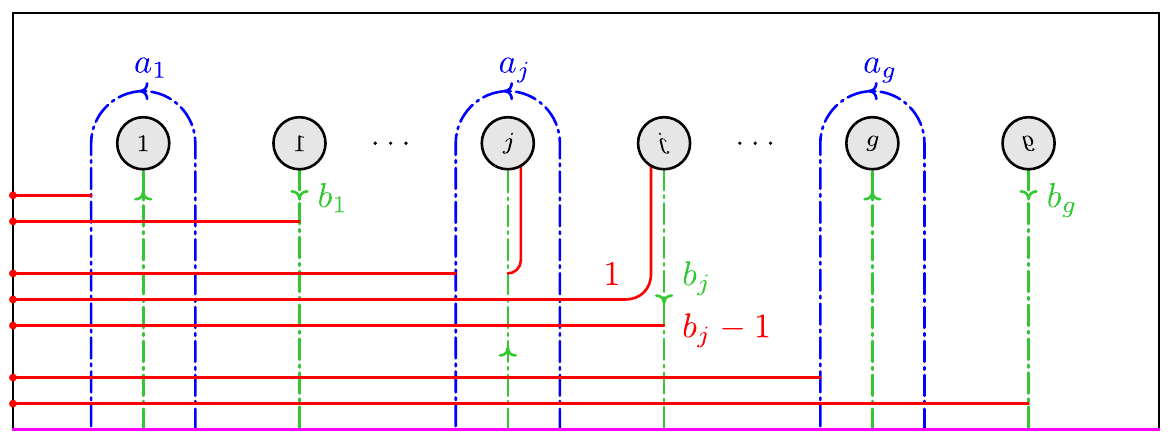} \\
  &\Beq \pic{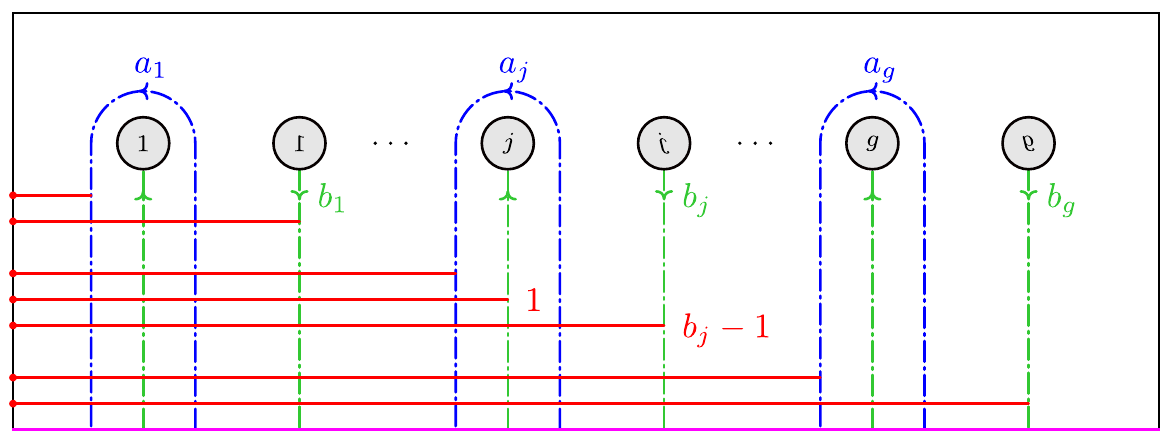} \\*
  &\hspace*{\parindent} \otimes \beta_j \\
  &\Peq \pic{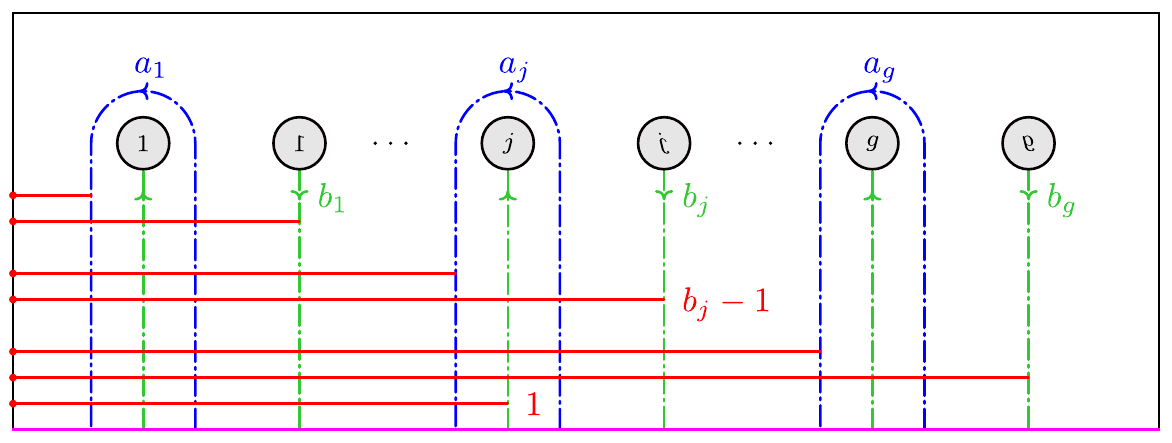} 
  \otimes q^{2(b_j-1)+2 \left| \bfa+\bfb \right|_{>j}} \beta_j
 \end{align*}
 Therefore, Definition~\ref{D:homological_E} gives
 \begin{align*}
  \calE(\basis(\bfa,\bfb)) 
  &= \sum_{j=1}^g \basis(\bfa-\bfe_j,\bfb) \otimes \left( q^{2b_j} - q^{-2(a_j-b_j-1)} \alpha_j \right) q^{2 \left| \bfa+\bfb \right|_{>j}} \\*
  &\hspace*{\parindent} + \basis(\bfa,\bfb-\bfe_j) \otimes \left( 1 - q^{2(b_j-1)} \beta_j \right) q^{2 \left| \bfa+\bfb \right|_{>j}}.
 \end{align*}
 Next, if we define $\calF^{(1)}_j(\basis(\bfa,\bfb))$ to be
 \[
  \pic{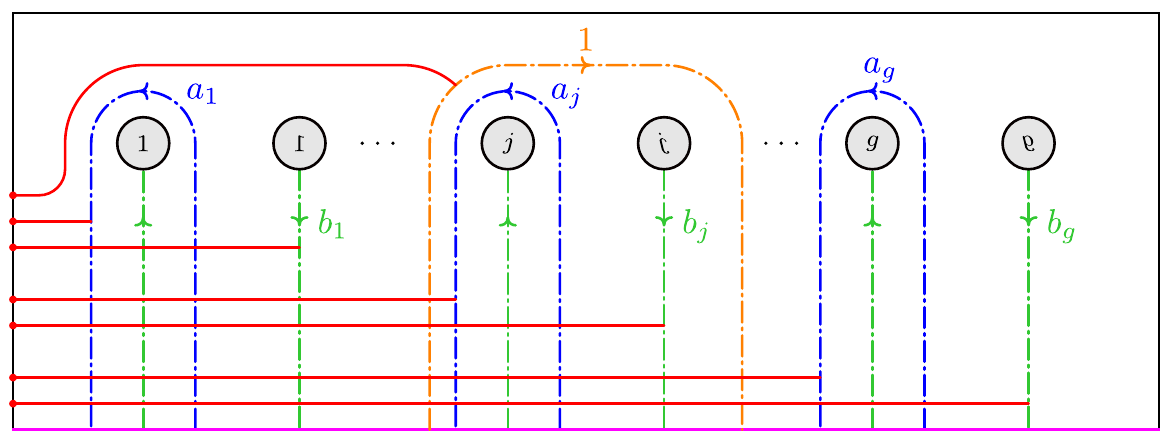}
 \]
 then, using Proposition~\ref{P:computation_rules} and  Lemma~\ref{L:F_right_regular}, we obtain
 \begin{align*}
  &\calF^{(1)}_j(\basis(\bfa,\bfb)) 
  \Beq \pic{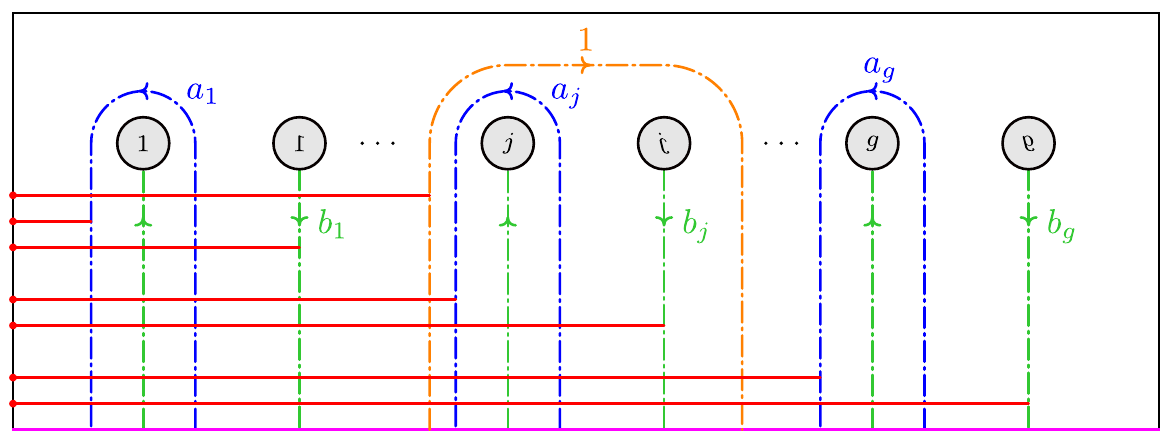} 
  \otimes \prod_{k=1}^{j-1} [\beta_k^{-1},\alpha_k^{-1}] \\
  &\Peq \pic{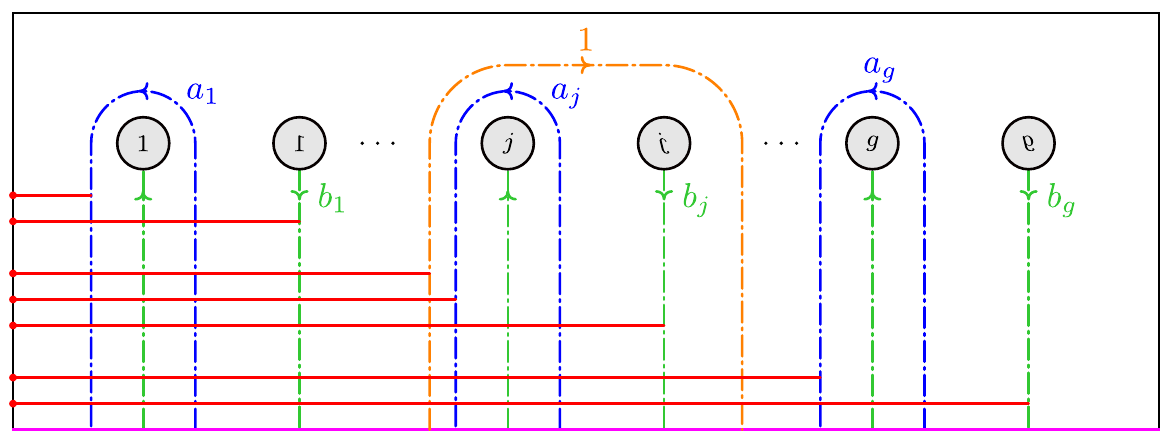} 
  \otimes q^{-2 \left| \bfa+\bfb \right|_{<j}-4(j-1)} \\
  &\Ceq \pic{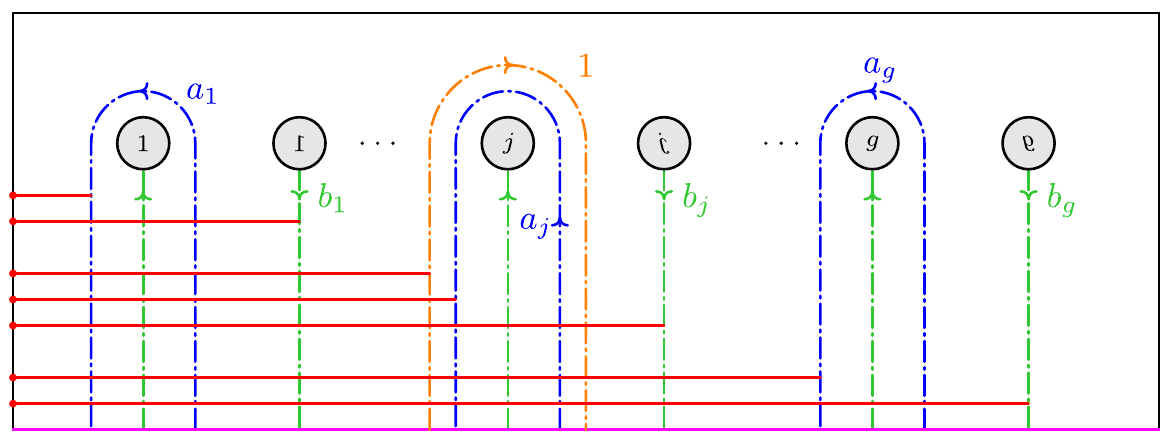} 
  \otimes q^{-2 \left| \bfa+\bfb \right|_{<j}-4(j-1)} \\*
  &+ \pic{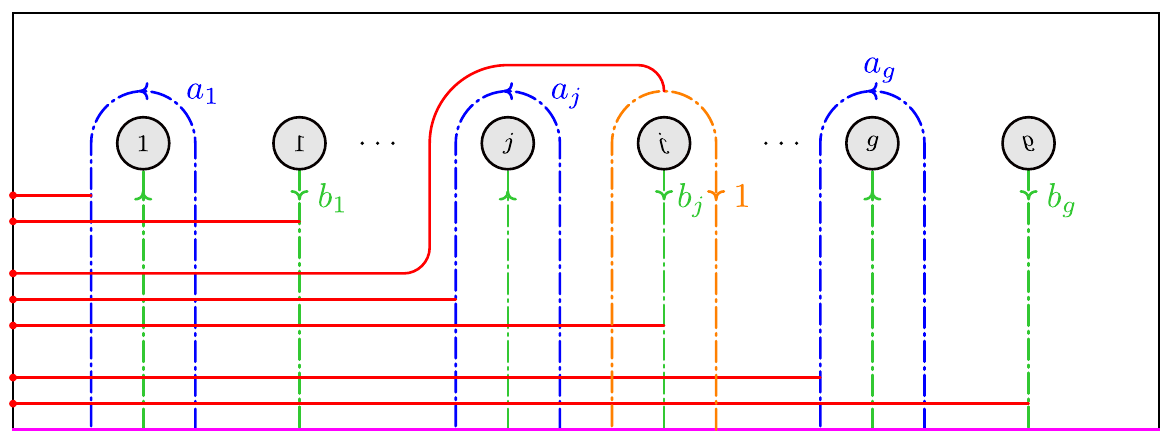} 
  \otimes q^{-2 \left| \bfa+\bfb \right|_{<j}-4(j-1)} \\
  &\stackrel{\substack{\mathclap{\eqref{E:orientation}} \\ \mathclap{\eqref{E:fusion}} \\ \mathclap{\eqref{E:F_right_regular}}}}{=}
  {}- \basis(\bfa+\bfe_j,\bfb) \otimes [a_j+1]_q q^{a_j-2 \left| \bfa+\bfb \right|_{<j}-4(j-1)} + \basis(\bfa+\bfe_j,\bfb) \otimes [a_j+1]_q q^{-a_j-4-2 \left| \bfa+\bfb \right|_{<j}-4(j-1)} \beta_j \\*
  &\hspace*{\parindent} {}+ \basis(\bfa,\bfb+\bfe_j) \otimes [b_j+1]_q \left( q^{-2a_j-b_j-4}-q^{-2a_j+b_j} \alpha_j \right) q^{-2 \left| \bfa+\bfb \right|_{<j}-4(j-1)}.
 \end{align*}
 Therefore, Definition~\ref{D:homological_F^(k)} gives
 \begin{align*}
  \calF^{(1)}(\basis(\bfa,\bfb)) 
  &\Ceq \sum_{j=1}^g \basis(\bfa+\bfe_j,\bfb) 
  \otimes [a_j+1]_q \left( - q^{a_j} + q^{-a_j-4} \beta_j \right) q^{-2 \left| \bfa+\bfb \right|_{<j}+2(g-2(j-1))} \\*
  &\hspace*{\parindent} + \basis(\bfa,\bfb+\bfe_j) 
  \otimes [b_j+1]_q \left( q^{-2a_j-b_j-4}-q^{-2a_j+b_j} \alpha_j \right) q^{-2 \left| \bfa+\bfb \right|_{<j}+2(g-2(j-1))}.
 \end{align*}
 Finally, the last equality follows directly from Definition~\ref{D:homological_K}.
\end{proof}

\subsection{Homological action of mapping class groups}\label{S:homological_mcg_computation}

Let us compute now the projective action of Dehn twists on the total modular Heisenberg homology group $\calH_g^V$. Notice that the projective representation $\bar{\rho}_g^V : \Mod(\varSigma_{g,1}) \to \PGL_{U_\zeta}(\homol_g^V)$ of Theorem~\ref{T:mcg_homological_projective_rep} is highly local in nature. In other words, the fact that the support of $\tau_{\alpha_j}$ and $\tau_{\beta_j}$ is contained in the $j$th summand $\varSigma_{1,1}$ inside $\varSigma_{g,1} \cong \varSigma_{1,1}^{\bcs g}$ is mirrored from the fact that $\psi_g^V(\tau_{\alpha_j})$ and $\psi_g^V(\tau_{\beta_j})$ only act on the $j$th factor of $V_{n,g} = (\fld^r)^{\otimes g}$, and similarly for $\tau_{\gamma_k}$. In other words, it is enough to compute the projective action of $\tau_\alpha = \tau_{\alpha_1}$ and $\tau_\beta = \tau_{\beta_1}$ on $\calH_1^V$, together with the one of $\tau_\gamma = \tau_{\gamma_1}$ on $\calH_2^V$. Let us start from the genus $1$ surface.

\begin{lemma}\label{L:homological_alpha}
 For all integers $0 \leqs a,b,c \leqs r-1$ we have
 \begin{align*}
  \rho^V_1(\tau_\alpha) \left( \basis(a,b) \otimes \bfv_c \right) 
  &\propto \zeta^{2(c+1)c} \sum_{i=0}^b \basis(a+i,b-i) \otimes \sqbinom{a+i}{a}_\zeta \zeta^{ai} \bfv_{c+i}.
\end{align*}
\end{lemma}

\begin{proof}
 Using Proposition~\ref{P:computation_rules} and Remark~\ref{R:computation_rules}, we obtain
 \begin{align*}
  &\rho^V_1(\tau_\alpha) \left( \basis(a,b) \otimes \bfv_c \right) 
  = \pic{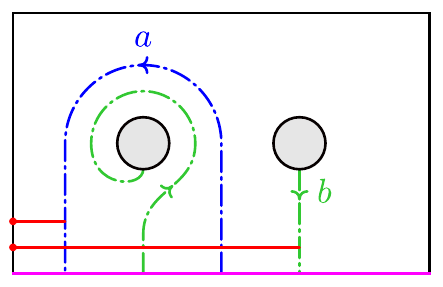} \otimes \psi_g^V(\twist{\alpha}{}) \bfv_c \\
  &\hspace*{\parindent}\tCeq \sum_{i=0}^b \pic{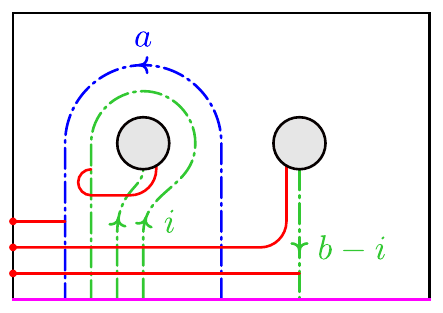} \otimes \left( \sum_{\ell=0}^{r-1}\zeta^{-2\ell(\ell-1)} A^\ell \right) \bfv_c \\
  &\hspace*{\parindent}\Beq \sum_{i=0}^b \pic{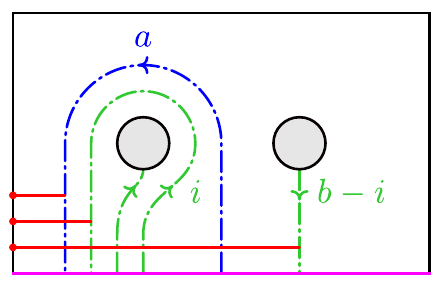} \otimes \left( \sum_{\ell=0}^{r-1}\zeta^{-2\ell(\ell-1) + 4c\ell} \right) B^i \bfv_c \\
  &\hspace*{\parindent}\Feq \fraki^{\frac{r-1}{2}} \sqrt{r} \zeta^{2(c+1)c+\frac{r+1}{2}} \sum_{i=0}^b \basis(a+i,b-i) \otimes \sqbinom{a+i}{a}_\zeta \zeta^{ai} \bfv_{c+i},
 \end{align*}
 where the last equality uses the identity
 \[
  \sum_{\ell=0}^{r-1}\zeta^{-2\ell(\ell-1) + 4c\ell} = \fraki^{\frac{r-1}{2}} \sqrt{r} \zeta^{2(c+1)c+\frac{r+1}{2}},
 \]
 see for instance \cite[Equation~(B.18)]{BD21}.
\end{proof}

\begin{lemma}\label{L:homological_beta}
 For all integers $0 \leqs a,b,c \leqs r-1$ we have
 \begin{align*}
  \rho^V_1(\tau_\beta) \left( \basis(a,b) \otimes \bfv_c \right) 
  &\propto \sum_{i=0}^a \sum_{j=0}^{r-1} (-1)^i \basis(a-i,b+i) \otimes \sqbinom{b+i}{b}_\zeta 
  \zeta^{(i+1)i-2(j+1)j-(2a-b-4c)i} \bfv_{c+j}.
 \end{align*}
\end{lemma}

\begin{proof}
 Using Proposition~\ref{P:computation_rules} and Remark~\ref{R:computation_rules}, we obtain
 \begin{align*}
  &\rho^V_1(\tau_\beta) \left( \basis(a,b) \otimes \bfv_c \right) 
  = \pic{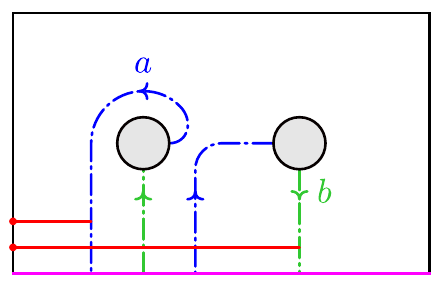} \otimes \psi_g^V(\twist{\beta}{}) \bfv_c \\
  &\hspace*{\parindent}\tCeq \sum_{i=0}^a \pic{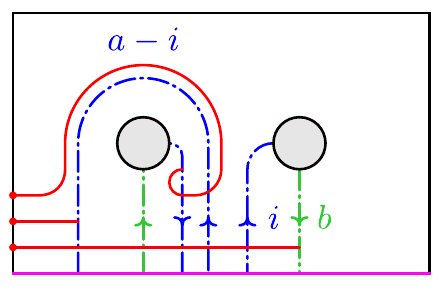} \otimes \left( \sum_{\ell=0}^{r-1} \zeta^{-2(\ell+1)\ell} B^\ell \right) \bfv_c \\
  &\hspace*{\parindent}= \sum_{i=0}^a \sum_{\ell=0}^{r-1} \pic{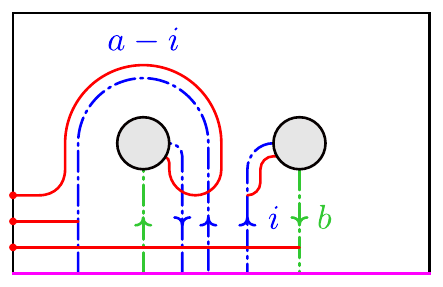} \otimes \zeta^{-2(\ell+1)\ell} \bfv_{c+\ell} \\
  &\hspace*{\parindent}\stackrel{\substack{\mathclap{\eqref{E:hook}} \\ \mathclap{\eqref{E:braid}}}}{=} \sum_{i=0}^a \sum_{\ell=0}^{r-1} \pic{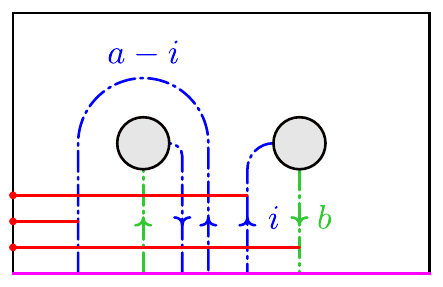} \otimes \zeta^{-2(\ell+1)\ell-i(i-1)}(B^{-1} A)^i \bfv_{c+\ell} \\
  &\hspace*{\parindent}\Peq \sum_{i=0}^a \sum_{\ell=0}^{r-1} \pic{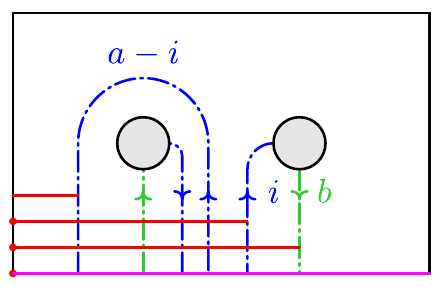} \otimes \zeta^{-2(\ell+1)\ell-3i(i-1)-2i(a-i)} B^{-i} A^i \bfv_{c+\ell} \\
  &\hspace*{\parindent}\stackrel{\substack{\mathclap{\eqref{E:orientation}} \\ \mathclap{\eqref{E:fusion}}}}{=}
  \sum_{i=0}^a \sum_{\ell=0}^{r-1} (-1)^i \basis(a-i,b+i) \otimes \sqbinom{b+i}{b}_\zeta \zeta^{bi-2(\ell+1)\ell-i(i-3)-2ai+4i(c+\ell)} \bfv_{c-i+\ell} \\
  &\hspace*{\parindent}= \sum_{i=0}^a \sum_{j=0}^{r-1} (-1)^i \basis(a-i,b+i) \otimes \sqbinom{b+i}{b}_\zeta 
  \zeta^{bi-2(i+j+1)(i+j)-i(i-3)-2(a-2c)i+4i(i+j)} \bfv_{c+j},
 \end{align*}
 where the last equality follows from the change of variable $\ell=i+j$.
\end{proof}

Let us move on to the genus $2$ surface.

\begin{lemma}\label{L:homological_gamma}
 For all integers $0 \leqs a_1,b_1,c_1,a_2,b_2,c_2 \leqs r-1$ we have we have
 \begin{align*}
  &\rho^V_2(\tau_\gamma) \left( \basis(a_1,b_1;a_2,b_2) \otimes \bfv_{(c_1,c_2)} \right) \\
  &\hspace*{\parindent} \propto \zeta^{2(b_1+c_1+a_2-c_2+1)(b_1+c_1+a_2-c_2)} \sum_{k_2=0}^{a_2} \sum_{j_2=0}^{b_2} \sum_{i_2=k_2}^{a_2} \sum_{k_1=0}^{b_1} \sum_{\ell=-k_1}^{i_2+j_2} \sum_{j_1=0}^{k_1+\ell} \sum_{i_1=0}^{k_1-j_1+\ell} (-1)^{j_1+\ell+i_2+k_2} \\*
  &\hspace*{2\parindent} \basis(a_1+i_1,b_1-i_1+\ell;a_2-\ell+j_2,b_2-j_2) \otimes \sqbinom{a_1+i_1}{a_1}_\zeta \sqbinom{b_1-i_1+\ell}{b_1-k_1,j_1}_\zeta \sqbinom{a_2-\ell+j_2}{a_2-i_2}_\zeta \sqbinom{k_1+i_2+j_2}{k_1,j_2,k_2}_\zeta \\*
  &\hspace*{2\parindent} \zeta^{k_1(k_1-2)+(\ell+1)\ell+i_1j_1-i_1k_1-j_1k_1+j_1\ell-k_1\ell+k_1i_2+k_1j_2-\ell i_2-2\ell j_2+i_2k_2+2j_2k_2} \\*
  &\hspace*{2\parindent} \zeta^{(a_1+b_1)i_1+(2b_1+4c_1+3)j_1-(3b_1+4c_1)k_1+(b_1-a_2)\ell-(4b_1+4c_1+a_2+3)i_2-(4b_1+4c_1+a_2+4)j_2-(2a_2-4c_2-1)k_2} \bfv_{(c_1+i_1,c_2+j_2)}.
 \end{align*}
\end{lemma}

Due to the technical nature of this computation, we postpone a detailed proof to Appendix~\ref{A:ugly_homological_computation}.

\section{Quantum representations of mapping class groups}\label{S:quantum_representations}

A Hopf algebra over a field $\Bbbk$ is a $\Bbbk$-vector space $H$ equipped with a family of $\Bbbk$-linear maps composed of a unit $\eta : \Bbbk \to H$, a product $\mu : H \otimes H \to H$, a counit $\varepsilon : H \to \Bbbk$, a coproduct $\Delta : H \to H \otimes H$, and an antipode $S : H \to H$. These structure morphisms are subject to a well-known list of axioms, that the reader can find in \cite[Definitions~III.1.1, III.2.2, \& III.3.2]{K95}. For all elements $x,y \in H$, we will use the short notation $\mu(x \otimes y) = xy$ (for the product), $\eta(1) = 1$ (for the unit), and $\Delta(x) = x_{(1)} \otimes x_{(2)}$ (for the coproduct, which hides a sum).

In Section~\ref{S:Hopf_algebras}, we will recall the algebraic setup required for the definition of quantum representations of mapping class groups based on Hopf algebras. Then, in Section~\ref{S:quantum_representations_of_mcg}, we will recall how to represent diagrammatically connected cobordisms, and how to compute their image under a TQFT by means of a diagrammatic calculus that is very different from the one introduced in Section~\ref{S:twisted_homology_modules} (which is more homological, and less algebraic). In particular, we will provide, in Proposition~\ref{P:quantum_action_of_mcg_generators}, explicit formulas for the action of generators of mapping class groups in Hopf algebraic terms. Finally, in Section~\ref{S:small_quantum_sl2}, we will discuss the example of the small quantum group of $\fsl_2$, which is the one we will relate to homology later on.

\subsection{Factorizable ribbon Hopf algebras}\label{S:Hopf_algebras}

In this section, we recall the main pieces of structure and properties of Hopf algebras required for the construction of quantum representations of mapping class groups.

\subsubsection{Ribbon structures and factorizability}

A \textit{ribbon structure} on $H$ is given by an R-matrix $R = R'_i \otimes R''_i \in H \otimes H$ (which hides a sum) and by a ribbon element $v \in H$, see \cite[Definitions~VIII.2.2. \& XIV.6.1]{K95}. We denote with $u \in H$ the Drinfeld element and with $M \in H \otimes H$ the M-matrix associated with the R-matrix $R$, which are defined by $u = S(R''_i)R'_i$ and by $M = R''_j R'_i \otimes R'_j R''_i$ respectively (with sums hidden in both notations). 
We also denote with $g \in H$ the unique pivotal element compatible with $v$, which is given by $g = uv^{-1}$.

A left integral $\lambda \in H^*$ of $H$ is a linear form on $H$ satisfying $\lambda(x_{(2)}) x_{(1)} = \lambda(x) 1$ for every $x \in H$, and a left cointegral $\Lambda \in H$ of $H$ is an element of $H$ satisfying $x \Lambda = \varepsilon(x) \Lambda$ for every $x \in H$, see \cite[Definition~10.1.1 \& 10.1.2]{R12}. Recall that, if $H$ is finite-dimensional, then a left integral and a left cointegral exist, they are unique up to scalar, and we can lock together their normalizations by requiring
\[
 \lambda(\Lambda) = 1,
\]
as follows from \cite[Theorem~10.2.2]{R12}. Now, recall that a finite-dimensional Hopf algebra $H$ is \textit{unimodular} if $S(\Lambda) = \Lambda$, compare with \cite[Definition~10.2.3]{R12}. If $H$ is unimodular, \cite[Theorem~10.5.4.(e)]{R12} implies that the left integral $\lambda$ satisfies
\begin{equation}\label{E:quantum_character}
 \lambda(xy) = \lambda(yS^2(x))
\end{equation}
for all $x,y \in H$, see also \cite[Theorem~7.18.12]{EGNO15}.

The Drinfeld map $D : H^* \to H$ of a ribbon Hopf algebra $H$ is the linear map determined by $D(f) := (f \otimes \id_H)(M)$ for every $f \in H^*$, where $M$ is the M-matrix of $H$. By definition, $H$ is \textit{factorizable} if $D$ is a linear isomorphism. This happens if and only if $\lambda(R'_j R''_i)R''_j R'_i$ is a cointegral, see \cite[Theorem~5]{K96} and \cite[Proposition~7.1]{BD21}, and we fix the normalization of both $\lambda$ and $\Lambda$ by asking that
\[
 \lambda(R'_j R''_i)R''_j R'_i = \Lambda.
\]

\subsubsection{Adjoint representations}\label{S:Quantum_adjoint_rep}

Let us denote by $\mods{H}$ the category of finite-di\-men\-sion\-al left $H$-modules. If $H$ is finite-dimensional, then its \textit{adjoint representation} $\ad \in \mods{H}$ is given by the $\Bbbk$-vector space $H$ itself, equipped with the adjoint action
\[
 x \triangleright y = x_{(1)}yS(x_{(2)})
\]
for all $x \in H$ and $y \in \ad$. This $H$-module coincides with the end
\[
 \ad = \int_{X \in \mods{H}} X \otimes X^*,
\]
which is defined as the universal dinatural transformation with target
\begin{align*}
  (\_ \otimes \_^*) : \mods{H} \times (\mods{H})^\op & \to \mods{H} \\*
  (X,Y) & \mapsto X \otimes Y^*,
\end{align*}
see \cite[Section~IX.5]{M71} for a definition. Roughly speaking, $\ad$ can be equipped with a dinatural family of intertwiners
\begin{align*}
 j_X : \ad &\to X \otimes X^* \\* 
 x &\mapsto \sum_{a=1}^n (x \cdot v_a) \otimes f^a
\end{align*}
for every $H$-module $X$, where $\{ v_a \in X | 1 \leqs a \leqs n \}$ and $\{ f^a \in X^* | 1 \leqs a \leqs n \}$ are dual bases for $X$ and $X^*$ respectively. Dinaturality means that $(f \otimes \id_{X^*}) \circ j_X = (\id_Y \otimes f^*) \circ j_Y$ for every intertwiner $f : X \to Y$, and these structure morphisms make $\ad$ into the initial object of the category of dinatural transformations with target $(\_ \otimes \_^*)$, see for instance \cite[Proposition~A.2]{BD20}.

\subsection{Quantum representations of mapping class groups}\label{S:quantum_representations_of_mcg}

In this section, we recall the construction of quantum representations of mapping class groups from TQFTs.

\subsubsection{Connected cobordism category}\label{S:cobordisms}

For each connected surface $\varSigma_{g,1}$, we specify a Lagrangian subspace $\calL_g \subset H_1(\varSigma_{g,1};\R)$, defined as the subspace generated by $\{ \alpha_1,\ldots,\alpha_g,\partial \}$, where $\alpha_j$ denotes the homology class of the corresponding simple closed curve in $\varSigma_{g,1}$ for every integer $1 \leqs j \leqs g$, following the notation of Section~\ref{S:homological_representations_of_mcg}, and $\partial$ denotes the homology class of the boundary of $\varSigma_{g,1}$.

The \textit{category $\RCob$ of connected framed cobordisms between connected surfaces} is defined as follows:
\begin{itemize}
 \item Objects of $\RCob$ are connected surfaces $\varSigma_{g,1}$ with one boundary component.
 \item Morphisms of $\RCob$ from $\varSigma_{g,1}$ to $\varSigma_{g',1}$ are pairs $(M,\sig)$, where $M$ is the diffeomorphism class of a connected $3$-di\-men\-sion\-al cobordism $M$ from $\varSigma_{g,1}$ to $\varSigma_{g',1}$, and $\sig \in \Z$ is an integer, called the \textit{signature defect}.
 \item The composition 
  \[
   (M',\sig') \circ (M,\sig) \in \RCob(\varSigma_{g,1},\varSigma_{g'',1})
  \]
  of morphisms $(M,\sig) \in \RCob(\varSigma_{g,1},\varSigma_{g',1})$, $(M',\sig') \in \RCob(\varSigma_{g',1},\varSigma_{g'',1})$ is given by
  \[
   (M \cup_{\varSigma_{g',1}} M', \sig + \sig' - \mu(M_*(\calL_g),\calL_{g'},(M')^*(\calL_{g''}))),
  \]
  where $M_*(\calL_g)$ and $(M')^*(\calL_{g''})$ are Lagrangian subspaces of $H_1(\varSigma_{g',1};\R)$ obtained by pushing forward $\calL_g$ through $M$ and by pulling back $\calL_{g''}$ through $M'$, respectively, and where $\mu$ is the Maslov index, see \cite[Section~2.2]{BD21} for more details.
\end{itemize}

\subsubsection{Top tangles in handlebodies}\label{S:top_tangles}

Let us recall Habiro's graphical notation for connected framed cobordisms between connected surfaces, which we adapt from \cite[Section~14.4]{H05}. For every $g \in \N$, we consider a connected $3$-di\-men\-sion\-al handlebody $H_g \subset \R^3$ of genus $g$, obtained by attaching $g$ copies of the $3$-di\-men\-sion\-al $1$-han\-dle $D^1 \times D^2$ to the bottom face $[0,1]^{\times 2} \times \{ 0 \}$ of the cube $[0,1]^{\times 3}$. We represent graphically $H_g$ through the projection to $\R \times \{ 0 \} \times \R$ as
\begin{align*}
 \pic{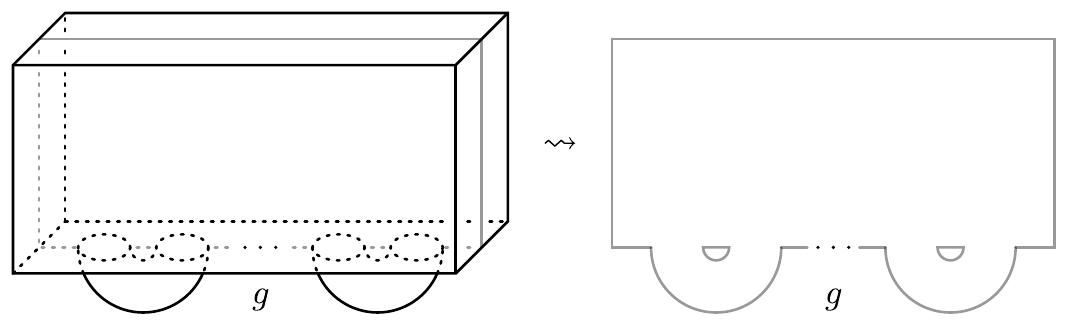}
\end{align*}
Notice that Habiro reads cobordisms from top to bottom, while we read them from bottom to top. For all $g,g' \in \N$, a \textit{top $(g,g')$-tangle} $T$, sometimes simply called a \textit{top tangle}, is an unoriented framed tangle inside the connected $3$-di\-men\-sion\-al handlebody $H_g$. Its set of boundary points is composed of $2g'$ points uniformly distributed on the top line $[0,1] \times \{ \frac 12 \} \times \{ 1 \} \subset H_g$. For every $1 \leqs j \leqs g$, a component of $T$ joins the $(2j)$th and the $(2j-1)$th boundary points. Here is an example of a top $(2,3)$-tangle, together with its projection:
\begin{align*}
 \pic{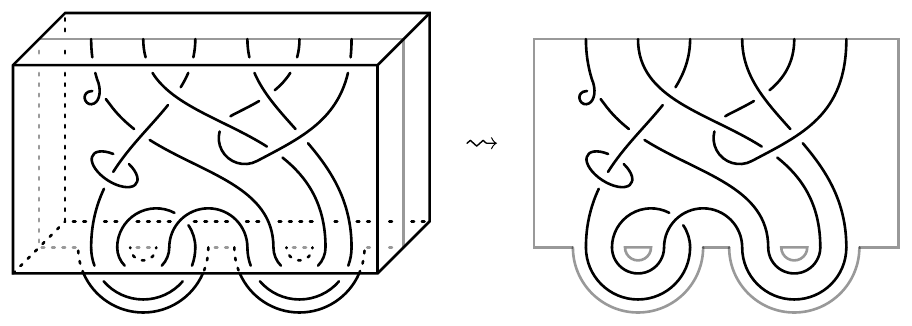}
\end{align*}
Every top $(g,g')$-tangle $T$ represents a connected framed cobordism $(M(T),\sig(T))$ from the surface $\varSigma_{g,1}$ to the surface $\varSigma_{g',1}$. The cobordism $M(T)$ is obtained from $H_g$ by carving out an open tubular neighborhood $N(\tilde{T})$ in $H_g$ of the subtangle $\tilde{T}$ of $T$ composed of all its arc components, and by performing $2$-surgery along the framed link $T \smallsetminus \tilde{T}$ composed of all its circle components, as prescribed by the framing. The signature defect $\sig(T)$ is the signature of the linking matrix of the framed link $T \smallsetminus \tilde{T} \subset H_g \subset \R^3$. Therefore, top tangles in handlebodies can be considered up to \textit{framed Kirby moves} of the following type:
\begin{gather*}
 \pic{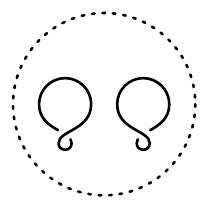} \leftrightsquigarrow \pic{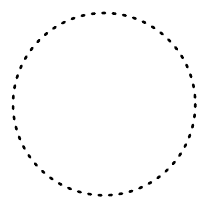} \tag{K$1'$}\label{E:rel_K1prime} \\*
 \pic{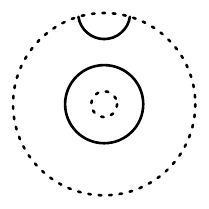} \leftrightsquigarrow \pic{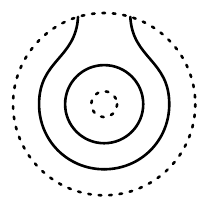} \tag{K$2$}\label{E:rel_K2} \\*
 \pic{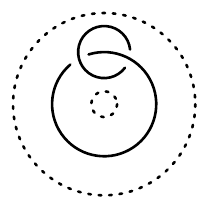} \leftrightsquigarrow \pic{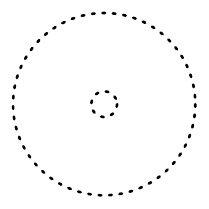} \tag{K$3$}\label{E:rel_K3}
\end{gather*}
Move~\eqref{E:rel_K1prime} is performed inside a ball $D^3$ embedded into $H_g$, while the remaining ones are performed inside a solid torus $S^1 \times D^2$ embedded into $H_g$. Notice that move~\eqref{E:rel_K1prime} is a framed version of the first Kirby move, which we need to restrict to if we want to take signature defects into account (just like we need to restrict to the framed version of the first Reidemeister move, when we want to keep track of framings of tangles).

The composition of a top $(g,g')$-tangle $T$ with a top $(g',g'')$-tangle $T'$ is the top $(g,g'')$-tangle $T' \circ T$ obtained by considering an open tubular neighborhood $N(\tilde{T})$ in $H_g$ of the subtangle $\tilde{T}$ of $T$ composed of all its arc components, by gluing vertically the complement $H_g \smallsetminus N(\tilde{T})$ to $H_{g'}$, identifying the top base of $H_g \smallsetminus N(\tilde{T})$ with the bottom base of $H_{g'}$ as prescribed by the top tangle $\tilde{T}$, and then by shrinking the result into $H_g$. Here is an example of the composition of a top $(2,1)$-tangle with a top $(1,2)$-tangle:
\begin{align*}
 \pic{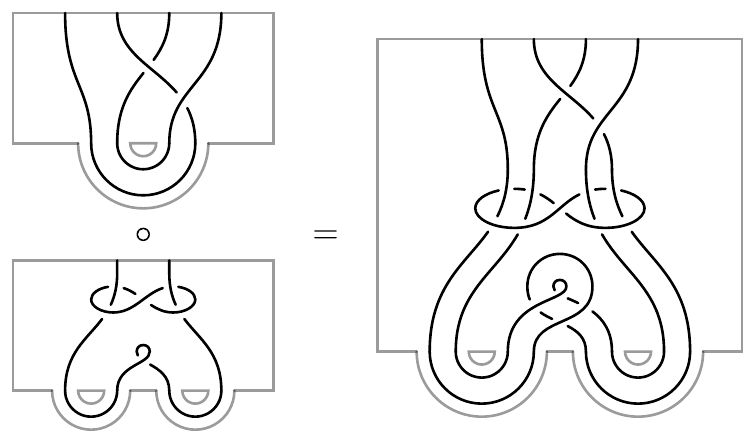}
\end{align*}
Similarly, the tensor product of a top $(g,g'')$-tangle $T$ and a top $(g',g''')$-tangle $T'$ is the top $(g+g',g''+g''')$-tangle $T \otimes T'$ obtained by gluing horizontally $H_g$ to $H_{g'}$, identifying the right face of $H_g$ with the left face of $H_{g'}$ as prescribed by the identity map, and then by shrinking the result into $H_{g+g'}$. We can even define a braiding given by
\[
 c_{g,g'} := \pic{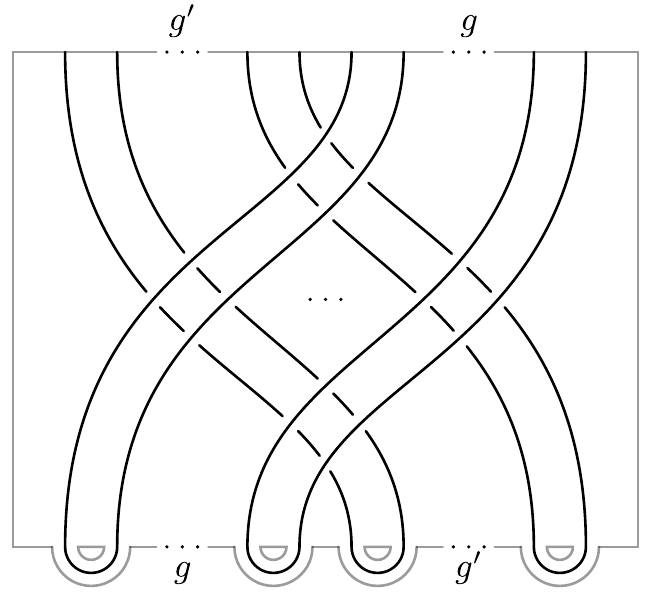}
\]
Thus, top tangles in handlebodies can be organized as the morphisms of a braided monoidal category that is equivalent to $\RCob$, see for instance \cite[Proposition~4.2]{BD21}. In particular, also $\RCob$ is a braided monoidal category, whose tensor product is induced by boundary connected sum. It is proved in \cite[Theorem~5.5.4]{BP11} that $\RCob$ is in fact the free braided monoidal category generated by a \textit{factorizable BPH algebra}\footnote{Short for \textit{Bobtcheva--Piergallini Hopf algebra}, in the terminology of \cite{BD21}.}, the surface $\varSigma_{1,1}$ (although we will not need this result, nor this notion, in the following).

\subsubsection{TQFTs and quantum representations}\label{S:TQFT_rep_of_mcg}

If $H$ admits a ribbon structure, then its adjoint representation can be given the structure of a braided Hopf algebra in $\mods{H}$. This dates back to the work of Majid \cite{M91}, who called the result the \textit{transmutation} of $H$. When $H$ is also factorizable, then its transmutation is a factorizable BPH algebra in $\mods{H}$, see \cite[Proposition~7.3]{BD21}. The following result is essentially a consequence of the construction of \cite{KL01}, compare with \cite[Theorem~7.4]{BD21}.

\begin{theoremark}[Kerler--Lyubashenko]
 If $H$ is a factorizable ribbon Hopf algebra, then there exists a braided monoidal functor
 \[
  J_H : \RCob \to \mods{H}
 \]
 sending every surface $\varSigma_{g,1}$ to the $H$-module $\ad^{\otimes g}$.
\end{theoremark}

As explained in Appendix~\ref{A:TQFT}, the functor $J_H : \RCob \to \mods{H}$ can also be seen as part of the non-semisimple TQFT constructed in \cite[Section~3]{DGP17} and in \cite[Section~4]{DGGPR19}, see \cite[Appendix~C]{DGGPR20} for a discussion of the equivalence between the two different approaches.

If $f \in \Mod(\varSigma_{g,1})$ is a positive diffeomorphism, then we denote with $\varSigma_{g,1} \times [0,1]_f$ its \textit{mapping cylinder}, which is the cobordism obtained from the cylinder $\varSigma_{g,1} \times [0,1]$ by specifying the identity of $\varSigma_{g,1}$ as the outgoing boundary identification, and $f$ as the incoming one. Mapping cylinders behave well under gluing, in the sense that
\[
 \varSigma_{g,1} \times [0,1]_{f'} \cup_{\varSigma_{g,1}} \varSigma_{g,1} \times [0,1]_f \cong \varSigma_{g,1} \times [0,1]_{f' \circ f}.
\]
Then, the \textit{quantum representation} associated with the factorizable ribbon Hopf algebra $H$ and the surface $\varSigma_{g,1}$ is the homomorphism
\begin{align}
 \bar{\rho}_g^H : \Mod(\varSigma_{g,1}) &\to \PGL_\Bbbk(\ad^{\otimes g}) \label{E:quantum_representation} \\*
 [f] &\mapsto [J_H(\varSigma_{g,1} \times [0,1]_f,0)]. \nonumber
\end{align}
The fact that this representation is only projective, instead of $\Bbbk$-linear, follows from the fact that the composition of two mapping cylinders with zero signature defect needs not have zero signature defect. However, it turns out that, for every cobordism $M$, we have
\[
 J_H(M,\sig) = \lambda(v)^{-\sig} J_H(M,0).
\]
Since the scalar $\lambda(v) \in \Bbbk$ is always invertible for a factorizable ribbon Hopf algebra $H$, forgetting about signature defects yields a homomorphism with target $\PGL_\Bbbk(\ad^{\otimes g})$. This can be lifted to a linear representation with target $\GL_\Bbbk(\ad^{\otimes g})$, but the source needs to be changed to a central extension of $\Mod(\varSigma_{g,1})$ that takes Maslov indices into account, see for instance \cite[Appendix~A]{DGGPR20}. Notice however that $\PGL_\Bbbk(\ad^{\otimes g})$ is a linear group, since its action by conjugation on $\GL_\Bbbk(\ad^{\otimes g})$ is faithful. Therefore, faithfulness of $\bar{\rho}_g^H$ would still directly imply linearity of $\Mod(\varSigma_{g,1})$.

\subsubsection{Diagrammatic calculus}\label{S:quantum_diagrammatic_calculus}

Let us recall a useful algorithm for the computation of the TQFT functor $J_H : \RCob \to \mods{H}$, that can also be found in \cite[Section~8.2]{BD21}. The procedure is based on \textit{singular diagrams} of framed tangles, which are obtained from regular diagrams by discarding framings, and forgetting the difference between overcrossings and undercrossings. On the set of singular diagrams, we consider the equivalence relation generated by all singular versions of the usual local moves corresponding to ambient isotopies of framed tangles, except for the first Reidemeister move. In particular, two equivalent singular diagrams represent homotopic tangles, but not all homotopies are allowed. In order to compute the intertwiner $J_H(M(T),\sig(T)) : \ad^{\otimes g} \to \ad^{\otimes g'}$ associated with a top $(g,g')$-tangle $T$, we start by considering a vector
\[
 x_{1} \otimes \ldots \otimes x_g \in \ad^{\otimes g}
\]
and a regular diagram of $T$. First of all, if the number of strands of $T$ running along the $j$th $1$-han\-dle of $H_g$ is $k_j$, we insert beads labeled by components of the $(k_j-1)$th iterated coproduct $(x_j)_{(1)} \otimes \ldots \otimes (x_j)_{(k_j)}$ for every integer $1 \leqs j \leqs g$, as shown:
\[
 \pic{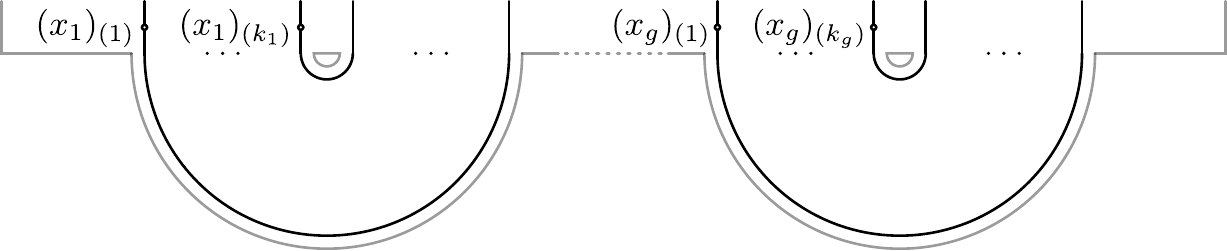}
\]
When $k_j = 0$, we add a multiplicative factor of $\varepsilon(x_j)$ in front of $T$. Next, we pass to the singular version of $T$, while also inserting beads labeled by components of the R-matrix as shown:
\begin{align*}
 \pic{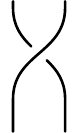} &\mapsto \pic{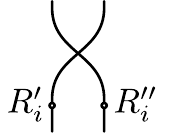} &
 \pic{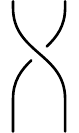} &\mapsto \pic{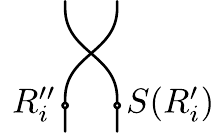}
\end{align*}
Then, we attach a complementary $2$-han\-dle canceling each $1$-han\-dle of $H_g$, as shown:
\begin{align*}
 \pic{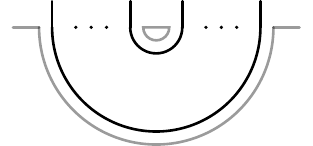} &\mapsto \pic{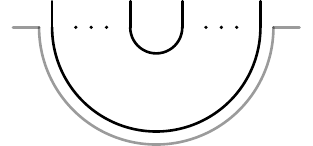}
\end{align*}
Next, we need to collect all beads sitting on the same strand in one place, which has to be next to the leftmost endpoint, for components which are not closed. As we slide beads past maxima, minima, and crossings, we change their labels according to the rule
\begin{align*}
 \pic{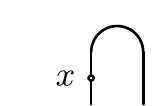} &= \pic{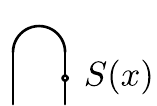} &
 \pic{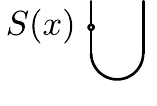} &= \pic{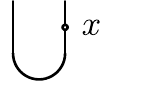} \\*[10pt]
 \pic{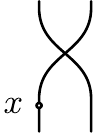} &= \pic{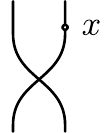} &
 \pic{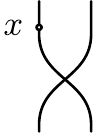} &= \pic{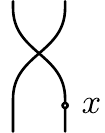}
\end{align*}
Next, we pass from our singular diagram to an equivalent one whose singular crossings all belong to singular versions of twist morphisms, and we replace them with beads labeled by pivotal elements according to the rule
\begin{align*}
 \pic{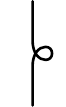} &= \pic{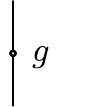} &
 \pic{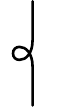} &= \pic{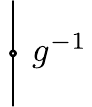}
\end{align*}
This is indeed possible, because we started from a top tangle. Finally, we collect all remaining beads, changing their labels along the way as before, and we multiply everything together according to the rule
\begin{align*}
 \pic{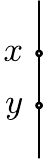} &= \pic{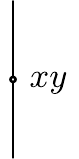}
\end{align*}
In the end, we are left with a planar tangle $B(T)$ carrying at most a single bead on each of its components. The intertwiner $J_H(M(T),\sig(T)) : \ad^{\otimes g} \to \ad^{\otimes g'}$ satisfies
\[
 J_H(M(T),\sig(T))(x_1 \otimes \ldots \otimes x_g) = 
 \left( \prod_{i=1}^g \lambda(y_i x_i) \right) 
 \left( \prod_{j=1}^k \lambda(z_j) \right)
 x'_1 \otimes \ldots \otimes x'_{g'}
\]
for every $x_{1} \otimes \ldots \otimes x_g \in \ad^{\otimes g}$, where
\begin{equation*}
 B(T) = \pic{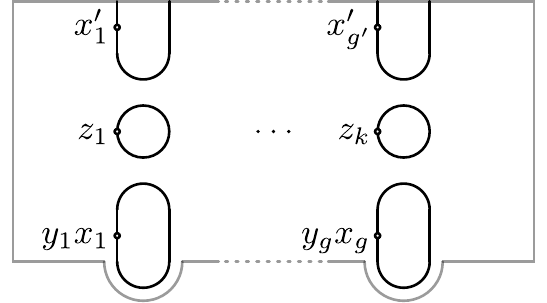}
\end{equation*} 
More details can be found in \cite[Section~8.2]{BD21}.

\subsubsection{Action of Dehn twists}\label{S:quantum_action_of_Dehn_twists}

Let $\gamma_\pm$ denote the knot $\gamma \times \left\{ \frac 12 \right\}$ inside $\varSigma_{g,1} \times [0,1]$ with framing $\pm 1$ relative to the surface $\varSigma_{g,1} \times \left\{ \frac 12 \right\}$. It is a classical remark, see for instance the proof of \cite[Theorem~2]{L62}, that there exists an isomorphism of cobordisms
\begin{equation*}
 \varSigma_{g,1} \times [0,1]_{\tau_\gamma^{\pm 1}} \cong (\varSigma_{g,1} \times [0,1])(\gamma_\mp),
\end{equation*}
where the cobordism $(\varSigma_{g,1} \times [0,1])(\gamma_\mp)$ is obtained from the cylinder $\varSigma_{g,1} \times [0,1]$ by performing surgery along the framed knot $\gamma_\pm$, with both incoming and outgoing boundary identifications induced by the identity of $\varSigma_{g,1}$.

\begin{proposition}\label{P:quantum_action_of_mcg_generators}
 The quantum representation $\bar{\rho}_g^H : \Mod(\varSigma_{g,1}) \to \PGL_\Bbbk(\ad^{\otimes g})$ of Equation~\eqref{E:quantum_representation} fits into the commutative diagram 
 \begin{center}
  \begin{tikzpicture}[descr/.style={fill=white}] \everymath{\displaystyle}
   \node (P0) at (0,0) {$F_g$};
   \node (P1) at (3,0) {$\GL_\Bbbk(\ad^{\otimes g})$};
   \node (P2) at (0,-1) {$\Mod(\varSigma_{g,1})$};
   \node (P3) at (3,-1) {$\PGL_\Bbbk(\ad^{\otimes g})$};
   \draw
   (P0) edge[->] node[above] {\scriptsize $\rho_g^H$}(P1)
   (P0) edge[->>] node[left] {} (P2)
   (P1) edge[->>] node[right] {} (P3)
   (P2) edge[->] node[below] {\scriptsize $\bar{\rho}_g^H$} (P3);
  \end{tikzpicture}
 \end{center}
 where $F_g$ is defined by Equation~\eqref{E:free_group}, and $\rho_g^H : F_g \to \GL_\Bbbk(\ad^{\otimes g})$ is defined by
 \begin{align*}
  \rho_g^H(\twist{\alpha}{j}) (x_1 \otimes \ldots \otimes x_g) &= x_1 \otimes \ldots \otimes x_{j-1} \otimes v^{-1} x_j \otimes x_{j+1} \otimes \ldots \otimes x_g, \\
  \rho_g^H(\twist{\beta}{j}) (x_1 \otimes \ldots \otimes x_g) &= \lambda(v_{(2)} x_j) \left( x_1 \otimes \ldots \otimes x_{j-1} \otimes S(v_{(1)}) \otimes x_{j+1} \otimes \ldots \otimes x_g \right), \\
  \rho_g^H(\twist{\gamma}{k}) (x_1 \otimes \ldots \otimes x_g) &= x_1 \otimes \ldots \otimes x_{k-1} \otimes x_k S(v^{-1}_{(1)}) \otimes v^{-1}_{(2)} x_{k+1} \otimes x_{k+2} \otimes \ldots \otimes x_g
 \end{align*}
 for all $x_1 \otimes \ldots \otimes x_g \in \ad^{\otimes g}$.
\end{proposition}

\begin{proof}
 First of all, the mapping cylinder $\varSigma_{g,1} \times [0,1]_{\tau_{\alpha_j}}$ is represented by
 \[
  \pic{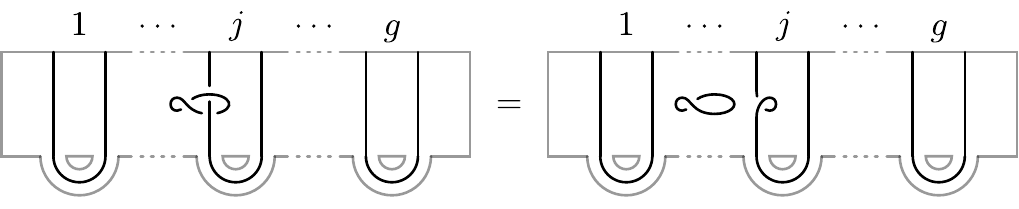}
 \]
 where the equality follows from move~\eqref{E:rel_K2}. Now, the algorithm for the computation of $J_H$ gives
 \[
  \pic{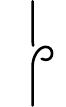} = \pic{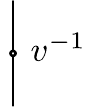}
 \]
 because
 \[
  R''_iS^2(R'_i)g = u^{-1}g = v^{-1}.
 \]
 Notice that
 \[
  \pic{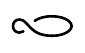} = \lambda(v)
 \]
 only contributes an invertible scalar, which can be ignored. Next, the mapping cylinder $\varSigma_{g,1} \times [0,1]_{\tau_{\beta_j}}$ is represented by
 \[
  \pic{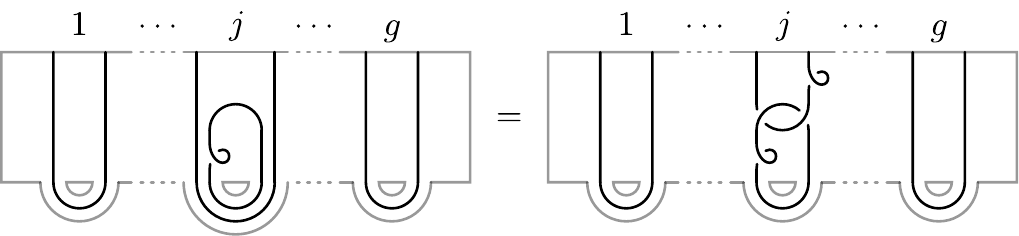}
 \]
 where the equality follows from move~\eqref{E:rel_K2}. Now, the algorithm for the computation of $J_H$ gives
 \[
  \pic{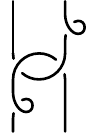} = \pic{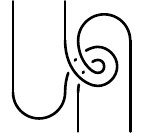} = \pic{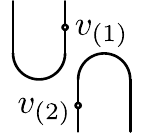}
 \]
 because
 \[
  S(R'_i)S^2(R''_i)g = S(u)g = S(g^{-1}u) = S(v) = v.
 \]
 Finally, the mapping cylinder $\varSigma_{g,1} \times [0,1]_{\tau_{\gamma_k}}$ is represented by
 \[
  \pic{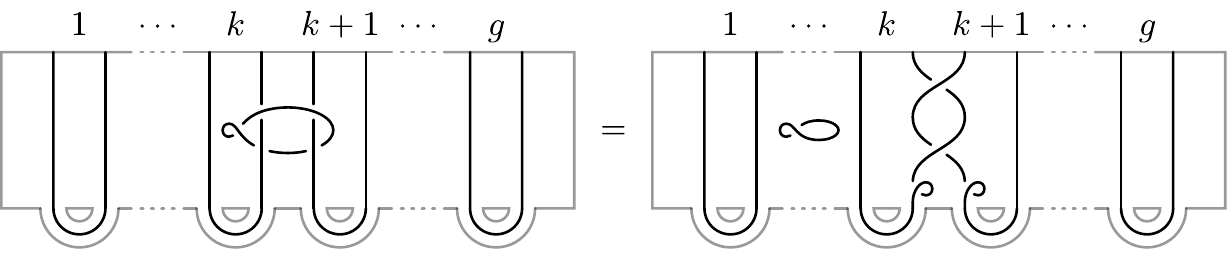}
 \]
 where the equality follows from move~\eqref{E:rel_K2}. Then, the algorithm for the computation of $J_H$ gives
 \[
  \pic{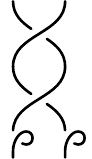} = \pic{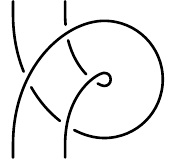} = \pic{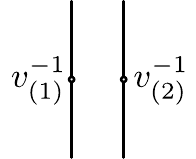}
 \]
 and we can discard again the scalar coefficient $\lambda(v)$.
\end{proof}

\subsection{Small quantum \texorpdfstring{$\fsl_2$}{sl(2)}}\label{S:small_quantum_sl2}

In this section, we recall the ribbon Hopf algebra structure on the small quantum group $\fraku_\zeta = \fraku_\zeta \fsl_2$ at a root of unity $\zeta$ of order $r \geqs 3$ odd. This example, and the quantum representations of mapping class groups it induces, will be given a homological model in Section~\ref{S:isomorphism}.

\subsubsection{Hopf algebra structure}

We start by recalling that, in Section~\ref{S:quantum_sl2_algebra}, we fixed an odd integer $3 \leqs r \in \Z$ and the primitive $r$th root of unity $\zeta = e^{\frac{2 \pi \fraki}{r}}$. The \textit{small quantum group} $\fraku_\zeta = \fraku_\zeta \fsl_2$, first defined by Lusztig in \cite{Lu90}, is the $\fld$-algebra with generators $\{ E,F^{(1)},K \}$ and relations
\begin{align*}
 E^r &= (F^{(1)})^r = 0, & K^r &= 1, &
 K E K^{-1} &= \zeta^2 E, & K F^{(1)} K^{-1} &= \zeta^{-2} F^{(1)}, & 
 [E,F^{(1)}] &= K - K^{-1}.
\end{align*}
Notice that $\fraku_\zeta$ can be identified with the quotient of the quantum group $\bar{U}_\zeta$ of Section~\ref{S:quantum_sl2_algebra} obtained by setting $E^r = 0$ and $K^r = 1$. Notice also that the presentation given here is equivalent to the one given in \cite{DGP17,DGGPR19,BD21}, which can be obtained by replacing the generator $F^{(1)}$ with the generator
\[
 F := \frac{F^{(1)}}{\{ 1 \}_\zeta}.
\]
Now $\fraku_\zeta$ admits a Hopf algebra structure obtained by setting
\begin{align*}
 \Delta(E) &= E \otimes K + 1 \otimes E, & \varepsilon(E) &= 0, & S(E) &= -E K^{-1}, \\*
 \Delta(F^{(1)}) &= K^{-1} \otimes F^{(1)} + F^{(1)} \otimes 1, & \varepsilon(F^{(1)}) &= 0, & S(F^{(1)}) &= - K F^{(1)}, \\*
 \Delta(K) &= K \otimes K, & \varepsilon(K) &= 1, & S(K) &= K^{-1}.
\end{align*}
Remark that Lusztig considers the opposite coproduct, while we are using the one of Kassel \cite[Section~VII.1]{K95}.

\subsubsection{Integral basis}\label{S:half-divided_basis}

For every integer $0 \leqs a \leqs r-1$, let us set
\begin{align*}
 T_a &:= \frac{1}{r} \sum_{b=0}^{r-1} \zeta^{2ab} K^b.
\end{align*}

\begin{lemma}\label{L:projectors}
 $\{ T_a \mid 0 \leqs a \leqs r-1 \}$ is an orthogonal family of projectors onto eigenspaces for the left regular action of $K$ on $\fraku_\zeta$, meaning that:
 \begin{enumerate}
  \item $T_a E = E T_{a+1}$, $T_a F^{(1)} = F^{(1)} T_{a-1}$, $T_a K = K T_a = \zeta^{-2a} T_a$; 
  \item $T_a T_b = \delta_{a,b} T_a$.
 \end{enumerate}
\end{lemma}

\begin{proof}
 Point $(i)$ is checked as follows:
 \begin{align*}
  T_a E 
  &= \frac{1}{r} \sum_{b=0}^{r-1} \zeta^{2ab} K^b E
  = \frac{1}{r} \sum_{b=0}^{r-1} \zeta^{2ab+2b} E K^b
  = E T_{a+1}, \\*
  T_a F^{(1)} 
  &= \frac{1}{r} \sum_{b=0}^{r-1} \zeta^{2ab} K^b F^{(1)} 
  = \frac{1}{r} \sum_{b=0}^{r-1} \zeta^{2ab-2b} F^{(1)} K^b
  = F^{(1)} T_{a-1}, \\*
  T_a K 
  &= \frac{1}{r} \sum_{b=0}^{r-1} \zeta^{2ab} K^{b+1}
  = \frac{1}{r} \sum_{c=0}^{r-1} \zeta^{2a(c-1)} K^c
  = \zeta^{-2a} T_a.
 \end{align*}
 Point $(ii)$ is checked as follows:
 \begin{align*}
  T_a T_b
  &= \frac{1}{r} \sum_{c=0}^{r-1} \zeta^{2ac} K^c T_b
  = \frac{1}{r} \sum_{c=0}^{r-1} \zeta^{2ac-2bc} T_b 
  = \delta_{a,b} T_b. \qedhere
 \end{align*}
\end{proof}

Notice that, for every integer $0 \leqs a \leqs r-1$, we have
\begin{equation}
 K^a = \sum_{b=0}^{r-1} \zeta^{-2ab} T_b. \label{E:from_T_to_K}
\end{equation}
If for every integer $0 \leqs a \leqs r-1$ we set
\[
 F^{(a)} := \frac{(F^{(1)})^a}{[ a ]_\zeta!},
\]
then the \textit{integral basis} of $\fraku_\zeta$, which is defined as
\[
 \{ E^\ell T_m F^{(n)} \mid 0 \leqs \ell,m,n \leqs r-1 \},
\]
is clearly a basis of $\fraku_\zeta$, thanks to \cite[Theorem~5.6]{Lu90}, and to Equation~\eqref{E:from_T_to_K}.

\begin{remark}\label{R:integral_quantum}
 The name \textit{integral basis} is motivated by Appendix~\ref{A:Identities_in_hdb}, where all computations performed in this basis turn out to have coefficients in $\Z[\zeta]$.
\end{remark}

\subsubsection{Ribbon structure and factorizability}

Next, $\fraku_\zeta$ supports the structure of a ribbon Hopf algebra. Indeed, an \textit{R-matrix} $R = R'_i \otimes R''_i \in \fraku_\zeta \otimes \fraku_\zeta$ is given by
\begin{align}
 R 
 &= \sum_{a,b=0}^{r-1} \zeta^{\frac{a(a-1)}{2}} K^{-b} E^a \otimes T_b F^{(a)} 
 = \sum_{a,b=0}^{r-1} \zeta^{\frac{a(a-1)}{2}} T_b E^a \otimes K^{-b} F^{(a)}, \label{E:R_Habiro}
\end{align}
compare with \cite[Example~3.4.3]{M95}. Furthermore, a \textit{pivotal element} $g \in \fraku_\zeta$ is given by $g := K$. Then, it is proved in \cite[Proposition~XIV.6.5]{K95} that $v := ug^{-1} \in \fraku_\zeta$ is a compatible \textit{ribbon element} with inverse $v^{-1} = u^{-1}g \in \fraku_\zeta$, where $u := S(R''_i)R'_i \in \fraku_\zeta$ is the \textit{Drinfeld element} associated with $R$, whose inverse is $u^{-1} = R''_iS^2(R'_i) \in \fraku_\zeta$. Explicit formulas for $v$ and $v^{-1}$ can be found in Lemma~\ref{L:ribbon}.

A non-zero left integral $\lambda$ of $\fraku_\zeta$ is given by
\begin{equation}
 \lambda(E^a F^{(b)} T_c) = \frac{\zeta^{-2c}}{\sqrt{r}} \delta_{a,r-1} \delta_{b,r-1}, \label{E:integral_Habiro}
\end{equation}
and a non-zero two-sided cointegral $\Lambda$ of $\fraku_\zeta$ satisfying $\lambda(\Lambda) = 1$ is given by
\begin{equation}
 \Lambda := \sqrt{r} E^{r-1} F^{(r-1)} T_0. \label{E:cointegral}
\end{equation}
The ribbon Hopf algebra $\fraku_\zeta$ is \textit{factorizable}, as first shown in \cite[Corollary~A.3.3]{L94}, see also \cite[Example~3.4.3]{M95}.

\section{Computation of quantum actions}\label{S:quantum_computations}

In this section, we compute explicitly the quantum action of generators of the small quantum group $\fraku_\zeta$ and of the mapping class group $\Mod(\varSigma_{g,1})$ on the $g$th tensor power $\ad^{\otimes g}$ of the adjoint representation of $\fraku_\zeta$, whose definition was recalled in Section~\ref{S:Quantum_adjoint_rep}.

\subsection{Tensor powers of the adjoint representation}\label{S:quantum_adjoint_computation}

Let us start by computing (tensor powers of) the left adjoint action of $\fraku_\zeta$ with respect to (tensor powers of) the integral basis of Section~\ref{S:half-divided_basis}.

\begin{lemma}\label{L:quantum_adjoint_torus}
 For all integers $0 \leqs \ell,m,n \leqs r-1$ we have
 \begin{align*}
  E \triangleright (E^\ell T_m F^{(n)}) 
  &= \zeta^{2(m-n)} E^{\ell+1} T_m F^{(n)} 
  - \zeta^{2(m-n+1)} E^{\ell+1} T_{m+1} F^{(n)} 
  - \{ 2m-n+1 \}_\zeta \zeta^{2(m-n+1)} E^\ell T_m F^{(n-1)}, \\*
  F^{(1)} \triangleright (E^\ell T_m F^{(n)})
  &= - [n+1]_\zeta \zeta^{-2(\ell-n)} E^\ell T_m F^{(n+1)}
  + [n+1]_\zeta E^\ell T_{m+1} F^{(n+1)} 
  - [\ell]_\zeta \{ \ell-2m-1 \}_\zeta E^{\ell-1} T_m F^{(n)}, \\*
  K \triangleright (E^\ell T_m F^{(n)})
  &= \zeta^{2(\ell-n)} E^\ell T_m F^{(n)}.
 \end{align*}
\end{lemma}

\begin{proof}
 Using the definition of $\fraku_\zeta$ from Section~\ref{S:small_quantum_sl2}, together with Lemma~\ref{L:commutators}, we obtain
 \begin{align*}
  &E \triangleright (E^\ell T_m F^{(n)}) 
  = E E^\ell T_m F^{(n)} K^{-1} - E^\ell T_m F^{(n)} E K^{-1} \\*
  &\hspace*{\parindent} \stackrel{\mathclap{\eqref{E:commutator_Habiro_T_left}}}{=} E^{\ell+1} T_m F^{(n)} K^{-1}
  - E^\ell T_m E F^{(n)} K^{-1} 
  - \{ 2m-n+1 \}_\zeta E^\ell T_m F^{(n-1)} K^{-1} \\*
  &\hspace*{\parindent} = \zeta^{2(m-n)} E^{\ell+1} T_m F^{(n)}
  - E^\ell E T_{m+1} F^{(n)} K^{-1} 
  - \{ 2m-n+1 \}_\zeta \zeta^{2(m-n+1)} E^\ell T_m F^{(n-1)} \\*
  &\hspace*{\parindent} = \zeta^{2(m-n)} E^{\ell+1} T_m F^{(n)}
  - \zeta^{2(m-n+1)} E^{\ell+1} T_{m+1} F^{(n)} 
  - \{ 2m-n+1 \}_\zeta \zeta^{2(m-n+1)} E^\ell T_m F^{(n-1)}, \\*
  &F^{(1)} \triangleright (E^\ell T_m F^{(n)})
  = - K^{-1} E^\ell T_m F^{(n)} K F^{(1)} + F^{(1)} E^\ell T_m F^{(n)} \\*
  &\hspace*{\parindent} \stackrel{\mathclap{\eqref{E:commutator_Habiro_T_right}}}{=} - \zeta^{-2(\ell-n)} E^\ell T_m F^{(n)} F^{(1)} 
  + E^\ell F^{(1)} T_m F^{(n)} - [\ell]_\zeta \{ \ell-2m-1 \}_\zeta E^{\ell-1} T_m F^{(n)} \\*
  &\hspace*{\parindent} = - [n+1]_\zeta \zeta^{-2(\ell-n)} E^\ell T_m F^{(n+1)} 
  + E^\ell T_{m+1} F^{(1)} F^{(n)} - [\ell]_\zeta \{ \ell-2m-1 \}_\zeta E^{\ell-1} T_m F^{(n)} \\*
  &\hspace*{\parindent} = - [n+1]_\zeta \zeta^{-2(\ell-n)} E^\ell T_m F^{(n+1)} 
  + [n+1]_\zeta E^\ell T_{m+1} F^{(n+1)} - [\ell]_\zeta \{ \ell-2m-1 \}_\zeta E^{\ell-1} T_m F^{(n)}, \\*
  &K \triangleright (E^\ell T_m F^{(n)})
  = K E^\ell T_m F^{(n)} K^{-1}
  = \zeta^{2(\ell-n)} E^\ell T_m F^{(n)}. \qedhere
 \end{align*}
\end{proof}

Recall that, in Section~\ref{S:homological_adjoint_computation}, for every vector $\bfa = (a_1,\ldots,a_g)$ with integer coordinates $a_1,\ldots,a_g \geqs 0$ and every integer $1 \leqs j \leqs g$, we set
\begin{align*}
 \left| \bfa \right|
 &= \sum_{k=1}^g a_k, &
 \left| \bfa \right|_{<j} 
 &= \sum_{k=1}^{j-1} a_k, &
 \left| \bfa \right|_{>j} 
 &= \sum_{k=j+1}^g a_k.
\end{align*}

\begin{lemma}\label{L:quantum_adjoint}
 For all vectors $\bfell = (\ell_1,\ldots,\ell_g)$, $\bfm = (m_1,\ldots,m_g)$, $\bfn = (n_1,\ldots,n_g)$ with integer coordinates $\ell_1,\ldots,\ell_g,m_1,\ldots,m_g,n_1,\ldots,n_g \geqs 0$ we have
 \begin{align*}
  E \triangleright (E^{\bfell} T_{\bfm} F^{(\bfn)}) 
  &= \sum_{j=1}^g \zeta^{2 \left| \bfell - \bfn \right|_{>j}} \left( \zeta^{2(m_j-n_j)} E^{\bfell + \bfe_j} T_{\bfm} F^{(\bfn)} \right. 
  - \zeta^{2(m_j-n_j+1)} E^{\bfell + \bfe_j} T_{\bfm + \bfe_j} F^{(\bfn)} \\*
  &\hspace*{\parindent}- \left. \{ 2m_j-n_j+1 \}_\zeta \zeta^{2(m_j-n_j+1)} E^{\bfell} T_{\bfm} F^{(\bfn - \bfe_j)} \right), \\*
  F^{(1)} \triangleright (E^{\bfell} T_{\bfm} F^{(\bfn)})  
  &= \sum_{j=1}^g \zeta^{-2 \left| \bfell-\bfn \right|_{<j}} \left( 
  - [n_j+1]_\zeta \zeta^{-2(\ell_j-n_j)} E^{\bfell} T_{\bfm} F^{(\bfn + \bfe_j)}
  + [n_j+1]_\zeta E^{\bfell} T_{\bfm + \bfe_j} F^{(\bfn + \bfe_j)} \right. \\*
  &\hspace*{\parindent}- \left. [\ell_j]_\zeta \{ \ell_j-2m_j-1 \}_\zeta E^{\bfell - \bfe_j} T_{\bfm} F^{(\bfn)} \right), \\*
  K \triangleright (E^{\bfell} T_{\bfm} F^{(\bfn)}) 
  &= \zeta^{2 \left| \bfell - \bfn \right|} E^{\bfell} T_{\bfm} F^{(\bfn)}.
 \end{align*}
\end{lemma}

\begin{proof}
 Since
 \begin{align*}
  E_{(1)} \otimes \ldots \otimes E_{(g)} 
  &= \sum_{j=1}^g 1^{\otimes j-1} \otimes E \otimes K^{\otimes g-j}, \\
  (F^{(1)})_{(1)} \otimes \ldots \otimes (F^{(1)})_{(g)} 
  &= \sum_{j=1}^g (K^{-1})^{\otimes j-1} \otimes F^{(1)} \otimes 1^{\otimes g-j}, \\
  K_{(1)} \otimes \ldots \otimes K_{(g)} 
  &= K^{\otimes g}, \\
 \end{align*}
 we obtain
 \begin{align*}
  E \triangleright (E^{\bfell} T_{\bfm} F^{(\bfn)}) 
  &= \sum_{j=1}^g \left( \bigotimes_{k=1}^{j-1} E^{\ell_k} T_{m_k} F^{(n_k)} \right) 
  \otimes \left( E \triangleright (E^{\ell_j} T_{m_j} F^{(n_j)}) \right) 
  \otimes \left( \bigotimes_{k=j+1}^g K \triangleright (E^{\ell_k} T_{m_k} F^{(n_k)}) \right), \\*
  F^{(1)} \triangleright (E^{\bfell} T_{\bfm} F^{(\bfn)}) 
  &=  \sum_{j=1}^g \left( \bigotimes_{k=1}^{j-1} K^{-1} \triangleright (E^{\ell_k} T_{m_k} F^{(n_k)}) \right) 
  \otimes \left( F^{(1)} \triangleright (E^{\ell_j} T_{m_j} F^{(n_j)}) \right) 
  \otimes \left( \bigotimes_{k=j+1}^g E^{\ell_k} T_{m_k} F^{(n_k)} \right), \\*
  K \triangleright (E^{\bfell} T_{\bfm} F^{(\bfn)}) 
  &= \bigotimes_{j=1}^g \left( K \triangleright (E^{\ell_j} T_{m_j} F^{(n_j)}) \right).
 \end{align*}
 Then the claim follows directly from Lemma~\ref{L:quantum_adjoint_torus}.
\end{proof}

\subsection{Quantum action of mapping class groups}\label{S:quantum_mcg_computation}

Let us compute now the projective action of Dehn twists on tensor powers of the adjoint representation $\ad^{\otimes g}$. Notice that, just like in the homological case, the projective representation $\bar{\rho}_g^{\fraku_\zeta} : \Mod(\varSigma_{g,1}) \to \PGL_{\fraku_\zeta}(\ad^{\otimes g})$ of Proposition~\ref{P:quantum_action_of_mcg_generators} is again highly local in nature. Therefore, it is enough to compute the projective action of $\tau_\alpha$ and $\tau_\beta$ on $\ad$, together with the one of $\tau_\gamma$ on $\ad^{\otimes 2}$. Let us start from the genus $1$ surface.

\begin{lemma}\label{L:quantum_alpha}
 For all integers $0 \leqs \ell,m,n \leqs r-1$ we have
 \begin{align*}
  \rho^{\fraku_\zeta}_1(\tau_\alpha) \left( E^\ell T_m F^{(n)} \right) 
  &= \zeta^{2(m+1)m} \sum_{i=0}^{r-1} \sqbinom{n+i}{n}_\zeta
  \zeta^{\frac{(i+3)i}{2}+2mi} E^{\ell+i} T_{m+i} F^{(n+i)}.
 \end{align*}
\end{lemma}

\begin{proof}
 Using Proposition~\ref{P:quantum_action_of_mcg_generators} and Lemma~\ref{L:ribbon}, together with the fact that the inverse ribbon element $v^{-1} \in \fraku_\zeta$ is central, we obtain
 \begin{align*}
  &\rho^{\fraku_\zeta}_1(\tau_\alpha) \left( E^\ell T_m F^{(n)} \right)
  = v^{-1} E^\ell T_m F^{(n)} 
  = E^\ell v^{-1} T_m F^{(n)} 
  \stackrel{\mathclap{\eqref{E:inverse_ribbon_Habiro_prime}}}{=} \sum_{a,b=0}^{r-1}
  \zeta^{\frac{(a+3)a}{2}+2(a+b+1)b} E^\ell E^a F^{(a)} T_b T_m F^{(n)} \\
  &\hspace*{\parindent} = \sum_{a,b=0}^{r-1}
  \zeta^{\frac{(a+3)a}{2}+2(a+b+1)b} \delta_{b,m} E^{\ell+a} F^{(a)} T_m F^{(n)} 
  = \sum_{a=0}^{r-1}
  \zeta^{\frac{(a+3)a}{2}+2(a+m+1)m} E^{\ell+a} T_{m+a} F^{(a)} F^{(n)} \\
  &\hspace*{\parindent} = \zeta^{2(m+1)m} \sum_{a=0}^{r-1} \sqbinom{n+a}{n}_\zeta
  \zeta^{\frac{(a+3)a}{2}+2ma} E^{\ell+a} T_{m+a} F^{(n+a)}.
  \qedhere
 \end{align*}
\end{proof}

\begin{lemma}\label{L:quantum_beta}
 For all integers $0 \leqs \ell,m,n \leqs r-1$ we have
 \begin{align*}
  \rho^{\fraku_\zeta}_1(\tau_\beta) \left( E^\ell T_m F^{(n)} \right) 
  &\propto \sum_{i,j=0}^{r-1} 
  \sqbinom{\ell}{i}_\zeta \zeta^{-\frac{i(i-5)}{2}-2j(j-1)-2ij+(\ell+2m-n)i+2\ell j} 
  E^{\ell-i} T_{m+j} F^{(n-i)}.
 \end{align*}
\end{lemma}

\begin{proof}
 First of all, using Lemmas~\ref{L:coproducts_antipodes} and \ref{L:ribbon}, we obtain
 \begin{align*}
  &S(v_{(1)}) \otimes v_{(2)}
  \stackrel{\substack{\mathclap{\eqref{E:ribbon_Habiro}} \\ \mathclap{\eqref{E:coproducts}}}}{=} \sum_{a,b,c=0}^{r-1} \sum_{i,j=0}^a (-1)^a 
  \sqbinom{a}{j}_\zeta \zeta^{-\frac{(a+3)a}{2}+2(a-b+1)b+(a+2b)i+aj-2c(i+j)-(i+j)^2} 
  S(F^{(a-i)} E^j T_{b-c}) \otimes F^{(i)} E^{a-j} T_c \\
  &\hspace*{\parindent} \stackrel{\mathclap{\eqref{E:antipodes}}}{=} \sum_{a,b,c=0}^{r-1} \sum_{i,j=0}^a (-1)^a 
  \sqbinom{a}{j}_\zeta \zeta^{-\frac{(a+3)a}{2}+2(a-b+1)b-(i+j)^2+(a+2(b-c))i+(a-2c)j} \\*
  &\hspace*{2\parindent} (-1)^{a+i+j} \zeta^{(a+2b-2c-i-j-1)(a-i-j)} 
  T_{-b+c} E^j F^{(a-i)} \otimes F^{(i)} E^{a-j} T_c \\
  &\hspace*{\parindent} = \sum_{a,b,c=0}^{r-1} \sum_{i,j=0}^a (-1)^{i+j} 
  \sqbinom{a}{j}_\zeta \zeta^{\frac{a(a-5)}{2}-2b(b-1)+4ab-2ac-(a-1)i-(a+2b-1)j} 
  T_{-b+c} E^j F^{(a-i)} \otimes F^{(i)} E^{a-j} T_c.
 \end{align*}
 Furthermore, we compute
 \begin{align*}
  &\lambda(F^{(i)} E^{a-j} T_c E^\ell T_m F^{(n)})
  \stackrel{\mathclap{\eqref{E:quantum_character}}}{=} \zeta^{-2i} \lambda(E^{a-j} E^\ell T_{\ell+c} T_m F^{(n)} F^{(i)}) 
  = \sqbinom{n+i}{n}_\zeta \zeta^{-2i} \delta_{c,-\ell+m} \lambda(E^{\ell+a-j} T_m F^{(n+i)}) \\*
  &\hspace*{\parindent} = \sqbinom{n+i}{n}_\zeta \zeta^{-2i} \delta_{c,-\ell+m} \lambda(E^{\ell+a-j} F^{(n+i)} T_{m-n-i}) 
  \stackrel{\mathclap{\eqref{E:integral_Habiro}}}{=} \sqbinom{n+i}{n}_\zeta \frac{\zeta^{-2m+2n}}{\sqrt{r}} \delta_{c,-\ell+m} \delta_{i,r-n-1} \delta_{j,\ell+a-r+1} \\*
  &\hspace*{\parindent} = \sqbinom{r-1}{n}_\zeta \frac{\zeta^{-2m+2n}}{\sqrt{r}} \delta_{c,-\ell+m} \delta_{i,r-n-1} \delta_{j,\ell+a-r+1} 
  = (-1)^n \frac{\zeta^{-2m+2n}}{\sqrt{r}} \delta_{c,-\ell+m} \delta_{i,r-n-1} \delta_{j,\ell+a-r+1},
 \end{align*}
 where the last equality uses the identity
 \[
  \frac{[r-1]_\zeta !}{[r-n-1]_\zeta !} = (-1)^n [n]_\zeta !.
 \]
 Then, using Proposition~\ref{P:quantum_action_of_mcg_generators}, we obtain
 \begin{align*}
  &\rho^{\fraku_\zeta}_1(\tau_\alpha) \left( E^\ell T_m F^{(n)} \right)
  = \lambda(v_{(2)} E^\ell T_m F^{(n)}) S(v_{(1)}) \\*
  &\hspace*{\parindent} = (-1)^n \frac{\zeta^{-2m+2n}}{\sqrt{r}} \sum_{a,b,c=0}^{r-1} \sum_{i,j=0}^a 
  (-1)^{i+j} \sqbinom{a}{j}_\zeta \zeta^{\frac{a(a-5)}{2}-2b(b-1)+4ab} \\* 
  &\hspace*{2\parindent} \zeta^{-2ac-(a-1)i-(a+2b-1)j} \delta_{c,-\ell+m} \delta_{i,r-n-1} \delta_{j,\ell+a-r+1} T_{-b+c} E^j F^{(a-i)} \\*
  &\hspace*{\parindent} = (-1)^n \frac{\zeta^{-2m+2n}}{\sqrt{r}} \sum_{a,b=0}^{r-1} 
  (-1)^{\ell+n+a} \sqbinom{a}{\ell+a-r+1}_\zeta \zeta^{\frac{a(a-5)}{2}-2b(b-1)+4ab} \\*
  &\hspace*{2\parindent} \zeta^{2a(\ell-m)+(a-1)(n+1)-(a+2b-1)(\ell+a+1)} T_{-\ell+m-b} E^{\ell+a-r+1} F^{(n+a-r+1)} \\* 
  &\hspace*{\parindent} = (-1)^\ell \frac{\zeta^{\ell-2m+n}}{\sqrt{r}} \sum_{a,b=0}^{r-1} 
  (-1)^a \sqbinom{a}{\ell+a-r+1}_\zeta \zeta^{-\frac{(a+3)a}{2}+2(a-b)b
  +(\ell-2m+n)a-2\ell b} E^{\ell+a-r+1} T_{m+a-b+1} F^{(n+a-r+1)}.
 \end{align*}
 If we change variables by setting $i=r-a-1$ and $j=a-b-r+1$, we obtain
 \begin{align*}
  &\rho^{\fraku_\zeta}_1(\tau_\alpha) \left( E^\ell T_m F^{(n)} \right) \\*
  &\hspace*{\parindent} = (-1)^\ell \frac{\zeta^{\ell-2m+n}}{\sqrt{r}} \sum_{i,j=0}^{r-1} 
  (-1)^i \sqbinom{r-i-1}{\ell-i}_\zeta \zeta^{-\frac{(i+1)(i-2)}{2}-2(i+j)(j-1)
  -(\ell-2m+n)(i+1)+2\ell(i+j)} E^{\ell-i} T_{m+j} F^{(n-i)}. 
 \end{align*}
 Now the identity
 \[
  \sqbinom{r-n-1}{k}_\zeta = (-1)^k \sqbinom{n+k}{n}_\zeta
 \]
 yields
 \begin{align*}
  &\rho^{\fraku_\zeta}_1(\tau_\alpha) \left( E^\ell T_m F^{(n)} \right) \\
  &\hspace*{\parindent} = (-1)^\ell \frac{\zeta^{\ell-2m+n}}{\sqrt{r}} \sum_{i,j=0}^{r-1} 
  (-1)^\ell \sqbinom{\ell}{i}_\zeta \zeta^{-\frac{(i+1)(i-2)}{2}-2(i+j)(j-1)
  -(\ell-2m+n)(i+1)+2\ell(i+j)} E^{\ell-i} T_{m+j} F^{(n-i)} \\*
  &\hspace*{\parindent} = \frac{\zeta}{\sqrt{r}} \sum_{i,j=0}^{r-1} 
  \sqbinom{\ell}{i}_\zeta \zeta^{-\frac{i(i-5)}{2}-2j(j-1)-2ij+(\ell+2m-n)i+2\ell j} 
  E^{\ell-i} T_{m+j} F^{(n-i)}. \qedhere
 \end{align*}
\end{proof}

Let us move on to the genus $2$ surface.

\begin{lemma}\label{L:quantum_gamma}
 For all integers $0 \leqs \ell_1,m_1,n_1,\ell_2,m_2,n_2 \leqs r-1$ we have
 \begin{align*}
  &\rho^{\fraku_\zeta}_2(\tau_\gamma) \left( E^{\ell_1} T_{m_1} F^{(n_1)} \otimes E^{\ell_2} T_{m_2} F^{(n_2)} \right) \\*
  &\hspace*{\parindent} \propto \zeta^{2(m_1-n_1+\ell_2-m_2+1)(m_1-n_1+\ell_2-m_2)} 
  \sum_{j_1=0}^{r-1} \sum_{i_1=0}^{r-j_1-1} \sum_{k_2=0}^{r-i_1-j_1-1} \sum_{b=-k_2}^{i_1+j_1} \sum_{i_2=0}^{b+k_2} (-1)^b \\*
  &\hspace*{2\parindent} \sqbinom{i_1+j_1+k_2}{i_1,j_1}_\zeta \sqbinom{\ell_2+k_2}{b-i_2+k_2}_\zeta \sqbinom{n_1-b+j_1}{n_1-i_1}_\zeta \sqbinom{n_2+i_2}{n_2}_\zeta \\*
  &\hspace*{2\parindent} \{ 2m_1-n_1+i_1+j_1;i_1 \}_\zeta \{ -\ell_2+2m_2+b;b-i_2+k_2 \}_\zeta \\*
  &\hspace*{2\parindent} \zeta^{\frac{(i_1+j_1+k_2+3)(i_1+j_1+k_2)}{2}-(i_1+j_1+k_2-1)b+2(m_1-n_1+2(\ell_2-m_2))(i_1+j_1)+2(\ell_2-m_2)k_2} \\*
  &\hspace*{2\parindent} E^{\ell_1+j_1} T_{m_1+j_1} F^{(n_1+j_1-b)} \otimes E^{\ell_2-b+i_2} T_{m_2+i_2} F^{(n_2+i_2)}.
 \end{align*}
\end{lemma}

Again, due to the technical nature of this computation, we postpone a detailed proof to Appendix~\ref{A:ugly_quantum_computation}.

\section{Identification between homological and quantum representations}\label{S:isomorphism}

In this section, we establish our main result, and discuss a few generalizations.

\subsection{Main result: isomorphism of representations}

Let us fix a positive integer $g \geqs 1$, for the genus of $\varSigma_{g,1}$, and an odd integer $r \geqs 3$, for the order of the primitive $r$th root of unity $\zeta = e^{\frac{2 \pi \fraki}{r}}$.

Recall that, in Definition~\ref{D:linear_homology}, we introduced a direct sum $\calH_g^V$ of twisted homology groups that we later endowed with commuting actions of the quantum group $U_\zeta$ and of the mapping class group $\Mod(\varSigma_{g,1})$, see Theorem~\ref{T:Uq_homological_representation}. These actions were computed in particular homological bases in Sections~\ref{S:homological_adjoint_computation} and \ref{S:homological_mcg_computation}, respectively. The following result provides an explicit isomorphism between a finite-dimensional linear subspace $\calH_g^{V(r)} \subset \calH_g^V$ and $\ad_\zeta^{\otimes g}$. What is remarkable about this linear isomorphism is that it intertwines the homological actions of $\bar{U}_\zeta$ (which is a subalgebra of $U_\zeta$) and of $\Mod(\varSigma_{g,1})$ with the quantum actions that, on the target space, arise from the non-semisimple TQFT associated with $\fraku_\zeta$ (which is a quotient of $\bar{U}_\zeta$). These quantum actions were computed in Sections~\ref{S:quantum_adjoint_computation} and \ref{S:quantum_mcg_computation}, respectively.

First, for all vectors $\bfa=(a_1 , \ldots, a_g)$, $\bfb = (b_1 , \ldots, b_g)$, and $\bfc = (c_1 , \ldots, c_r)$ with integer coordinates, we set
\begin{align*}
 N_j(a_j,b_j,c_j) &:= \zeta^{2(a_j+b_j)(j-1)+\frac{a_j(a_j-1)}{2}+2a_jb_j-2(b_j-1)c_j} \in \Z[\zeta], \\*
 N(\bfa,\bfb,\bfc) &:= \prod_{1 \leqs j < k \leqs g} \zeta^{2(a_j+b_j)(a_k+b_k)} \prod_{j=1}^g N_j(a_j,b_j,c_j) \in \Z[\zeta],
\end{align*}
where 
\begin{align*}
 \bar{\bfk} &= (k_g,\ldots,k_1), &
 \iota(\bfk) &= (r-k_1-1,\ldots,r-k_g-1)
\end{align*}
for every $\bfk = (k_1,\ldots,k_g)$.

\begin{theorem}\label{T:The_Theorem}
 The linear isomorphism
 \begin{align}
  \Phi_g^V : \calH_g^{V(r)} &\to \ad^{\otimes g} \label{E:isomorphism} \\*
  \basis(\bfa,\bfb) \otimes \bfv_{\bfc} &\mapsto 
 N(\bfa,\bfb,\bfc) E^{\iota(\bar{\bfb})} T_{\bar{\bfc}} F^{(\bar{\bfa})} \nonumber
 \end{align}
 is a $\fraku_\zeta$-module isomorphism, where the basis of $\calH_g^{V(r)}$ is given in Section~\ref{S:fin-dim_homol_rep}, while the one of $\ad^{\otimes g}$ is given in Section~\ref{S:half-divided_basis}. Furthermore, it defines an isomorphism between the homological representation
 \[
  \bar{\rho}_g^{V(r)} : \Mod(\varSigma_{g,1}) \to \PGL_{\fraku_\zeta}(\homol_g^{V(r)})
 \]
 of Theorem~\ref{T:uzeta_homological_representation} and the quantum representation
 \[
  \bar{\rho}_g^{\fraku_\zeta} : \Mod(\varSigma_{g,1}) \to \PGL_{\fraku_\zeta}(\ad^{\otimes g})
 \]
 of Proposition~\ref{P:quantum_action_of_mcg_generators}.
\end{theorem}

\begin{proof} 
 First of all, in order to show that $\Phi_g^V$ intertwines the homological action of $\bar{U}_\zeta$ on $\calH_g^{V(r)}$ with the quantum one of its quotient $\fraku_\zeta$ on $\ad^{\otimes g}$, we simply need to compare the computations of Lemma~\ref{L:homological_adjoint} with those of Lemma~\ref{L:quantum_adjoint}. On the one hand, Lemma~\ref{L:homological_adjoint} gives
 \begin{align*}
  \calE \left( \basis(\bfa,\bfb) \otimes \bfv_{\bfc} \right) 
  &= \sum_{j=1}^g \zeta^{2 \left| \bfa+\bfb \right|_{>j}} 
  \left( \basis(\bfa-\bfe_j,\bfb) \otimes \{ a_j-2c_j-1 \}_\zeta \zeta^{-a_j+2b_j+2c_j+1} \bfv_{\bfc} \right. \\*
  &\hspace*{\parindent} \left. + \basis(\bfa,\bfb-\bfe_j) 
  \otimes \left( \bfv_{\bfc} - \zeta^{2b_j-2} \bfv_{\bfc+\bfe_j} \right) \right), \\
  \calF^{(1)} \left( \basis(\bfa,\bfb) \otimes \bfv_{\bfc} \right) 
  &= \sum_{j=1}^g \zeta^{-2 \left| \bfa+\bfb \right|_{<j}+2(g-2(j-1))} 
  \left( \basis(\bfa+\bfe_j,\bfb) 
  \otimes [a_j+1]_\zeta \left( - \zeta^{a_j} \bfv_{\bfc} + \zeta^{-a_j-4} \bfv_{\bfc+\bfe_j} \right) \right. \\*
  &\hspace*{\parindent} \left. - \basis(\bfa,\bfb+\bfe_j) \otimes [b_j+1]_\zeta \{ b_j+2c_j+2 \}_\zeta \zeta^{-2a_j+2c_j-2} \bfv_{\bfc} \right), \\
  \calK \left( \basis(\bfa,\bfb) \otimes \bfv_{\bfc} \right)
  &= \basis(\bfa,\bfb) \otimes q^{-2 \left| \bfa+\bfb \right|-2g} \bfv_{\bfc}.
 \end{align*}
 On the other hand, Lemma~\ref{L:quantum_adjoint} gives
 \begin{align*}
  E \triangleright (E^{\iota(\bar{\bfb})} T_{\bar{\bfc}} F^{(\bar{\bfa})}) 
  &= \sum_{j=1}^g \zeta^{-2 \left| \bfa + \bfb \right|_{<j} - 2(j-1)} \left( \zeta^{-2(a_j-c_j)} E^{\iota(\bar{\bfb} - \bar{\bfe}_j)} T_{\bar{\bfc}} F^{(\bar{\bfa})} \right. 
  - \zeta^{-2(a_j-c_j-1)} E^{\iota(\bar{\bfb} - \bar{\bfe}_j)} T_{\bar{\bfc} + \bar{\bfe}_j} F^{(\bar{\bfa})} \\*
  &\hspace*{\parindent}+ \left. \{ a_j-2c_j-1 \}_\zeta \zeta^{-2(a_j-c_j-1)} E^{\iota(\bar{\bfb})} T_{\bar{\bfc}} F^{(\bar{\bfa} - \bar{\bfe}_j)} \right), \\
  F^{(1)} \triangleright (E^{\iota(\bar{\bfb})} T_{\bar{\bfc}} F^{(\bar{\bfa})})  
  &= \sum_{j=1}^g \zeta^{2 \left| \bfa+\bfb \right|_{>j} + 2(g-j)} \left( 
  - [a_j+1]_\zeta \zeta^{2(a_j+b_j+1)} E^{\iota(\bar{\bfb})} T_{\bar{\bfc}} F^{(\bar{\bfa} + \bar{\bfe}_j)}
  + [a_j+1]_\zeta E^{\iota(\bar{\bfb})} T_{\bar{\bfc} + \bar{\bfe}_j} F^{(\bar{\bfa} + \bar{\bfe}_j)}  \right. \\*
  &\hspace*{\parindent}- \left. [b_j+1]_\zeta \{ b_j+2c_j+2 \}_\zeta E^{\iota(\bar{\bfb} + \bar{\bfe}_j)} T_{\bar{\bfc}} F^{(\bar{\bfa})} \right), \\
  K \triangleright (E^{\iota(\bar{\bfb})} T_{\bar{\bfc}} F^{(\bar{\bfa})})
  &= q^{-2 \left| \bfa+\bfb \right|-2g} E^{\iota(\bar{\bfb})} T_{\bar{\bfc}} F^{(\bar{\bfa})}.
 \end{align*}
 Now the claim follows from the computations
 \begin{align*}
  \frac{N(\bfa,\bfb,\bfc)}{N(\bfa-\bfe_j,\bfb,\bfc)}
  &= \zeta^{2 \left| \bfa+\bfb \right|_{<j}+2 \left| \bfa+\bfb \right|_{>j}+a_j+2b_j+2j-3}, \\*
  \frac{N(\bfa,\bfb,\bfc)}{N(\bfa,\bfb-\bfe_j,\bfc)}
  &= \zeta^{2 \left| \bfa+\bfb \right|_{<j}+2 \left| \bfa+\bfb \right|_{>j}+2a_j-2c_j+2j-2}, \\*
  \frac{N(\bfa,\bfb,\bfc)}{N(\bfa,\bfb-\bfe_j,\bfc+\bfe_j)}
  &= \zeta^{2 \left| \bfa+\bfb \right|_{<j}+2 \left| \bfa+\bfb \right|_{>j}+2a_j+2b_j-2c_j+2j-6}, \\
  \frac{N(\bfa,\bfb,\bfc)}{N(\bfa+\bfe_j,\bfb,\bfc)}
  &= \zeta^{-2 \left| \bfa+\bfb \right|_{<j}-2 \left| \bfa+\bfb \right|_{>j}-a_j-2b_j-2j+2}, \\*
  \frac{N(\bfa,\bfb,\bfc)}{N(\bfa+\bfe_j,\bfb,\bfc+\bfe_j)}
  &= \zeta^{-2 \left| \bfa+\bfb \right|_{<j}-2 \left| \bfa+\bfb \right|_{>j}-a_j-2j}, \\*
  \frac{N(\bfa,\bfb,\bfc)}{N(\bfa,\bfb+\bfe_j,\bfc)}
  &= \zeta^{-2 \left| \bfa+\bfb \right|_{<j}-2 \left| \bfa+\bfb \right|_{>j}-2a_j+2c_j-2j+2}.
 \end{align*}

 Next, in order to show that $\Phi_g^V$ simultaneously intertwines the homological and the quantum projective actions of $\Mod(\varSigma_{g,1})$ on $\calH_g^{V(r)}$ and on $\ad^{\otimes g}$, respectively, we need to compare the computations of Lemmas~\ref{L:homological_alpha}--\ref{L:homological_gamma} with those of Lemmas~\ref{L:quantum_alpha}--\ref{L:quantum_gamma}, one by one. First, Lemma~\ref{L:quantum_alpha} gives
 \begin{align*} 
  \rho^{\fraku_\zeta}_1(\tau_\alpha) \left( E^{r-b-1} T_c F^{(a)} \right) 
  &\propto \zeta^{2(c+1)c} \sum_{i=0}^{r-1} \sqbinom{a+i}{a}_\zeta
  \zeta^{\frac{(i+3)i}{2}+2ic} E^{r-(b-i)-1} T_{c+i} F^{(a+i)}.
 \end{align*}
 Now the claim follows from the computation
 \begin{align*}
  \frac{N(a,b,c)}{N(a+i,b-i,c+i)}
  &= \zeta^{-\frac{(i+3)i}{2}+(a-2c)i}.
 \end{align*}

 Next, Lemma~\ref{L:quantum_beta} gives
 \begin{align*}
  &\rho^{\fraku_\zeta}_1(\tau_\beta) \left( E^{r-b-1} T_c F^{(a)} \right) 
  \propto \sum_{i,j=0}^{r-1} 
  \sqbinom{r-b-1}{i}_\zeta \zeta^{-\frac{i(i-5)}{2}-2j(j-1)-2ij-(a+b-2c+1)i-2(b+1)j} 
  E^{r-(b+i)-1} T_{c+j} F^{(a-i)} \\*
  &\hspace*{\parindent} = \sum_{i,j=0}^{r-1} 
  (-1)^i \sqbinom{b+i}{b}_\zeta \zeta^{-\frac{i(i-3)}{2}-2j^2-2ij-(a+b-2c)i-2bj} 
  E^{r-(b+i)-1} T_{c+j} F^{(a-i)}.
 \end{align*}
 Now the claim follows from the computation
 \begin{align*}
  \frac{N(a,b,c)}{N(a-i,b+i,c+j)}
  &= \zeta^{\frac{(3i-1)i}{2}+2ij-(a-2(b+c))i+2(b-1)j}.
 \end{align*}

 Finally, Lemma~\ref{L:quantum_gamma}, together with Lemma~\ref{L:Murakami}, give
 \begin{align*}
  &\rho^{\fraku_\zeta}_2(\tau_\gamma) \left( E^{\iota(b_2)} T_{c_2} F^{(a_2)} \otimes E^{\iota(b_1)} T_{c_1} F^{(a_1)} \right) \\* 
  &\hspace*{\parindent} = \zeta^{2(b_1+c_1+a_2-c_2+1)(b_1+c_1+a_2-c_2)} 
  \sum_{j_2=0}^{r-1} \sum_{i_2=0}^{r-j_2-1} \sum_{k_1=0}^{r-i_2-j_2-1} \sum_{\ell=-k_1}^{i_2+j_2} \sum_{i_1=0}^{k_1+\ell} (-1)^\ell \\*
  &\hspace*{2\parindent} \sqbinom{k_1+i_2+j_2}{i_2,j_2}_\zeta \sqbinom{r-b_1+k_1-1}{k_1-i_1+\ell}_\zeta \sqbinom{a_2-\ell+j_2}{a_2-i_2}_\zeta \sqbinom{a_1+i_1}{a_1}_\zeta \\*
  &\hspace*{2\parindent} \{ -a_2+2c_2+i_2+j_2;i_2 \}_\zeta \{ b_1+2c_1+\ell+1;k_1-i_1+\ell \}_\zeta \\*
  &\hspace*{2\parindent} \zeta^{\frac{(k_1+i_2+j_2+3)(k_1+i_2+j_2)}{2}-(k_1+i_2+j_2-1)\ell-2(2(b_1+c_1)+a_2-c_2+2)(i_2+j_2)-2(b_1+c_1+1)k_1} \\*
  &\hspace*{2\parindent} E^{\iota(b_2-j_2)} T_{c_2+j_2} F^{(a_2-\ell+j_2)} \otimes E^{\iota(b_1-i_1+\ell)} T_{c_1+i_1} F^{(a_1+i_1)} \\
  &\hspace*{\parindent} \stackrel{\mathclap{\eqref{E:Murakami}}}{=} \zeta^{2(b_1+c_1+a_2-c_2+1)(b_1+c_1+a_2-c_2)} \\*
  &\hspace*{2\parindent} \sum_{j_2=0}^{r-1} \sum_{i_2=0}^{r-j_2-1} \sum_{k_1=0}^{r-i_2-j_2-1} \sum_{\ell=-k_1}^{i_2+j_2} \sum_{i_1=0}^{k_1+\ell} \sum_{k_2=0}^{i_2} \sum_{j_1=0}^{k_1-i_1+\ell} (-1)^{j_1+\ell+i_2+k_2} \\*
  &\hspace*{2\parindent} \sqbinom{k_1+i_2+j_2}{i_2,j_2}_\zeta \sqbinom{b_1-i_1+\ell}{k_1-i_1+\ell}_\zeta \sqbinom{a_2-\ell+j_2}{a_2-i_2}_\zeta \sqbinom{a_1+i_1}{a_1}_\zeta \\* 
  &\hspace*{2\parindent} \sqbinom{i_2}{k_2}_\zeta \zeta^{(a_2-2c_2-i_2-j_2)(i_2-2k_2)+\frac{i_2(i_2-1)}{2}-(i_2-1)k_2} \\* 
  &\hspace*{2\parindent} \sqbinom{k_1-i_1+\ell}{j_1}_\zeta \zeta^{-(b_1+2c_1+\ell+1)(\ell-i_1+k_1-2j_1)+\frac{(\ell-i_1+k_1)(\ell-i_1+k_1-1)}{2}-(\ell-i_1+k_1-1)j_1} \\*
  &\hspace*{2\parindent} \zeta^{\frac{(k_1+i_2+j_2+3)(k_1+i_2+j_2)}{2}-(k_1+i_2+j_2-1)\ell-2(2(b_1+c_1)+a_2-c_2+2)(i_2+j_2)-2(b_1+c_1+1)k_1} \\*
  &\hspace*{2\parindent} E^{\iota(b_2-j_2)} T_{c_2+j_2} F^{(a_2-\ell+j_2)} \otimes E^{\iota(b_1-i_1+\ell)} T_{c_1+i_1} F^{(a_1+i_1)} \\
  &\hspace*{\parindent} = \zeta^{2(b_1+c_1+a_2-c_2+1)(b_1+c_1+a_2-c_2)} \\*
  &\hspace*{2\parindent} \sum_{j_2=0}^{r-1} \sum_{i_2=0}^{r-j_2-1} \sum_{k_1=0}^{r-i_2-j_2-1} \sum_{\ell=-k_1}^{i_2+j_2} \sum_{i_1=0}^{k_1+\ell} \sum_{k_2=0}^{i_2} \sum_{j_1=0}^{k_1-i_1+\ell} (-1)^{j_1+\ell+i_2+k_2} \\*
  &\hspace*{2\parindent} \sqbinom{k_1+i_2+j_2}{k_1,j_2,k_2}_\zeta \sqbinom{b_1-i_1+\ell}{b_1-k_1,j_1}_\zeta \sqbinom{a_2-\ell+j_2}{a_2-i_2}_\zeta \sqbinom{a_1+i_1}{a_1}_\zeta \\* 
  &\hspace*{2\parindent} \zeta^{\frac{(i_1+3)i_1}{2}+k_1(k_1-2)+\frac{j_2(j_2-7)}{2}-\frac{(\ell+1)\ell}{2}-i_1j_1-(i_1+j_1-j_2+\ell)k_1+(j_1-j_2)\ell+(k_1+k_2-\ell)i_2+2j_2k_2} \\*
  &\hspace*{2\parindent} \zeta^{-(b_1+2c_1)(i_1-2j_1+\ell)+3j_1-(3b_1+4c_1)k_1-(4b_1+4c_1+a_2+3)i_2-(4b_1+4c_1+2a_2-2c_2-1)j_2-(2a_2-4c_2-1)k_2} \\*
  &\hspace*{2\parindent} E^{\iota(b_2-j_2)} T_{c_2+j_2} F^{(a_2-\ell+j_2)} \otimes E^{\iota(b_1-i_1+\ell)} T_{c_1+i_1} F^{(a_1+i_1)}.
 \end{align*}
 Now the claim follows from the computation
 \begin{align*}
  &\frac{N(a_1,b_1,c_1;a_2,b_2,c_2)}{N(a_1+i_1,b_1-i_1+\ell,c_1+i_1;a_2-\ell+j_2,b_2-j_2,c_2+j_2)} \\*
  &\hspace*{\parindent}= \zeta^{-\frac{(i_1+3)i_1}{2}-\frac{(j_2+3)j_2}{2}+3\frac{(\ell+1)\ell}{2}-j_2\ell+(a_1-2c_1)i_1+(a_2-2c_2)j_2+(2b_1+2c_1-a_2)\ell}. \qedhere
 \end{align*} 
\end{proof}

\subsection{Corollary: integrality of non-semisimple quantum representations}

Let us highlight a direct consequence of our results. First of all, let us endow the $\Z[\zeta]$-module $W_g := (\Z[\zeta])^{r^g}$ with the left $\Z[\pi_{n,g}]$-module structure determined by the $\Z[\zeta]$-linear representation $\varphi_{n,g}^W : \Z[\pi_{n,g}] \to \End_{\Z[\zeta]}(W_g)$ determined by Equation~\eqref{E:varphi}. We define for every integer $n \geqs 1$ the \textit{$n$th integral Heisenberg homology group} of $\varSigma_{g,1}$ as
\[
 \homol_{n,g}^W := H^\BM_n(X_{n,g},Y_{n,g};\phi_{n,g}^W).
\]
Next, recall that the \textit{Torelli subgroup} of $\Mod(\varSigma_{g,1})$, denoted by $\calI(\varSigma_{g,1})$, consists of those diffeomorphisms that act trivially on the (standard) homology $H_1(\varSigma_{g,1})$. It was proved in \cite[Proposition~28]{BPS21} that $\calI(\varSigma_{g,1})$ coincides with the subgroup of $\Mod(\varSigma_{g,1})$ that acts by inner homomorphisms on the quotient $\pi_{n,g}/K_{n,g}$, in the notation of Remark~\ref{R:quotient_group}. This has the following consequence.

\begin{lemma}\label{L:integral_Torelli}
 There exists a homomorphism $\bar{\psi}_g^W : \calI(\varSigma_{g,1}) \to \PGL_{\Z[\zeta]}(W_g)$ that fits into the commutative diagram
 \begin{center}
  \begin{tikzpicture}[descr/.style={fill=white}] \everymath{\displaystyle}
   \node (P0) at (0,0) {$\calI(\varSigma_{g,1})$};
   \node (P1) at (3,0) {$\PGL_{\Z[\zeta]}(W_g)$};
   \node (P2) at (0,-1) {$\Mod(\varSigma_{g,1})$};
   \node (P3) at (3,-1) {$\PGL_{\fld}(V_g)$};
   \draw
   (P0) edge[->] node[above] {\scriptsize $\bar{\psi}_g^W$}(P1)
   (P0) edge[right hook->] node[left] {} (P2)
   (P1) edge[right hook->] node[right] {} (P3)
   (P2) edge[->] node[below] {\scriptsize $\bar{\psi}_g^V$} (P3);
  \end{tikzpicture}
 \end{center}
\end{lemma}

\begin{proof}
 Thanks to \cite[Proposition~28]{BPS21}, for every $f \in \calI(\varSigma_{g,1})$, there exists some $\vartheta_n(f) \in \pi_{n,g}$ such that
 \begin{equation}\label{E:Torelli_inner}
  \varphi_{n,g}^V((f^{\times n})_*(\bloop)) = \varphi_{n,g}^V(\vartheta_n(f)) \circ \varphi_{n,g}^V(\bloop) \circ \varphi_{n,g}^V(\vartheta_n(f)^{-1}) 
 \end{equation}
 for every $\bloop \in \pi_{n,g}$. Then, for every $f \in F_g$ representing an element of $\calI(\varSigma_{g,1}) < \Mod(\varSigma_{g,1})$, Equation~\eqref{E:fundamental_identity} implies that
 \[
  \varphi_{n,g}^V(\vartheta_n(f)^{-1}) \circ \psi_g^V(f) \circ \varphi_{n,g}^V(\bloop) 
  = \varphi_{n,g}^V(\bloop) \circ \varphi_{n,g}^V(\vartheta_n(f)^{-1}) \circ \psi_g^V(f)
 \]
 for every $\bloop \in \pi_{n,g}$. Thanks to Equation~\eqref{E:AB_centralizer}, this means that
 \[
  \varphi_{n,g}^V(\vartheta_n(f)^{-1}) \circ \psi_g^V(f) \in 
  C_{\GL_{r^n}(\Z[\zeta])} \left( A_1,B_1,\ldots,A_n,B_n \right) 
  = (\Z[\zeta])^\times.
 \]
 Therefore, we can set
 \[
  \bar{\psi}_g^W(f) := [\varphi_{n,g}^V(\vartheta_n(f))] \in \PGL_{r^n}(\Z[\zeta])
 \]
 for every $f \in \calI(\varSigma_{g,1})$.
\end{proof}

Following Theorem~\ref{T:mcg_homological_projective_rep}, we obtain a projective representation
\[
 \bar{\rho}_{n,g}^W : \calI(\varSigma_{g,1}) \to \PGL_{\Z[\zeta]}(\homol_{n,g}^W)
\]
for all $n \in \N$. Then, Definition~\ref{D:fin_dim_homol} can be directly generalized to yield
\[
 \homol^{W(r)}_g \subset \homol_g^W := \bigoplus_{n \geqs 0} \homol_{n,g}^W.
\]
The proof of Theorem~\ref{T:uzeta_homological_representation} can be replicated word-by-word to show that $\homol^{W(r)}_{g}$ is closed under the projective action of $\calI(\varSigma_{g,1})$. Similarly, the $\Z[\zeta]$-linear map $\Phi_g^W : \calH_g^{W(r)} \to \ad^{\otimes g}$ determined by Equation~\eqref{E:isomorphism} defines an integral $\Z[\zeta]$-lattice $\Phi_g^W(\calH_g^{W(r)}) \subset \ad^{\otimes g}$.

\begin{corollary}\label{C:non-semisimple_TQFTs_are_integral}
 The projective action of the Torelli group $\calI(\varSigma_{g,1})$ on $\Phi_g^W(\calH_g^{W(r)}) \subset \ad^{\otimes g}$ is integral, in the sense that 
 \[
  \bar{\rho}_g^{\fraku_\zeta}(\calI(\varSigma_{g,1})) \subset \PGL_{r^g}(\Z[\zeta]) \cong \PGL_{\Z[\zeta]}(\Phi_g^W(\calH_g^{W(r)})) \subset \PGL_{\fld}(\ad^{\otimes g}).
 \]
\end{corollary}

\begin{remark}
 In order to obtain integrality of the projective action of $\Mod(\varSigma_{g,1})$, we need to find a $\Z[\zeta]$-lattice $W'_g \subset V_g$ that is preserved by the image of $\psi_g$. Since $\psi_g$ factors through a finite group, which is only slightly bigger than a finite symplectic group, this should be a fairly reasonable task. For instance, when $g=1$ and $r=3$, we can consider the $\Z[\zeta]$-lattice $W'_1$ with basis 
 \begin{align*}
  w'_1 &:= (1-\zeta^2) v_1, & 
  w'_2 &:= (\zeta-\zeta^2) v_1 + (1-\zeta^2) v_2, & 
  w'_3 &:= \zeta^2 v_1 + \zeta v_2 + v_3.
 \end{align*}
 In this basis, we obtain $\varphi_{n,1}^{W'} : \Z[\pi_{n,1}] \to \End_{\Z[\zeta]}(W'_1)$ and $\bar{\psi}_1^{W'} : \Mod(\varSigma_{1,1}) \to \PGL_{\Z[\zeta]}(W'_1)$ satisfying
 \begin{align*}
  \varphi_{n,1}^{W'}(\braid{\alpha}{}) &=
  \begin{pmatrix}
   1 & 1-\zeta^2 & 0 \\
   0 & \zeta & -1 \\
   0 & 0 & \zeta^2
  \end{pmatrix}, &
  \varphi_{n,1}^{W'}(\braid{\beta}{}) &=
  \begin{pmatrix}
   \zeta^2 & 0 & 0 \\
   1 & 1 & 0 \\
   0 & 1-\zeta^2 & \zeta
  \end{pmatrix}, \\
  \bar{\psi}_1^{W'}(\tau_\alpha) &=
  \begin{pmatrix}
   1 & 1-\zeta^2 & \zeta \\
   0 & \zeta & \zeta^2 \\
   0 & 0 & 1
  \end{pmatrix}, &
  \bar{\psi}_1^{W'}(\tau_\beta) &=
  \begin{pmatrix}
   \zeta & 0 & 0 \\
   -\zeta & 1 & 0 \\
   -\zeta^2 & 0 & 1
  \end{pmatrix}.
 \end{align*}
 We leave the discussion of the general case, for arbitrary values of $g$ and $r$, to a future work.
\end{remark}

In some sense, the homological model for quantum representations highlights quite naturally the exact place where their integrality features emerge. These properties have been the subject of some deep investigation in the semisimple case, see \cite{G01,GM04,BCL10}.

\subsection{Generalizations: an overview}

\subsubsection{Recovering other non-semisimple quantum representations}\label{S:recovering_BCGP}

Here, we modify slightly the ring homomorphism $\varphi_{n,g}^V : \Z[\pi_{n,g}] \to \End_{\fld}(V_g)$ defined in Equation~\eqref{E:varphi}, that was used to linearize Heisenberg group coefficients. Namely, we consider a list $\formv = (s_1,t_1,\ldots,s_g,t_g)$ of $2g$ formal variables, and we denote by 
\begin{align*}
 \Z[\zeta,\formv^{\pm 1}] &:= \Z[\zeta,s_1^{\pm 1},t_1^{\pm 1},\ldots,s_g^{\pm 1},t_g^{\pm 1}], 
\end{align*}
the ring of Laurent polynomials in these variables. If, for all $1 \leqs j \leqs g$, we set
\begin{align*}
 A_j^\formv &:= s_j A_j, &
 B_j^\formv &:= t_j B_j,
\end{align*}
where the matrices $A_j$ and $B_j$ were recursively defined in Equations~\eqref{E:A_1_1-B_1_1}--\eqref{E:A_n_n-B_n_n}, then we can again endow the free $\Z[\zeta,\formv^{\pm 1}]$-module $W_g^\formv := (\Z[\zeta,\formv^{\pm 1}])^{r^g}$ with a left $\Z[\pi_{n,g}]$-module structure determined by the $\Z[\zeta,\formv^{\pm 1}]$-linear representation
\begin{align}
 \varphi_{n,g}^\formv : \Z[\pi_{n,g}] & \to \End_{\Z[\zeta,\formv^{\pm 1}]}(W_g^\formv) \label{E:varphi_omega} \\*
 \braid{\sigma}{i} & \mapsto -\zeta^{-2} I_{r^g} \nonumber \\*
 \braid{\alpha}{j} & \mapsto A_j^\formv \nonumber \\*
 \braid{\beta}{j} & \mapsto B_j^\formv \nonumber 
\end{align}
This is a well-defined homomorphism because the matrices 
\[
 A_1^\formv, B_1^\formv, \ldots, A_g^\formv, B_g^\formv \in M_{r^g \times r^g}(\Z[\zeta,\formv^{\pm 1}])
\]
still satisfy Equation~\eqref{E:AB_zeta-commute}. They also satisfy (the analogue of) Equation~\eqref{E:AB_centralizer}, but notice that they no longer satisfy Equation~\eqref{E:AB_order_r}. We define for every integer $n \geqs 1$ the \textit{$n$th modular Heisenberg homology group} of $(\varSigma_{g,1},\formv)$ as
\[
 \homol_{n,g}^\formv := H^\BM_n(X_{n,g},Y_{n,g};\phi_{n,g}^\formv).
\]
Notice that $\homol_{n,g}^\formv$ is again a free $\Z[\zeta,\formv^{\pm 1}]$-module with exactly the same basis of Remark~\ref{R:infinite_linear_basis}. Just like in Lemma~\ref{L:integral_Torelli}, we can use \cite[Proposition~28]{BPS21} to define a homomorphism 
\[
 \bar{\psi}_g^\formv : \calI(\varSigma_{g,1}) \to \PGL_{\Z[\zeta,\formv^{\pm 1}]}(W_g^\Z[\zeta,\formv^{\pm 1}]).
\]
Then, following Theorem~\ref{T:mcg_homological_projective_rep}, we obtain a projective representation
\begin{equation}\label{E:The_future_of_Torelli}
 \bar{\rho}_n^\formv : \calI(\varSigma_{g,1}) \to \PGL_{\Z[\zeta,\formv^{\pm 1}]}(\homol_{n,g}^\formv)
\end{equation}
for all $n \in \N$. Once again, Definition~\ref{D:fin_dim_homol} can be directly generalized to yield
\[
 \homol^{\formv(r)}_g \subset \homol_g^\formv := \bigoplus_{n \geqs 0} \homol_{n,g}^\formv.
\]
The proof of Theorem~\ref{T:uzeta_homological_representation} can be replicated word-by-word to show that $\homol^{\formv(r)}_{g}$ is closed under the projective action of $\calI(\varSigma_{g,1})$.

For every $\alpha \in \bbC$, let us set
\[
 \zeta^\alpha := e^{\frac{2 \alpha \pi i}{r}}.
\]
Then, every cohomology class $\cohom \in H^1(\varSigma_{g,1};\bbC) \cong \Hom_\Z(H_1(\varSigma_{g,1}),\bbC)$ determines a ring homomorphism
\begin{align*}
 \Z[\zeta,\formv^{\pm 1}] &\to \bbC \\*
 s_j &\mapsto \zeta^{\cohom(\alpha_j)} \\*
 t_j &\mapsto \zeta^{\cohom(\beta_j)}
\end{align*}
where $\alpha_j$ and $\beta_j$ denote the homology classes of the usual simple closed curves, meaning those considered in Section~\ref{S:homological_representations_of_mcg}, and $\zeta^\cohom = (\zeta^{\cohom(\alpha_1)},\zeta^{\cohom(\beta_1)},\ldots,\zeta^{\cohom(\alpha_g)},\zeta^{\cohom(\beta_g)}) \in \bbC^{2g}$. If we set
\[
 \homol^{\cohom(r)}_g := \homol^{\formv(r)}_g \otimes_{\Z[\zeta,\formv^{\pm 1}]} \bbC,
\]
then we obtain a projective representation 
\[
 \bar{\rho}_g^{\cohom(r)} : \calI(\varSigma_{g,1}) \to \PGL_{\bbC} \left( \homol^{\cohom(r)}_g \right).
\]
For the zero cohomology class $0 \in H^1(\varSigma_{g,1};\bbC)$, the projective representation $\bar{\rho}_g^{0(r)}$ coincides with the restriction of $\bar{\rho}_g^{V(r)}$ in Theorem~\ref{T:uzeta_homological_representation}, and can therefore be extended to the whole mapping class group $\Mod(\varSigma_{g,1})$. The interpretation is that diffeomorphisms also act on cohomology classes, and only those that fix $\cohom$ are assigned matrices. However, the cohomology class $0$ is fixed by all diffeomorphisms.

In general, $\bar{\rho}_g^{\cohom(r)}$ is a representation of the Torelli group $\calI(\varSigma_{g,1})$ that depends on the choice of a cohomology class $\cohom \in H^1(\varSigma_{g,1};\bbC)$, and we ask the following.

\begin{question}\label{Q:recovering_BCGP}
 Does the representation $\bar{\rho}_g^{\cohom(r)}$ coincide with a quantum representation arising from a non-semisimple TQFT?
\end{question}

As a candidate, we suggest a graded TQFT similar to the one constructed by Blanchet, Costantino, Geer, and Patureau in \cite{BCGP14}. Indeed, the \textit{unrolled quantum group} $U^H_\zeta = U^H_\zeta \fsl_2$ at the odd root of unity $\zeta$ induces a graded TQFT, as follows from \cite[Section~6.2]{D17} and \cite[Theorem~1.3]{DGP18}. We would need to extend this graded TQFT to the category $\adCob_{U^H_\zeta}$ of \textit{admissible} cobordisms decorated not only with cohomology classes (with coefficients in $\bbC/2\Z$), but also with \textit{bichrome graphs} (with labels in $\mods{U^H_\zeta}$) like those considered in Appendix~\ref{A:TQFT}. This corresponds to the approach of \cite{GHP20} to the invariants of \cite{CGP12}. Although this graded TQFT, that we will denote by $\bbV_{U^H_\zeta}$ here, has not appeared in the literature yet, all the ingredients for its construction are already available. Then, let us consider the closed surface $\varSigma_{g,1} \cup_{S^1} D^2$ of genus $g$ obtained from $\varSigma_{g,1}$ by gluing a disc $D^2$ along its boundary, and let $P$ denote the center of $D^2$. We also consider a finite-dimensional weight $U^H_\zeta$-module $\fraku_\zeta^H$ lifting the regular representation of $\fraku_\zeta$ (determined by left multiplication of $\fraku_\zeta$ onto itself). Then, we denote by $P_{(-,\fraku_\zeta^H)}$ the decoration of $D^2$ obtained by giving $P$ negative orientation and label $\fraku_\zeta^H$.

\begin{conjecture}\label{C:BCGP_conjecture}
 For every cohomology class $\cohom \in H^1(\varSigma_{g,1};\bbC)$, if $\bar{\cohom} \in H^1(\varSigma_{g,1} \cup_{S^1} (D^2 \smallsetminus P);\bbC/2\Z) \cong H^1(\varSigma_{g,1};\bbC/2\Z)$ denotes its image, then the projective representation
 \[
  \bar{\rho}_g^{\cohom(r)} : \calI(\varSigma_{g,1}) \to \PGL_{\Q[\zeta,\zeta^\cohom]} \left( \homol^{\cohom(r)}_g \right)
 \]
 is isomorphic to the projective representation
 \[
  \bar{\rho}_g^{U^H_\zeta} : \calI(\varSigma_{g,1}) \to \PGL_\bbC \left( \bbV_{U^H_\zeta}(\varSigma_{g,1} \cup_{S^1} D^2,P_{(-,\fraku_\zeta^H)},\bar{\cohom}) \right)
 \]
 induced from the non-semisimple graded TQFT $\bbV_{U^H_\zeta}$.
\end{conjecture}

The homological action of the Torelli group still intertwines an action of $\bar{U}_\zeta$, like in Theorem~\ref{T:Uq_homological_representation}, while, on the quantum side, this property would only follow naturally from a graded version of the construction of Kerler and Lyubashenko. Indeed, the universal construction of \cite{BHMV95} required by the use of modified traces typically hides this kind of behavior.

Furthermore, for every cohomology class $\cohom \in H^1(\varSigma_{g,1};\bbC)$, a basis of $\homol^{\cohom}_g$ is given by Remark~\ref{R:infinite_linear_basis}. We can compute actions of elements of the Torelli group in this basis by applying the same techniques used in Section~\ref{S:homological_mcg_computation}. The result is a representation of $\calI(\varSigma_{g,1})$ of dimension $r^{3g}$, with polynomial dependence on the list of $2g$ parameters $(\cohom(\alpha_1),\cohom(\beta_1),\ldots,\cohom(\alpha_g),\cohom(\beta_g))$. One could obtain explicit matrices for images of generators under \eqref{E:The_future_of_Torelli}, or try to apply Bigelow's strategy to this representation, taking advantage of polynomial coefficients. This very promising for studying faithfulness of $\bar{\rho}_n^{\formv}$ (already for $n=2$), that would imply linearity of Torelli groups. Conjecture~\ref{C:BCGP_conjecture} would then provide explicit bases for state spaces of non-semisimple TQFTs associated with the unrolled quantum group $U^H_\zeta \fsl_2$ for all possible choices of cohomology classes (at least for the decorated surfaces appearing in the statement), thus yielding matrices depending on $2g$ parameters. Taking limits in these parameters would also be possible. This would raise the following natural question: can we to construct homologically a TQFT that associates to every surface $\varSigma_{g,1}$ the total modular Heisenberg homology $\homol_g^\formv$ (possibly only for the category of Lagrangian cobordisms of \cite[Section~1.1]{CHM07}, which is equivalent to the category of bottom tangles in homology handlebodies of \cite[Section~14.4]{H05})? As in Corollary~\ref{C:non-semisimple_TQFTs_are_integral}, such a construction would naturally be integral, and a positive answer to Conjecture~\ref{C:BCGP_conjecture} would transfer these integrality properties to the quantum representations arising from the construction of \cite{BCGP14}.

Proving Conjecture~\ref{C:BCGP_conjecture} following the same strategy used for Theorem~\ref{T:The_Theorem} would require constructing the appropriate TQFT associated with $U^H_\zeta$, that does not exist in the literature yet. Furthermore, generators of the Torelli group are more complicated than single Dehn twists. For these two reasons, we decided to postpone more detailed computations.

\subsubsection{Linear mapping class group representations}\label{S:unprojectivize}

In Theorem~\ref{T:mcg_homological_projective_rep} we defined homological representations of mapping class groups, which we identified with quantum representations arising from non-semisimple TQFTs in Theorem~\ref{T:The_Theorem}. It is a well-known fact that TQFTs associated with small quantum groups at roots of unity only yield projective representations of mapping class groups. Here, we explain a standard natural way for obtaining honest linear representations from these projective ones. Then, in the particular case of the homological representations of \ref{T:mcg_homological_projective_rep}, we also give a homological method for doing the same. Both methods increase the dimension of the representation, but we only pay attention to the fact that, if the projective representation is faithful, then so is the linear one.

We begin with an algebraic method that applies to any projective representation
\[
 \rho: G \to \PGL_n(\bbC),
\]
where $G$ is a group and $n \geqs 1$ is an integer. Indeed, for all $A \in \PGL_n(\bbC)$ and $B \in M_{n \times n}(\bbC)$, the adjoint action
\[
 \ad_A(B) := ABA^{-1}
\]
is well-defined, since rescaling $A$ does not affect the operation of conjugation. This induces a linear representation
\begin{align*}
 \ad : \PGL_n(\bbC) &\to \GL_\bbC(M_{n \times n}(\bbC)) \\*
 A &\mapsto \ad_A
\end{align*}
and since the center of $\PGL_n(\bbC)$ is trivial, $\ad$ is injective. Therefore, if $\rho$ is injective, so is $\ad \circ \rho$.

\begin{corollary}
 If a projective representation of a group is faithful, the group is linear.
\end{corollary}

In the special case of the homological representations of Theorem~\ref{T:mcg_homological_projective_rep}, there is a homological way of obtaining linear representations. For instance, in \cite{DM22} it is shown that a homological representation of a subgroup of $\Mod(\varSigma_{g,1})$ on the homology of $X_{n,g}$ with twisted coefficients in $\Z[G_{n,g}]$, for a given quotient $G_{n,g}$ of $\pi_{n,g}$, naturally extends to an action of the whole $\Mod(\varSigma_{g,1})$ on the homology of $X_{n,g}$ with twisted coefficients in $\Z[G_{n,g} \rtimes \Aut(G_{n,g})]$. Since all generators of $\pi_{n,g}$ are sent to finite order matrices by $\phi_{n,g}^V$, this means that $\phi_{n,g}^V$ factors through a finite quotient $G_{n,g}$ of $\pi_{n,g}$. Namely, there exists a commutative diagram
\begin{center}
 \begin{tikzpicture}[descr/.style={fill=white}]
  \node (P1) at (-1.5,1) {$\Z[\pi_{n,g}]$};
  \node (P2) at (1.5,1) {$\End_\bbC(V_g)$};
  \node (P3) at (0,0) {$\Z[G_{n,g}]$};
  \draw
  (P1) edge[->] node[above] {\scriptsize $\varphi_{n,g}^V$} (P2)
  (P1) edge[->>] (P3)
  (P3) edge[right hook->] (P2);
 \end{tikzpicture}
\end{center}
where $G_{n,g}$ is a finite group. It turns out that, at least for $n \geqs 2$, the group $G_{n,g}$ is isomorphic to the finite Heisenberg group $\bbH_g(\Z/r\Z)$, which is a linear group. In particular, $G_{n,g} \rtimes \Aut(G_{n,g})$ is linear too, by considering the induced representation. This yields a linear representation of the whole mapping class group $\Mod(\varSigma_{g,1})$. It is not the case for $\phi_{n,g}^\cohom$, when $\cohom$ is not the zero cohomology class. Computing by how much the dimension increases would tell us whether this homological method is less expensive than the algebraic one, or not.

\subsubsection{Other roots of unity}

In this paper, we have performed constructions and computations only for $\zeta = e^{\frac{2 \pi i}{r}}$, with $r \geqs 3$ odd. Both the homological and the quantum construction can be carried out for even roots of unity too, although computations and explicit formulas would change slightly. These variants are not discussed here, but Theorem~\ref{T:The_Theorem} clearly admits an analogue version for even roots of unity.

\subsubsection{Irreducibility of absolute representations}

The $U_q$-module structure on $\homol_g^{\dHeis}$ given by Theorem~\ref{T:Uq_homological_representation} is new, and has coefficients in $\Z[\dHeis_g]$. What is particularly interesting is that $\homol_g^{\dHeis}$ also carries an action of (a subgroup of) the mapping class group $\Mod(\varSigma_{g,1})$ which intertwines this action of $U_q$. Furthermore, this can be specialized and subsequently extended to a projective action of $\Mod(\varSigma_{g,1})$ that recovers the quantum representation arising from the non-semisimple TQFT associated with the small quantum group $\fraku_\zeta \fsl_2$, thanks to Theorem~\ref{T:The_Theorem}. We also conjecture that its deformations are related to other non-semisimple TQFTs associated with the unrolled quantum group $U^H_\zeta \fsl_2$, see Conjecture~\ref{C:BCGP_conjecture}. We find similarities between the structure of the $U_q$-module $\homol_g^{\dHeis}$ and that of a tensor product of universal Verma modules, see for instance \cite[Definition~5.8]{M20} (although coefficients here are extended to $\Z[\dHeis_g]$). Jackson and Kerler have studied the action of braid groups over tensor products of universal Verma modules, and they have decomposed these braid group representations into irreducible summands \cite{JK09}. Each irreducible summand is generated, over the braid group, by a highest weight vector for the action of $U_q$, namely by an eigenvector for the action of $K$ which lies in the kernel of the action of $E$. Since the homological action of $E$ is implemented by the connection homomorphism of a long exact sequence of a triple (see Definition~\ref{D:homological_E}), highest weight vectors come from absolute homology classes (or from homology classes with smaller relative part, see the long exact sequence just above Defnition~\ref{D:homological_E}). We hope that the techniques of Jackson and Kerler for decomposing tensor products of Verma modules into irreducible summands for the action of braid groups might shed light onto the action of $\Mod(\varSigma_{g,1})$ on $\homol_g^{\dHeis}$, and that irreducible subrepresentations might be described in terms of absolute homology groups (or homology groups with smaller relative part), as it was done in \cite[Corollary~7.1]{M20}.

\subsubsection{Other Lie algebras}

Theorem~\ref{T:The_Theorem} is an extension of results from \cite{M20}, where the braiding of $U_q \fsl_2$ on tensor products of universal Verma modules was recovered from the twisted homology of configuration spaces of punctured discs. Here, on the other hand, we recover homologically quantum representations of mapping class groups of positive genus surfaces extracted from a non-semisimple TQFT built from $\fraku_\zeta \fsl_2$. In a joint work in progress of the second author with S.~Bigelow, the results of \cite{M20} are generalized to the quantum group $U_q \frakg$ associated with a simple Lie algebra $\frakg$. This requires a decorated version of configuration spaces involving points labeled by simple roots of $\frakg$, and the result recovers quantum representations of braid groups associated with $U_q \frakg$. In the case of surfaces, we would need to adapt the homomorphism $\varphi_{n,g}^{\dHeis}$ to these new configuration spaces labeled by roots, in order to obtain an analogue of Theorem~\ref{T:The_Theorem} for every simple Lie algebra $\frakg$.

\appendix

\section{Homological preliminaries}

In this appendix, we recall the tools required for the definition of homological representations of mapping class groups. Throughout, $X$ will always denote a connected topological space, and $R$ will always denote a ring.

\subsection{Twisted homology with local coefficients}\label{A:twisted}

Let $X$ be semi-locally simply connected. Let us denote by $\tilde{p} : \tilde{X} \to X$ its universal cover, and by $\pi := \pi_1(X,\xi)$ its fundamental group with respect to a base point $\xi \in X$. By definition, every point $\tilde{x} \in \tilde{X}$ is the homotopy class of a path $\tilde{x} : [0,1] \to X$ with $\tilde{x}(0) = \xi$, and the left action of $\gamma \in \pi$ on $\tilde{x} \in \tilde{X}$ by deck transformation is the homotopy class $\gamma \cdot \tilde{x} \in \tilde{X}$ of the concatenation of paths $\gamma \ast \tilde{x} : [0,1] \to X$. This is a left action because, in our conventions, $\alpha \ast \beta$ stands for ``first $\alpha$, then $\beta$''. The complex $C_\ast(\tilde{X})$ (where the absence of coefficients stands for $\Z$-coefficients) is thus naturally endowed with a left $\Z[\pi]$-module structure, where $\Z[\pi]$ denotes the group ring of $\pi$. Let $M$ be a left $\Z[\pi]$-module, with action of $\Z[\pi]$ determined by a homomorphism
\[
 \varphi^M : \Z[\pi] \to \End(M).
\]
The \textit{homology of $X$ with local coefficients in $M$ twisted by $\varphi^M$}, or simply the \textit{twisted homology of $X$}, is defined as the homology $H_\ast(X;\varphi^M)$ of the complex
\[
 C_\ast(X;\varphi^M) := C_\ast(\tilde{X}) \otimes_{\Z[\pi]} M,
\]
where the right $\Z[\pi]$-module structure on $C_\ast(\tilde{X})$ is determined by
\[
 \tilde{\chain} \cdot \bloop := \bloop^{-1} \cdot \tilde{\chain}
\]
for all $\tilde{\chain} \in C_*(\tilde{X})$ and $\bloop \in \pi$. Similarly, if $Y \subset X$ is a subspace, then the \textit{twisted homology of $X$ relative to  $Y$} is defined as the homology $H_\ast(X,Y;\varphi^M)$ of the complex
\[
 C_*(X,Y;\varphi^M) := C_*(\tilde{X},\tilde{p}^{-1}(Y)) \otimes_{\Z[\pi]} M.
\]

\begin{remark}\label{R:why_we_conjugate}
 The reader should be careful when deriving the computation rules of Proposition~\ref{P:computation_rules}, since moving a $\Z[\pi]$-coefficient from the left tensor factor to the right one involves taking inverses, meaning that
 \[
  \left( \bloop \cdot \tilde{\chain} \right) \otimes m = 
  \tilde{\chain} \otimes \varphi^M(\bloop^{-1})(m)
 \]
 for all $\bloop \in \pi$, $\tilde{\chain} \in C_*(\tilde{X},\tilde{p}^{-1}(Y))$, and $m \in M$.
\end{remark}

We refer the reader to \cite[Section~5.1]{DK01} or \cite[Section~2.1]{Ma} for an introduction to twisted homology. Notice that, if $M$ is an $R$-module and $\varphi^M$ takes values in the group of $R$-module endomorphisms $\End_R(M)$ (which means $M$ is a $(\Z[\pi],R)$-bimodule), then $H_\ast(X,Y;\varphi^M)$ is naturally an $R$-module.

If $G$ is a group, and $\varphi^G : \Z[\pi] \to \Z[G]$ is a ring homomorphism, then we abuse notation, and denote by
\[
 \varphi^G : \Z[\pi] \to \End_{\Z[G]}(\Z[G])
\]
also the associated representation induced by left multiplication, meaning that
\[
 \varphi^G(\bloop)(g) = \varphi^G(\bloop)g
\]
for all $\bloop \in \pi$ and $g \in G$ (notice that $\varphi^G(\bloop)$ is $\Z[G]$-linear with respect to right multiplication). In particular, we use the notation $H_\ast(X,Y;\varphi^G)$ every time we have a ring homomorphism $\varphi^G : \Z[\pi] \to \Z[G]$. Then, if $K \triangleleft \pi$ is a normal subgroup, and if $\varphi^{\pi/K} : \Z[\pi] \to \Z[\pi/K]$ is induced by the projection to the quotient, we have
\[
 H_\ast(X,Y;\varphi^{\pi/K}) \cong H_\ast(\hat{X},\hat{p}^{-1}(Y)),
\]
where $\hat{p} : \hat{X} \to X$ denotes the regular cover associated with $K$, see \cite[Section~2.1, Example~1.(3)]{Ma}. Therefore, $H_\ast(X,Y;\varphi^{\pi/K})$ is endowed with a natural $\Z[\pi/K]$-module structure induced by the deck transformation group of $\hat{p} : \hat{X} \to X$, which coincides with $\pi/K$. We point this out because the twisted homology groups studied in this paper are very closely related (with only minor differences) to those considered in \cite{DM22}, where the theory of regular covers is extensively exploited in order to derive formulas that could be used to describe the mapping class group representations constructed here.

Let us finish this section with a remark that is used several times in the paper, and which can be found in \cite[Section~2.1, Lemme~6]{Ma}.

\begin{remark}\label{R:twisted_universal_coeff_free}
 The universal coefficient theorem for twisted homology usually involves a spectral sequence, rather than the short exact sequence appearing in the untwisted case. However, if $C_\ast$ is a complex of $R$-modules whose homology is free over $R$, and $M$ is an $R$-module, then we have a natural isomorphism
 \[
  H_\ast(C_\ast \otimes_R M) \cong H_\ast(C_\ast) \otimes_R M.
 \]
\end{remark}

\subsection{Borel--Moore homology}\label{A:Borel--Moore}

Let $X$ be locally compact. The \textit{Borel--Moore homology of $X$} is defined as the homology $H^\BM_\ast(X)$ of the complex
\[
 C^\BM_\ast(X) := \varprojlim_{K \in \calK(X)} C_\ast(X,X \smallsetminus K),
\]
where $\calK(X)$ denotes the set of compact subspaces of $X$, equipped with the partial order induced by inclusion. If $Y \subset X$ is a subspace, then the \textit{Borel--Moore homology of $X$ relative to $Y$} is defined as the homology $H^\BM_\ast(X,Y)$ of the complex
\[
 C^\BM_\ast(X,Y) := \varprojlim_{K \in \calK(X)} C_\ast(X,Y \cup (X \smallsetminus K)).
\]
For a homomorphism $\varphi^M : \Z[\pi] \to \End(M)$ inducing a $\Z[\pi]$-module structure on $M$, the \textit{twisted Borel--Moore homology of $X$ relative to $Y$} is defined as the homology $H^\BM_\ast(X,Y;\varphi^M)$ of the complex
\[
 C^\BM_\ast(X,Y;\varphi^M) := \varprojlim_{K \in \calK(X)} C_\ast(\tilde{X},\tilde{p}^{-1}(Y \cup (X \smallsetminus K))) \otimes_{\Z[\pi]} M.
\]

\begin{remark}\label{R:twisted_universal_coeff_free_BM}
 Notice that Remark~\ref{R:twisted_universal_coeff_free} implies that, if $\varphi^\pi : \Z[\pi] \to \Z[\pi]$ denotes the identity, and if the homology of every complex appearing in the inverse limit defining $C^\BM_\ast(X,Y;\varphi^\pi)$ is free over $\Z[\pi]$, then
 \[
  H^\BM_\ast(X,Y;\varphi^M) \cong H^\BM_\ast(X,Y;\varphi^\pi) \otimes_{\Z[\pi]} M.
 \]
\end{remark}

\subsection{Group actions on twisted homology}\label{A:twisted_homological_action_of_mcg}

Let $X$ be an oriented manifold. If $\Diff(X)$ denotes the group of orientation-preserving self-diffeomorphisms of $X$ fixing the boundary $\partial X$ point-wise, and if $\Diff_0(X)$ denotes the normal subgroup of those diffeomorphisms which are isotopic to the identity (trough an isotopy fixing $\partial X$ point-wise), then the \textit{mapping class group of $X$} is the group
\[
 \Mod(X) := \Diff(X) / \Diff_0(X).
\]

The mapping class group $\Mod(X)$ usually acts naturally on the homology of $X$. However, within the framework of twisted homology, this is no longer true, or at least the action is no longer completely natural. A topological interpretation of these issues is exposed in detail in \cite{DM22}. We explain here below how to define actions of diffeomorphisms on twisted homology that are adapted to our present needs.

We first recall that, if $\tilde{p} : \tilde{X} \to X$ denotes the regular cover, then every diffeomorphism $f$ of $X$ fixing the base point $\xi \in X$ can be uniquely lifted (by lifting $f \circ \tilde{p}$) to a diffeomorphism $\tilde{f}$ of the universal cover $\tilde{X}$ fixing the base point $\tilde{\xi} \in \tilde{X}$ given by the homotopy class of the constant path based at $\xi$. Let us assume from now on that $\partial X \neq \varnothing$, so that we can consider a base point $\xi \in \partial X$ and a subspace $Y \subset \partial X$. In particular, every element of $\Diff(X)$ fixes both $\xi$ and $Y$.

If $F$ is a group, every homomorphism $\chi : F \to \Mod(X)$ lifts to a homomorphism $\tilde{\chi} : F \to \Mod(\tilde{X})$, and we denote by
\begin{align*}
 \chi_* : F &\to \Aut(\Z[\pi]), & 
 \tilde{\chi}_* : F &\to \Aut(H^\BM_\ast(X,Y;\varphi^\pi))
\end{align*}
the associated homorphisms induced by the actions of $\Mod(X)$ and $\Mod(\tilde{X})$ on $\Z[\pi]$ and $H^\BM_\ast(X,Y;\varphi^\pi)$, respectively. We also define
\[
 \calM^M := \left\{ f \in F \mid \varphi^M \circ \chi_*(f) = \varphi^M \right\}.
\]

\begin{proposition}\label{P:specializing}
 Let $\chi : F \to \Mod(X)$ be a homomorphism and $M$ be an $R$-module equipped with a representation $\varphi^M : \Z[\pi] \to \End_R(M)$. If $H^\BM_\ast(X,Y;\varphi^\pi)$ is a free $\Z[\pi]$-module, then:
 \begin{enumerate}
  \item There exists an $R$-linear action of $\calM^M$ on $H^\BM_\ast(X,Y;\varphi^M)$ given by the homomorphism
 \[
  \rho^M : \calM^M \to \GL_R(H^\BM_\ast(X,Y;\varphi^M))
 \]
 defined by 
 \[
  \rho^M(f)(\tilde{\chain} \otimes m) = \tilde{\chi}_*(f)(\tilde{\chain}) \otimes m
 \]
 for all $f \in \calM^M$, $\tilde{\chain} \in H^\BM_\ast(X,Y;\varphi^\pi)$, and $m \in M$;
 \item If $M$ is also equipped with a representation $\psi : F \to \GL_R(M)$ satisfying
 \begin{equation}
  \psi(f) \circ \varphi^M(\bloop) = \varphi^M(\chi_*(f)(\bloop)) \circ \psi(f) \label{E:f_gamma}
 \end{equation}
 for all $f \in F$ and $\bloop \in \pi$, then the $R$-linear action of $\calM^M$ extends to an $R$-linear action of $F$ given by the homomorphism
 \[
  \rho^M : F \to \GL_R(H^\BM_\ast(X,Y;\varphi^M))
 \]
 defined by 
 \[
  \rho^M(f)(\tilde{\chain} \otimes m) = \tilde{\chi}_*(f)(\tilde{\chain}) \otimes \psi(f)(m)
 \]
 for all $f \in F$, $\tilde{\chain} \in H^\BM_\ast(X,Y;\varphi^\pi)$, and $m \in M$.
\end{enumerate}
\end{proposition}

Both representations are denoted $\rho^M$, since one extends the other.

\begin{proof}
 First of all, notice that, if $f$ is a diffeomorphism of $X$, and if $\tilde{f}$ lifts $f \circ \tilde{p}$ to a diffeomorphism of $\tilde{X}$, then
 \[
  \tilde{f}(\tilde{x} \cdot \bloop) = \tilde{f}(\tilde{x}) \cdot f_*(\bloop)
 \]
for all $\tilde{x} \in \tilde{X}$ and $\bloop \in \pi$. Furthermore, thanks to Remak~\ref{R:twisted_universal_coeff_free}, we have
 \[
  H^\BM_\ast(X,Y;\varphi^M) \cong H^\BM_\ast(X,Y;\varphi^\pi) \otimes_{\Z[\pi]} M.
 \]

 Then, let us start from claim~$(i)$. For all $f \in \calM^M$, $\tilde{\chain} \in H^\BM_\ast(X,Y;\varphi^\pi)$, and $m \in M$, the action 
 \[
  \rho^M(f)(\tilde{\chain} \otimes m) = \tilde{\chi}_*(f)(\tilde{\chain}) \otimes m
 \]
 is clearly well-defined on the tensor product $H^\BM_\ast(X,Y;\varphi^\pi) \otimes_\Z M$, so it remains to check that it passes to the quotient $H^\BM_\ast(X,Y;\varphi^\pi) \otimes_{\Z[\pi]} M$. However, by definition, $\calM^M$ corresponds exactly to the subgroup whose action does pass to the quotient. Indeed, we have
 \begin{align*}
  &\rho^M(f) \left( (\tilde{\chain} \cdot \bloop) \otimes m \right) 
  = \left( \tilde{\chi}_*(f)(\tilde{\chain}) \cdot \chi_*(f)(\bloop) \right) \otimes m 
  = \tilde{\chi}_*(f)(\tilde{\chain}) \otimes \varphi^M(\chi_*(f)(\bloop))(m) 
  = \tilde{\chi}_*(f)(\tilde{\chain}) \otimes \varphi^M(\bloop)(m) \\
  &\hspace*{\parindent} = \rho^M(f) \left( \tilde{\chain} \otimes \varphi^M(\bloop)(m) \right)
 \end{align*}
 for all $f \in \calM^M$, $\tilde{\chain} \in H^\BM_\ast(X,Y;\varphi^\pi)$, $\bloop \in \pi$, and $m \in M$, where we used the definition of $\calM^M$ in order to obtain the second equality. This is precisely the identity we were looking for.

 Next, let us focus on claim~$(ii)$. Again, for all $f \in F$, $\tilde{\chain} \in H^\BM_\ast(X,Y;\varphi^\pi)$, and $m \in M$, the action 
 \[
  \rho^M(f)(\tilde{\chain} \otimes m) = \tilde{\chi}_*(f)(\tilde{\chain}) \otimes \psi(f)(m)
 \]
 is well-defined on the tensor product $H^\BM_\ast(X,Y;\varphi^\pi) \otimes_\Z M$, so we only need to check that it passes to the quotient $H^\BM_\ast(X,Y;\varphi^\pi) \otimes_{\Z[\pi]} M$. In this case, it is Equation~\eqref{E:f_gamma} that ensures this happens. Indeed, we have
 \begin{align*}
  &\rho^M(f) \left( (\tilde{\chain} \cdot \bloop) \otimes m \right) 
  = \left( \tilde{\chi}_*(f)(\tilde{\chain}) \cdot \chi_*(f)(\bloop) \right) \otimes \psi(f)(m) 
  = \tilde{\chi}_*(f)(\tilde{\chain}) \otimes \varphi^M(\chi_*(f)(\bloop))(\psi(f)(m)) \\
  &\hspace*{\parindent} = \tilde{\chi}_*(f)(\tilde{\chain}) \otimes \psi(f)(\varphi^M(\bloop)(m)) 
  = \rho^M(f) \left( \tilde{\chain} \otimes \varphi^M(\bloop)(m) \right)
 \end{align*}
 for all $f \in F$, $\tilde{\chain} \in H^\BM_\ast(X,Y;\varphi^\pi)$, $\bloop \in \pi$, and $m \in M$, where we used Equation~\eqref{E:f_gamma} in order to obtain the second equality.

 Notice that $\rho^M$ is indeed a representation (meaning a group homomorphism), because homology is functorial, and because $\psi$ is a homomorphism too.
\end{proof}

\begin{remark}\label{R:action_of_mapping_classes_on_intermediate_covers}
 Let $\chi : F \to \Mod(X)$, $\varphi^M : \Z[\pi] \to \End_R(M)$, and $\psi : F \to \GL_R(M)$ be homomorphisms satisfying the hypotheses of Proposition~\ref{P:specializing}.$(ii)$, let $K \triangleleft \pi$ be a normal subgroup, and suppose that $\varphi^M$ fits into the commutative diagram
 \begin{center}
  \begin{tikzpicture}[descr/.style={fill=white}]
   \node (P1) at (-1.5,1) {$\Z[\pi]$};
   \node (P2) at (1.5,1) {$\End_R(M)$};
   \node (P3) at (0,0) {$\Z[\pi/K]$};
   \draw
   (P1) edge[->] node[above] {\scriptsize $\varphi^M$} (P2)
   (P1) edge[->>] node[left,xshift=-7.5pt] {\scriptsize $\varphi^{\pi/K}$} (P3)
   (P3) edge[right hook->] (P2);
  \end{tikzpicture}
 \end{center}
 where $\varphi_{n,g}^{\pi/K}$ is induced by the projection $\pi \twoheadrightarrow \pi/K$. Then, if $\hat{p} : \hat{X} \to X$ denotes the regular cover associated with $K$, and if the homology $H^\BM_\ast(X,Y;\varphi^{\pi/K})$ is free over $\Z[\pi/K]$, we have
 \[
  H^\BM_\ast(X,Y;\varphi^M) \cong H^\BM_\ast(X,Y;\varphi^{\pi/K}) \otimes_{\Z[\pi/K]} M.
 \]
 Notice that Equation~\eqref{E:f_gamma} automatically implies that $\chi(f)_*$ stabilizes $\Z[K]$. Therefore, by the lifting property of regular covers, there exists a homomorphism $\hat{\chi} : F \to \hat{X}$ sending every $f \in F$ to the unique lift of $\chi(f) \circ \hat{p}$ fixing the (projection to $\hat{X}$ of the) constant path at the base point $\braid{\xi}{}$. Then, up to isomorphism, the homological action $\rho^M : F \to \GL_R(H^\BM_\ast(X,Y;\varphi^M))$ is given by 
 \[
  \rho^M(f)(\hat{\chain} \otimes m) = \hat{\chi}_*(f)(\hat{\chain}) \otimes \psi(f)(m)
 \]
 for all $f \in F$, $\hat{\chain} \in H^\BM_\ast(X,Y;\varphi^{\pi/K})$, and $m \in M$.
\end{remark}

In Proposition~\ref{P:specializing}, we defined an action of the group $F$ on the tensor product between $H^\BM_\ast(X,Y;\varphi^\pi)$ and the $\Z[\pi]$-module $M$, assuming that $H^\BM_\ast(X,Y;\varphi^\pi)$ is a free $\Z[\pi]$-module. When this is not the case, a similar construction can be carried out, but we have to define the action at the level of chain complexes instead.

Proposition~\ref{P:specializing} yields in particular a strategy for finding interesting linear representations of $F = \Mod(\varSigma_{g,1})$ in the case $X = \Conf_n(\varSigma_{g,1})$. However, since the procedure requires a representation $\psi$ of $\Mod(\varSigma_{g,1})$ into $\GL_R(M)$ to start with, it might be hard to apply it in general. Luckily, if the image of $\varphi^M$ is large enough, then the linear representation $\rho_M$ of 
$F$ descends to a projective representation $\bar{\rho}_M$ of $F/\ker \chi$, as it will be explained in Proposition~\ref{P:lifting_actions} below. Then, it is sufficient to apply the previous procedure to a group $F$ whose quotient $F/\ker \chi$ is isomorphic to $\Mod(\varSigma_{g,1})$. This time, the main example to keep in mind arises when $F$ is the free group over a system of generators for $\Mod(\varSigma_{g,1})$, in which case the homomorphism $\chi : F \to \Mod(X)$ defined by 
\[
 \chi(f)(\myuline{x}) = \{ f(x_1),\ldots,f(x_n) \}
\]
for every $\myuline{x} = \{ x_1,\ldots,x_n \} \in X = \Conf_n(\varSigma_{g,1})$ satisfies $F/\ker \chi \cong \Mod(\varSigma_{g,1})$. Notice that defining a linear representation of $F$ is considerably easier than defining a linear representation of $F/\ker \chi$, in this case.

\begin{proposition}\label{P:lifting_actions}
 Let $\chi : F \to \Mod(X)$ be a homomorphism and $M$ be an $R$-module equipped with representations $\varphi^M : \Z[\pi] \to \End_R(M)$ and $\psi : F \to \GL_R(M)$ satisfying Equation~\eqref{E:f_gamma}. If $H^\BM_\ast(X,Y;\varphi^\pi)$ is a free $\Z[\pi]$-module, and if furthermore 
 \begin{equation}
  C_{\End_R(M)}(\varphi^M(\Z[\pi])) = R^{\times}, \label{E:centralizer}
 \end{equation}
 meaning that the centralizer of $\varphi^M(\Z[\pi])$ in $\End_R(M)$ is the group of invertible scalar multiples of the identity $\id_M$, then the $R$-linear action of $F$ given by Proposition~\ref{P:specializing}.$(ii)$ descends to an $R$-projective action
 \[
  \bar{\rho}^M : F / \ker \chi \to \PGL_R(H^\BM_\ast(X,Y;\varphi^M))
 \]
 which fits into the commutative diagram 
 \begin{center}
  \begin{tikzpicture}[descr/.style={fill=white}] \everymath{\displaystyle}
   \node (P0) at (0,0) {$F$};
   \node (P1) at (3.5,0) {$\GL_R(H^\BM_\ast(X,Y;\varphi^M))$};
   \node (P2) at (0,-1) {$F/\ker \chi$};
   \node (P3) at (3.5,-1) {$\PGL_R(H^\BM_\ast(X,Y;\varphi^M))$};
   \draw
   (P0) edge[->] node[above] {\scriptsize $\rho^M$}(P1)
   (P0) edge[->>] node[left] {} (P2)
   (P1) edge[->>] node[right] {} (P3)
   (P2) edge[->] node[below] {\scriptsize $\bar{\rho}^M$} (P3);
  \end{tikzpicture}
 \end{center}
\end{proposition}

\begin{proof}
 We need to check that
 \[
  \rho^M(f) \propto \id_{H^\BM_\ast(X,Y;\varphi^M)}
 \]
 for every $f \in \ker \chi$. Notice however that $\ker \chi \subset \ker \chi_*$, where $\chi_* : F \to \Aut(\Z[\pi])$ is the homomorphism induced by the natural action of $\Mod(X)$ onto $\Z[\pi]$. Then, the claim follows immediately from the fact that
 \[
  \psi(f) \circ \varphi^M(\bloop) = \varphi^M(\chi_*(f)(\bloop)) \circ \psi(f) = \varphi^M(\bloop) \circ \psi(f),
 \]
 which implies that $\psi(f)$ belongs to the centralizer of $\varphi^M(\Z[\pi])$ in $\End_R(M)$.
\end{proof}

\section{Homological computation rules}\label{A:graphs_and_rules}

In Section~\ref{S:diagrammatic_notation}, we explained how to interpret homologically a particular class of diagrams representing \textit{embedded twisted cycles}. In Section~\ref{S:diagrammatic_rules}, we presented a set of rules for manipulating these diagrams without changing the homology classes they represent. We claim that all these rules can be straightforwardly deduced from \cite[Section~4.2]{M20}, but we give elements of proofs here below.

\begin{proof}[Sketch of proof of Proposition~\ref{P:computation_rules}]
 All these rules are established in \cite[Section~4]{M20}. Even though the reference focuses on twisted Borel--Moore homologies of configuration spaces of punctured discs, rather than arbitrary surfaces of positive genus with one boundary component, most of these diagrammatic rules involve modifications which only take place inside a disc embedded into the surface, so that the corresponding proofs are exactly the same. We sketch the arguments here, and give precise references rule by rule:

 \paragraph*{\textbf{Cutting rule}} This is a direct consequence of \cite[Proposition~4.4]{M20}, see the diagrammatic equality in \cite[Example~4.6]{M20}, which resembles the one considered here. In order to prove the claim, it is convenient to introduce slightly more general diagrams, where dashed-dotted curves representing multisimplices are allowed to intersect the boundary of $\varSigma_{g,1}$ also in their interior. Then, there exists an ambient isotopy of the form ${[0,1[} \times \varSigma_{g,1} \to \varSigma_{g,1}$ (which can be extended to $X_{k,g}$ coordinate-by-coordinate) yielding the equality
 \[
  \pic{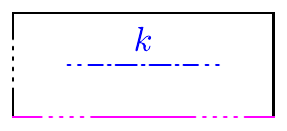} = \pic{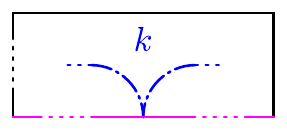}
 \]
 between standard homology classes (that is, untwisted ones, for which we have not chosen a red multithread yet, which would specify a lift to the universal cover).  If we denote by $\simp : [0,1] \to \varSigma_{g,1}$ the $1$-multisimplex represented by the dashed-dotted curves appearing on the right-hand side of the previous equality, then, without loss of generality, we can assume that 
  \[
   \simp \Big( {]0,1[} \Big) \cap \partial_- \varSigma_{g,1} = \simp \left( \frac{1}{2} \right).
  \]
  We can subdivide the $k$-simplex $\Delta^k$ into the union
  \[
   \Delta^k = \bigcup_{\ell=0}^k \Delta^{(\ell,k)},
  \]
  where every $\Delta^{(\ell,k)}$ on the right-hand side of the equality is the product of an $\ell$-simplex with a $(k-\ell)$-simplex, given by
  \[
   \Delta^{(\ell,k)} = \left\{ (t_1,\ldots,t_k) \in \R^k \Bigm\vert 0 < t_1 < \cdots < t_\ell \leqs \frac{1}{2} < t_{\ell+1} < \cdots < t_k < 1 \right\}.
  \]
  Up to reparametrization, and also up to another ambient isotopy, this time of the form ${]0,1]} \times \varSigma_{g,1} \to \varSigma_{g,1}$, the subdivision of $\Delta^k$ considered above induces the equality
 \[
  \pic{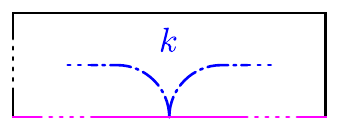} = \sum_{\ell=0}^k \pic{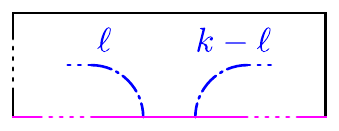}
 \]
 Again, this equality holds already before homology classes are lifted to the universal cover. Then, it is easy to check that the red multithreads appearing in the claim are exactly the ones ensuring that the two lifts match.

 \paragraph*{\textbf{Permutation rule}} In order to prove the claim, it is convenient to introduce even more general diagrams, where red multithreads are allowed to self-intersect. In these diagrams, red double points correspond to points in the surface where the multithread passes at different times, and we keep track of which strand passes after the other by making them the over-strand and the under-strand, respectively, in a crossing (this corresponds to considering the link diagram of the red multithread with respect to the standard projection $\varSigma_{g,1} \times [0,1] \to \varSigma_{g,1}$, as seen from above in the $[0,1]$-direction). Then, up to isotopy, we have
 \begin{align*}
  \pic{computation_rules_permutation_1} 
  &= \pic{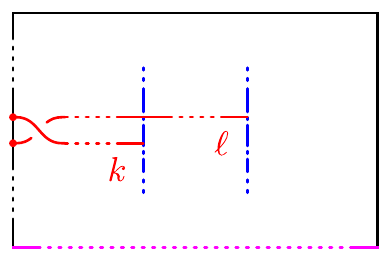}
 \end{align*}
 Notice that, if we denote by $\thread_{(k,\ell)}$ the multithread appearing on the right-hand side of the previous equality, and by $\perm_{\thread_{(k,\ell)}}$ the induced permutation in $\frakS_n$ (following the notation of Section~\ref{S:diagrammatic_notation}), then we have
 \[
  \sgn(\perm_{\thread_{(k,\ell)}}) = (-1)^{k\ell}.
 \]
 This means that, if we denote by $\braid{\sigma}{(k,\ell)} \in \pi_{n,g}$ the standard braid represented by
 \[
  \braid{\sigma}{(k,\ell)} = \pic{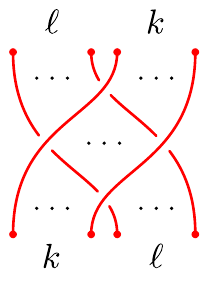}
 \]
 around the components of the basepoint $\braid{\xi}{}$ that appear in the pictures above, then Remarks~\ref{R:thread_and_orientation} and \ref{R:why_we_conjugate} yield
 \begin{align*}
  \pic{computation_rules_permutation_proof_1}
  &= (-1)^{k\ell} \braid{\sigma}{(k,\ell)} \cdot \pic{computation_rules_permutation_2} 
  = \pic{computation_rules_permutation_2} \otimes (-1)^{k\ell} \varphi^\dHeis_{n,g}(\braid{\sigma}{(k,\ell)}^{-1}).
 \end{align*}
 Therefore, we simply need to notice that
 \[
  \varphi^\dHeis_{n,g}(\braid{\sigma}{(k,\ell)}) = (-q^{-2})^{k\ell}.
 \]

 \paragraph*{\textbf{Fusion rule}} The case $\ell=1$ is exactly the same as in \cite[Lemma~4.8]{M20}. The case $\ell>1$ is established in \cite[Corollary~4.11]{M20} as a straightforward consequence of \cite[Corollary~4.10]{M20}, which itself follows directly from the case $\ell=1$ by a combinatorial expansion. We will only prove the case $\ell=1$ here, since the case $\ell>1$ follows from the same combinatorial argument. Then, let us prove
 \begin{equation}\label{E:fusion_k_1}
  \pic{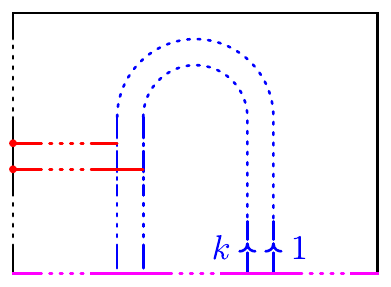} 
  = \pic{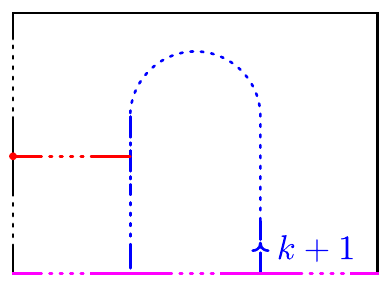} \otimes [k+1]_q q^k.  
 \end{equation}
 Notice that the band contained between the $k$-labeled and the $1$-labeled dashed-dotted curve is homeomorphic to the interior of a rectangle, whose horizontal sides are blue (and identified with the two curves), and whose vertical sides are magenta (and identified with a pair of arcs in $\partial_- \varSigma_{g,1}$). There exists a vertical deformation retraction of this rectangle onto its $k$-labeled dashed-dotted horizontal side, and we can extend this to a deformation retraction of the surface $\varSigma^0_{g,1} = \varSigma_{g,1}$ onto a subsurface $\varSigma^1_{g,1} \subset \varSigma_{g,1}$ containing the $k$-labeled dashed-dotted curve, as shown in the following picture:
 \[
  \pic{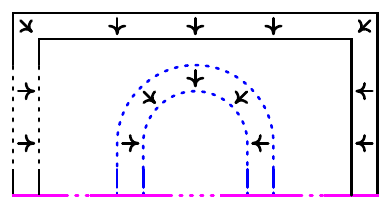}
 \]
 Furthermore, we can safely assume, without loss of generality, there exists a metric on the surface which makes this deformation into a contraction at all times. Let $d$ be the associated distance, and let us extend the deformation retraction to $X_{n,g}$ coordinate-by-coordinate. We also denote by $X_{n,g}^t$ and by $Y_{n,g}^t$ the images of $X_{n,g}$ and of $Y_{n,g}$, respectively, under this deformation retraction at the time $t \in [0,1]$. Then, for every $\epsilon > 0$, there exists a time $t \in {]0,1[}$ at which the deformation retraction maps the $\ell$-labeled blue dashed-dotted curve into an $\frac{\epsilon}{2}$-neighborhood of the $k$-labeled one with respect to the distance $d$. Since the deformation retraction is a contraction, we obtain an isomorphism
 \[
  H_n \left( X^0_{n,g},Y^0_{n,g} \cup (X^0_{n,g} \smallsetminus K_\epsilon(X^0_{n,g})) \right) 
  \cong H_n \left( X^t_{n,g},Y^t_{n,g} \cup (X^t_{n,g} \smallsetminus K_\epsilon(X^t_{n,g})) \right),
 \]
 where $K_\epsilon(Z_{n,g}) := \{ \{ z_1,\ldots,z_n \} \in Z_{n,g} \mid \Forall i \neq j \quad d(z_i,z_j) \geqs \epsilon \}$ for every subset $Z_{n,g} \subset X_{n,g}$. Then, in the rectangle, every configuration with at least a pair of coordinates that are vertically aligned lies in the relative part of the homology on the right-hand side. Thus, we can excise the set $E^t_{n,g}$ of all configurations of this form, and get the isomorphism
 \begin{align*}
  H_n \left( X^t_{n,g},Y^t_{n,g} \cup (X^t_{n,g} \smallsetminus K_\epsilon(X^t_{n,g})) \right) 
  \cong H_n \left( X^t_{n,g} \smallsetminus E^t_{n,g}, \left( Y^t_{n,g} \cup (X^t_{n,g} \smallsetminus K_\epsilon(X^t_{n,g})) \right) \smallsetminus E^t_{n,g} \right).
 \end{align*}
 Since in $X^t_{n,g} \smallsetminus E^t_{n,g}$ no configuration has vertically aligned coordinates inside the rectangle, the vertical projection yields an isomorphism
 \begin{equation*}
  H_n \left( X^0_{n,g},Y^0_{n,g} \cup (X_{n,g} \smallsetminus K_\epsilon(X^0_{n,g})) \right) 
  \cong H_n \left( X^1_{n,g},Y^1_{n,g} \cup (X^1_{n,g} \smallsetminus K_\epsilon(X^1_{n,g})) \right).
 \end{equation*}
 Thus, we obtain an isomorphism
 \[
  \iota_* : H_n^\BM \left( X^1_{n,g},Y^1_{n,g} \right) \to H_n^\BM \left( X^0_{n,g},Y^0_{n,g} \right)
 \]
 induced by the inclusion of pairs, since both families of complements of compact sets $X^0_{n,g} \smallsetminus K_\epsilon(X^0_{n,g})$ and $X^1_{n,g} \smallsetminus K_\epsilon(X^1_{n,g})$, indexed by $\varepsilon > 0$, are cofinal.  Then, for $n = k+1$, let us denote by 
 \begin{align*}
  \bfsimp_0(k,1) &\in H_n^\BM ( X^0_{n,g},Y^0_{n,g} ), &
  \bfsimp_1(k+1) &\in H_n^\BM ( X^1_{n,g},Y^1_{n,g} )
 \end{align*}
 the homology classes represented by the diagrams obtained from the left-hand side and from the right-hand side of Equation~\eqref{E:fusion_k_1}, respectively, by forgetting red multithreads. Notice that we can subdivide the interval $[0,1[$ into the union of $k+1$ intervals
 \[
   {[0,1[} = \bigcup_{i=1}^{k+1} \left[ \frac{i-1}{k+1}, \frac{i}{k+1} \right[
 \]
 and consider, for all integers $1 \leqs i \leqs k+1$, the permutation homeomorphism 
 \begin{align*}
  \perm_{(i)} : {]0,1[}^{\times (k+1)} &\to {]0,1[}^{\times (k+1)} \\*
  (t_1,\ldots,t_{k+1}) &\mapsto (t_2,\ldots,t_i,t_1,t_{i+1},\ldots,t_{k+1}).
 \end{align*}
 Therefore, up to decomposing and reparametrizing the first component of the hypercube, the homology class $\iota_*^{-1}(\bfsimp_0(k,1))$ is the sum of the $k+1$ homology classes $(\bfsimp_1(k+1))_{(i)}$ determined by the $k+1$ embeddings
 \begin{center}
  \begin{tikzpicture}
   \node (P1) at (0,0) {${]0,1[}^{\times (k+1)}$};
   \node (P2) at (3,0) {${]0,1[}^{\times (k+1)}$};
   \node (P3) at (6,0) {$\Delta^{k+1}$};
   \node (P4) at (9,0) {$X_{k+1,g}$,};
   \draw
   (P1) edge[->] node[above] {\scriptsize $\perm_{(i)}$} (P2)
   (P2) edge[->] node[above] {\scriptsize $\bfDelta^{k+1}$} (P3)
   (P3) edge[->] node[above] {\scriptsize $\bfsimp_1^{k+1}$} (P4);
  \end{tikzpicture}
 \end{center}
 where we are using the notation introduced in Section~\ref{S:diagrammatic_notation} for homeomorphisms between hypercubes and simplices. If we denote by 
 \[
  \tilde{\iota}_* : H_n^\BM \left( X^1_{n,g},Y^1_{n,g};\varphi_{n,g}^\pi \right) \to H_n^\BM \left( X^0_{n,g},Y^0_{n,g};\varphi_{n,g}^\pi \right) 
 \]
 the lift of the isomorphism $\iota_*$, where $H^\BM_\ast \left( X^t_{n,g},Y^t_{n,g};\varphi_{n,g}^\pi \right)$ denotes the homology of the complex
 \[
  C^\BM_\ast \left( X^t_{n,g},Y^t_{n,g};\varphi_{n,g}^\pi \right) := \varprojlim_{K \in \calK(X^t_{n,g})} C_\ast(\tilde{p}^{-1}(X^t_{n,g}),\tilde{p}^{-1}(Y^t_{n,g} \cup (X^t_{n,g} \smallsetminus K)))
 \]
 in the notation of Appendix~\ref{A:Borel--Moore}, then we have
 \begin{align*}
  \tilde{\iota}_*^{-1} \left( \pic{computation_rules_fusion_3} \right) 
  &= \sum_{i=1}^{k+1} \pic{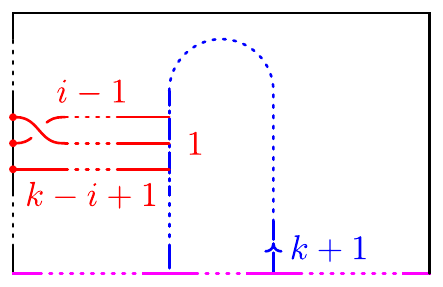} 
  = \sum_{i=1}^{k+1} \pic{computation_rules_fusion_4} \otimes q^{2(i-1)},
 \end{align*}
 where the last equality follows from the proof of the permutation rule. Now the claim follows from
 \[
  \sum_{i=1}^{k+1} q^{2(i-1)} = [k+1]_q q^k.
 \]

 \paragraph*{\textbf{Orientation rule}} Changing the orientation of a dashed-dotted curve amounts to reparametrizing the corresponding $1$-multisimplex $\simp : [0,1] \to \varSigma_{g,1}$ by precomposition with 
 \begin{align*}
  {[0,1]} &\to {[0,1]} \\*
  t &\mapsto 1-t.
 \end{align*}

 \paragraph*{\textbf{Braid rule}} Recalling Remark~\ref{R:why_we_conjugate}, we have
 \begin{align*}
  \pic{computation_rules_braid_1} 
  &= (\thread_\alpha \ast \thread_{\beta}^{-1}) \cdot \pic{computation_rules_braid_2} 
  = \pic{computation_rules_braid_2} \otimes \varphi^\dHeis_{n,g}(\thread_\beta \ast \thread_{\alpha}^{-1})
 \end{align*}
 where
 \[
  \thread_\alpha \ast \thread_{\beta}^{-1} = \pic{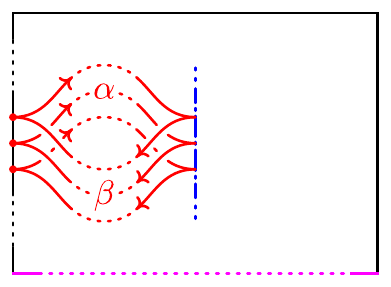}
 \]
 Then, notice that this braid is of the form
 \[
  \thread_\alpha \ast \thread_{\beta}^{-1} = \braid{(\alpha \ast \beta^{-1})}{(i+k-1)} \ast \braid{(\alpha \ast \beta^{-1})}{(i+k-2)} \ast \ldots \ast \braid{(\alpha \ast \beta^{-1})}{(i)}.
 \]
 for some $1 \leqs i \leqs n-k+1$. Indeed, we can start by moving the $(i+k-1)$th component of the basepoint $\braid{\xi}{}$ along $\alpha \ast \beta^{-1}$, then proceed to move the $(i+k-2)$th component (after the previous one parked back at $\braid{\xi}{}$), and so on, up until the $i$th component. Notice that changing the component of the basepoint $\braid{\xi}{}$ that travels along $\alpha \ast \beta^{-1}$ does not change the braid itself (since this only amounts to conjugating by the central element $\sigma = -q^2$). Therefore, we obtain
 \[
  \varphi^\dHeis_{n,g}(\thread_\beta \ast \thread_{\alpha}^{-1}) = \varphi^\dHeis_{n,g}(\braid{(\beta \ast \alpha^{-1})}{(i)})^k. \qedhere
 \]
\end{proof}

The additional rules reported here below can be deduced from the previous ones, and are extensively used in computations.

\begin{remark}\label{R:computation_rules}
 Notice that the following additional relations are a direct consequence of the computation rules of Proposition~\ref{P:computation_rules}:
 \begin{description}
  \item[\normalfont Hook rule (plus variant).] For every integer $0 \leqs k \leqs n$ we have
   \begin{align*}
    \pic{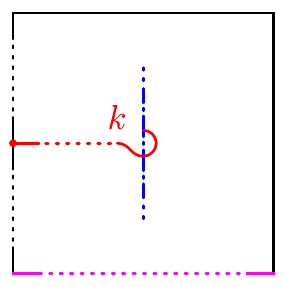}
    &= \pic{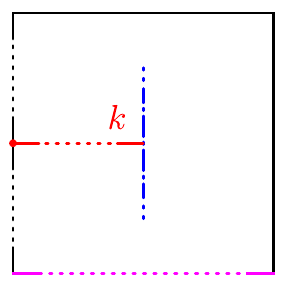} \otimes q^{k(k-1)} \tag{H}\label{E:hook} \\
    \pic{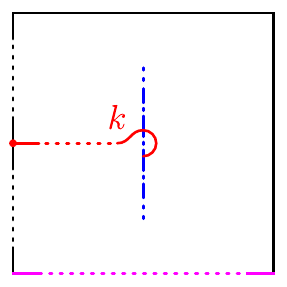} 
    &= \pic{computation_rules_permutation_3} \otimes q^{-k(k-1)} \tag{$\tilde{\rmH}$}\label{E:tilde_hook}
   \end{align*}
  \item[\normalfont Cutting rule (variant).] For every integer $0 \leqs k \leqs n$ we have 
   \[
    \pic{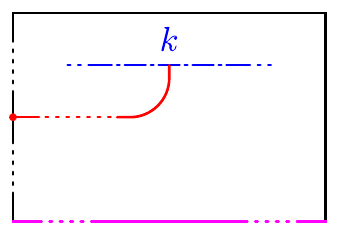} 
    = \sum_{\ell = 0}^k \pic{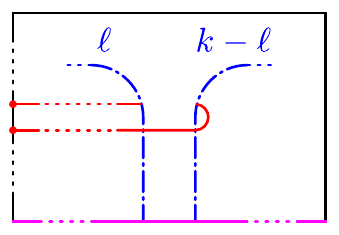} \tag{$\tilde{\rmC}$}\label{E:tilde_cutting}
   \]
   where the pair of dotted red arcs on the right-hand side of the equality runs parallel in surface.
 \end{description} 
 Indeed, the hook rule (and its variatn) are a direct consequence of the permutation rule, and they can be shown by induction on the integer $k \geqs 1$. Similarly, the variant of the cutting rule follows directly from the hook rule.
\end{remark}

\section{Relations between TQFTs}\label{A:TQFT}

In Section~\ref{S:TQFT_rep_of_mcg}, we recalled the construction, due to Kerler and Lyubashenko, of a non-semisimple TQFT $J_H : \RCob \to \mods{H}$ out of a factorizable ribbon Hopf algebra $H$, following the approach of \cite{BD21}. In this appendix, we explain how this TQFT can be understood as part of the TQFT constructed in \cite{DGP17, DGGPR19}, thanks to a diagram of functors
\begin{equation}\label{E:commutative_diagram}
 \raisebox{-0.5\height+2.5pt}{
 \begin{tikzpicture}[descr/.style={fill=white}] \everymath{\displaystyle}
  \node (P0) at (0,0) {$\RCob$};
  \node (P1) at (3,0) {$\mods{H}$};
  \node (P2) at (0,-1.5) {$\adCob_H$};
  \node (P3) at (3,-1.5) {$\Vect_\Bbbk$};
  \draw
  (P0) edge[->] node[above] {\scriptsize $J_H$}(P1)
  (P0) edge[->] node[left] {\scriptsize $D^2_{(-,H)}$} (P2)
  (P1) edge[->] node[right] {\scriptsize $F$} (P3)
  (P2) edge[->] node[below] {\scriptsize $\rmV_H$} (P3);
 \end{tikzpicture}
 }
\end{equation}
Let us define the categories and functors appearing in this diagram. First of all, $F : \mods{H} \to \Vect_\Bbbk$ denotes the forgetful functor with target the category of $\Bbbk$-vector spaces, which simply forgets the action of $H$.

Next, the \textit{category $\adCob_H$ of admissible $H$-decorated cobordisms} is defined as follows:
\begin{itemize}
 \item Objects of $\adCob_H$ are triples $\bbSigma = (\Sigma,P,\calL)$ where $\Sigma$ is a closed surface, while $P \subset \Sigma$ is a finite set of oriented framed points labeled by $H$-modules, and $\calL \subset H_1(\Sigma;\R)$ is a Lagrangian subspace.
 \item Morphisms of $\adCob_H$ from $\bbSigma = (\Sigma,P,\calL)$ to $\bbSigma' = (\Sigma',P',\calL')$ are equivalence classes of admissible triples $\bbM = [M,T,n]$, where $M$ is a 3-dimensional cobordism from $\Sigma$ to $\Sigma'$, while $T \subset M$ is a \textit{bichrome graph}\footnote{A bichrome graph is a ribbon graph whose edges can be either red (and unlabeled) or blue (and labeled by $H$-modules). Similarly, coupons can be either bichrome (and unlabeled) or blue (and labeled by $H$-module intertwiners). A precise definition can be found in \cite[Section~3.1]{DGGPR19}.} from $P$ to $P'$, and $n \in \Z$ is a signature defect. A triple $(M,T,n)$ is \textit{admissible} if every connected component of $M$ disjoint from the outgoing boundary $\partial_+ M \cong \Sigma$ contains an \textit{admissible} subgraph of $T$ (meaning a bichrome graph whose blue edges feature a projective $H$-module among their labels). Two triples $(M,T,n)$ and $(M',T',n')$ are equivalent if $n = n'$, and if there exists an isomorphism of cobordisms $f : M \rightarrow M'$ satisfying $f(T) = T'$. 
 \item The composition 
  \[
   \bbM' \circ \bbM : \bbSigma \rightarrow \bbSigma''
  \]
  of morphisms $\bbM' \in \adCob_H(\bbSigma',\bbSigma'')$, $\bbM \in \adCob_H(\bbSigma,\bbSigma')$ is the equivalence class of the triple
  \[
   ( M \cup_{\Sigma'} M', T \cup_{P'} T', n + n' - \mu ( M_*(\calL),\calL',M'^* (\calL'') ) ).
  \]
\end{itemize}
Notice that the admissibility condition we are imposing here on morphisms is different from (but equivalent to) the one of \cite{DGP17,DGGPR19}, in the sense that the requirement only affects connected components of cobordisms disjoint from the outgoing boundary, instead of those disjoint from the incoming one. In other words, we are essentially considering the opposite category to the one appearing in the references. The category $\adCob_H$ is a symmetric monoidal category, whose tensor product is induced by disjoint union.

Then, the functor $D^2_{(-,H)} : \RCob \to \adCob_H$ sends every surface $\varSigma_{g,1}$ to the triple 
\[
 (\varSigma_{g,1} \cup_{S^1} D^2,P_{(-,H)},\iota_*(\calL_g)),
\]
where $\varSigma_{g,1} \cup_{S^1} D^2$ denotes the closed surface of genus $g$ obtained from $\varSigma_{g,1}$ by gluing a disc $D^2$ along its boundary, $P_{(-,H)}$ denotes the center of the disc with negative orientation and label given by the regular representation $H$ (determined by left multiplication of $H$ onto itself), and where $\iota_* : H_1(\varSigma_{g,1};\R) \to H_1(\varSigma_{g,1} \cup_{S^1} D^2;\R)$ is induced by inclusion. Similarly, it sends every morphism $(M,\sig)$ to the equivalence class of the triple
\[
 (M \cup_{S^1 \times [0,1]} (D^2 \times [0,1]),P_{(-,H)} \times [0,1],\sig).
\]

Finally, $\rmV_H : \adCob_3 \to \Vect_\Bbbk$ denotes the TQFT constructed in \cite[Section~3]{DGP17}, or equivalently in \cite[Section~4]{DGGPR19}, see \cite[Appendix~C]{DGGPR20} for an explanation of the equivalence of the two approaches. Notice that, since the definition of $\adCob_3$ provided above is naturally equivalent to the opposite of the category of admissible cobordisms considered in all these references, the TQFT $\rmV_H$ is naturally isomorphic to the contravariant TQFT of \cite{DGP17, DGGPR19}.

\begin{proposition}\label{P:identification_of_state_spaces}
 The diagram of functors~\eqref{E:commutative_diagram} is commutative.
\end{proposition}

\begin{proof}[Sketch of proof]
 In order to show the claim, let us exchange the two algebraic models for covariant and contravariant state spaces given in \cite[Section~4.1]{DGGPR19}, and define
 \begin{align*}
  \calX_{g,V} &:= \Hom_H(\one,\ad^{\otimes g} \otimes V), &
  \calX'_{g,V} &:= \Hom_H(\coad^{\otimes g} \otimes V,P_{\one})
 \end{align*}
 for every integer $g \geqs 0$ and for every $H$-module $V$. Here, $P_{\one}$ denotes the projective cover of the trivial $H$-module, while $\coad$ denotes the \textit{coadjoint representation}, which is the $\Bbbk$-vector space $H^*$ equipped with the coadjoint action
 \[
  (x \triangleright f)(y) = f(S(x_{(1)})yx_{(2)})
 \]
 for all $x,y \in H$ and $f \in \coad$. This $H$-module coincides with the coend
 \[
  \coad = \int^{X \in \mods{H}} X^* \otimes X,
 \]
 which is defined as the universal dinatural transformation with source
 \begin{align*}
  (\_^* \otimes \_) : (\mods{H})^\op \times \mods{H} & \to \mods{H} \\*
  (X,Y) & \mapsto X^* \otimes Y,
 \end{align*}
 see \cite[Section~IX.5]{M71} for a definition. Roughly speaking, $\coad$ can be equipped with a dinatural family of intertwiners
 \begin{align*}
  i_X : X^* \otimes X &\to \coad \\* 
  f \otimes v &\mapsto f(\_ \cdot v)
 \end{align*}
 for every $H$-module $X$, where $f(\_ \cdot v)(x) = f(x \cdot v)$ for all $x \in H$. Dinaturality means that $i_X \circ (f^* \otimes \id_X) = i_Y \circ (\id_{Y^*} \otimes f)$ for every intertwiner $f : X \to Y$, and these structure morphisms make $\coad$ into the terminal object of the category of dinatural transformations with source $(\_^* \otimes \_)$, see for instance \cite[Proposition~7.1]{FGR17}.

 Exchanging the algebraic models for state spaces requires changing the pairing $\langle \_,\_ \rangle_{g,V} : \calX'_{g,V} \times \calX_{g,V} \to \Bbbk$ by appropriately using the projection $\epsilon_{\one} : P_{\one} \to \one$ instead of the injection $\eta_{\one} : \one \to P_{\one}$. Now the statement of \cite[Lemma~4.1]{DGGPR19}, and its proof, can be straightforwardly translated, and the same goes for \cite[Proposition~4.17]{DGGPR19}. This proves that
 \[
  \rmV_H(D^2_{(-,H)}(\varSigma_{g,1})) \cong \Hom_H(H,\ad^{\otimes g}).
 \]
 Then, it is easy to see that 
 \begin{align*}
  \Hom_H(H,\ad^{\otimes g}) &\to \ad^{\otimes g} \\*
  f &\mapsto f(1)
 \end{align*}
 is a linear isomorphism, because the regular representation $H$ is a free $H$-module of rank $1$. This establishes commutativity of the diagram at the level of objects.

 In order to prove commutativity at the level of morphisms, it is useful to recall the skein description of state spaces of \cite[Section~4.7]{DGGPR19}, which provides an isomorphism
 \begin{align*}
  \ad^{\otimes g} &\to \rmV_H(D^2_{(-,H)}(\varSigma_{g,1})) \\*
  x_1 \otimes \ldots \otimes x_g &\mapsto [H_g,T_{x_1 \otimes \ldots \otimes x_g}0]
 \end{align*}
 where, using the diagrammatic conventions of Section~\ref{S:top_tangles}, the bichrome graph $T_{x_1 \otimes \ldots \otimes x_g}$ is given by
 \[
  T_{x_1 \otimes \ldots \otimes x_g} = \pic{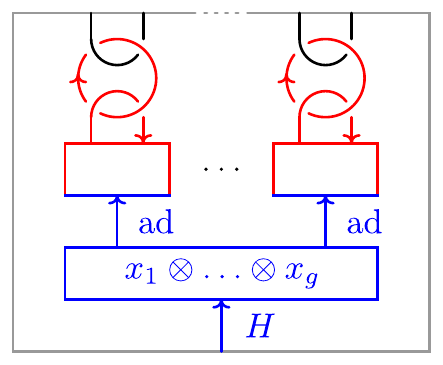}
 \]
 with the vector $x_1 \otimes \ldots \otimes x_g \in \ad^{\otimes g}$ also abusively denoting the corresponding intertwiner in $\Hom_H(H,\ad^{\otimes g})$. Then, we need to check that, for every top tangle $T$ representing a framed cobordism $(M(T),\sig(T))$ in $\RCob$, we have
 \[
  \rmV_H(D_{(-,H)}(M(T),\sig(T)))[H_g,T_{x_1 \otimes \ldots \otimes x_g},0] 
  = J_H(M(T),\sig(T))(x_1 \otimes \ldots \otimes x_g).
 \]
 Notice that, thanks to \cite[Theorem~5.5.4]{BP11}, it is sufficient to restrict our attention to the set of generating morphisms of $\RCob$ given in \cite[Section~6.3]{BD21}. Then, using the BPH algebra structure on the end given in \cite[Section~7.1]{BD21}, the claim can be established by suitable skein equivalences. For instance, the product cobordism yields
 \[
  \pic{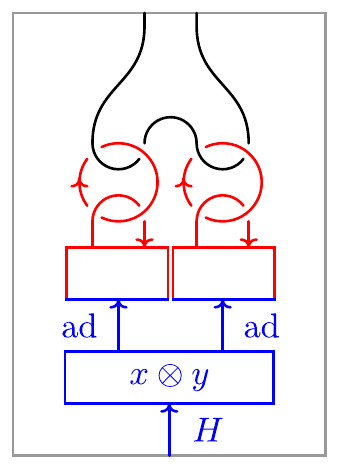} \doteq \pic{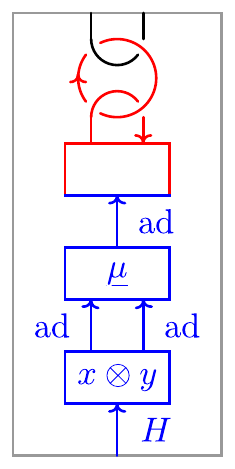}
 \]
 while the coproduct cobordism yields
 \[
  \pic{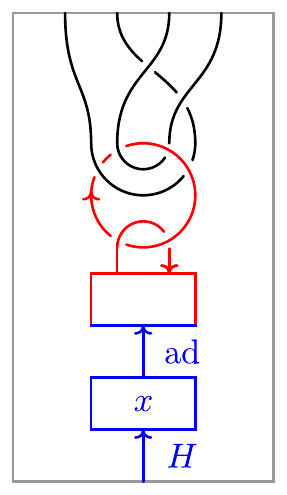} \doteq \pic{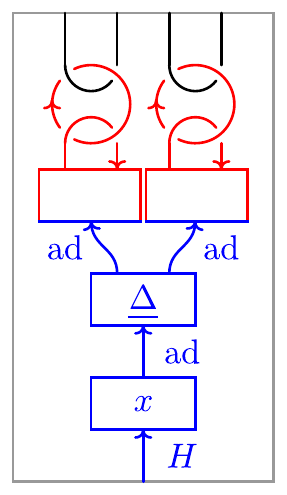}
 \]
 and the antipode cobordism yields
 \[
  \pic{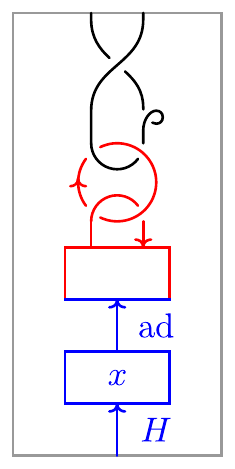} \doteq \pic{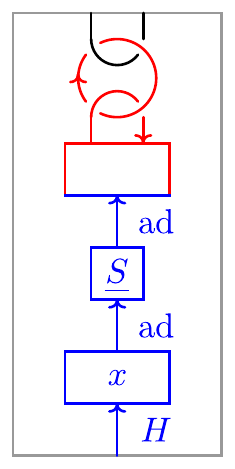}
 \]
 Then, by comparing with the explicit case of the transmutation discussed in \cite[Section~7.2]{BD21}, it is easy to check that the algorithm of Section~\ref{S:quantum_diagrammatic_calculus} provides the same computations. We leave details, as well as the rest of the generating morphisms, to the reader.
\end{proof}

In particular, thanks to \cite[Theorem~4.4]{DGGPR20}, the quantum representations of Section~\ref{S:quantum_representations_of_mcg} are equivalent to those defined by Lyubashenko in \cite{L94}.

\section{Quantum identities}\label{A:Quantum_identities}

In this section, we collect identities that are used extensively in all quantum computations.

\subsection{Computations in integral bases}\label{A:Identities_in_hdb}

First of all, we compute coproducts and antipodes in integral bases.

\begin{lemma}\label{L:coproducts_antipodes}
 For all integers $0 \leqs a,b,c \leqs r-1$ we have
 \begin{align}
  \Delta(F^{(a)} E^b T_c)
  &= \sum_{d=0}^{r-1} \sum_{i=0}^a \sum_{j=0}^b \sqbinom{b}{j}_\zeta \zeta^{(a+2c)i+bj-2d(i+j)-(i+j)^2} 
  F^{(a-i)} E^j T_{c-d} \otimes F^{(i)} E^{b-j} T_d, \label{E:coproducts} \\*  
  S(F^{(a)} E^b T_c) 
  &= (-1)^{a+b} \zeta^{(a-b+2c-1)(a-b)} T_{-c} E^b F^{(a)}. \label{E:antipodes} 
 \end{align}
\end{lemma}

\begin{proof}
 \cite[Chapter~3, Proposition~5]{KS97} gives
 \begin{align*}
  \Delta(F^\ell K^m E^n) 
  &= \sum_{i=0}^\ell \sum_{j=0}^n \sqbinom{\ell}{i}_\zeta \sqbinom{n}{j}_\zeta \zeta^{i(\ell-i)-j(n-j)} F^{\ell-i} K^{m-i} E^j \otimes F^i K^{m+j} E^{n-j}, \\*
  S(F^\ell K^m E^n) 
  &= (-1)^{\ell+n} \zeta^{\ell(\ell-1)-n(n-1)} E^n K^{\ell-m-n} F^\ell.
 \end{align*}
 Furthermore
 \begin{align*}
  \Delta(T_c) 
  &= \frac{1}{r} \sum_{b=0}^{r-1} \zeta^{2bc} \Delta(K^b) 
  = \frac{1}{r} \sum_{d=0}^{r-1} \zeta^{2bc} K^b \otimes K^b 
  = \frac{1}{r} \sum_{c,d=0}^{r-1} \zeta^{2b(c-d)} K^b \otimes T_d
  = \sum_{d=0}^{r-1} T_{c-d} \otimes T_d, \\*
  S(T_c) &= \frac{1}{r} \sum_{b=0}^{r-1} \zeta^{2bc} S(K^b)
  = \frac{1}{r} \sum_{b=0}^{r-1} \zeta^{2bc} K^{-b}
  = T_{-c}.
 \end{align*}
 Using this, it easy to check the claim.
\end{proof}

Next, we derive an explicit formula for the ribbon element and its inverse.

\begin{lemma}\label{L:ribbon}
 The ribbon element $v \in \fraku_\zeta$ and its inverse $v^{-1} \in \fraku_\zeta$ are given by
 \begin{align}
  v 
  &= \sum_{a,b=0}^{r-1} (-1)^a 
  \zeta^{-\frac{(a+3)a}{2}+2(a-b+1)b} F^{(a)} E^a T_b \label{E:ribbon_Habiro} \\*
  &= \sum_{a,b=0}^{r-1} (-1)^a 
  \zeta^{-\frac{(a+3)a}{2}-2(a+b+1)b} E^a F^{(a)} T_b, \label{E:ribbon_Habiro_prime} \\*
  v^{-1}
  &= \sum_{a,b=0}^{r-1}
  \zeta^{\frac{(a+3)a}{2}-2(a-b+1)b} F^{(a)} E^a T_b \label{E:inverse_ribbon_Habiro} \\*
  &= \sum_{a,b=0}^{r-1}
  \zeta^{\frac{(a+3)a}{2}+2(a+b+1)b} E^a F^{(a)} T_b. \label{E:inverse_ribbon_Habiro_prime}
 \end{align}
\end{lemma}

\begin{proof}
 On one hand, Equation~\eqref{E:ribbon_Habiro} can be obtained by computing the Drinfeld element
 \begin{align*}
  u &= \sum_{a,b=0}^{r-1}
  \zeta^{\frac{a(a-1)}{2}} S(K^{-b} F^{(a)}) T_b E^a 
  = \sum_{a,b=0}^{r-1}
  (-1)^a \zeta^{-\frac{(a+3)a}{2}} F^{(a)} K^{a+b} T_b E^a 
  = \sum_{a,b=0}^{r-1}
  (-1)^a \zeta^{-\frac{(a+3)a}{2}-2(a+b)b} F^{(a)} T_b E^a \\*
  &= \sum_{a,b=0}^{r-1}
  (-1)^a \zeta^{-\frac{(a+3)a}{2}-2(a+b)b} F^{(a)} E^a T_{a+b},
 \end{align*}
 where the first equality follows from
 \[
  S(K^{-b} F^{(a)}) = (-1)^a \zeta^{-(a+1)a} F^{(a)} K^{a+b}.
 \]
 Since $v = uK^{-1}$, we obtain
 \begin{align*}
  v &= \sum_{a,b=0}^{r-1}
  (-1)^a \zeta^{-\frac{(a+3)a}{2}-2(a+b)b} F^{(a)} E^a T_{a+b} K^{-1} 
  = \sum_{a,b=0}^{r-1}
  (-1)^a \zeta^{-\frac{(a+3)a}{2}-2(a+b)(b-1)} F^{(a)} E^a T_{a+b} \\*
  &= \sum_{a,b=0}^{r-1}
  (-1)^a \zeta^{-\frac{(a+3)a}{2}+2(a-b+1)b} F^{(a)} E^a T_b.
 \end{align*}
 Equation~\eqref{E:ribbon_Habiro_prime} follows from the identity $S(v) = v$, together with Equation~\eqref{E:antipodes}.

 Equation~\eqref{E:inverse_ribbon_Habiro} can be obtained by computing the Drinfeld element
 \begin{align*}
  u^{-1} &= \sum_{a,b=0}^{r-1} 
  \zeta^{\frac{a(a-1)}{2}} T_b F^{(a)} S^2(K^{-b} E^a) 
  = \sum_{a,b=0}^{r-1}
  \zeta^{\frac{(a+3)a}{2}} T_b F^{(a)} K^{-b} E^a 
  = \sum_{a,b=0}^{r-1}
  \zeta^{\frac{(a+3)a}{2}-2(a-b)b} F^{(a)} T_{-a+b} E^a \\*
  &= \sum_{a,b=0}^{r-1}
  \zeta^{\frac{(a+3)a}{2}-2(a-b)b} F^{(a)} E^a T_b,
 \end{align*}
 where the first equality follows from
 \[
  S^2(K^{-b} E^a) = \zeta^{2a} K^{-b} E^a.
 \]
 Since $v^{-1} = u^{-1}K$, we obtain
 \begin{align*}
  v^{-1} &= \sum_{a,b=0}^{r-1}
  \zeta^{\frac{(a+3)a}{2}-2(a-b)b} F^{(a)} E^a T_b K 
  = \sum_{a,b=0}^{r-1}
  \zeta^{\frac{(a+3)a}{2}-2(a-b+1)b} F^{(a)} E^a T_b.
 \end{align*}
 Equation~\eqref{E:inverse_ribbon_Habiro_prime} follows from the identity $S(v^{-1}) = v^{-1}$, together with Equation~\eqref{E:antipodes}.
\end{proof}

Finally, we provide formulas for commutators.

\begin{lemma}\label{L:commutators}
 For all integers $0 \leqs a,b,m \leqs r-1$, we have
 \begin{align}
  F^{(a)} E^b T_c &= \sum_{k=0}^{\min \{ a,b \}} \sqbinom{b}{k}_\zeta \{ a-b+2c;k \}_\zeta E^{b-k} F^{(a-k)} T_c, \label{E:commutator_Habiro_T_right} \\*
  T_a F^{(b)} E^c &= \sum_{k=0}^{\min \{ b,c \}} \sqbinom{c}{k}_\zeta \{ 2a-b+c;k \}_\zeta T_a E^{c-k} F^{(b-k)}, \label{E:commutator_Habiro_T_left}
 \end{align}
 where $\{ n;k \}_\zeta := \prod_{j=0}^{k-1} \{ n-j \}_\zeta$ for all integers $0 \leqs k \leqs n$.
\end{lemma}

\begin{proof}
 If we apply the algebra isomorphism $\omega : \fraku_\zeta \to \fraku_\zeta$ defined by
 \begin{align*}
  \omega(E) &= F, &
  \omega(F) &= E, &
  \omega(K) &= K^{-1},
 \end{align*}
 to \cite[Equation~(5), Section~3.1.1]{KS97}, we obtain 
 \begin{align*}
  F^a E^b 
  &= \sum_{k=0}^{\min \{ a,b \}} \sqbinom{a}{k}_\zeta \sqbinom{b}{k}_\zeta [k]_\zeta! E^{b-k} F^{a-k} \left( \prod_{j=0}^{k-1} \frac{\zeta^{a-b-j}K^{-1} - \zeta^{-a+b+j}K}{\{ 1 \}_\zeta} \right).
 \end{align*}
 which means that
 \begin{align*}
  &F^{(a)} E^b T_c
  = \sum_{k=0}^{\min \{ a,b \}} \sqbinom{b}{k}_\zeta E^{b-k} F^{(a-k)} \left( \prod_{j=0}^{k-1} \zeta^{a-b-j}K^{-1} - \zeta^{-a+b+j}K \right) T_c \\*
  &\hspace*{\parindent} = \sum_{k=0}^{\min \{ a,b \}} \sqbinom{b}{k}_\zeta \{ a-b+2c;k \}_\zeta E^{b-k} F^{(a-k)} T_c.
 \end{align*} 
 This implies
 \begin{align*}
  &T_a F^{(b)} E^c 
  = F^{(b)} E^c T_{a-b+c} 
  = \sum_{k=0}^{\min \{ b,c \}} \sqbinom{c}{k}_\zeta \{ 2a-b+c;k \}_\zeta E^{c-k} F^{(b-k)} T_{a-b+c} \\*
  &\hspace*{\parindent} = \sum_{k=0}^{\min \{ b,c \}} \sqbinom{c}{k}_\zeta \{ 2a-b+c;k \}_\zeta T_a E^{c-k} F^{(b-k)}. \qedhere
 \end{align*}
\end{proof}

\subsection{A formula by Murakami}

The following formula holds in $\Z[q,q^{-1}]$, where $q$ is a generic parameter, and is equivalent to \cite[Lemma~A]{Mu08}. We prove it using quantum calculus, but the reader might guess a homological diagrammatic proof using computation rules in punctured discs (see \cite[Example~4.6 \& Corollary~4.11]{M20}).

\begin{lemma}\label{L:Murakami}
 For all $n \in \Z$ and $k \in \N$ we have
 \begin{equation}
  \{ n;k \}_q = \sum_{\ell=0}^k (-1)^{k+\ell} \sqbinom{k}{\ell}_q q^{\frac{k(k-1)}{2}-n(k-2\ell)-(k-1)\ell}. \label{E:Murakami}
 \end{equation}
\end{lemma}

\begin{proof}
 Equation~\eqref{E:Murakami} can be proven by induction on $k \in \N$. For $k = 0$ we have
 \[
  \{ n;0 \}_q = 1 = \sum_{\ell=0}^0 (-1)^\ell \sqbinom{0}{\ell}_q q^{2n\ell+\ell}.
 \]
 For $k > 0$ we have
 \begin{align*}
  &\{ n;k \}_q = \{ n;k-1 \}_q \{ n-k+1 \}_q 
  = \sum_{\ell=0}^{k-1} (-1)^{k+\ell+1} \sqbinom{k-1}{\ell}_q \{ n-k+1 \}_q q^{\frac{(k-1)(k-2)}{2}-n(k-2\ell-1)-(k-2)\ell} \\*
  &\hspace*{\parindent} = \sum_{\ell=0}^{k-1} (-1)^{k+\ell+1} \sqbinom{k-1}{\ell}_q (q^{n-k+1}-q^{-n+k-1}) q^{\frac{(k-1)(k-2)}{2}-n(k-2\ell-1)-(k-2)\ell} \\*
  &\hspace*{\parindent} = \sum_{\ell=0}^{k-1} (-1)^{k+\ell+1} \sqbinom{k-1}{\ell}_q q^{\frac{(k-1)(k-2)}{2}-n(k-2(\ell+1))-(k-2)(\ell+1)-1} 
  + \sum_{\ell=0}^{k-1} (-1)^{k+\ell} \sqbinom{k-1}{\ell}_q q^{\frac{k(k-1)}{2}-n(k-2\ell)-(k-2)\ell} \\*
  &\hspace*{\parindent} = \sum_{\ell=1}^k (-1)^{k+\ell} \sqbinom{k-1}{l-1}_q q^{\frac{(k-1)(k-2)}{2}-n(k-2\ell)-(k-2)\ell-1} 
  + \sum_{\ell=0}^{k-1} (-1)^{k+\ell} \sqbinom{k-1}{\ell}_q q^{\frac{k(k-1)}{2}-n(k-2\ell)-(k-2)\ell} \\*
  &\hspace*{\parindent} = (-1)^k \sqbinom{k-1}{0}_q q^{\frac{k(k-1)}{2}-nk} 
  + \sum_{\ell=1}^{k-1} (-1)^{k+\ell} \left( \sqbinom{k-1}{l-1}_q q^{-k+\ell} + \sqbinom{k-1}{\ell}_q q^{\ell} \right) q^{\frac{k(k-1)}{2}-n(k-2\ell)-(k-1)\ell} \\*
  &\hspace*{2\parindent} + \sqbinom{k-1}{k-1}_q q^{\frac{(k-1)(k-2)}{2}+nk-(k-2)k-1}.
 \end{align*}
 Now we can use the identity
 \[
  \sqbinom{k}{\ell}_q = \sqbinom{k-1}{l-1}_q q^{-k+\ell} + \sqbinom{k-1}{\ell}_q q^{\ell},
 \]
 compare with \cite[Section~0.2,~Equation~(1)]{J96}. This yields
 \begin{align*}
  &\{ n;k \}_q 
  = (-1)^k q^{\frac{k(k-1)}{2}-nk} + \left( \sum_{\ell=1}^{k-1} (-1)^{k+\ell} \sqbinom{k}{\ell}_q q^{\frac{k(k-1)}{2}-n(k-2\ell)-(k-1)\ell} \right) + q^{-\frac{k(k-1)}{2}+nk} \\*
  &\hspace*{\parindent} = \sum_{\ell=0}^k (-1)^{k+\ell} \sqbinom{k}{\ell}_q q^{\frac{k(k-1)}{2}-n(k-2\ell)-(k-1)\ell}. \qedhere
 \end{align*}
\end{proof}

\section{Computations for higher genus surfaces}

In this appendix, we provide detailed computations for the homological and quantum actions of the Dehn twist $\tau_\gamma = \tau_{\gamma_1} \in \Mod(\varSigma_{2,1})$ of Section~\ref{S:homological_representations_of_mcg}, whose results were announced in Sections~\ref{S:homological_mcg_computation} and \ref{S:quantum_mcg_computation}, respectively.

\subsection{Homological computation}\label{A:ugly_homological_computation}

We start with the homological action.

\begin{proof}[Proof of Lemma~\ref{L:homological_gamma}]
 Using Proposition~\ref{P:computation_rules} and Remark~\ref{R:computation_rules}, we obtain
 \begin{align*}
  &\rho^V_2(\tau_\gamma) \left( \basis(a_1,b_1;a_2,b_2) \otimes \bfv_{(c_1,c_2)} \right) \\
%
%
  &= \pic{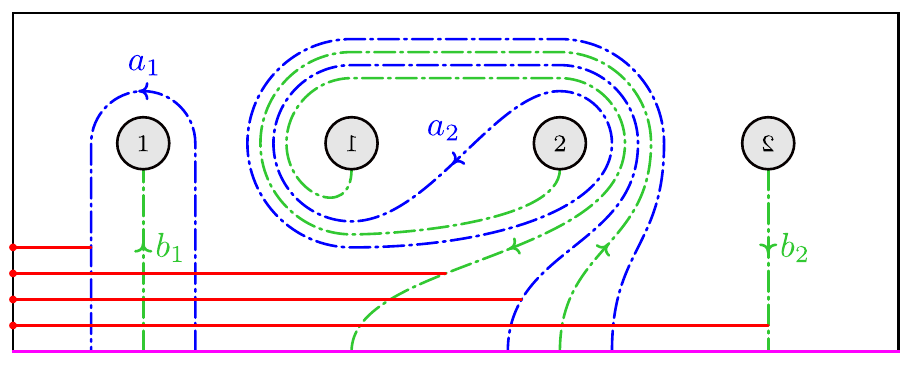} \otimes \psi_g^V(\twist{\gamma}{}) \bfv_{(c_1,c_2)} \\
%
%
  &\tCeq \sum_{k_2=0}^{a_2} \pic{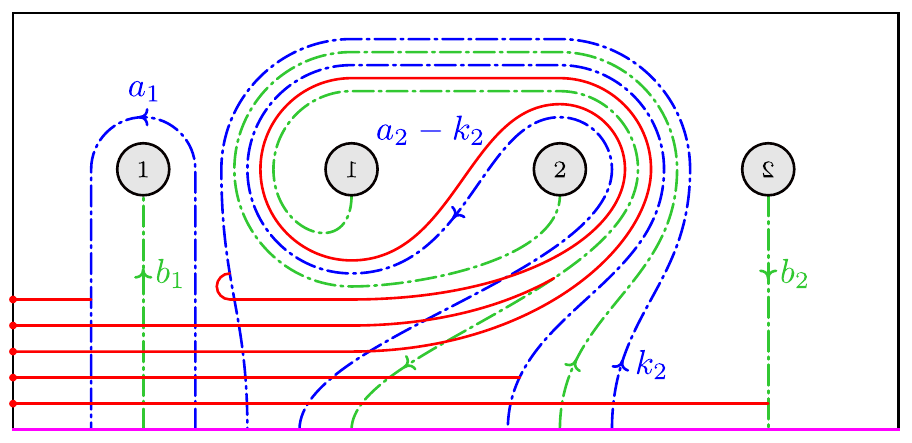} 
  \otimes \left( \sum_{\ell=0}^{r-1} \zeta^{-2(\ell+1)\ell} A_1^{-\ell} A_2^\ell \right) \bfv_{(c_1,c_2)} \\
%
%
  &\stackrel{\substack{\mathclap{\eqref{E:permutation}} \\ \mathclap{\eqref{E:braid}}}}{=} \sum_{k_2=0}^{a_2} \pic{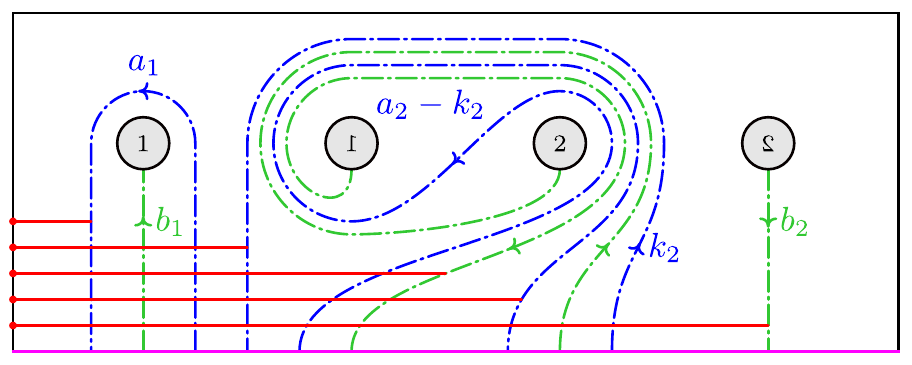} 
  \otimes \left( \sum_{\ell=0}^{r-1} \zeta^{-2(\ell+1)\ell -4c_1\ell + 4c_2\ell} \right) \\*
  &\hspace*{\parindent} \zeta^{2b_1k_2} (A_2 B_1^{-1} A_1 B_1 A_2^{-1})^{k_2} \bfv_{(c_1,c_2)} \\
%
%
  &\tCeq \fraki^{\frac{r-1}{2}} \sqrt{r} \zeta^{2(c_1-c_2+1)(c_1-c_2)+\frac{r+1}{2}} \sum_{k_2=0}^{a_2} \sum_{j_2=0}^{b_2} 
  \pic{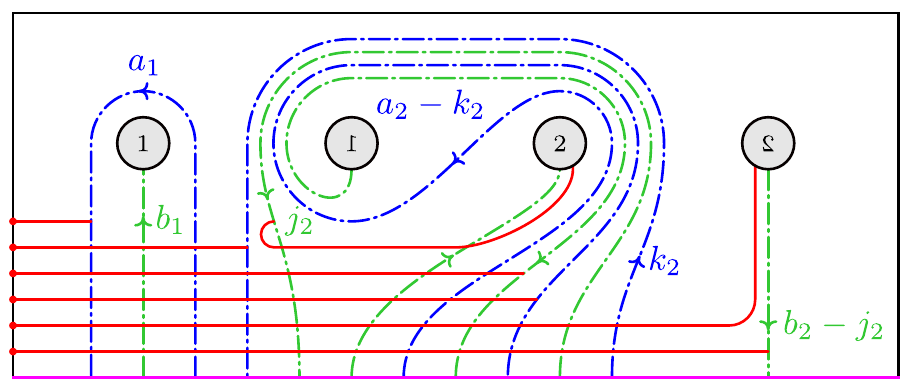} \\*
  &\hspace*{\parindent} \otimes \zeta^{2b_1k_2+4(c_1+1)k_2} \bfv_{(c_1,c_2)} \\
%
%
  &\stackrel{\substack{\mathclap{\eqref{E:permutation}} \\ \mathclap{\eqref{E:braid}}}}{\propto} \zeta^{2(c_1-c_2+1)(c_1-c_2)} \sum_{k_2=0}^{a_2} \sum_{j_2=0}^{b_2} 
  \pic{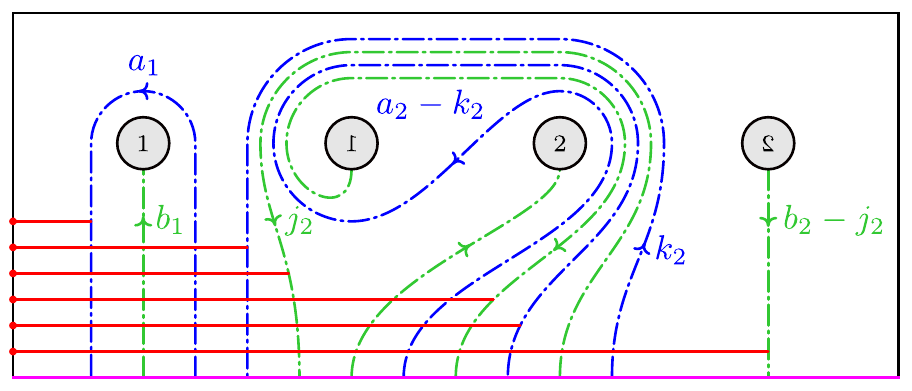} \\*
  &\hspace*{\parindent} \otimes \zeta^{2(b_1+2c_1+2)k_2+2(b_1+a_2-k_2)j_2} B_2^{j_2} \bfv_{(c_1,c_2)} \\
%
%
  &\Feq \zeta^{2(c_1-c_2+1)(c_1-c_2)} \sum_{k_2=0}^{a_2} \sum_{j_2=0}^{b_2} 
  \pic{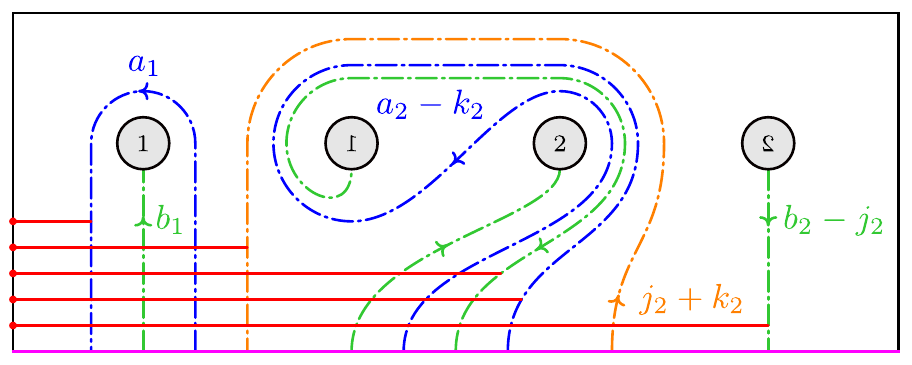} \\*
  &\hspace*{\parindent} \otimes \sqbinom{j_2+k_2}{j_2}_\zeta \zeta^{j_2k_2-2j_2k_2+2(b_1+a_2)j_2+2(b_1+2c_1+2)k_2} \bfv_{(c_1,c_2+j_2)} \\
%
%
  &\Ceq \zeta^{2(c_1-c_2+1)(c_1-c_2)} \sum_{k_2=0}^{a_2} \sum_{j_2=0}^{b_2} \sum_{i_2=k_2}^{a_2} 
  \pic{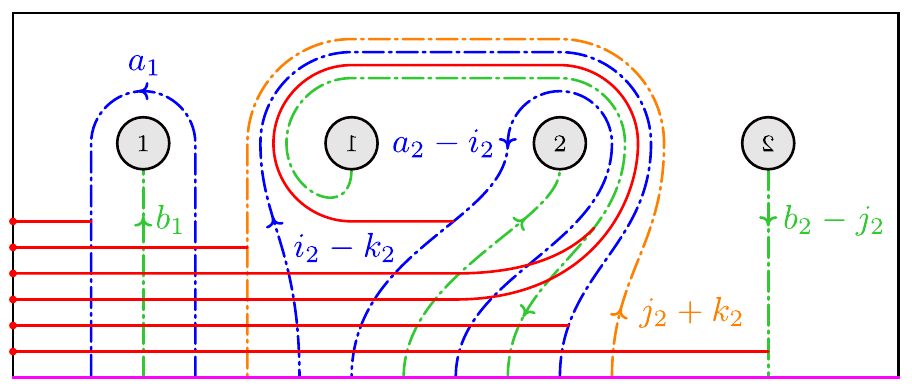} \\*
  &\hspace*{\parindent} \otimes \sqbinom{j_2+k_2}{j_2}_\zeta \zeta^{-j_2k_2+2(b_1+a_2)j_2+2(b_1+2c_1+2)k_2} \bfv_{(c_1,c_2+j_2)} \\
%
%
  &\stackrel{\substack{\mathclap{\eqref{E:permutation}} \\ \mathclap{\eqref{E:tilde_hook}} \\ \mathclap{\eqref{E:braid}}}}{=} \zeta^{2(c_1-c_2+1)(c_1-c_2)} \sum_{k_2=0}^{a_2} \sum_{j_2=0}^{b_2} \sum_{i_2=k_2}^{a_2} 
  \pic{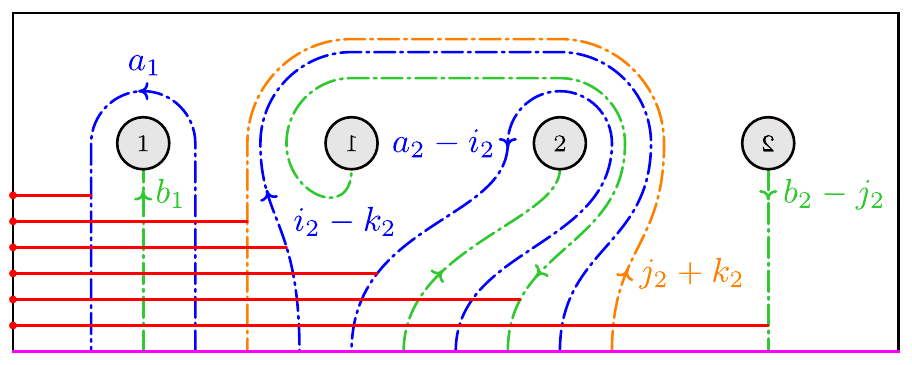} \\*
  &\hspace*{\parindent} \otimes \sqbinom{j_2+k_2}{j_2}_\zeta \zeta^{-j_2k_2+2(b_1+a_2)j_2+2(b_1+2c_1+2)k_2} \\*
  &\hspace*{\parindent} \zeta^{(i_2-k_2)(i_2-k_2-1)+2(a_2-i_2)(a_2-k_2-1)+2b_1(a_2-k_2)} (B_1^{-1} A_1 B_1 A_2^{-1})^{a_2-k_2} \bfv_{(c_1,c_2+j_2)} \\
%
%
  &\stackrel{\substack{\mathclap{\eqref{E:orientation}} \\ \mathclap{\eqref{E:fusion}}}}{=} \zeta^{2a_2(b_1+a_2-1)+2(c_1-c_2+1)(c_1-c_2)} \sum_{k_2=0}^{a_2} \sum_{j_2=0}^{b_2} \sum_{i_2=k_2}^{a_2} (-1)^{i_2+k_2} \\*
  &\hspace*{\parindent} \pic{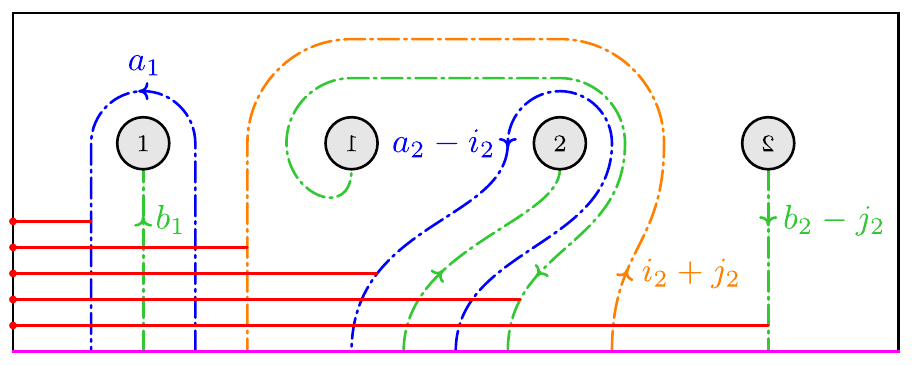} 
  \otimes \sqbinom{i_2+j_2}{j_2,k_2}_\zeta \zeta^{(i_2-k_2)(j_2+k_2)} \\*
  &\hspace*{\parindent} \zeta^{(i_2+1)i_2+(k_2+5)k_2-j_2k_2-2a_2i_2+2(b_1+a_2)j_2+2(2c_1-a_2)k_2
  +4(a_2-k_2)(c_1-c_2-j_2+1)} \bfv_{(c_1,c_2+j_2)} \\
%
%
  &\Ceq \zeta^{2a_2(b_1+a_2-1)+2(c_1-c_2+1)(c_1+2a_2-c_2)} \sum_{k_2=0}^{a_2} \sum_{j_2=0}^{b_2} \sum_{i_2=k_2}^{a_2} \sum_{k_1=0}^{b_1} (-1)^{i_2+k_2} \\*
  &\hspace*{\parindent} \pic{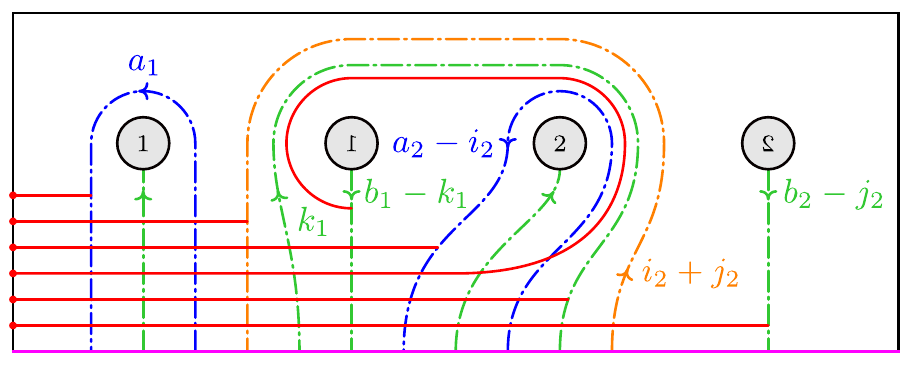} 
  \otimes \sqbinom{i_2+j_2}{j_2,k_2}_\zeta \\*
  &\hspace*{\parindent} \zeta^{(i_2+1)i_2+i_2j_2+i_2k_2-2j_2k_2-2a_2i_2+2(b_1+a_2)j_2-(2a_2-4c_2-1)k_2} \bfv_{(c_1,c_2+j_2)} \\   
%
%
  &\stackrel{\substack{\mathclap{\eqref{E:permutation}} \\ \mathclap{\eqref{E:tilde_hook}} \\ \mathclap{\eqref{E:braid}}}}{=} \zeta^{2a_2(b_1+a_2-1)+2(c_1-c_2+1)(c_1+2a_2-c_2)} \sum_{k_2=0}^{a_2} \sum_{j_2=0}^{b_2} \sum_{i_2=k_2}^{a_2} \sum_{k_1=0}^{b_1} (-1)^{i_2+k_2} \\*
  &\hspace*{\parindent} \pic{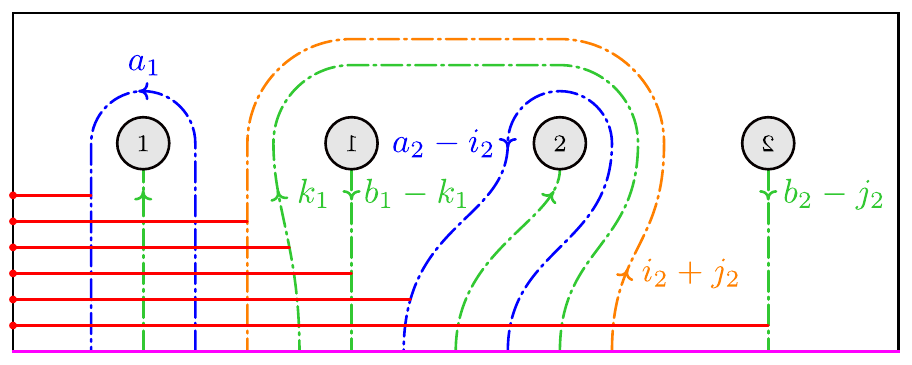} 
  \otimes \sqbinom{i_2+j_2}{j_2,k_2}_\zeta \\*
  &\hspace*{\parindent} \zeta^{(i_2+1)i_2+i_2j_2+i_2k_2-2j_2k_2-2a_2i_2+2(b_1+a_2)j_2-(2a_2-4c_2-1)k_2} \\*
  &\hspace*{\parindent} \zeta^{k_1(k_1-1)+2(b_1-k_1)(b_1-1)+2b_1(a_2-i_2)} (B_1^{-1} A_1 B_1 A_2^{-1})^{b_1} \bfv_{(c_1,c_2+j_2)} \\
%
%
  &\stackrel{\substack{\mathclap{\eqref{E:orientation}} \\ \mathclap{\eqref{E:fusion}}}}{=} \zeta^{2(b_1+a_2)(b_1+a_2-1)+2(c_1-c_2+1)(c_1+2a_2-c_2)} \sum_{k_2=0}^{a_2} \sum_{j_2=0}^{b_2} \sum_{i_2=k_2}^{a_2} \sum_{k_1=0}^{b_1} (-1)^{k_1+i_2+k_2} \\*
  &\hspace*{\parindent} \pic{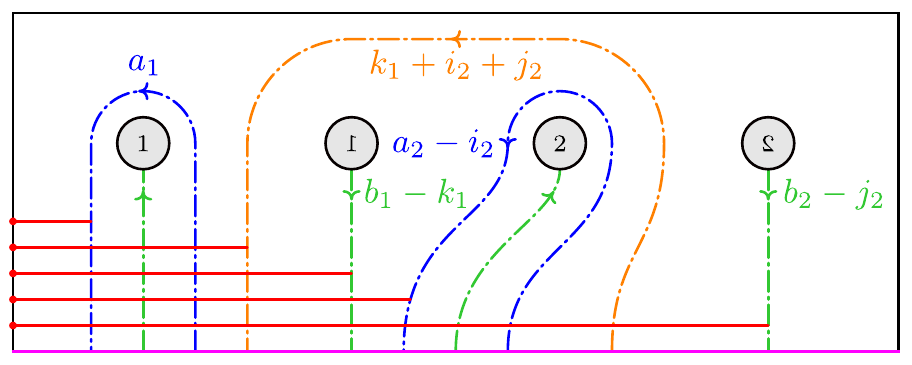} 
  \otimes \sqbinom{k_1+i_2+j_2}{k_1,j_2,k_2}_\zeta \zeta^{k_1(i_2+j_2)} \\*
  &\hspace*{\parindent} \zeta^{(k_1+1)k_1+(i_2+1)i_2+i_2j_2+i_2k_2+2j_2k_2
  -2b_1k_1-2(b_1+a_2)i_2+2(b_1-a_2)j_2-(2a_2-4c_2-1)k_2+4b_1(c_1-c_2-j_2+1)} \bfv_{(c_1,c_2+j_2)} \\
%
%
  &\Ceq \zeta^{2(b_1+c_1+a_2-c_2+1)(b_1+c_1+a_2-c_2)} \sum_{k_2=0}^{a_2} \sum_{j_2=0}^{b_2} \sum_{i_2=k_2}^{a_2} \sum_{k_1=0}^{b_1} \sum_{\ell=-k_1}^{i_2+j_2} (-1)^{k_1+i_2+k_2} \\*
  &\hspace*{\parindent} \pic{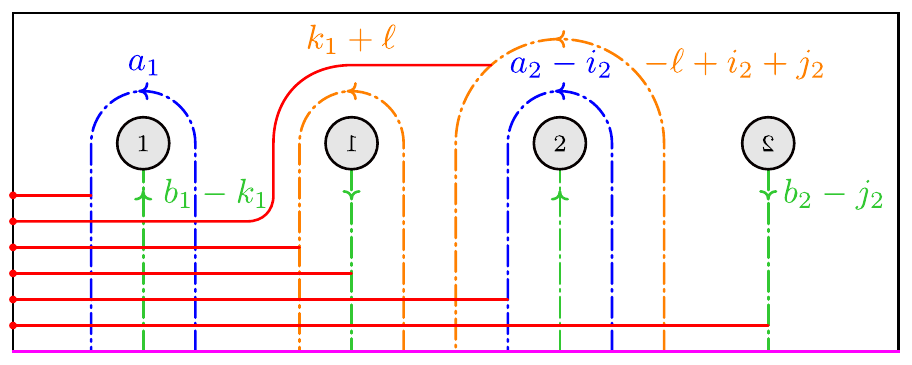} 
  \otimes \sqbinom{k_1+i_2+j_2}{k_1,j_2,k_2}_\zeta \\*
  &\hspace*{\parindent} \zeta^{(k_1+1)k_1+(i_2+1)i_2+k_1i_2+k_1j_2+i_2j_2+i_2k_2+2j_2k_2
  -2b_1k_1-2(b_1+a_2)i_2-2(b_1+a_2)j_2-(2a_2-4c_2-1)k_2} \bfv_{(c_1,c_2+j_2)} \\
%
%
  &\stackrel{\substack{\mathclap{\eqref{E:braid}} \\ \mathclap{\eqref{E:permutation}}}}{=} \zeta^{2(b_1+c_1+a_2-c_2+1)(b_1+c_1+a_2-c_2)} \sum_{k_2=0}^{a_2} \sum_{j_2=0}^{b_2} \sum_{i_2=k_2}^{a_2} \sum_{k_1=0}^{b_1} \sum_{\ell=-k_1}^{i_2+j_2} (-1)^{k_1+i_2+k_2} \\*
  &\hspace*{\parindent} \pic{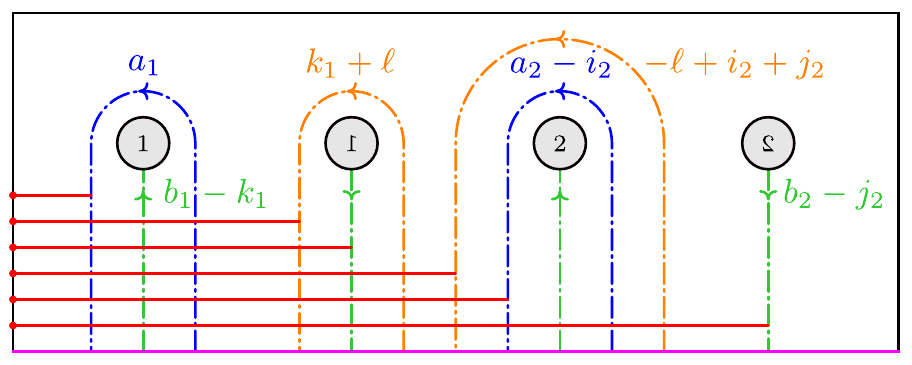} 
  \otimes \sqbinom{k_1+i_2+j_2}{k_1,j_2,k_2}_\zeta \\*
  &\hspace*{\parindent} \zeta^{(k_1+1)k_1+(i_2+1)i_2+k_1i_2+k_1j_2+i_2j_2+i_2k_2+2j_2k_2-2b_1k_1-2(b_1+a_2)i_2-2(b_1+a_2)j_2-(2a_2-4c_2-1)k_2} \\*
  &\hspace*{\parindent} \zeta^{2(b_1+\ell)(\ell-i_2-j_2)} 
  (B_1^{-1} A_1^{-1} B_1)^{-\ell+i_2+j_2} \bfv_{(c_1,c_2+j_2)} \\
%
%
  &\stackrel{\substack{\mathclap{\eqref{E:permutation}} \\ \mathclap{\eqref{E:braid}} \\ \mathclap{\eqref{E:fusion}}}}{=}  \zeta^{2(b_1+c_1+a_2-c_2+1)(b_1+c_1+a_2-c_2)} \sum_{k_2=0}^{a_2} \sum_{j_2=0}^{b_2} \sum_{i_2=k_2}^{a_2} \sum_{k_1=0}^{b_1} \sum_{\ell=-k_1}^{i_2+j_2} (-1)^{k_1+i_2+k_2} \\*
  &\hspace*{\parindent} \pic{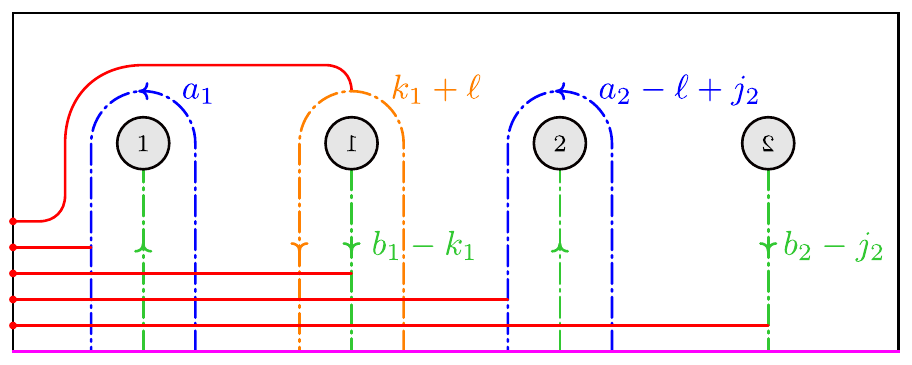} 
  \otimes \sqbinom{k_1+i_2+j_2}{k_1,j_2,k_2}_\zeta \sqbinom{a_2-\ell+j_2}{a_2-i_2}_\zeta 
  \zeta^{-(a_2-i_2)(\ell-i_2-j_2)} \\*
  &\hspace*{\parindent} \zeta^{(k_1+1)k_1+2\ell^2+(i_2+1)i_2+k_1i_2+k_1j_2-2\ell i_2-2\ell j_2+i_2j_2+i_2k_2+2j_2k_2
  -2b_1k_1+2b_1\ell-2(2b_1+a_2)i_2-2(2b_1+a_2)j_2-(2a_2-4c_2-1)k_2} \\*
  &\hspace*{\parindent} \zeta^{2a_1(k_1+\ell)+4(c_1+1)(\ell-i_2-j_2)} A_1^{-k_1-\ell} \bfv_{(c_1,c_2+j_2)} \\
%
%
  &\stackrel{\substack{\mathclap{\eqref{E:orientation}} \\ \mathclap{\eqref{E:F_right_regular}}}}{=} \zeta^{2(b_1+c_1+a_2-c_2+1)(b_1+c_1+a_2-c_2)} \sum_{k_2=0}^{a_2} \sum_{j_2=0}^{b_2} \sum_{i_2=k_2}^{a_2} \sum_{k_1=0}^{b_1} \sum_{\ell=-k_1}^{i_2+j_2} \sum_{j_1=0}^{k_1+\ell} \sum_{i_1=0}^{k_1-j_1+\ell} (-1)^{j_1+\ell+i_2+k_2} \\*
  &\hspace*{\parindent} \basis(a_1+i_1,b_1-i_1+\ell;a_2-\ell+j_2,b_2-j_2) 
  \otimes \sqbinom{k_1+i_2+j_2}{k_1,j_2,k_2}_\zeta \sqbinom{a_2-\ell+j_2}{a_2-i_2}_\zeta \sqbinom{a_1+i_1}{a_1}_\zeta \sqbinom{b_1-i_1+\ell}{b_1-k_1,j_1}_\zeta \\*
  &\hspace*{\parindent} \zeta^{-(k_1+\ell+3)(k_1+\ell)-(2a_1+b_1-k_1)(k_1+\ell)+i_1j_1+(a_1+b_1-k_1)i_1+(2b_1-k_1+\ell+3)j_1} \\*
  &\hspace*{\parindent} \zeta^{(k_1+1)k_1+2(\ell+2)\ell+k_1i_2+k_1j_2-\ell i_2-2\ell j_2+i_2k_2+2j_2k_2} \\*
  &\hspace*{\parindent} \zeta^{2(a_1-b_1)k_1+(2a_1+2b_1+4c_1-a_2)\ell-(4b_1+4c_1+a_2+3)i_2-(4b_1+4c_1+a_2+4)j_2-(2a_2-4c_2-1)k_2} \\*
  &\hspace*{\parindent} \zeta^{-4c_1(k_1+\ell)} B_1^{i_1} A_1^{j_1} \bfv_{(c_1,c_2+j_2)} \\ 
%
%
  &= \zeta^{2(b_1+c_1+a_2-c_2+1)(b_1+c_1+a_2-c_2)} \sum_{k_2=0}^{a_2} \sum_{j_2=0}^{b_2} \sum_{i_2=k_2}^{a_2} \sum_{k_1=0}^{b_1} \sum_{\ell=-k_1}^{i_2+j_2} \sum_{j_1=0}^{k_1+\ell} \sum_{i_1=0}^{k_1-j_1+\ell} (-1)^{j_1+\ell+i_2+k_2} \\*
  &\hspace*{\parindent} \basis(a_1+i_1,b_1-i_1+\ell;a_2-\ell+j_2,b_2-j_2) 
  \otimes \sqbinom{k_1+i_2+j_2}{k_1,j_2,k_2}_\zeta \sqbinom{a_2-\ell+j_2}{a_2-i_2}_\zeta \sqbinom{a_1+i_1}{a_1}_\zeta \sqbinom{b_1-i_1+\ell}{b_1-k_1,j_1}_\zeta \\*
  &\hspace*{\parindent} \zeta^{k_1(k_1-2)+(\ell+1)l+i_1j_1-i_1k_1-j_1k_1+j_1\ell-k_1\ell+k_1i_2+k_1j_2-\ell i_2-2\ell j_2+i_2k_2+2j_2k_2} \\*
  &\hspace*{\parindent} \zeta^{(a_1+b_1)i_1+(2b_1+3)j_1-(3b_1+4c_1)k_1+(b_1-a_2)\ell-(4b_1+4c_1+a_2+3)i_2-(4b_1+4c_1+a_2+4)j_2-(2a_2-4c_2-1)k_2+4c_1j_1} \bfv_{(c_1+i_1,c_2+j_2)}. \qedhere
 \end{align*} 
\end{proof}

\subsection{Quantum computation}\label{A:ugly_quantum_computation}

We finish with the quantum action.

\begin{proof}[Proof of Lemma~\ref{L:quantum_gamma}]
 First of all, using Lemmas~\ref{L:coproducts_antipodes} and \ref{L:ribbon}, we obtain
 \begin{align*}
  &S(v^{-1}_{(1)}) \otimes v^{-1}_{(2)}
  \stackrel{\substack{\mathclap{\eqref{E:inverse_ribbon_Habiro}} \\ \mathclap{\eqref{E:coproducts}}}}{=} \sum_{a,b,c=0}^{r-1} \sum_{i,j=0}^a 
  \sqbinom{a}{j}_\zeta \zeta^{\frac{(a+3)a}{2}-2(a-b+1)b+(a+2b)i+aj-2c(i+j)-(i+j)^2} 
  S(F^{(a-i)} E^j T_{b-c}) \otimes F^{(i)} E^{a-j} T_c \\
  &\hspace*{\parindent} \stackrel{\mathclap{\eqref{E:antipodes}}}{=} \sum_{a,b,c=0}^{r-1} \sum_{i,j=0}^a 
  \sqbinom{a}{j}_\zeta \zeta^{\frac{(a+3)a}{2}-2(a-b+1)b+(a+2b)i+aj-2c(i+j)-(i+j)^2} \\*
  &\hspace*{2\parindent} (-1)^{a+i+j} \zeta^{(a+2b-2c-i-j-1)(a-i-j)} 
  T_{-b+c} E^j F^{(a-i)} \otimes F^{(i)} E^{a-j} T_c \\
  &\hspace*{\parindent} = \sum_{a,b,c=0}^{r-1} \sum_{i,j=0}^a (-1)^{a+i+j} 
  \sqbinom{a}{j}_\zeta \zeta^{\frac{(3a+1)a}{2}+2b(b-1)-2ac-(a-1)i-(a+2b-1)j} 
  T_{-b+c} E^j F^{(a-i)} \otimes F^{(i)} E^{a-j} T_c.
 \end{align*}
 Then, using Proposition~\ref{P:quantum_action_of_mcg_generators}, we obtain
 \begin{align*}
  &\rho^{\fraku_\zeta}_2(\tau_\gamma) \left( E^{\ell_1} T_{m_1} F^{(n_1)} \otimes E^{\ell_2} T_{m_2} F^{(n_2)} \right)
  = E^{\ell_1} T_{m_1} F^{(n_1)} S(v_{(1)}) \otimes v_{(2)} E^{\ell_2} T_{m_2} F^{(n_2)} \\*
  &\hspace*{\parindent} = \sum_{a,b,c=0}^{r-1} \sum_{i,j=0}^a (-1)^{a+i+j} 
  \sqbinom{a}{j}_\zeta \zeta^{\frac{(3a+1)a}{2}+2b(b-1)-2ac-(a-1)i-(a+2b-1)j} \\*
  &\hspace*{2\parindent} E^{\ell_1} T_{m_1} F^{(n_1)} T_{-b+c} E^j F^{(a-i)} \otimes F^{(i)} E^{a-j} T_c E^{\ell_2} T_{m_2} F^{(n_2)} \\
  &\hspace*{\parindent} = \sum_{a,b,c=0}^{r-1} \sum_{i,j=0}^a (-1)^{a+i+j} 
  \sqbinom{a}{j}_\zeta \zeta^{\frac{(3a+1)a}{2}+2b(b-1)-2ac-(a-1)i-(a+2b-1)j} \\*
  &\hspace*{2\parindent} E^{\ell_1} T_{m_1} T_{n_1-b+c} F^{(n_1)} E^j F^{(a-i)} \otimes F^{(i)} E^{\ell_2+a-j} T_{\ell_2+c} T_{m_2} F^{(n_2)} \\
  &\hspace*{\parindent} = \sum_{a,b,c=0}^{r-1} \sum_{i,j=0}^a (-1)^{a+i+j} 
  \sqbinom{a}{j}_\zeta \zeta^{\frac{(3a+1)a}{2}+2b(b-1)
  -2ac-(a-1)i-(a+2b-1)j} \delta_{b,-m_1+n_1-\ell_2+m_2} \delta_{c,-\ell_2+m_2} \\*
  &\hspace*{2\parindent} E^{\ell_1} T_{m_1} F^{(n_1)} E^j F^{(a-i)} \otimes F^{(i)} E^{\ell_2+a-j} T_{m_2} F^{(n_2)} \\
  &\hspace*{\parindent} = \sum_{a=0}^{r-1} \sum_{i,j=0}^a (-1)^{a+i+j} 
  \sqbinom{a}{j}_\zeta \zeta^{\frac{(3a+1)a}{2}+2(m_1-n_1+\ell_2-m_2+1)(m_1-n_1+\ell_2-m_2)
  +2(\ell_2-m_2)a-(a-1)i+(2m_1-2n_1+2\ell_2-2m_2-a+1)j} \\*
  &\hspace*{2\parindent} E^{\ell_1} T_{m_1} F^{(n_1)} E^j F^{(a-i)} \otimes F^{(i)} E^{\ell_2+a-j} T_{m_2} F^{(n_2)}.
 \end{align*}
 Using Lemma~\ref{L:commutators}, we obtain
 \begin{align*}
  &\rho^{\fraku_\zeta}_2(\tau_\gamma) \left( E^{\ell_1} T_{m_1} F^{(n_1)} \otimes E^{\ell_2} T_{m_2} F^{(n_2)} \right) \\
  &\hspace*{\parindent} \stackrel{\substack{\mathclap{\eqref{E:commutator_Habiro_T_right}} \\ \mathclap{\eqref{E:commutator_Habiro_T_left}}}}{=} \zeta^{2(m_1-n_1+\ell_2-m_2+1)(m_1-n_1+\ell_2-m_2)} 
  \sum_{a=0}^{r-1} \sum_{i,j=0}^a \sum_{k=0}^j \sum_{h=0}^i (-1)^{a+i+j} \\*
  &\hspace*{2\parindent} \sqbinom{a}{j}_\zeta \sqbinom{j}{k}_\zeta \sqbinom{\ell_2+a-j}{h}_\zeta 
  \{ 2m_1-n_1+j;k \}_\zeta \{ -\ell_2+2m_2-a+i+j;h \}_\zeta \\*
  &\hspace*{2\parindent} \zeta^{\frac{(3a+1)a}{2}+2(\ell_2-m_2)a-(a-1)i+(2m_1-2n_1+2\ell_2-2m_2-a+1)j} 
  E^{\ell_1} T_{m_1} E^{j-k} F^{(n_1-k)} F^{(a-i)} \otimes E^{\ell_2+a-j-h} F^{(i-h)} T_{m_2} F^{(n_2)} \\
  &\hspace*{\parindent} = \zeta^{2(m_1-n_1+\ell_2-m_2+1)(m_1-n_1+\ell_2-m_2)} 
  \sum_{a=0}^{r-1} \sum_{i,j=0}^a \sum_{k=0}^j\sum_{h=0}^i (-1)^{a+i+j} \\*
  &\hspace*{2\parindent} \sqbinom{a}{j-k,k}_\zeta \sqbinom{\ell_2+a-j}{h}_\zeta \sqbinom{n_1+a-i-k}{n_1-k}_\zeta 
  \{ 2m_1-n_1+j;k \}_\zeta \{ -\ell_2+2m_2-a+i+j;h \}_\zeta \\*
  &\hspace*{2\parindent} \zeta^{\frac{(3a+1)a}{2}+2(\ell_2-m_2)a-(a-1)i+(2m_1-2n_1+2\ell_2-2m_2-a+1)j} \\*
  &\hspace*{2\parindent} E^{\ell_1+j-k} T_{m_1+j-k} F^{(n_1+a-i-k)} \otimes E^{\ell_2+a-j-h} T_{m_2+i-h} F^{(i-h)} F^{(n_2)} \\
  &\hspace*{\parindent} = \zeta^{2(m_1-n_1+\ell_2-m_2+1)(m_1-n_1+\ell_2-m_2)} 
  \sum_{a=0}^{r-1} \sum_{i,j=0}^a \sum_{k=0}^j \sum_{h=0}^i (-1)^{a+i+j} \\*
  &\hspace*{2\parindent} \sqbinom{a}{j-k,k}_\zeta \sqbinom{\ell_2+a-j}{h}_\zeta \sqbinom{n_1+a-i-k}{n_1-k}_\zeta \sqbinom{n_2+i-h}{n_2}_\zeta 
  \{ 2m_1-n_1+j;k \}_\zeta \{ -\ell_2+2m_2-a+i+j;h \}_\zeta \\*
  &\hspace*{2\parindent} \zeta^{\frac{(3a+1)a}{2}+2(\ell_2-m_2)a-(a-1)i+(2m_1-2n_1+2\ell_2-2m_2-a+1)j} \\*
  &\hspace*{2\parindent} E^{\ell_1+j-k} T_{m_1+j-k} F^{(n_1+a-i-k)} \otimes E^{\ell_2+a-j-h} T_{m_2+i-h} F^{(n_2+i-h)}.
 \end{align*}
 If we change variables by setting $b=-a+i+j$, $i_1=k$, $j_1=j-k$, $i_2=i-h$, and $k_2=a-j$, we obtain
 \begin{align*}
  &\rho^{\fraku_\zeta}_2(\tau_\gamma) \left( E^{\ell_1} T_{m_1} F^{(n_1)} \otimes E^{\ell_2} T_{m_2} F^{(n_2)} \right) \\
  &\hspace*{\parindent} = \zeta^{2(m_1-n_1+\ell_2-m_2+1)(m_1-n_1+\ell_2-m_2)} 
  \sum_{j_1=0}^{r-1} \sum_{i_1=0}^{r-j_1-1} \sum_{k_2=0}^{r-i_1-j_1-1} \sum_{b=-k_2}^{i_1+j_1} \sum_{i_2=0}^{b+k_2} (-1)^b \\*
  &\hspace*{2\parindent} \sqbinom{i_1+j_1+k_2}{i_1,j_1}_\zeta \sqbinom{\ell_2+k_2}{b-i_2+k_2}_\zeta \sqbinom{n_1-b+j_1}{n_1-i_1}_\zeta \sqbinom{n_2+i_2}{n_2}_\zeta \\*
  &\hspace*{2\parindent} \{ 2m_1-n_1+i_1+j_1;i_1 \}_\zeta \{ -\ell_2+2m_2+b;b-i_2+k_2 \}_\zeta \\*
  &\hspace*{2\parindent} \zeta^{\frac{(3i_1+3j_1+3k_2+1)(i_1+j_1+k_2)}{2}+2(\ell_2-m_2)(i_1+j_1+k_2)-(i_1+j_1+k_2-1)(b+k_2)
  +(2m_1-2n_1+2\ell_2-2m_2-i_1-j_1-k_2+1)(i_1+j_1)} \\*
  &\hspace*{2\parindent} E^{\ell_1+j_1} T_{m_1+j_1} F^{(n_1+j_1-b)} \otimes E^{\ell_2-b+i_2} T_{m_2+i_2} F^{(n_2+i_2)} \\
  &\hspace*{\parindent} = \zeta^{2(m_1-n_1+\ell_2-m_2+1)(m_1-n_1+\ell_2-m_2)} 
  \sum_{j_1=0}^{r-1} \sum_{i_1=0}^{r-j_1-1} \sum_{k_2=0}^{r-i_1-j_1-1} \sum_{b=-k_2}^{i_1+j_1} \sum_{i_2=0}^{b+k_2} (-1)^b \\*
  &\hspace*{2\parindent} \sqbinom{i_1+j_1+k_2}{i_1,j_1}_\zeta \sqbinom{\ell_2+k_2}{b-i_2+k_2}_\zeta \sqbinom{n_1-b+j_1}{n_1-i_1}_\zeta \sqbinom{n_2+i_2}{n_2}_\zeta \\*
  &\hspace*{2\parindent} \{ 2m_1-n_1+i_1+j_1;i_1 \}_\zeta \{ -\ell_2+2m_2+b;b-i_2+k_2 \}_\zeta \\*
  &\hspace*{2\parindent} \zeta^{\frac{(i_1+j_1+k_2+3)(i_1+j_1+k_2)}{2}-(i_1+j_1+k_2-1)b+2(m_1-n_1+2(\ell_2-m_2))(i_1+j_1)+2(\ell_2-m_2)k_2} \\*
  &\hspace*{2\parindent} E^{\ell_1+j_1} T_{m_1+j_1} F^{(n_1+j_1-b)} \otimes E^{\ell_2-b+i_2} T_{m_2+i_2} F^{(n_2+i_2)}. \qedhere
 \end{align*}
\end{proof}

\end{document}